\documentclass[a4paper,10pt]{article}
\usepackage{amssymb,amsmath,amsthm}
\usepackage{amsfonts}
\usepackage[english]{babel}
\usepackage{hyperref}  
\usepackage{todonotes}

\usepackage{color}

\numberwithin{equation}{section}

\let\OLDthebibliography\thebibliography
\renewcommand\thebibliography[1]{
  \OLDthebibliography{#1}
  \setlength{\parskip}{0pt}
  \setlength{\itemsep}{0pt plus 0.3ex} }
    
\newenvironment{psmallmatrix}
  {\left(\begin{smallmatrix}}
  {\end{smallmatrix}\right)}
  
\newcommand{\Om}{\Omega}
\newcommand{\la}{\langle}
\newcommand{\ra}{\rangle}

\newcommand{\Op}{\mathrm{Op}\,}
\newcommand{\Lm}{\Lambda}
\newcommand{\h}{\mathtt{h}}  % depth
\newcommand{\Lip}{\mathrm{Lip}}

\newcommand{\tLm}{{\mathtt \Lambda} }

\newtheorem{theorem}{Theorem}[section]
\newtheorem{proposition}[theorem]{Proposition}
\newtheorem{lemma}[theorem]{Lemma}
\newtheorem{corollary}[theorem]{Corollary}

\newtheorem{definition}[theorem]{Definition}

\theoremstyle{definition}
\newtheorem{remarkbase}[theorem]{Remark}
\newenvironment{remark}{\pushQED{\qed} \remarkbase}{\popQED\endremarkbase}

\newcommand{\be}{\begin{equation}}
\newcommand{\ee}{\end{equation}}
\newcommand{\teta}{\theta}
\newcommand{\om}{\omega}

\newcommand{\e}{\varepsilon}
\newcommand{\ep}{\epsilon}

\newcommand{\ov}{\overline}
\newcommand{\wtilde}{\widetilde}
\newcommand{\R}{\mathbb R}
\newcommand{\C}{\mathbb C}

\newcommand{\Z}{\mathbb Z}

\newcommand{\N}{\mathbb N}
\newcommand{\T}{\mathbb T}

\newcommand{\wl}{\varsigma}

\newcommand{\sign}{{\rm sign}}
\renewcommand{\a }{\alpha }
\renewcommand{\b }{\beta }
\newcommand{\s }{\sigma }
\newcommand{\ii }{{\rm i} }
\renewcommand{\d }{\delta }
\newcommand{\D }{\Delta}
\newcommand{\g }{\gamma}

\renewcommand{\l }{\lambda }

\newcommand{\vphi}{\varphi }

\renewcommand{\t }{\tau }

\newcommand{\mN}{\mathcal{N}}
\newcommand{\norma}{| \hspace{-2.2pt} |} % {|\!\!|}

\newcommand{\ppp}{\langle p_7^{(m)} \rangle_{\ph,x}}
\newcommand{\pppuno}{\langle p_7^{(1)} \rangle_{\ph,x}}
\newcommand{\Dom}{\om \cdot \pa_\vphi}
\newcommand{\mG}{\mathcal G}
\newcommand{\tOm}{\mathtt \Omega}
\newcommand{\ph}{\varphi}

\newcommand{\mH}{\mathcal{H}}

\newcommand{\mZ}{\mathcal{Z}}
\newcommand{\mL}{\mathcal{L}}
\newcommand{\lm}{\lambda}
\newcommand{\mM}{\mathcal{M}}
\newcommand{\fm}{\mathfrak{m}}

\newcommand{\mA}{\mathcal{A}}
\newcommand{\mR}{\mathcal{R}}
\newcommand{\mF}{\mathcal{F}}
\newcommand{\mV}{\mathcal{V}}
\newcommand{\mE}{\mathcal{E}}

\newcommand{\mQ}{\mathcal{Q}}
\newcommand{\mP}{\mathcal{P}}
\newcommand{\mU}{\mathcal{U}}
\newcommand{\pa}{\partial}
\newcommand{\ompaph}{\om \cdot \partial_\ph}
\newcommand{\fracchi}{{\mathfrak{I}}}
\newcommand{\mD}{\mathcal{D}}
\newcommand{\odd}{\text{odd}}
\newcommand{\even}{\text{even}}

\def\ba{\begin{aligned}}
\def\ea{\end{aligned}}
\def\beginm{\begin{multline}}
\def\endm{\end{multline}}
\newcommand{\B}{B}
\newcommand{\mB}{\mathcal{B}}

\newcommand{\mC}{\mathcal{C}}

\newcommand{\kug}{{k+1,\gamma}} % MACRO per norme Whitney
\newcommand{\ttb}{\mathtt{b}}

\newcommand{\perd}{\mathtt{d}}
\newcommand{\sym}{{sym}}
\newcommand{\barka}{k_0^*}
\newcommand{\Id}{\mathrm{Id}}
\renewcommand{\Re}{\mathrm{Re}}
\renewcommand{\Im}{\mathrm{Im}}
\newcommand{\DC}{\mathtt{DC}}
\newcommand{\Bignorma}{\Big| \hspace{-2.7pt} \Big|}

\newcommand{\Splus}{\mathbb{S}^+} % Siti Tangenziali positivi 
\DeclareMathOperator{\diag}{diag} 

\newcommand{\vnu}{\mathtt{n}} % la vecchia \nu dell'iterazione
\newcommand{\axi}{a} % a_j che sostituiscono \sqrt{\xi_j} nel Theorem 1.1
\newcommand{\rin}{\mathfrak{r}} % correzione complessiva negli autovalori finali 

\begin{document}

\title{{\bf Time quasi-periodic  gravity water waves \\
in  finite depth}}

\date{}

 \author{Pietro Baldi, Massimiliano Berti, Emanuele Haus, Riccardo Montalto}

\maketitle

\noindent
{\bf Abstract:}
We prove the existence and the linear stability of Cantor families of small amplitude time {\it quasi-periodic} 
standing water wave solutions -- namely periodic and even in the space variable $ x $ -- 
of a bi-dimensional ocean with finite depth under the action of pure gravity. 
Such a result holds for all the values of the depth parameter
in a Borel set of asymptotically full measure. 
This is a small divisor problem. The main difficulties are the fully nonlinear
nature of the  gravity water waves equations -- the highest order $ x $-derivative appears in the nonlinear term
but not in the linearization at the origin -- and the fact that 
the linear frequencies  grow just in a sublinear way at infinity. 
We overcome these problems by first reducing 
the linearized operators, obtained at each approximate quasi-periodic  solution  along a Nash-Moser iterative scheme, 
to constant coefficients up to smoothing operators,
using pseudo-differential  changes of variables that are quasi-periodic in time. 
Then we apply a KAM reducibility scheme which requires very weak Melnikov non-resonance conditions
which lose derivatives both in time and space. Despite the fact that the 
depth parameter moves  the linear frequencies by just exponentially small quantities, 
we are able to verify such non-resonance conditions 
for most values of the depth, extending degenerate KAM theory. %  arguments. 
 \\[2mm]
{\it Keywords:}  Water waves, KAM for PDEs, quasi-periodic solutions, standing waves.
%\footnote{Supported by MIUR 
%``Variational Methods and Nonlinear Differential Equations" and by the European Research Council under FP7.}.
\\[1mm]
% 2010AMS subject classification: 

\tableofcontents

\section{Introduction}

We consider the Euler equations of hydrodynamics for a 2-dimensional perfect, incompressible, inviscid, irrotational fluid
under the action of gravity, filling an ocean with finite depth $ h $  and with space periodic boundary conditions, 
namely the fluid occupies the region
\be\label{domain-fluid}
{\cal D}_\eta := \big\{ (x, y) \in \T \times \R \, : \, - h < y < \eta (t,x) \big\} \, , \quad \T := \T_x := \R / 2 \pi \Z \, . 
\ee
In this paper we prove the existence and the linear stability of small amplitude quasi-periodic in time solutions of 
the {\it pure gravity} water waves system
\begin{equation}\label{water}
\begin{cases}
\partial_t \Phi + \frac12 | \nabla \Phi |^2 + g \eta = 0  \ \ 
\qquad \qquad \qquad    {\rm at} \ y = \eta (t,x)  \cr
\Delta \Phi =0 \qquad \qquad \qquad \  \  \    \qquad  \qquad \ \qquad {\rm in} \ {\cal D}_{\eta}  \cr
\pa_y \Phi = 0  \qquad  \qquad \qquad \quad \qquad \qquad \qquad \,  {\rm at} \  
y = - h  \cr
\partial_t \eta = \partial_y \Phi - \partial_x \eta \cdot
\partial_x \Phi \qquad \qquad \qquad \quad  {\rm at} \  y = \eta (t,x) 
\end{cases}
\end{equation}
where $ g >  0 $ is the acceleration of gravity. 
The unknowns of the problem 
are the free surface $ y = \eta (t,x) $
and the  velocity potential $ \Phi : {\cal D}_\eta \to \R $, 
i.e.\ the irrotational velocity field   
$ v =\nabla_{x,y} \Phi   $ 
of the fluid.
The first equation in \eqref{water}  is the Bernoulli condition stating  the continuity 
of the pressure at the free surface. 
The last equation in \eqref{water} expresses the fact that 
the fluid particles on the free surface always remain part of it. 

\smallskip

Following Zakharov \cite{Zakharov1968} and Craig-Sulem \cite{CrSu}, the evolution problem \eqref{water} may be written as an infinite-dimensional Hamiltonian system
in the  unknowns $ (\eta(t,x), \psi(t,x) ) $ where
$ \psi(t,x)=\Phi(t,x,\eta(t,x))  $
is, at each instant $ t $,   the trace at the free boundary of the velocity potential. 
Given the shape $\eta(t,x)$ of the domain top boundary 
and  the Dirichlet value $\psi(t,x)$ of the velocity potential at the top boundary, there is a unique solution 
$\Phi(t,x,y; h)$ of the elliptic problem % (see e.g.  \cite{ABZ1}, \cite{LannesLivre})
\begin{equation} \label{BoundaryPr}
\begin{cases}
\Delta \Phi = 0 & \text{in } 
\{-  h < y < \eta(t,x)\} \\
\partial_y \Phi = 0 & \text{on } y = - h \\
\Phi = \psi & \text{on } \{y = \eta(t,x)\} \, . 
\end{cases}
\end{equation}
As proved in \cite{CrSu}, 
system \eqref{water} is then equivalent to the Craig-Sulem-Zakharov system
\begin{equation}\label{WW}
\begin{cases}
\pa_t \eta =  G(\eta, h) \psi \\
\pa_t \psi = - g \eta - \dfrac{\psi_x^2}{2} + \dfrac{1}{2(1+\eta_x^2)} \big( G(\eta, h) \psi + \eta_x \psi_x \big)^2 
\end{cases}
\end{equation}
where  $ G(\eta, h)$ is the Dirichlet-Neumann operator defined as
\be\label{DN}
G(\eta, h) \psi := \big( \Phi_y - \eta_x \Phi_x \big)_{|y = \eta (t,x)} 
\ee
(we denote by $ \eta_x $ the space derivative $ \partial_x \eta $).
The reason of the name ``Dirichlet-Neumann'' is that 
$G(\eta,h)$ maps the Dirichlet datum $\psi$ 
into the (normalized) normal derivative $G(\eta,h)\psi$ at the top boundary (Neumann datum). 
The operator $ G(\eta, h) $ is linear in $ \psi $, 
self-adjoint with respect to the $ L^2 $ scalar product and positive-semidefinite, 
and its kernel contains only the constant functions. 
The Dirichlet-Neumann operator is a {\it pseudo-differential} operator with principal symbol 
$ D \tanh (hD) $, with the property that 
$ G(\eta, h ) - D \tanh (hD) $ is in $ OPS^{-\infty}  $
when $ \eta (x) \in {\cal C}^\infty $. This operator has been introduced in Craig-Sulem \cite{CrSu} and its properties are nowdays well-understood thanks to the works of Lannes \cite{Lannes}-\cite{LannesLivre}, Alazard-M\'etivier \cite{AlM}, Alazard-Burq-Zuily \cite{ABZ2}, Alazard-Delort \cite{AlDe1}. In Appendix \ref{subDN} we provide a self-contained analysis of the Dirichlet-Neumann operator adapted to our purposes.   

Furthermore, equations \eqref{WW} are the Hamiltonian system (see \cite{Zakharov1968}, \cite{CrSu})
\be\label{HS}
\begin{aligned}
& \qquad \pa_t \eta = \nabla_\psi H (\eta, \psi) \, , \quad  \pa_t \psi = - \nabla_\eta H (\eta, \psi)  \\
& \pa_t u  = J \nabla_u H (u)  \, , \quad u := \begin{pmatrix} 
\eta \\
\psi \\
\end{pmatrix} \, , \quad J := 
\begin{pmatrix} 
0 & {\rm Id} \\
- {\rm Id}  & 0 \\
\end{pmatrix} \, , 
\end{aligned}
\ee
where $ \nabla $ denotes the $ L^2 $-gradient,  
and  the Hamiltonian 
\be\label{Hamiltonian}
H(\eta, \psi) := H(\eta, \psi, h ) 
:= \frac12 \int_\T \psi \, G(\eta, h ) \psi \, dx + \frac{g}{2} \int_{\T} \eta^2  \, dx  
\ee
is the sum of the kinetic and  potential energies 
expressed in terms of the variables  $ (\eta, \psi ) $. 
The symplectic structure induced by  \eqref{HS} is the standard Darboux $ 2 $-form 
\be\label{2form tutto}
{\cal W}(u_1, u_2):= ( u_1, J u_2 )_{L^2(\T_x)} = ( \eta_1, \psi_2 )_{L^2 (\T_x)} - ( \psi_1, \eta_2 )_{L^2 (\T_x)}   
\ee
for all $ u_1 = (\eta_1, \psi_1) $, $ u_2 = (\eta_2, \psi_2) $.
In the paper we will often write $G(\eta), H(\eta,\psi)$ 
instead of $G(\eta,h), H(\eta,\psi, h)$, 
omitting for simplicity to denote the dependence on the depth parameter $ h$.  

The phase space of \eqref{WW} is 
\be\label{phase-space}
(\eta, \psi) \in H^1_0 (\T) \times {\dot H}^1 (\T)  \qquad {\rm where} \qquad  {\dot H}^1 (\T) := H^1 (\T) \slash_{ \sim}  
\ee
is the homogeneous space obtained by the equivalence relation $ \psi_1 (x) \sim \psi_2 (x) $ if and only if $ \psi_1 (x) - \psi_2 (x) = c $ is a constant, 
and $H^1_0(\T)$ is the subspace of $H^1(\T)$ of zero average functions. 
For simplicity of notation we denote the equivalence class $ [\psi] $ by $ \psi $. Note that the second equation in \eqref{WW}
is in $  {\dot H}^1 (\T) $, as it is natural because only the gradient of the velocity potential has a physical meaning. 
Since the quotient map induces an isometry of $ {\dot H}^1 (\T) $ onto $ H^1_0 (\T) $, 
it is often convenient to identify $ \psi $ with a function with zero average. 

\smallskip

The water waves system \eqref{WW}-\eqref{HS} exhibits several symmetries. 
First of all, the mass $ \int_\T \eta \, dx $ is a first integral of \eqref{WW}. 
In addition, the subspace of functions that are even in $ x $, 
\be\label{even in x}
\eta (x) = \eta (-x) \, , \quad \psi (x) = \psi (- x )  \, , 
\ee
is invariant under \eqref{WW}. 
In this case also the velocity potential $ \Phi(x,y) $ is even  and $ 2 \pi $-periodic in $ x $ and so the 
$ x$-component  of the velocity field $ v = (\Phi_x, \Phi_y) $ vanishes  at $ x = k \pi $, for all $ k \in \Z $. 
Hence   there is no flow of fluid through the lines $ x = k \pi $, $ k \in \Z $, and 
a solution of \eqref{WW} satisfying \eqref{even in x} describes 
the motion of a liquid confined  between two vertical walls.

Another important symmetry 
of the water waves system  is reversibility, namely equations
 \eqref{WW}-\eqref{HS} are reversible with respect to the involution  $ \rho : (\eta, \psi) \mapsto (\eta, - \psi) $,
 or, equivalently, the Hamiltonian $H$ in \eqref{Hamiltonian} 
 is even in $ \psi $: 
\be\label{defS}
H \circ \rho = H \, , \quad H( \eta, \psi) = H ( \eta, - \psi ) \, , \quad \rho : (\eta, \psi) \mapsto (\eta, - \psi) \, .
\ee
As a consequence it is natural to look for solutions of \eqref{WW} satisfying 
\be\label{odd-even}
u(-t ) = \rho u(t) \, , \quad i.e. \quad \eta (-t, x) = \eta(t, x) \, , \ 
\psi(-t,x ) = - \psi (t,x ) \quad \forall t, x \in \R \, , 
\ee
namely $ \eta  $ is even  in time and $ \psi $ is odd in time. 
Solutions of the water waves equations \eqref{WW} satisfying 
\eqref{even in x} and \eqref{odd-even} are called gravity {\it standing water waves}. 

\smallskip

In this paper we prove the first existence result 
of small amplitude time \emph{quasi-periodic} standing  waves solutions of the pure gravity water waves equations \eqref{WW},
  for most values of the depth $ h $, see Theorem \ref{thm:main0}.

The existence of standing water waves is a small divisor problem, 
which is particularly difficult because
\eqref{WW} is a fully nonlinear system of PDEs, the nonlinearity 
contains derivatives of order higher than those present in the  linearized system at the origin, and the linear frequencies 
grow as $ \sim j^{1/2} $. 
The existence of small amplitude time-periodic gravity standing wave solutions 
for bi-dimensional fluids has been first proved by Plotinkov and Toland \cite{PlTo}
in finite depth and by Iooss, Plotnikov and Toland in \cite{IPT}  
in infinite depth, see also \cite{IP-SW2}, \cite{IP-SW1}. 
More recently, the existence of time periodic gravity-capillary standing wave solutions in infinite depth 
has been proved by Alazard and Baldi \cite{Alaz-Bal}. 
Next, both the existence and the linear stability of time quasi-periodic gravity-capillary 
standing wave solutions,  in infinite depth,  
have been proved by Berti and Montalto in \cite{BertiMontalto}, 
see also the expository paper \cite{BM}.

We also mention that the bifurcation of small amplitude one-dimensional traveling
gravity water wave solutions (namely traveling waves in bi-dimensional fluids like \eqref{WW}) 
dates back to Levi-Civita \cite{LC}; note that standing waves are not traveling 
because they are even in space, see \eqref{even in x}. 
For three-dimensional fluids, the existence of small amplitude traveling 
water wave solutions with space periodic boundary conditions has been proved by Craig and Nicholls 
\cite{CN} for the gravity-capillary case (which is not a small divisor problem) 
and by Iooss and Plotinikov \cite{IP-Mem-2009}-\cite{IP2} 
in the pure gravity case 
(which is a small divisor problem). 

\smallskip

From a physical point of view,   
it is  natural to consider the depth $ h $ of the ocean as a fixed physical quantity 
and to introduce the space wavelength $ 2  \pi \wl $ as an {\it internal} parameter.  
Rescaling  time, space and amplitude of the solution $(\eta(t,x), \psi(t,x))$ 
of \eqref{WW} as 
\[
\t := \mu t, \quad \tilde x:= \wl x \, ,  \quad 
\tilde\eta(\t,\tilde x) :=\wl \eta( \mu^{-1} \t, \wl^{-1} \tilde x) = \wl \eta(t,x) \, , \quad 
\tilde\psi(\t,\tilde x) := \alpha  \psi(\mu^{-1} \t , \wl^{-1} \tilde x) = \alpha \psi(t,x)\, , 
\]
we get that $(\tilde\eta(\t,\tilde x), \tilde\psi(\t,\tilde x))$ satisfies
$$
\begin{cases}
\pa_\t \tilde \eta =  \dfrac{\wl^2}{\alpha  \mu} G(\tilde\eta,\wl h) \tilde\psi \\
\pa_\t \tilde\psi  = - \dfrac{g \alpha}{\wl \mu} \tilde\eta - \dfrac{ \wl^2 \tilde\psi_{\tilde x}^2}{\alpha  \mu 2} + 
\dfrac{\wl^2}{\alpha  \mu 2(1+\tilde\eta_{\tilde x}^2)} \Big(G(\tilde\eta,\wl h)\tilde\psi + \tilde\eta_{\tilde x} \tilde\psi_{\tilde x} \Big)^2 \, . 
\end{cases}
$$
Choosing the scaling parameters $ \wl, \mu, \alpha $ such that 
$ \frac{\wl^2}{ \alpha \mu} = 1 $, $ \frac{g \alpha }{\wl \mu} = 1 $ 
we obtain   
system \eqref{WW} where the gravity constant $ g $ has been replaced by 1 and the depth parameter $ h$ 
by  
\begin{equation} \label{def rescaled h}
\h := \wl h \, .  
\end{equation}
Changing the parameter $\h$ can be interpreted as changing the space period $ 2 \pi \wl $ of the 
solutions  
and not the depth $ h $ of the water, giving results for a {\it fixed} equation \eqref{WW}. 

In the sequel we shall look for time quasi-periodic solutions of the water waves system 
\begin{equation}\label{WW0}
\begin{cases}
\pa_t \eta =  G(\eta,\h) \psi \\
\pa_t \psi = - \eta  - \dfrac{\psi_x^2}{2} 
+ \dfrac{1}{2(1+\eta_x^2)} \big( G(\eta,\h) \psi + \eta_x \psi_x \big)^2 
\end{cases}
\end{equation}
with $\eta(t) \in H^1_0(\T_x) $ and $\psi(t) \in {\dot H}^1 (\T_x) $, 
actually belonging to more regular  Sobolev spaces. 
 
 \subsection{Main result}

% PRIMA ERA COSI:
%We look for small amplitude solutions of \eqref{WW0}. 
%Of main importance is therefore the dynamics of the  
%system obtained linearizing \eqref{WW0} at the equilibrium $(\eta, \psi) = (0,0)$, namely 
%
% NUOVO:
We look for small amplitude solutions of \eqref{WW0}. 
Hence a fundamental r\^ole is played by the dynamics of the system 
obtained linearizing \eqref{WW0} at the equilibrium $(\eta, \psi) = (0,0)$, namely 
%\textbf{OPPURE:}
%Looking for small amplitude solutions of the water waves equations,  
%the nonlinear system \eqref{WW0} can be viewed as a perturbation of its linear part,   
%which is the system obtained linearizing \eqref{WW0} at the equilibrium 
%$(\eta, \psi) = (0,0)$, namely 
\begin{equation} \label{Lom}
\left\{
\begin{aligned}
&\partial_t \eta = G(0, {\mathtt h})\psi \\
&\partial_t\psi = - \eta 
\end{aligned}
\right.
\end{equation}
where $ G(0, {\mathtt h}) = D \tanh ({\mathtt h} D) $ is the Dirichlet-Neumann operator at the flat surface $ \eta = 0  $. 
In the compact Hamiltonian form as in \eqref{HS}, system \eqref{Lom} reads 
\begin{equation}\label{definizione Omega}
\pa_t u  = J \Omega u  \, , \quad 
\Om  :=  \begin{pmatrix} 
1    & 0 \\
0 &  G(0, {\mathtt h}) \\
\end{pmatrix},  
\end{equation}
which is  the Hamiltonian system generated by the quadratic Hamiltonian (see \eqref{Hamiltonian})
\be\label{Hamiltonian linear}
H_{L}  := \frac12 ( u, \Omega u )_{L^2}
= \frac12 \int_\T \psi\, G(0, {\mathtt h}) \psi \, dx 
+ \frac12 \int_{\T}  \eta^2   \, dx \, . 
\ee
The solutions of the linear system \eqref{Lom}, i.e.\ \eqref{definizione Omega}, even in $x$, satisfying 
\eqref{odd-even} and \eqref{phase-space}, are 
\begin{equation}  \label{eta psi sin cos series}
\eta(t,x) = \sum_{ j \geq 1}  a_j  \cos (\om_j t)  \cos(jx), \quad 
\psi(t,x) = - \sum_{ j \geq 1}  a_j  \om_j^{-1}   \sin (\omega_j t)  \cos(jx) \, , 
\end{equation} 
with linear frequencies of oscillation
\be\label{LIN:fre}
\om_j := \om_j ({\mathtt h}) :=  \sqrt{j \tanh({\mathtt h} j)} \, , \quad j \geq 1 \, . 
\ee
Note that, since $ j \mapsto j \tanh({\mathtt h} j) $ is monotone increasing, all the linear frequencies are simple.

The main result of the paper proves that most solutions \eqref{eta psi sin cos series} 
of the linear system \eqref{Lom}
can be continued to solutions of the nonlinear water waves equations \eqref{WW0} for most 
values of the parameter $ {\mathtt h} \in [{\mathtt h}_1, {\mathtt h}_2] $. 
More precisely we look for quasi-periodic solutions  $ u (\wtilde \om t) = (\eta, \psi)( \wtilde \om t) $ of  \eqref{WW0}, 
with frequency $ \wtilde \om \in \R^\nu$ (to be determined),  
close to solutions \eqref{eta psi sin cos series} of \eqref{Lom}, 
in the Sobolev spaces of functions  
$$
H^s(\T^{\nu+1}, \R^2) := \big\{ u = (\eta, \psi) : \eta, \psi \in H^s \big\}
$$
\begin{equation} \label{unified norm}
H^s := H^s(\T^{\nu+1}, \R)
= \Big\{ f = \sum_{(\ell,j) \in \Z^{\nu+1}} f_{\ell j} \, e^{\ii(\ell \cdot \ph + jx)} : \ 
\| f \|_s^2 := \sum_{(\ell,j) \in \Z^{\nu+1}} | f_{\ell j}|^2 \langle \ell,j \rangle^{2s} < \infty \Big\},
\end{equation}
where  $\langle \ell,j \rangle := \max \{ 1, |\ell|, |j| \} $. For 
\be\label{def:s0}
s \geq s_0 := \Big[ \frac{\nu +1}{2} \Big] +1 \in \N  
\ee 
one has $ H^s ( \T^{\nu+1},\R) \subset L^\infty ( \T^{\nu+1},\R)$, and $H^s(\T^{\nu+1},\R)$ is an algebra. 

Fix an arbitrary finite subset $ \Splus \subset \N^+ := \{1,2, \ldots \} $ (tangential sites) and consider the solutions of the linear equation \eqref{Lom}
\begin{equation}  \label{eta psi sin cos series-QP}
\eta(t,x) = \sum_{ j \in \Splus}  \axi_j \cos \big( \om_j(\h) t \big) \cos(jx), \quad 
\psi(t,x) = - \sum_{ j \in \Splus} \frac{\axi_j}{\om_j(\h)} \sin \big( \om_j(\h) t \big) \cos(jx), \quad \axi_j >  0 \, , 
\end{equation}
which are Fourier supported on $ \Splus $. 
We denote by $ \nu := | \Splus| $ the cardinality of $ \Splus $.

\begin{theorem} \label{thm:main0}  {\bf (KAM for gravity water waves in finite depth)}
For every choice of the tangential sites $ \Splus \subset \N  \setminus \{0\} $,
there exists $ \bar s >  \frac{|\Splus| + 1}{2} $,  
$ \e_0 \in (0,1) $ such that 
for every vector $ \vec{\axi} := (\axi_j)_{j \in \Splus} $, 
with $\axi_j > 0$ for all $j \in \Splus$
and $ |\vec{\axi}| \leq \e_0 $, 
there exists a Cantor-like set $ \mG \subset [{\mathtt h}_1, {\mathtt h}_2] $ 
with asymptotically full measure as $ \vec{\axi} \to 0 $, i.e. 
$$
\lim_{\vec{\axi} \to 0} | \mG |  = {\mathtt h}_2- {\mathtt h}_1 \, ,  
$$
such that, for any  $ {\mathtt h} \in \mG $,  the gravity water waves system \eqref{WW0}
has a time quasi-periodic solution 
$ u( \widetilde \om t, x ) = (\eta ( \widetilde \om t, x), \psi ( \widetilde \om t, x) ) $,
with Sobolev regularity $ (\eta, \psi)   \in H^{\bar s} ( \T^\nu \times \T, \R^2) $, with a Diophantine frequency vector $ \widetilde \omega := \widetilde \omega(\mathtt h, \vec{\axi}) := 
(\widetilde \omega_j )_{j \in \Splus} \in \R^\nu $,
of the form
\be\label{QP:soluz}
\begin{aligned}
& \eta(\widetilde \omega t,x) = \sum_{ j \in \Splus}  \axi_j \cos ({\widetilde \om}_j t)  \cos(jx) + r_1 ( \widetilde \om t,x ), \\
& \psi(\widetilde \omega t,x) = - \sum_{ j \in \Splus} \frac{\axi_j}{\om_j(\h)} \sin ({\widetilde \omega}_j t)  \cos(jx)+ r_2 ( \widetilde \om t,x ) 
\end{aligned}
\ee
with 
$ {\widetilde \omega}(\mathtt h, \vec{\axi}) 
\to \vec \omega ({\mathtt h}) := (\omega_j ({\mathtt h}))_{j \in {\mathbb S}^+} $ 
as $ \vec{\axi} \to 0 $,
and the functions 
$ r_1 ( \vphi ,x ), r_2(\vphi, x)$ are $o( |\vec{\axi}| )$-small in $ H^{\bar s} ( \T^\nu \times \T, \R) $, i.e.\ 
$\| r_i \|_{\bar s}/ |\vec{\axi}| \to 0$ as $|\vec{\axi}| \to 0$ for $i = 1, 2$. 
The solution  $ (\eta(\widetilde \omega t, x), \psi(\widetilde \omega t, x)) $ %  \in H^{\bar s} $ 
is even in $ x $, $ \eta $ is even  in $ t $ and $ \psi $ is odd in $ t $. 
In addition these quasi-periodic solutions are  linearly stable, see Theorem \ref{thm:lin stab}.  
\end{theorem}

Let us make some comments on the result.

No global wellposedness results concerning the initial value problem of the water waves equations 
\eqref{WW} under {\it periodic} boundary conditions are known so far. 
Global existence results have been proved 
for smooth Cauchy data rapidly decaying at infinity in $ \R^d $, $ d = 1, 2 $, exploiting the dispersive properties
of the flow. For three dimensional fluids (i.e.\  $ d = 2 $) it has been proved independently by 
Germain-Masmoudi-Shatah \cite{GMS} and Wu \cite{Wu2}.  
In the more difficult case of bi-dimensional fluids (i.e.\  $ d = 1 $) it has been proved
by  Alazard-Delort \cite{AlDe1} and Ionescu-Pusateri \cite{IP}. 

In the case of periodic boundary conditions, Ifrim-Tataru \cite{IT} proved for small initial data 
a cubic life span time of existence, which is longer than the one just provided 
by the local existence theory, see for example \cite{ABZ3}.
For longer times, we mention the almost global existence result in Berti-Delort \cite{BD}
for gravity-capillary space periodic water waves. 

The Nash-Moser-KAM iterative procedure used to prove Theorem \ref{thm:main0}
selects many values of the parameter
$ {\mathtt h} \in [{\mathtt h}_1, {\mathtt h}_2] $ that give rise
to the quasi-periodic solutions  \eqref{QP:soluz}, which are defined for all times.
By a Fubini-type argument it also results that, for most values of 
${\mathtt h} \in [{\mathtt h}_1, {\mathtt h}_2] $,
there exist quasi-periodic solutions of \eqref{WW0} for most values 
of the amplitudes $ | \vec{\axi} | \leq \e_0 $. 
The fact that we find quasi-periodic solutions only restricting to a proper subset of 
parameters is not a technical issue, because 
the gravity water waves equations \eqref{WW} are expected to be not
integrable, see \cite{Craig-Worfolk}, \cite{DLZ} in the case of infinite depth. 

The dynamics of the pure gravity and gravity-capillary water waves equations is very different:
\begin{itemize}
%\item[$(i)$]\label{singul} the pure gravity water waves system \eqref{WW0} 
%is a \emph{singular perturbation} of the linearized system \eqref{Lom}
% at the origin. The linearization (or para-linearization) of  \eqref{WW0} at a nontrivial  function is,
%%  after symmetrization and the introduction of the Alinach good unknown, 
% \[
%\pa_t + \ii |D_x|^{\frac12}{\tanh^{\frac12}(\h |D_x|)}  + V ( t, x )  \pa_x + \ldots 
%\]
%and the first order transport vector field $V\pa_x$ is a singular perturbation of 
%$ \ii |D_x|^{\frac12}{\tanh^{\frac12}(\h |D_x|)} $. 

\item[$(i)$]\label{singul} the pure gravity water waves vector field in  \eqref{WW0}  
 is a \emph{singular perturbation} of the linearized vector field 
at the origin in  \eqref{Lom}, which, after symmetrization, is $ |D_x|^{\frac12}{\tanh^{\frac12}(\h |D_x|)}$; in fact, the linearization of the nonlinearity gives rise to a transport vector field $V\pa_x$, see \eqref{def:L0-intro}.
On the other hand, the gravity capillary vector field is quasi-linear and contains derivatives of the same order as the linearized vector field at the origin, which
is $\sim |D_x|^{\frac32}$. This difference, which is well known in the water waves literature, 
requires a very different analysis of the linearized operator
(Sections  \ref{linearizzato siti normali}-\ref{sezione descent method})  with respect to the gravity capillary case in 
\cite{Alaz-Bal}, \cite{BertiMontalto}, see Remark \ref{rem:GC}.

\item[$(ii)$]\label{growth}  
The linear frequencies $ \omega_j $ in \eqref{LIN:fre} 
of the pure gravity water waves grow like $\sim j^{\frac12}$ as $ j \to + \infty $, 
while, in presence of surface tension $ \kappa $, 
the  linear frequencies are $ \sqrt{(1+ \kappa j^2)j \tanh ({\mathtt h}j) }  
% of the gravity capillary water waves 
\sim j^{\frac32}$. This makes a substantial difference for the development of KAM theory. In presence of a sublinear growth of the linear frequencies $\sim j^\a$, $\a<1$, 
one may impose only very weak second order Melnikov non-resonance conditions, see e.g.\  
\eqref{2nd melnikov perd}, which lose also space (and not only  time) derivatives along the KAM reducibility scheme.
This is not the case of the abstract infinite-dimensional KAM theorems \cite{K1}, \cite{Kuksin-Oxford}, \cite{Po2}
where the linear frequencies 
grow as 
$ j^\a $, $ \a \geq 1 $, and  the perturbation is bounded.
In this paper we overcome the difficulties posed by the sublinear growth $\sim j^{\frac12}$ and by the unboundedness of the water waves vector field thanks to a regularization procedure performed on the linearized PDE at each approximate quasi-periodic solution obtained along a Nash-Moser iterative scheme, see the regularized system \eqref{op:redu1}. 
This 
regularization strategy is in principle applicable to a broad class of PDEs
where the second order Melnikov non-resonance conditions lose space derivatives.

\item[$(iii)$]
The linear frequencies \eqref{LIN:fre} vary with $\mathtt h$ only by exponentially small quantities: they admit the 
asymptotic expansion 
\begin{equation}\label{espansione asintotica degli autovalori}
\sqrt{j \tanh ({\mathtt h} j)} = \sqrt{j} + r(j,\h) \quad {\rm where} \quad 
\big| \partial_{\mathtt h}^k r(j,\h) \big| \leq C_k e^{ - {\mathtt h}  j} 
\quad \forall k \in \N \, , \ \forall j \geq 1,
\end{equation}
uniformly in $ {\mathtt h} \in [{\mathtt h}_1, {\mathtt h}_2] $, 
where the constant $C_k$ depends only on $k$ and $\h_1$. Nevertheless
we shall be able, 
extending the degenerate KAM theory approach in \cite{BaBM}, \cite{BertiMontalto},  
to use the finite depth parameter   $ \h $  to impose the required 
Melnikov non-resonance conditions, see \eqref{2nd melnikov perd}  
and Sections \ref{sec:degenerate KAM} and \ref{sec:measure}. 
On the other hand, for the gravity capillary water waves considered in \cite{BertiMontalto},
 the surface tension parameter $ \kappa $ moves the  linear frequencies
 $ \sqrt{(1+ \kappa j^2)j \tanh ({\mathtt h}j) }  $  of polynomial  quantities $ O( j^{3/2})$.
\end{itemize}

\noindent
{\bf Linear stability.} The quasi-periodic solutions $ u( \widetilde \om t) 
= (\eta ( \widetilde \om t), \psi ( \widetilde \om t) ) $ found in Theorem \ref{thm:main0}  
are linearly stable.
Since this  is not only a dynamically relevant information, 
but also an essential ingredient of the existence proof 
(it is not necessary 
for time periodic  solutions as in \cite{PlTo}, \cite{IPT}, \cite{IP-SW2}, \cite{IP-SW1}, \cite{Alaz-Bal}),
we state precisely the result.

The quasi-periodic solutions  \eqref{QP:soluz} are mainly supported in Fourier space on the tangential sites $ \Splus $.
As a consequence, the dynamics of the water waves equations \eqref{WW} on  the
symplectic  subspaces 
\begin{equation}\label{splitting S-S-bot}
H_{\Splus} := \Big\{ v = \sum_{j \in \Splus} \begin{pmatrix} 
 \eta_j  \\
\psi_j \\
\end{pmatrix} \cos (jx)   \Big\} \, , 
\quad 
H_{\Splus}^\bot := \Big\{ 
z = \sum_{j \in \N \setminus \Splus} \begin{pmatrix} 
 \eta_j  \\
\psi_j \\
\end{pmatrix} \cos (jx) \in H^1_0(\T_x)  \Big\} ,
\end{equation}
is quite different. We shall call $ v \in H_{\Splus} $  the {\it tangential} variable and $ z \in H_{\Splus}^\bot $ the {\it normal} one. 
In the finite dimensional subspace $ H_{\Splus} $
we shall describe the dynamics  by introducing  the action-angle variables $ (\theta, I) \in  \T^\nu \times \R^\nu $ in Section \ref{sec:functional}. 

The classical normal form formulation of KAM theory for lower dimensional tori, see for instance \cite{K1}-\cite{Kuksin-Oxford}, \cite{Po2}, \cite{KaP}, \cite{EK}, \cite{PP}, \cite{Berti-Biasco-Procesi-Ham-DNLW}-\cite{Berti-Biasco-Procesi-rev-DNLW}, \cite{Zhang-Gao-Yuan}, \cite{Liu-Yuan}, provides, when applicable, existence and linear stability of quasi-periodic solutions at the same time. On the other hand, existence (without linear stability) of periodic and quasi-periodic solutions of PDEs has been proved by using the Lyapunov-Schmidt decomposition combined with Nash-Moser implicit function theorems, see e.g. \cite{B5}, \cite{CraigPanorama}, 
\cite{PlTo}, \cite{IPT}, \cite{IP-SW2}, \cite{IP-SW1}, \cite{CN}, \cite{Baldi Benj-Ono}, \cite{Alaz-Bal} 
 and references therein. 
In this paper we follow the Nash Moser approach to KAM theory outlined in \cite{BB13} and implemented in \cite{BBM-auto}, \cite{BertiMontalto}, which combines ideas of both formulations, 
%For this reason we introduce the action-angle variables on the tangential sites 
%instead of the Cartesian variables used in the 
see Section \ref{ideas of the proof} ``Analysis of the linearized operators'' 
and Section \ref{sezione approximate inverse}. 
%For a comparison with the Lyapunov-Schmidt approach in  and Remark \ref{rem:LS}. 

We prove that around each torus filled by the 
quasi-periodic solutions \eqref{QP:soluz}  
of the Hamiltonian system \eqref{WW0} 
constructed in Theorem \ref{thm:main0}
there exist symplectic coordinates 
$ (\phi, y, w) = (\phi, y, \eta, \psi) \in \T^\nu \times \R^\nu \times  H_{\Splus}^\bot  $ 
(see \eqref{trasformazione modificata simplettica} and \cite{BB13}) 
in which the water waves Hamiltonian  reads
\begin{align}\label{weak-KAM-normal-form}
\widetilde\om \cdot y  + \frac12 K_{2 0}(\phi) y \cdot y +  \big( K_{11}(\phi) y , w \big)_{L^2(\T_x)} 
+ \frac12 \big(K_{02}(\phi) w , w \big)_{L^2(\T_x)} + K_{\geq 3}(\phi, y, w) 
\end{align}
where $ K_{\geq 3} $ collects the terms at least cubic in the variables $ (y, w )$
(see \eqref{KHG} and note that, at a solution, 
one has $ \partial_\phi K_{00} = 0 $, $ K_{10} = \widetilde \omega $, 
$ K_{01} = 0 $ by Lemma \ref{coefficienti nuovi}). 
The $ (\phi, y) $ coordinates are modifications of the action-angle variables and 
$ w $ is a translation of the cartesian variable $ z$ in the normal subspace, see \eqref{trasformazione modificata simplettica}.
In these coordinates the quasi-periodic solution reads 
$ t \mapsto (\widetilde\om t , 0, 0 ) $ 
and  the corresponding  linearized water waves equations are
\begin{equation}\label{linear-torus-new coordinates}
\begin{cases}
\dot{ \phi} = K_{20}(\widetilde\omega t)[ y] + K_{11}^T(\widetilde\omega t)[ w] \\
\dot{ y} = 0 \\
\dot{ w} = J K_{02}(\widetilde\omega t)[ w] + J K_{11}(\widetilde\omega t)[ y]\, .
\end{cases}
\end{equation}
The self-adjoint operator $  K_{02} (\widetilde\om t) $ is defined in \eqref{KHG}
and $ J K_{02} (\widetilde\om t) $ is the restriction to $ H_{\Splus}^\bot $ of the linearized
water waves vector field  $ J \pa_u \nabla_u H (u (\widetilde\om t ))$ (computed  explicitly in \eqref{linearized vero})
up to a finite dimensional remainder, see Lemma \ref{thm:Lin+FBR}.

We have the following result of linear stability for the quasi-periodic solutions found in Theorem \ref{thm:main0}. 
\begin{theorem}{ \bf (Linear stability)} \label{thm:lin stab}
The quasi-periodic solutions \eqref{QP:soluz} of the pure gravity water waves system are linearly stable, meaning that for all $s$ belonging to a suitable interval $[s_1,s_2]$, for any initial datum $y(0) \in \R^\nu$, $w(0) \in H^{s - \frac14}_x \times H^{s + \frac14}_x$, the solutions $y(t)$, $w(t)$ of system \eqref{linear-torus-new coordinates} satisfy 
$$
y(t) = y(0), \quad  \| w (t)\|_{H^{s - \frac14}_x \times H^{s + \frac14}_x} \leq C \big( \| w(0)\|_{H^{s - \frac14}_x \times H^{s + \frac14}_x} + |y(0)| \big) \quad \forall t\in\R. %\quad \forall s \in [s_1, s_2]\,. 
$$
\end{theorem}

In fact, by \eqref{linear-torus-new coordinates},  
the actions $ y (t) = y(0) $ do not evolve in time and 
the third equation reduces to  the linear PDE
\begin{equation}\label{San pietroburgo modi normali}
\dot{w} = J K_{02}(\widetilde\omega t)[ w] + J K_{11}(\widetilde\omega t)[ y (0)] \, .
\end{equation}
Sections \ref{linearizzato siti normali}-\ref{sec: reducibility} 
imply the existence of a transformation $(H^s_x \times H^s_x) \cap H_{\Splus}^\bot  \to (H^{s - \frac14}_x \times H^{s + \frac14}_x) \cap H_{\Splus}^\bot $, 
bounded and invertible for all % on Sobolev spaces in a range 
$s \in [s_1, s_2]$,
% see \eqref{semiconiugio cal L8}, \eqref{final semi conjugation}, 
such that, in the new variables ${\mathtt w}_\infty$, 
the homogeneous equation $\dot{w} = J K_{02}(\widetilde\omega t)[ w]$ transforms into 
a system of infinitely many uncoupled % complex 
scalar and time independent ODEs of the form
\begin{equation}\label{san pietroburgo siti normali ridotta}
\partial_t {\mathtt w}_{\infty, j} = - \ii \mu_j^\infty {\mathtt w}_{\infty, j} \, , 
\quad \forall j \in {\mathbb S}_0^c \, , 
\end{equation}
% (together with their complex conjugate),
where $\ii$ is the imaginary unit, 
$\mathbb S_0^c := \Z \setminus \mathbb S_0$,
${\mathbb S}_0 := \Splus \cup (- \Splus) \cup \{ 0 \} \subseteq \Z$, 
the eigenvalues $\mu_j^{\infty}$ are (see \eqref{mu j infty kappa}, \eqref{autovalori in kappa})
% (see \eqref{autovalori finali riccardo}, \eqref{mu j nu}, \eqref{op-diago0}, \eqref{code - 1/2 inizio riducibilita}, \eqref{stima code - 1/2 inizio riducibilita}) 
\be\label{Floquet-exp}
\mu_j^{\infty} 
:= \mathtt m_{\frac12}^{\infty} |j|^{\frac12} \tanh^{\frac12}(\mathtt h |j|) 
+ \rin_j^\infty \in \R , 
\quad j  \in {\mathbb S}_0^c \, , 
\quad \rin_j^\infty = \rin_{-j}^\infty\,, 
\ee
and $ \mathtt m_{\frac12}^{\infty} = 1 + O( | \vec{\axi} |^{c}) $, 
$\sup_{j \in {\mathbb S}_0^c} |j|^{\frac12} |\rin_j^{\infty}| = O( |\vec{\axi} |^{c} ) $ 
for some $ c  >  0 $, see \eqref{stime coefficienti autovalori in kappa}. Since $\mu_j^\infty$ are even in $j$, equations \eqref{san pietroburgo siti normali ridotta} can be 
equivalently written in the basis $(\cos(jx))_{j\in\N\setminus\Splus}$ of functions even in $x$; in Section \ref{sec: reducibility}, for convenience, we represent even operators in the exponential basis $(e^{\ii j x})_{j\in{\mathbb S}_0^c}$.
The above result is the {\it reducibility} of the linearized quasi-periodically time dependent equation 
$\dot{w} = J K_{02}(\widetilde\omega t)[w]$. 
The {\it Floquet exponents} of the quasi-periodic solution are the purely imaginary numbers 
$ \{ 0,  \ii \mu_j^{\infty}, j \in {\mathbb S}_0^c  \}$
(the null Floquet exponent comes from the action component $\dot y = 0$).
Since $\mu_j^\infty$ are real, the Sobolev norms of the solutions of 
\eqref{san pietroburgo siti normali ridotta} are constant.

%The conjugation above % of $\pa_t - J K_{02}(\widetilde \om t)$  
%that transforms \eqref{San pietroburgo modi normali} 
%into \eqref{san pietroburgo siti normali ridotta} 
The reducibility of the linear equation 
$\dot{w} = J K_{02}(\widetilde\omega t)[w]$ 
is obtained by 
two well-separated procedures: 

\begin{enumerate}
\item
First, we perform a reduction of the linearized operator 
into a constant coefficient pseudo-differential operator, up to smoothing remainders, via 
changes of variables that are quasi-periodic transformations of the phase space, see \eqref{op:redu1}.  
We perform such a reduction in Sections \ref{linearizzato siti normali}-\ref{coniugio cal L omega}.
\item
Then, 
we  implement in Section \ref{sec: reducibility} a KAM iterative scheme 
which completes the diagonalization of  the linearized operator. 
This scheme uses very weak second order Melnikov non-resonance conditions 
which lose derivatives both in time and in space. 
This loss is compensated along the KAM scheme 
by the smoothing nature of the variable coefficients remainders. 
Actually, in Section \ref{sec: reducibility} we explicitly state only a result of almost-reducibility 
(in Theorems \ref{iterazione riducibilita}-\ref{Teorema di riducibilita} we impose only finitely many Melnikov non-resonance conditions and there appears a remainder $\mR_n$ of size $O(N_n^{-\mathtt a})$, where $\mathtt a > 0$ is the large parameter fixed in \eqref{alpha beta}), because this is sufficient for the construction of the quasi-periodic solutions. 
However the frequencies of the quasi-periodic solutions 
that we construct in Theorem \ref{thm:main0} 
satisfy all the infinitely many Melnikov non-resonance conditions in 
\eqref{Cantor set infinito riccardo1} 
and Theorems \ref{iterazione riducibilita}-\ref{Teorema di riducibilita} pass to the limit as $n\to\infty$, 
leading to \eqref{san pietroburgo siti normali ridotta}.
\end{enumerate}

We shall explain these steps in detail in Section \ref{ideas of the proof}.
In the pioneering works of Plotnikov-Toland \cite{PlTo} and Iooss-Plotnikov-Toland \cite{IPT} dealing with time-periodic solutions of the water waves equations, 
as well as in \cite{Alaz-Bal}, % \cite{Baldi Benj-Ono}, 
the latter diagonalization is not required. 
The key difference is that,
in the periodic problem, 
a sufficiently regularizing operator in the space variable is also regularizing in the  time  variable,  
on the ``characteristic'' Fourier indices which correspond to 
the small divisors. This is definitely not true for quasi-periodic solutions.  

\medskip

\noindent
{\bf Literature about KAM for quasilinear PDEs.} 
KAM theory for PDEs has been developed to a large extent for bounded perturbations 
and for linear frequencies growing in a superlinear way, as $ j^\a $, $ \a \geq 1 $.  
The case $ \a = 1 $, which corresponds to 1d wave and Klein-Gordon equations, is more delicate. 
In the sublinear case $ \a < 1$, as far as we know, there are no general KAM results, 
since the second order Melnikov conditions lose space derivatives. 
Since the eigenvalues of $- \Delta$ on $\T^d$ 
grow, 
according to the Weyl law, like $ \sim j^{2/d}$, $j \in \N$, 
one could regard the KAM results for multidimensional Schr\"odinger and wave equations
 in
\cite{B5}, \cite{EK}, \cite{BB13JEMS}, \cite{BCP}, \cite{PP},
 under this perspective. 
Actually the proof of  these results 
 exploits specific properties of clustering of the eigenvalues of the Laplacian.
 
The existence of quasi-periodic solutions of PDEs with \emph{unbounded}  perturbations (i.e.\ the nonlinearity contains derivatives) has been first proved by Kuksin \cite{Kuksin-Oxford} and Kappeler-P\"oschel \cite{KaP} for KdV, 
then by Liu-Yuan \cite{Liu-Yuan}, Zhang-Gao-Yuan \cite{Zhang-Gao-Yuan} for
derivative NLS, and by Berti-Biasco-Procesi 
\cite{Berti-Biasco-Procesi-Ham-DNLW}-\cite{Berti-Biasco-Procesi-rev-DNLW} for derivative wave equation.
All these previous results still refer to semilinear perturbations, i.e.\ where the order of the derivatives in the nonlinearity is strictly lower than the order of the constant coefficient (integrable) linear differential operator. 

For quasi-linear or fully nonlinear PDEs 
the first KAM results have been recently proved by Baldi-Berti-Montalto  in \cite{BBM-Airy}, 
\cite{BBM-auto}, \cite{BBM-mKdV} 
for perturbations of Airy, KdV and mKdV equations, introducing tools of pseudo-differential calculus for the
spectral analysis of the linearized equations.
In particular, \cite{BBM-Airy} concerns quasi-periodically forced perturbations of the Airy equation
\begin{equation}\label{Airy intro}
u_t + u_{xxx} + \e f(\om t, x,u,u_x,u_{xx},u_{xxx}) = 0
\end{equation}
where the forcing frequency $\om$ is an external parameter. 
The key step is the reduction of the linearized operator at an approximate solution to constant coefficients up to a sufficiently \emph{smoothing} remainder, followed by a KAM reducibility scheme leading to its complete diagonalization. 
Once such a reduction has been achieved, the second order Melnikov nonresonance conditions required for the diagonalization are easily imposed since the frequencies are $\sim j^3$ and using $\om$ as external parameters.
Because of the purely differential structure of \eqref{Airy intro}, the required tools of pseudo-differential calculus are mainly multiplication operators and Fourier multipliers.
These techniques have been extended by Feola-Procesi \cite{FP} 
for quasi-linear forced perturbations of Schr\"odinger equations and by Montalto \cite{riccardo-kirchhoff} for the forced Kirchhoff equation.  

The paper \cite{BBM-auto} deals with the more difficult case of completely resonant autonomous Hamiltonian perturbed KdV equations of the form
\begin{equation}\label{kdv autonoma}
u_t + u_{xxx} - 6 u u_x + f(x,u,u_x,u_{xx},u_{xxx}) = 0\,.
\end{equation}
Since the Airy equation $u_t + u_{xxx}=0$ possesses only $2\pi$-periodic solutions, the existence of quasi-periodic solutions of \eqref{kdv autonoma} is entirely due to the nonlinearity, 
which determines the modulation of the tangential frequencies of the solutions with respect to its amplitudes. This is achieved via ``weak'' Birkhoff normal form transformations that are close to the identity up to finite rank operators.
The paper \cite{BBM-auto} implements the general symplectic procedure proposed in \cite{BB13} for autonomous PDEs, which reduces the construction of an approximate inverse of the linearized operator to the construction of an approximate inverse of its restriction to the normal directions. This is obtained along the lines of \cite{BBM-Airy}, but with more careful size estimates because 
\eqref{kdv autonoma} is a completely resonant PDE.
The symplectic procedure of \cite{BB13} is also applied in \cite{BertiMontalto} and in Section \ref{sezione approximate inverse} of the present paper. We refer to \cite{CorsiFeolaProcesi} and \cite{Feola-Procesi-arxiv} 
for a similar reduction which applies also to PDEs which are not Hamiltonian, 
but for example reversible. 

By further extending these techniques,  
the existence of quasi-periodic solutions of gravity capillary water waves has been recently proved 
in \cite{BertiMontalto}. In items ($i$)-($iii$) after  Theorem \ref{thm:main0} we have described 
the major differences between the pure gravity  and gravity-capillary water waves equations
and we postpone to Remark \ref{rem:GC} more  comments about the differences regarding 
the reducibility  of the linearized equations.

\medskip

\noindent
{\it Acknowledgements}. 
This research was supported by 
PRIN 2015 ``Variational methods, with applications to problems in mathematical physics and geometry",
%PRIN 2012 ``Variational and perturbative aspects of nonlinear differential problems", 
by the European Research Council under FP7, project no.\,306414 ``Hamiltonian PDEs and small divisor problem: a dynamical systems approach'' (HamPDEs),
partially by the Swiss National Science Foundation,
and partially by the Programme STAR, funded by Compagnia di San Paolo and UniNA.

\subsection{Ideas of the proof}
\label{ideas of the proof}

The three major difficulties in proving the existence of time quasi-periodic solutions of the gravity water waves equations \eqref{WW0} are: 
\begin{itemize}
\item[$(i)$] The nonlinear water waves system \eqref{WW0} is a singular perturbation
of \eqref{Lom}.

\item[$(ii)$] The dispersion relation \eqref{LIN:fre} is sublinear, i.e.\  $ \om_j  \sim \sqrt{j} $ for $ j \to \infty $. 

\item[$(iii)$] 
The linear frequencies $ \omega_j ( \mathtt h) = j^{\frac12}\tanh^{\frac12}(\h j)$ vary with $ \mathtt h$ of just exponentially small quantities.
\end{itemize}

We present below the key ideas to solve these three major problems.

\medskip

\noindent
\textbf{Nash-Moser Theorem \ref{main theorem} of hypothetical conjugation.}
In Section \ref{sec:functional} we rescale $u \mapsto \e u$ and introduce the action angle variables $(\theta, I) \in \T^\nu \times \R^\nu$ on the tangential sites (see \eqref{splitting S-S-bot})
\be\label{AA}
\eta_j := \sqrt{\frac{2}{\pi}} \, \om_j^{\frac12} \sqrt{\xi_j + I_j}  \cos (\theta_j ) , \  \
\psi_j := - \sqrt{\frac{2}{\pi}} \, \om_j^{- \frac12} \sqrt{\xi_j + I_j} \, \sin (\theta_j ) \, ,  \ \  j \in \Splus \, ,
\ee
where $ \xi_j > 0 $, $ j = 1, \ldots, \nu $, the variables $I_j$ satisfy $ | I_j | < \xi_j $, 
so that system \eqref{WW0} becomes the Hamiltonian system generated by %the Hamiltonian function
\begin{equation}\label{acca epsilon}
H_\e = {\vec \om} (\h) \cdot  I  + 
\frac12 (z, \Om z)_{L^2} + \e P \,, 
 \quad {\vec \om} (\h):= (j^{\frac12}\tanh^{\frac12}(\h j))_{j\in\Splus}, 
\end{equation}
where $P$ is given in \eqref{definizione cal N P}.
The % amplitudes 
unperturbed actions
$\xi_j$ in \eqref{AA} and the  unperturbed amplitudes  % amplitudes 
$\axi_j$ in \eqref{eta psi sin cos series-QP} and Theorem \ref{thm:main0} 
are related by the identity 
$a_j = \e \sqrt{(2/\pi)}\,\om_j^{\frac12} \sqrt{\xi_j}$ for all $j \in \Splus$.

The  expected quasi-periodic solutions of the autonomous Hamiltonian system generated by  $H_\e$
will have shifted frequencies $ \widetilde \om_j $ -- to be found -- close to the linear frequencies $  \om_j ({\mathtt h}) $ in \eqref{LIN:fre}. 
The perturbed frequencies depend on the nonlinearity and on the amplitudes $ \xi_j $.
Since the Melnikov non-resonance conditions are naturally imposed on $ \om $, 
it is convenient to use the frequencies $\om \in \R^\nu$ as parameters, introducing ``counter-terms'' $ \a \in \R^\nu $ (as in \cite{BertiMontalto}, in the spirit of Herman-F\'ejoz \cite{HF}) in the family of Hamiltonians (see \eqref{H alpha})
$$
H_\a := \a \cdot I + \frac12 (z, \Om z)_{L^2} + \e P\, .
$$
Then the first goal (Theorem \ref{main theorem}) is to prove that,  for $ \e $ small enough,
there exist 
$ \a_\infty (\om, {\mathtt h}, \e)$, close to $\omega$, and a $ \nu $-dimensional embedded  torus $i_\infty(\ph; \om,{\mathtt h}, \e)$ of the form
$$
i : \T^\nu \to \T^\nu \times \R^\nu \times H_{\Splus}^\bot, \quad  
\vphi \mapsto i (\vphi) := ( \theta (\vphi), I (\vphi), z (\vphi)), 
$$
close to $ (\vphi, 0, 0 ) $,   
defined for all $(\om,\h) \in \R^\nu \times  [{\mathtt h}_1, {\mathtt h}_2]$,
such that, for all $(\om,\h)$ belonging to the set ${\cal C}^\gamma_\infty$ defined in \eqref{Cantor set infinito riccardo},  $(i_\infty, \a_\infty )(\om,{\mathtt h}, \e)$ is a zero of the nonlinear operator (see \eqref{operatorF})
\begin{align}\label{calF-intr}
{\cal F} (i, \a, \om, {\mathtt h}, \e )  
&   :=  \left(
\begin{array}{c}
\Dom \theta (\vphi) -  \a - \e \partial_I P ( i(\vphi)   )   \\
\Dom I (\vphi)  +  \e \partial_\teta P( i(\vphi)  )  \\
\Dom z (\vphi) -  J (\Om z(\vphi) + \e \nabla_z P ( i(\vphi) ))  
\end{array}
\right).
\end{align}
%$ {\cal F} (i, \a, \omega, {\mathtt h}, \varepsilon ) = 0 $ defined in \eqref{operatorF}. 
The explicit set ${\cal C}^\gamma_\infty $ requires $ \omega $ to satisfy, in addition to the Diophantine property  
$$
| \om \cdot \ell | \geq \gamma \langle \ell \rangle^{- \tau} 
\quad \forall \ell \in \Z^\nu \setminus \{0\} \, , 
\qquad \langle \ell \rangle := \max \{ 1, |\ell| \}, 
\quad |\ell| := \max_{i = 1, \ldots, \nu} |\ell_i| \,,
$$
the  first and second Melnikov non-resonance conditions stated 
in \eqref{Cantor set infinito riccardo}, in particular
\begin{equation}\label{2nd melnikov perd}
 |\omega \cdot \ell  + 
 \mu_j^\infty (\om, {\mathtt h}) -  \mu_{j'}^\infty (\om, {\mathtt h}) | \geq 
 4 \gamma j^{-\perd} j'^{-\perd} \langle \ell  \rangle^{-\tau}, \,\,
 \forall \ell   \in \Z^\nu ,\,\,j, j' \in \N^+ \setminus \Splus , \, (\ell,j,j') \neq (0,j,j) \, , 
\end{equation}
where $ \mu_{j}^\infty (\om, {\mathtt h})  $ are the
``final eigenvalues" in \eqref{autovalori infiniti}, defined for all $ (\om,\h) \in \R^\nu \times  [{\mathtt h}_1, {\mathtt h}_2]$ 
(we use the abstract  Whitney  extension theorem in 
Appendix \ref{sec:U}).
The torus $ i_\infty $, the conter-term $ \a_\infty $ and  the final eigenvalues $ \mu_{j}^\infty (\om, {\mathtt h})  $ are
$ {\cal C}^{k_0} $ differentiable with respect to the parameters $ (\omega, {\mathtt h }) $.
The value of $ k_0 $ is fixed in
Section \ref{sec:degenerate KAM}, depending only on  the unperturbed linear frequencies, so that 
transversality conditions 
like \eqref{unperturbed measure} hold, see  Proposition \ref{Lemma: degenerate KAM}.  
The value of the counterterm $ \a := \a_\infty (\om, {\mathtt h}, \e) $ is adjusted along the Nash-Moser iteration 
in order to control the average of the first component 
of the Hamilton equation \eqref{operatorF}, 
especially for solving the linearized 
equation \eqref{operatore inverso approssimato proiettato}, 
in particular \eqref{equazione psi hat}. 

Theorem \ref{main theorem} follows by the Nash-Moser Theorem \ref{iterazione-non-lineare}  
which relies on the analysis 
of the linearized operators $ d_{i,\alpha} {\cal F} $ at an approximate  solution,
performed in  Sections \ref{sezione approximate inverse}-\ref{sec: reducibility}. 
The formulation of  Theorem \ref{main theorem} is convenient as it completely decouples the Nash-Moser 
iteration required to prove Theorem \ref{thm:main0}
and the discussion about the measure of the set of parameters $ {\cal C}^\gamma_\infty $ 
where all the Melnikov non-resonance conditions are verified. 
In Section \ref{sec:measure} we are able to prove positive measure estimates, 
if  the exponent $ {\mathtt d} $ in \eqref{2nd melnikov perd}
is large enough and $\g = o(1)$ as $\e \to 0$. 
Since such a value of $ {\mathtt d} $
determines  the number of regularization steps  to be performed on the linearized operator,  
we prefer to first discuss how we fix it, applying degenerate KAM theory.

\medskip

\noindent
\textbf{Proof of Theorem \ref{thm:main0}: degenerate KAM theory and measure estimates.} 
In order to prove the existence of quasi-periodic solutions of the system with Hamiltonian $H_\e$ in \eqref{acca epsilon},
thus \eqref{WW0}, 
and not only
of the system with modified Hamiltonian $ H_\a $ with $ \a := \a_\infty (\om, {\mathtt h} , \e) $, 
we have to prove that the curve of the unperturbed linear tangential frequencies 
\be\label{tang-vec}
[{\mathtt h} _1, {\mathtt h} _2] \ni {\mathtt h}  \mapsto \vec \om ({\mathtt h}  ) := ( \sqrt{j \tanh({\mathtt h} j)} )_{j \in \Splus} \in \R^\nu 
\ee
intersects the image $ \a_\infty ({\cal C}^\gamma_\infty ) $
of the set $ {\cal C}^\gamma_\infty  $ 
under the map $ \a_\infty $,
for ``most" values of $ {\mathtt h}  \in [{\mathtt h} _1, {\mathtt h} _2] $.  
%This is proved in Theorem \ref{Teorema stima in misura} by using degenerate KAM theory. 
Setting \begin{equation}\label{omega epsilon}
 \om_\e ({\mathtt h} ) := \a_\infty^{-1} ( \vec \om ({\mathtt h}  ), {\mathtt h}  )\,, 
\end{equation}
where $\alpha_\infty^{- 1}(\cdot, {\mathtt h} )$ is the inverse of the function $\alpha_\infty(\cdot, {\mathtt h} )$ at a fixed ${\mathtt h}  \in [{\mathtt h}_1, {\mathtt h}_2]$,
if the vector $( \om_\e ({\mathtt h} ), {\mathtt h} )$ belongs to ${\cal C}^\gamma_\infty$, then Theorem  \ref{main theorem} implies the existence of a quasi-periodic solution of the system with Hamiltonian $H_\e$ with 
Diophantine frequency $\om_\e ({\mathtt h} )$.

In Theorem \ref{Teorema stima in misura} we prove that for all the values of $\h\in[\h_1,\h_2]$ except a set of small measure $O(\g^{1/k_0^*})$ (where $k_0^*$ is the index of non-degeneracy appearing below in \eqref{unperturbed measure}), the vector $( \om_\e ({\mathtt h} ), {\mathtt h} )$ belongs to ${\cal C}^\gamma_\infty$.
Since the parameter interval $ [{\mathtt h}_1, {\mathtt h}_2 ] $ is fixed, independently of the 
$ O(\e ) $-neighborhood of  the origin where we look for the solutions,  
the  small divisor constant $ \g $ in the definition of $\mC^\g_\infty$ 
(see e.g.\ \eqref{2nd melnikov perd})
can be taken as 
%$ \gamma = o(1) $ as $ \e \to 0 $,   and, for simplicity, we take 
$ \gamma = \e^a $ with $ a > 0 $ 
% In fact, we take $a$
 as small as needed, see \eqref{relazione tau k0}, 
so that all the quantities $\e \g^{-\kappa}$ that we encounter along the proof 
are $\ll 1$.

The first task is to prove a transversality property for the unperturbed tangential 
frequencies $\vec\om ({\mathtt h})$ in \eqref{tang-vec} and the normal ones 
$\vec\Omega(\h):= (\Omega_j(\h))_{j\in\N^+\setminus\Splus} :=(\om_j(\h))_{j\in\N^+\setminus\Splus}$. Exploiting the fact that 
the maps $  {\mathtt h} \mapsto \om_j ({\mathtt h}^4) $ 
are {\it analytic},
simple -- namely injective in $j$ -- in the subspace of functions even in $ x $,  and
they grow asymptotically like $ \sqrt{j} $
for $ j \to \infty $,  we first prove that the linear frequencies $\om_j(\h)$ are {\it non-degenerate} in the sense of 
Bambusi-Berti-Magistrelli \cite{BaBM} (i.e.\ they are not contained in any hyperplane). 
This is  verified in Lemma \ref{non degenerazione frequenze imperturbate} %as in \cite{BaBM} 
using a generalized Vandermonde determinant (see Lemma \ref{lemma-VDM}).
Then in Proposition \ref{Lemma: degenerate KAM} we translate this qualitative non-degeneracy condition into quantitative transversality information: 
there exist $k_0^*>0, \rho_0>0$ such that, for all $\h \in [\h_1,\h_2]$,
\be\label{unperturbed measure}
\max_{0 \leq k \leq k_0^*} \big| \pa_{\mathtt h}^k  \big( {\vec \om} ({\mathtt h}) \cdot \ell + { \Om}_j ({\mathtt h}) 
- { \Om}_{j'} ({\mathtt h})  \big) \big| 
\geq \rho_0 \langle \ell \rangle  \, ,   \quad 
\forall \ell \neq 0\, , \  j, j' \in \N^+ \setminus \Splus \, , 
\ee
and similarly for the $ 0 $th,  
$ 1$st and $ 2$nd order Melnikov non-resonance condition with the $ +$ sign.
We call (following \cite{Ru1}) $ \barka $ the index of non-degeneracy 
and $ \rho_0 $  the amount of non-degeneracy.
Note that 
the restriction to the subspace of functions with zero average 
in $ x $  eliminates the zero frequency $ \om_0 = 0 $, which is trivially resonant
(this is used also in  \cite{Craig-Worfolk}).
%Property \eqref{unperturbed measure} 
%implies 
%that for ``most" parameters $ {\mathtt h} \in [{\mathtt h}_1, {\mathtt h}_2] $ the unperturbed linear frequencies $ ( \vec \om ({\mathtt h}), \vec \Omega ({\mathtt h}) ) $ satisfy the Melnikov conditions of $ 0 $-th,  
%$ 1$-st and $ 2$-nd order (but we do not use it explicitly). 

%Along % Actually, along 
%the Nash-Moser-KAM iteration we need to verify that the \emph{perturbed} frequencies, 
%and not only the unperturbed linear ones, are Diophantine and satisfy first and second order Melnikov non-resonance conditions, see the explicit conditions in \eqref{Cantor set infinito riccardo}. 
%% It is for this purpose that
The transversality condition  \eqref{unperturbed measure} is stable under
perturbations that are small in $ {\cal C}^{k_0}$-norm, where $ k_0 := k_0^* + 2 $,
see Lemma \ref{Lemma: degenerate KAM perturbato}.
Since $\om_\e ({\mathtt h})$ in \eqref{omega epsilon} and the perturbed Floquet exponents $\mu_j^\infty(\h) = \mu_j^\infty(\om_\e ({\mathtt h}), \h)$ in \eqref{mu j infty kappa} 
are  small perturbations of the unperturbed linear frequencies
$ \sqrt{j \tanh({\mathtt h} j)} $ in $ {\cal C}^{k_0}$-norm,
% (see \eqref{stima omega epsilon kappa} and \eqref{stime coefficienti autovalori in kappa})
the transversality property \eqref{unperturbed measure} still holds   
for the  perturbed frequencies.
%$ \om_\e ({\mathtt h}) $ defined in \eqref{omega epsilon}. 
As a consequence, 
by applying the classical R\"ussmann lemma (Theorem 17.1 in \cite{Ru1}) 
we prove that, for most $\h \in [\h_1,\h_2]$, the 0th, 1st and 2nd Melnikov conditions on the perturbed frequencies hold if $ \perd > \frac34 \, k_0^*$, see Lemma \ref{stima risonanti Russman} and \eqref{stima-complementary-set}.

The larger is $ \perd $, the weaker are the Melnikov conditions \eqref{2nd melnikov perd}, 
and the stronger will be the loss of space derivatives 
due to the small divisors in the reducibility scheme of Section \ref{sec: reducibility}. 
In order to guarantee the convergence of such a KAM reducibility scheme,
these losses of derivatives will be compensated 
by the regularization procedure of 
Sections \ref{linearizzato siti normali}-\ref{coniugio cal L omega}, 
where we reduce the linearized operator to constant coefficients 
up to very regularizing terms $ O( |D_x|^{-M}) $ for some $ M := M(\perd, \tau)$ large enough,  
fixed in \eqref{relazione mathtt b N}, 
which is large with respect to $ \perd $ and $ \tau $ by \eqref{alpha beta}. 
We will explain in detail this procedure below.

\medskip

\noindent
\textbf{Analysis of the linearized operators.} \label{par-ai} In order to prove the existence of 
a solution of $ {\cal F} (i, \alpha ) = 0 $ in \eqref{calF-intr}, proving the Nash-Moser Theorem \ref{main theorem},  
the key step  is to show that the 
linearized operator $ d_{i, \alpha} {\cal F} $ obtained at any approximate solution along the  iterative scheme 
admits
an {\it almost approximate inverse} satisfying  tame estimates in Sobolev spaces
with loss of derivatives, 
see Theorem \ref{thm:stima inverso approssimato}.
Following the terminology of Zehnder \cite{Zehnder}, an \emph{approximate inverse} is an operator which is an exact inverse at a solution
(note that the operator $\mP$ in \eqref{stima inverso approssimato 2} is zero when  ${\cal F}(i, \a) = 0 $).
The adjective \emph{almost} refers to the fact that at the $n$-th step of the Nash-Moser iteration we shall require
only finitely many non-resonance conditions of Diophantine type, therefore there remain
operators (like \eqref{stima cal G omega})
that are Fourier supported on high frequencies of magnitude larger than $O(N_n)$ and thus they 
can be estimated as $O(N_n^{-a})$ for some $a > 0$  (in suitable norms). 
The tame estimates \eqref{stima inverso approssimato 2}-\eqref{stima cal G omega bot alta} are sufficient for the convergence of a differentiable Nash-Moser scheme: the remainder \eqref{stima inverso approssimato 2} produces a quadratic error since it is of order  $O(\mF(i_n,\a_n))$; the remainder \eqref{stima cal G omega} arising from the almost-reducibility is small enough by taking $\mathtt a>0$ sufficiently large, as in \eqref{alpha beta}; the remainder \eqref{stima cal G omega bot bassa} arises by ultraviolet cut-off truncations and its contribution is small by usual differentiable Nash-Moser mechanisms, see for instance \cite{BBP}. These abstract tame estimates 
imply the Nash-Moser Theorem \ref{iterazione-non-lineare}.

In order to find an almost approximate inverse of $ d_{i, \alpha} {\cal F} $ we  first implement the strategy of Section 
\ref{sezione approximate inverse} 
introduced in Berti-Bolle \cite{BB13}, which is based on the following simple observation: 
around an invariant torus there are symplectic coordinates $(\phi, y, w)$ in which  the Hamiltonian 
assumes the normal form  \eqref{weak-KAM-normal-form} 
and therefore 
the linearized equations  at the quasi-periodic solution assume the {\it triangular}  form as in 
\eqref{linear-torus-new coordinates}. 
In these new coordinates 
it  is immediate to solve the equations in the variables  $ \phi, y $, 
and it remains to invert an operator acting on the $ w $ component, 
which is precisely ${\cal L}_\omega$ defined in \eqref{Lomega def}. 
By Lemma \ref{thm:Lin+FBR}
the operator 
${\cal L}_\omega$ is a finite rank perturbation (see \eqref{representation Lom})  of 
the restriction  to the normal subspace $ H_{\Splus}^\bot $ in \eqref{splitting S-S-bot} of 
\be\label{LINOP}
\mL = \om \cdot \pa_\vphi + \begin{pmatrix} 
\partial_x V + G(\eta) B & - G(\eta) \\
(1 + B V_x) + B G(\eta) B \  &  V \partial_x - B G(\eta) \end{pmatrix} 
\ee
where the functions $B, V$ are given in \eqref{def B V}, which is obtained 
linearizing  the water waves equations \eqref{WW0} 
at a quasi-periodic approximate solution $ (\eta, \psi) (\omega t, x)  $ 
and changing $ \pa_t $ into the directional derivative $ \om \cdot \pa_\vphi $.

If $ {\cal F}(i, \alpha ) $ is not zero but it is small,  
we say that $ i $ is  approximately invariant for $ X_{H_\a} $, 
and, following \cite{BB13}, 
in  Section \ref{sezione approximate inverse} we transform 
 $ d_{i, \alpha} {\cal F} $  into  an  \emph{approximately} triangular operator, with an error of size $ O({\cal F}(i, \a))$.  
In this way, we have reduced the  problem of almost approximately inverting
 $ d_{i, \alpha} {\cal F} $  
to the task of almost inverting the operator
${\cal L}_\omega$. 
The precise invertibility properties of $ {\cal L}_\om $ are stated in \eqref{inversion assumption}-\eqref{tame inverse}.

\begin{remark}  \label{rem:LS}
The main advantage of this approach is that 
the problem of inverting $ d_{i, \a} {\cal F} $ on the whole space (i.e.\ both tangential and normal modes)
is  reduced to invert a  PDE on the normal subspace $ H_{\Splus}^\bot $ only.
In this sense this  
is reminiscent of the Lyapunov-Schmidt decomposition, 
where
the complete nonlinear problem is split into  a bifurcation and a range equation on the orthogonal of the kernel. 
However, the Lyapunov-Schmidt approach is based on a splitting of the space 
$H^s(\T^{\nu+1})$ of functions $u(\ph,x)$ of time and space, 
whereas the approach of \cite{BB13} splits the phase space (of functions of $ x $ only) 
into $ H_{\Splus} \oplus H_{\Splus}^\bot $ more similarly to a classical  KAM theory formulation.
\end{remark}

The procedure of Section \ref{sezione approximate inverse}  is a preparation for the
reducibility of the linearized water waves equations in the normal subspace
developed in Sections \ref{linearizzato siti normali}-\ref{sec: reducibility}, 
where we conjugate the operator $ {\cal L }_\omega  $ to a diagonal system of infinitely many decoupled, constant coefficients, scalar linear equations, see \eqref{final-lin} below. 
First, in Sections \ref{linearizzato siti normali}-\ref{sezione descent method}, in order to use the tools of pseudo-differential calculus, it is convenient to ignore the projection on the normal subspace $ H_{\Splus}^\bot $ and to perform a regularization procedure on the operator ${\cal L}$ acting on the whole space, see Remark \ref{remark riccardo inverse}. Then, in Section \ref{coniugio cal L omega}, we project back on $ H_{\Splus}^\bot $.
Our approach involves two  well separated procedures that we describe in detail: 
\begin{enumerate} 
\item { \bf Symmetrization and  diagonalization of $ {\cal L } $ up to smoothing operators}.
 The goal of Sections \ref{linearizzato siti normali}-\ref{sezione descent method} is to conjugate $ {\cal L }$
 to an operator of the form 
\be\label{op:redu1}
 \omega \cdot \partial_\vphi \mathbb + \ii {\mathtt m}_{\frac12} |D|^{\frac12} \tanh^{\frac12}(\h |D|) + 
 \ii r (D) +    {\cal T}_8 (\vphi) 
\ee
where $ {\mathtt m}_{\frac12} \approx 1 $ is a real constant, 
independent of $ \vphi $, 
the symbol  $ r (\xi) $ is real and 
 independent of $ (\vphi, x) $, 
 of order $ S^{-1/2} $,   
and the remainder $  {\cal T}_8  (\vphi) $,
as well as $ \pa_\vphi^\beta {\cal T}_8  $ for all $ | \beta | \leq \b_0 $ large enough, 
 is a small, still variable  coefficient operator,  
which is regularizing at a sufficiently high order, and  satisfies tame estimates 
in Sobolev spaces.

\item {\bf KAM reducibility}. In Section \ref{coniugio cal L omega} we restrict the operator in \eqref{op:redu1} to $H_{\mathbb S_+}^\bot$ and in Section \ref{sec: reducibility} we implement an iterative
diagonalization scheme 
to reduce quadratically the  size of the perturbation, completing the conjugation
of  $ {\cal L }_\omega $ to  a diagonal, constant coefficient system of the form 
\be\label{final-lin}
\Dom  + \ii {\rm Op} (\mu_j  )
\ee
where
$ \mu_j =  {\mathtt m}_{\frac12} |j|^{\frac12} \tanh^{\frac12}(\h |j|) +  r (j) + \tilde r(j) $ are real and 
$ \tilde r(j)  $ are small.
%The numbers
%$ \ii \mu_j $ are the perturbed Floquet exponents of the quasi-periodic solution. 
\end{enumerate}
We underline that all the  transformations performed in Sections 
\ref{linearizzato siti normali}-\ref{sec: reducibility}
are quasi-periodically-time-dependent changes of variables acting in phase spaces
of functions of $ x $ (quasi-periodic Floquet operators).
Therefore, they preserve 
the dynamical system structure of the conjugated  linear operators.

All these changes of variables are bounded 
and satisfy tame  estimates between Sobolev spaces. As a consequence, the estimates that we 
shall obtain inverting the final operator \eqref{final-lin} directly provide good tame estimates for the inverse of 
the operator ${\cal L}_\om$ in \eqref{representation Lom}. 

We also note that the original system $ {\cal L} $ is reversible and even 
and that all the transformations that we perform are 
reversibility preserving and even. The preservation of these properties ensures 
that  in the final system  \eqref{final-lin} the $ \mu_j $ are real valued. Under this respect, 
the linear stability of the quasi-periodic standing wave solutions proved in Theorem \ref{thm:main0}
is obtained as a consequence of the reversible nature of the water waves equations. 
We could also preserve the Hamiltonian nature of $ {\cal L} $ performing symplectic transformations, but it would be more complicated.

\begin{remark}\label{rem:GC} {\bf (Comparison with the gravity-capillary linearized PDE)}
With respect to the gravity capillary water waves in infinite depth 
in \cite{Alaz-Bal}, \cite{BertiMontalto}, 
the reduction in decreasing orders of the linearized operator is completely different. 
The linearized operator in the gravity-capillary case is like 
$$ 
\om \cdot \pa_\vphi + \ii |D_x|^\frac32 + V(\vphi, x) \pa_x  \, , 
$$
the term $V \pa_x$ is a lower order perturbation of $|D_x|^{\frac32}$, 
and it can be reduced to constant coefficients by conjugating the operator 
with a ``semi-Fourier Integral Operator'' $A$ of type $(\frac12, \frac12)$ 
(like in \cite{Alaz-Bal} and \cite{BertiMontalto}):
the commutator of $|D_x|^{\frac32}$ and $A$ produces a new operator of order $1$, 
and one chooses appropriately the symbol of $A$  for the reduction of $V \pa_x$. 
Instead, in the pure gravity case we have a linearized operator of the type
\[
\om \cdot \pa_\ph + \ii |D_x|^{\frac12} + V(\ph,x) \pa_x
\]
where the term $V \pa_x$ is a  \emph{singular} perturbation of $  \ii |D_x|^{\frac12} $. %   see \eqref{def:L0-intro} below. 
The commutator between $|D_x|^{\frac12}$ and 
% any bounded linear transformation 
any bounded pseudo-differential operator
produces operators of order $\leq 1/2$, which do not interact with $V \pa_x$. 
Hence one uses the commutator with $\om \cdot \pa_\ph$ 
(which is the leading term of the unperturbed operator)
to produce operators of order 1 that cancel out $V \pa_x$. 
This is why 
our first task is to straighten the first order vector field \eqref{highest order trasporto}, which corresponds to a time quasi-periodic transport operator. 
Furthermore, the fact that the unperturbed linear operator is $\sim |D|^{\frac12}$, unlike $\sim |D|^{\frac32}$, also affects the conjugation analysis of the lower order operators, where the contribution of the commutator with $\ompaph$ is always of order higher than the commutator with $|D_x|^{\frac12}$.
As a consequence, in the procedure of reduction of the symbols to constant coefficients in Sections \ref{sec:primo semi-FIO}-\ref{sezione descent method}, we remove first their dependence on $\ph$, and then their dependence on $x$. 
We also note that in \cite{BertiMontalto}, since the second order Melnikov conditions do not lose space derivatives, there is no need to perform such reduction steps at negative orders before starting with the KAM reducibility algorithm.
%in \cite{Alaz-Bal} this reduction stops at negative order $-3/2$  
%(which, for time periodic solutions, is sufficient for using Neumann series). 
\end{remark}

We now explain in details the steps of the conjugation of the quasi-periodic linear operator \eqref{LINOP} described in the items 1 and 2 above.
We underline that all the coefficients of the linearized operator 
$ {\cal L } $ in \eqref{LINOP} are $ {\cal C}^\infty $ in $ (\vphi, x) $ because 
each approximate solution $  (\eta (\vphi, x), \psi (\vphi, x)) $ 
at which we linearize along the Nash-Moser iteration 
is a  trigonometric  polynomial in $ (\vphi, x) $ (at each step we apply the projector
$ \Pi_n $ defined in \eqref{truncation NM}) and the water waves vector field is analytic. 
This allows us to work in the usual framework of  $ {\cal C}^\infty $ pseudo-differential symbols, as recalled in Section \ref{sec:pseudo}. 

\medskip 

\noindent
\textbf{1. Linearized good unknown of Alinhac.}
The first step is to introduce in Section \ref{sez:Alinhac} the linearized 
good unknown of Alinhac, as in \cite{Alaz-Bal} and \cite{BertiMontalto}. 
This is indeed the same change of variable introduced by Lannes \cite{Lannes} (see also
\cite{LannesLivre}) for proving energy estimates for the local existence theory. 
Subsequently, the nonlinear good unknown of Alinhac has been introduced by Alazard-M\'etivier \cite{AlM},
see also \cite{ABZ2}-\cite{AlDe1} to perform the paralinearization of the Dirichlet-Neumann operator.  
In these new variables, the linearized operator \eqref{LINOP} becomes the more symmetric operator 
(see \eqref{mZ mL0})
\be\label{def:L0-intro}
\mL_0 = \om \cdot \pa_\vphi + \begin{pmatrix}
\partial_x V   &  - G(\eta) \\ 
a  &  V \partial_x 
\end{pmatrix}
 = \om \cdot \pa_\vphi + \begin{pmatrix}
V \partial_x    &  0 \\ 
0 &  V \partial_x 
\end{pmatrix} +
\begin{pmatrix}
 V_x   & - G(\eta) \\ 
a  &  0
\end{pmatrix},
\ee
where the Dirichlet-Neumann operator admits the expansion 
$$
G(\eta) = |D| \tanh(\h |D|) +  \mR_G
$$ 
and $\mR_G $ is an $ OPS^{-\infty} $ smoothing operator. 
In Appendix \ref{subDN} we provide a self-contained 
proof of such  a representation. 
%by transforming
%the elliptic problem \eqref{BoundaryPr},  which is defined in the variable fluid domain 
% $  \{ - {\mathtt h} \leq y \leq \eta (x) \} $,  
%into the elliptic problem \eqref{BVP-new}, which is defined % \eqref{fixed-strip}
%on the straight strip $ \{  - {\mathtt h} - c \leq Y \leq 0 \} $ and  can be solved by an explicit integration. 
We cannot directly use a result already existing in the literature (for the Cauchy problem) because we have to provide tame estimates for the action of $G(\eta)$ on Sobolev spaces of  time-space variables $(\vphi, x)$ and to control its smooth dependence with respect to the parameters $(\omega, \mathtt h)$.  We can neither directly apply the corresponding result of \cite{BertiMontalto}, which is given in the case $\mathtt h = + \infty$. % in order to deal with the finite bottom. 

Notice that the first order transport operator $ V \pa_x $ in \eqref{def:L0-intro} is a singular perturbation of $\mL_0$ evaluated at $(\eta,\psi) = 0$, i.e. 
$\om \cdot \pa_\vphi 
+ \big( \begin{smallmatrix} 0   & - G(0) \\ 1  &  0 \end{smallmatrix} \big)$.

\medskip

\noindent
\textbf{2. Straightening the first order vector field $\ompaph + V(\ph,x) \pa_x $.} 
The next step is  to conjugate the variable coefficients vector field 
(we regard equivalently a vector field 
as a differential operator)
\begin{equation}\label{highest order trasporto}
\ompaph + V(\ph,x) \pa_x 
\end{equation}
 to the constant coefficient vector field $\ompaph $ on the torus $ 
 \T^\nu_{\vphi} \times \T_x $
 for $ V (\vphi, x) $ small. 
This a perturbative problem of rectification of a close to constant vector field on a torus, 
which  is a classical small divisor problem.
For perturbation of a Diophantine vector field this problem was solved at the beginning of KAM theory, we refer e.g.\ to  
\cite{Zehnder} and references therein.
Notice  that, despite the fact that $  \om \in \R^\nu $ is Diophantine,  the constant vector field $ \om \cdot \pa_\vphi  $
is resonant on the higher dimensional torus $ \T^{\nu}_{\vphi} \times \T_x $. 
We exploit in a crucial way the {\it symmetry} induced by the {\it reversible} structure of the water waves equations, i.e.  $ V(\vphi, x ) $ is odd in $ \vphi $, 
to prove that
 it is possible to conjugate $ \ompaph + V(\ph,x) \pa_x  $ to the constant vector field $ \ompaph  $ 
 without changing the frequency $ \omega $.
 
From a functional point of view we have to solve a linear transport equation which depends on time in quasi-periodic way, 
see equation \eqref{transport}. Actually we solve equation \eqref{1202.11}  
for the inverse diffeomorphism.
This problem  amounts to prove that all the solutions 
of the quasi periodically time-dependent scalar characteristic 
equation $ \dot x = V(\om t, x )$ are quasi-periodic in time with frequency $ \omega $, 
see Remark \ref{rem:geo}, \cite{PlTo}, \cite{IPT} and \cite{Mo67}. 
We solve this problem in Section \ref{sec: change-transport equation} 
using a Nash-Moser implicit function theorem.
Actually, after having inverted the linearized operator at an approximate solution (Lemma \ref{lemma:inv transport}),  we apply 
the Nash-Moser-H\"ormander Theorem  \ref{thm:NMH}, proved in Baldi-Haus \cite{Baldi-Haus}.
%The  main advantage of this approach is to provide in Theorem \ref{thm:nonlinear transport}
%the optimal higher order regularity estimates \eqref{0811.3} of the solution in terms of $ V $.
We cannot directly use already existing results for equation \eqref{1202.11}
because we have to prove tame estimates and Lipschitz dependence of the solution with respect to the approximate torus, as well as its smooth dependence with respect to the parameters $(\om,\h)$,  see Lemmata \ref{lemma:uV Lip}-\ref{stime trasformazione cal B}.

We remark that, when 
searching for time periodic solutions as in \cite{IPT}, \cite{PlTo}, the corresponding transport equation  
is not a small-divisor problem and has been solved
in \cite{PlTo} by a direct ODE analysis. 

In Lemma \ref{lem:stime coefficienti cal L1} we apply this change of variable to the whole operator $ \mL_0 $ in \eqref{def:L0-intro}, obtaining
the new conjugated system (see \eqref{mL1})
$$
\mL_1 
= \ompaph + \begin{pmatrix} 
  a_1 &  - a_2 |D| T_\h + {\cal R}_1 \\ 
a_3 & 0 
\end{pmatrix}, 
\qquad  T_\h := \tanh(\h |D|),  
$$
where the remainder $ {\cal R}_1 $ is  in $ OPS^{-\infty} $. 

\medskip

\noindent
\textbf{3. Change of the space variable.}
In Section \ref{sezione nuova cambio di variabile} 
we introduce a change of variable induced by a diffeomorphism of $ \T_x $ of the form
(independent of $ 	\vphi $)
\be\label{diffeo:intro}
y = x + \alpha (x)  \qquad \Leftrightarrow \qquad x = y + \breve \alpha (y) \, . 
\ee
Conjugating $ {\cal L}_1$  by the change of variable 
$ u (x) \mapsto u (x + \a(x) )$, we obtain an operator of the same form 
$$
{\cal L}_2 = \omega \cdot \partial_\vphi + \begin{pmatrix}
a_4 & - a_5 |D| T_{\mathtt h} + {\cal R}_2 \\
a_6 & 0 
\end{pmatrix}\,,
$$
see \eqref{nuovo cal L2 prima}, where $ {\cal R}_2 $ is  in $ OPS^{-\infty} $, and  
% (see \eqref{coefficienti cal L2 nuovo}-\eqref{coefficienti cal L2 nuovo-a6}) 
the functions  $ a_5, a_6 $ are given by
$$ 
a_5 =\big[a_2(\vphi, x) (1 + \alpha_x(x)) \big]_{| x = y + \breve \alpha(y)} \, , \quad
 a_6 = a_3(\vphi, y + \breve \alpha(y))\, . 
 $$ 
We shall choose in Section \ref{sec:primo semi-FIO}  the function
$ \alpha (x) $ (see \eqref{equazione omologica per alpha}) in order to eliminate the  dependence on $  x $ from the 
time average $\langle a_7 \rangle_\ph (x)$ in  \eqref{eq:passaggio}-\eqref{def:a51}
of the coefficient of $|D_x|^{\frac12}$.
The advantage of 
introducing the diffeomorphism \eqref{diffeo:intro} at this step, rather than in Section \ref{sec:primo semi-FIO} where it is used, is that it is easier to study the conjugation under this change of variable of differentiation and multiplication operators, Hilbert transform, and integral operators in $ OPS^{-\infty }$, see Section \ref{sub:integral-op} 
(on the other hand, performing this transformation in Section \ref{sec:primo semi-FIO} would require delicate estimates of the symbols obtained after an Egorov-type analysis). 

\medskip

\noindent
\textbf{4. Symmetrization of the order $1/2$}. 
In Section \ref{sec:Symm}
we apply two simple conjugations with a Fourier multiplier and a multiplication operator, whose goal is to 
obtain a new operator of the form
$$
\mL_3 = \om \cdot \pa_\vphi + \begin{pmatrix} \breve a_4 &  - a_7 |D|^{\frac12} T_\h^{\frac12} \\ 
a_7  |D|^{\frac12} T_\h^{\frac12} & 0 \end{pmatrix} + \ldots \, ,
 %T_\h := \tanh(\h |D|) \, , 
$$
see \eqref{0104.2}-\eqref{0104.14},  up  to 
lower order operators.
The function $ a_7 $ is close to $ 1 $ and $ \breve a_4 $ is small in $ \e $, see
\eqref{stime an bn cn (3)}. 
Notice that the off-diagonal operators in $\mL_3$ are opposite to each other, unlike in $\mL_2$.
%The novelty of $\mL_3$ with respect to $\mL_2$ is that the off diagonal operators are opposite with respect to each other.  
Then, in the complex unknown $ h = \eta + \ii \psi $, the first component of such an operator reads
\begin{equation}\label{equaz h bar h}
(h, \bar h) \mapsto
\ompaph h + \ii a_7 |D|^{\frac12} T_\h^{\frac12} h 
+ a_8 h +  P_5 h  + Q_5 {\bar h} 
\end{equation}
(which corresponds to \eqref{0104.21} neglecting the projector $ \ii \Pi_0 $)
where $ P_5 (\vphi)  $ is a  $ \vphi$-dependent 
families of pseudo-differential operators of order $ - 1/ 2 $,  
and $ Q_5 (\vphi) $  of order $ 0 $. 
We shall call the former operator
``diagonal",   and the latter  ``off-diagonal", 
with respect to the variables $ (h, \bar h ) $.

In Sections \ref{sec:10}-\ref{sezione descent method} we perform the reduction to constant coefficients of \eqref{equaz h bar h} up to smoothing operators, dealing separately with the diagonal and off-diagonal operators. 

\medskip

\noindent
\textbf{5. Symmetrization of the lower  orders}.
In Section \ref{sec:10} we reduce the off-diagonal term $ Q_5 $ to 
a pseudo-differential operator with  very negative order, 
i.e.\  we conjugate the above operator to 
another one of the form (see Lemma \ref{Lemma finale decoupling}) 
\be\label{op:dopo-deco}
(h, \bar h) \mapsto
\ompaph h + \ii a_7 (\vphi, x) |D|^{\frac12} T_\h^{\frac12} h 
+ a_8 h +  P_6 h  + Q_6 {\bar h} \,, 
\ee
where 
$ P_6 $ is in $ OPS^{- \frac12 } $ and 
$ Q_6 \in OPS^{-M } $ for a constant $ M $ large enough 
fixed in Section \ref{sec: reducibility}, in view of the reducibility scheme. 

\medskip

\noindent
\textbf{6. Time and space reduction at the order $1/2$}.
In Section \ref{sec:primo semi-FIO} we eliminate
the $ \vphi$- and the  $ x $-dependence from the coefficient of the 
 leading operator $ \ii a_7 (\vphi, x) |D|^{\frac12} T_\h^{\frac12} $. 
We conjugate  the  operator \eqref{op:dopo-deco} by 
the %  transformation defined  as the 
time-$ 1$ flow of the pseudo-PDE
$$
\partial_\t u = \ii \beta(\vphi, x) |D|^{\frac12} u  
$$
where $ \beta (\vphi, x) $ is a small function to be chosen. 
This kind of transformations 
-- which are ``semi-Fourier integral operators'', 
namely pseudo-differential operators of type $(\frac12, \frac12)$ 
in H\"ormander's notation -- has been introduced in \cite{Alaz-Bal} 
and studied as flows in \cite{BertiMontalto}. 

Choosing appropriately the functions $ \beta (\vphi, x) $ and $ \alpha (x) $ 
(introduced in Section \ref{sezione nuova cambio di variabile}), see formulas
\eqref{defbetaSFIO} and \eqref{equazione omologica per alpha}, 
the final outcome is a linear operator of the form, 
see \eqref{cal LM (1) Egorov},  
\begin{equation} \label{quella con P7}
(h, \bar h) \mapsto
\omega \cdot \partial_\vphi h  + \ii \mathtt m_{\frac12} |D|^{\frac12} T_\h^{\frac12} h 
+ (a_8 + a_9 {\cal H}) h  + P_7 h + {\cal T}_7 (h, \bar h) \,,
\end{equation}
where $\mH$ is the Hilbert transform.  
This linear operator has the constant coefficient $ \mathtt m_{\frac12} \approx 1 $ at the order $1/2$, while $ P_7 $ is in $ OPS^{-1/2} $
and the operator $ {\cal T}_7 $ is small, smoothing and satisfies tame estimates in Sobolev spaces, see \eqref{stima tame cal T (1)}. 
%Notice that we first remove the dependence on $\ph$, and then the dependence on $x$, 
%from the coefficient of the operator $\ii a_7(\ph,x) |D|^{\frac12} T_\h^{\frac12} $,
%in accordance with the fact that the contribution coming from the commutator with $\ompaph$ 
%is of higher order than the commutator with $|D|^{\frac12}T_\h^{\frac12} $.

\medskip

\noindent
\textbf{7. Reduction of the lower orders.} 
In Section \ref{sezione descent method} we further diagonalize 
the linear operator in \eqref{quella con P7}, 
reducing it to constant coefficients  up to regularizing smoothing operators of very negative order $ | D |^{-M} $. This step,
based on standard pseudo-differential calculus, 
is not needed in \cite{BertiMontalto}, because the second order Melnikov conditions in \cite{BertiMontalto} do not lose space derivatives. 
We apply
an iterative sequence of pseudo-differential transformations that eliminate first the $ \vphi $- and then the $ x $-dependence of the diagonal symbols.  The final system has the form 
\be\label{op:QuasiF}
(h, \bar h) \mapsto  
\omega \cdot \partial_\vphi  h 
+ \ii {\mathtt m}_{\frac12} |D|^{\frac12} T_\h^{\frac12} h + 
\ii r (D) h  +    {\cal T}_8 (\vphi) (h, \bar h)
\ee
where the constant Fourier multiplier  $ r (\xi )  $ is real, even $ r ( \xi) = r (- \xi ) $,
it satisfies (see \eqref{stima code - 1/2 inizio riducibilita})
$$
\sup_{j \in \Z}|j|^{\frac12} |r_j|^{k_0, \gamma} \lesssim_M \e \gamma^{- (2M + 1)}\, , 
$$
and the variable coefficient operator $  {\cal T}_8  (\vphi) $ is regularizing and 
satisfies tame estimates, see more precisely \eqref{stima tame cal T 8}. 
We also remark that the operator \eqref{op:QuasiF}  is reversible and even, since all the previous
transformations that we performed are reversibility preserving and even.

At this point the procedure of diagonalization of $\mL$ up to smoothing operators is complete. 
Thus, in Section \ref{coniugio cal L omega}, restricting the operator \eqref{op:QuasiF} to $ H_{\Splus}^\bot $, we obtain the reduction of $\mL_\om$ up to smoothing remainders. We are now ready to begin the KAM reduction procedure.

\medskip

\noindent
\textbf{8. KAM reducibility.} 
In order to decrease quadratically the size of the resulting perturbation $ {\cal R}_0 $ (see \eqref{defL0-red})
we apply the KAM diagonalization iterative scheme of Section \ref{sec: reducibility},
which converges because  the operators
\be\label{operatori-tames}
\langle D \rangle^{\fm + \mathtt b} {\cal R}_0 \langle D \rangle^{\fm + \mathtt b +  1}, \quad \partial_{\vphi_i}^{s_0 +  {\mathtt b}} \langle D \rangle^{\fm + \mathtt b} {\cal R}_0 \langle D \rangle^{\fm + \mathtt b +  1}\,  , \
i = 1, \ldots, \nu \, , 
 \ee  
satisfy tame estimates for some $ \mathtt b := \mathtt b (\tau, k_0) \in \N $
and $ \fm := \fm (k_0) $ that are large enough (independently of $s$), see Lemma \ref{lem:tame iniziale}.
Such conditions hold under the assumption that % having assumed that 
$ M $ (the order of regularization of the remainder)
is chosen large enough as in \eqref{relazione mathtt b N} 
(essentially $ M = O(\fm + {\mathtt b})$).  
This is the property 
that compensates, along the KAM iteration, the loss of derivatives in $ \vphi  $ and $  x $ 
produced by the small divisors in the second order Melnikov non-resonance conditions.
Actually, for the construction of the quasi-periodic solutions, it is sufficient to prove the almost-reducibility of the linearized operator, in the sense that the remainder $\mR_n$ in  Theorem \ref{Teorema di riducibilita} is not zero but it is of order $ O( \e \gamma^{-2(M+1)} N_{n-1}^{- \mathtt a})$, which can be obtained imposing only the finitely many Diophantine conditions \eqref{Cantor set}, \eqref{Omega nu + 1 gamma}.% This is sufficient for the construction of the quasi-periodic solutions.  
%The expression ``almost-diagonalization" refers to the fact that in Theorem \ref{Teorema di riducibilita} the remainders $ {\cal R}_n $  that are left in \eqref{cal L infinito} 
%are not zero, but they are as small as $ O( \e \gamma^{-2(M+1)} N_{n-1}^{- \mathtt a})$, because we only require
%the finitely many Diophantine conditions \eqref{Cantor set}, \eqref{Omega nu + 1 gamma}.

The big difference of the
 KAM reducibility scheme of Section \ref{sec: reducibility} with respect to the one 
developed in \cite{BertiMontalto} is that the second order 
Melnikov non-resonance conditions that we impose are very weak, see \eqref{Omega nu + 1 gamma}, 
in particular they lose regularity, not only in the $ \vphi $-variable, but also in the space variable $ x $.
For this reason we apply at each iterative step a smoothing procedure also in the space variable
(see the Fourier truncations $|\ell|, |j-j'| \leq N_{\vnu-1}$ 
in \eqref{Omega nu + 1 gamma}).

%We underline that this is a general approach to KAM reducibility in cases where the second order Melnikov conditions lose space derivatives.

\medskip

After the above almost-diagonalization of the linearized operator we almost-invert it, by 
imposing the first order Melnikov non-resonance conditions in \eqref{prime di melnikov}, see Lemma \ref{lem:first-Mel}. Since all the changes of variables that we performed in the diagonalization process 
satisfy tame estimates in Sobolev spaces, we finally 
conclude the existence of an almost  inverse of ${\cal L}_\omega $ %  of the linearized operator % $ {\cal L} $ 
which satisfies tame estimates, see Theorem \ref{inversione parziale cal L omega}.  

At this point the proof of the Nash-Moser Theorem \ref{main theorem}, given in Section \ref{sec:NM}, follows in a usual way, in the same setting of \cite{BertiMontalto}. 

\medskip

\noindent
\textbf{Notation.}
Given a function $ u(\vphi, x ) $ we write that it is $ \even(\ph) \even(x)$ 
if it is even in $ \vphi $ for any $ x $ and, separately, even in $ x $ for any $ \vphi $.   
With similar meaning we say that $ u (\vphi, x) $ is 
 $ \even(\ph) \odd(x)$,  $ \odd(\ph) \even(x)$ and $ \odd(\ph) \even(x)$.

The notation $ a \lesssim_{s, \a, M} b $ means
that $a \leq C(s, \a, M ) b$ for some constant $C(s, \a, M ) > 0$ depending on 
the Sobolev index $ s $ and the constants $ \a, M $.  Sometimes, along the paper, we omit to write the dependence $ \lesssim_{s_0, k_0} $ with respect to $ s_0, k_0 $, 
 because $ s_0 $ (defined in \eqref{def:s0}) and $ k_0 $
 (determined  in Section \ref{sec:degenerate KAM}) are considered as fixed constants. 
Similarly, the set $\Splus$ of tangential sites is considered as fixed along the paper.

\section{Functional setting}

\subsection{Function spaces}
\label{subsec:function spaces}

In the paper we will use Sobolev norms for real or complex functions $u(\om, \h, \ph, x)$, 
$(\ph,x) \in \T^\nu \times \T$, depending on parameters $(\om,\h) \in F$ 
in a Lipschitz way together with their derivatives in the sense of Whitney, 
where $F$ is a closed subset of $\R^{\nu+1}$. 
We use the compact notation $\lm := (\om,\h)$ to collect the frequency $\om$ and the depth $\h$ into a parameter vector. 

We use the multi-index notation: if 
$ k = ( k_1, \ldots , k_{\nu+1}) \in \N^{\nu+1} $ we denote $|k| := k_1 + \ldots + k_{\nu+1}$ and $k! := k_1!  \cdots k_{\nu+1}!$ and if $ \lambda = (\lambda_1, \ldots, \lambda_{\nu+1}) \in \R^{\nu+1}$, we denote  the derivative $ \pa_\lambda^k := \pa_{\lambda_1}^{k_1} \ldots \pa_{\lambda_{\nu+1}}^{k_{\nu+1}} $ and $\lambda^k := \lambda_1^{k_1}  \cdots \lambda_{\nu+1}^{k_{\nu+1}} $. Recalling that $\| \ \|_s$ denotes the  norm of the Sobolev space $H^s(\T^{\nu+1}, \C) = H^s_{(\ph,x)}$ introduced in \eqref{unified norm}, we now define  the ``Whitney-Sobolev'' norm $\| \cdot \|_{s,F}^{k+1,\g}$.
\begin{definition} {\bf (Whitney-Sobolev functions)} \label{def:Lip F uniform}
Let $F$ be a closed subset of $\R^{\nu+1}$.  
Let $k \geq 0$ be an integer,  $\g \in (0,1]$, and $s \geq s_0 > (\nu+1)/2$. 
We say that a function $u : F \to H^s_{(\ph,x)}$ 
belongs to $\Lip(k+1,F,s,\g)$ if there exist functions $u^{(j)} : F \to H^s_{(\ph,x)}$, $j \in \N^\nu$, $0 \leq |j| \leq k$ with $u^{(0)} = u$, and a constant $M > 0$ 
such that, if $  R_j(\lm,\lm_0) := R_j^{(u)}(\lm,\lm_0) $ is defined by 
\begin{equation} \label{16 Stein uniform}
u^{(j)}(\lm) 
= \sum_{\ell \in \N^{\nu+1} : 
|j+\ell| \leq k} \frac{1}{\ell!} \, u^{(j+\ell)}(\lm_0) \, (\lm - \lm_0)^\ell 
+ R_j(\lm,\lm_0), \quad \lm, \lm_0 \in F, 
\end{equation} 
%(recall the multi-index notation \eqref{multi-index-notation})
then 
\begin{equation} \label{17 Stein uniform}
\g^{|j|} \| u^{(j)}(\lm) \|_s \leq M, \quad 
\g^{k+1} \| R_j(\lm, \lm_0) \|_s \leq M |\lm - \lm_0|^{k+1 - |j|} \quad 
\forall \lm, \lm_0 \in F, \ |j| \leq k.
\end{equation} 
An element of $\Lip(k+1,F,s,\g)$ 
is in fact the collection $\{ u^{(j)} : |j| \leq k \}$. 
The norm of $u \in \Lip(k+1,F,s,\g)$ is defined as
\begin{equation} \label{def norm Lip Stein uniform}
\| u \|_{s,F}^\kug 
:= \| u \|_s^\kug
:= \inf \{ M > 0 : \text{\eqref{17 Stein uniform} holds} \}.
\end{equation}
If $ F = \R^{\nu+1} $ by $ \Lip(k+1,\R^{\nu+1},s, \g) $  we shall mean the % linear 
space of the functions $ u = u^{(0)} $ for which
there exist $ u^{(j)} = \pa_{\lm}^j u $, $ |j| \leq k $,  
satisfying  \eqref{17 Stein uniform}, with the same norm \eqref{def norm Lip Stein uniform}.
\end{definition}

We make some remarks.
\begin{enumerate}
\item 
If $ F = \R^{\nu+1} $, and $ u \in \Lip(k+1,F,s,\g) $
the $ u^{(j)} $, $ | j | \geq 1 $, are uniquely determined as the partial derivatives 
$ u^{(j)} = \pa_{\lm}^j u $, $ | j | \leq k $,  of $ u = u^{(0)} $. Moreover all the derivatives $ \pa_{\lm}^j u $, $ | j | = k $
are Lipschitz.  Since $ H^s $ is a Hilbert space we have that % the space of functions 
$ \Lip(k+1,\R^{\nu+1},s,\g) $ coincides with the Sobolev space
$ W^{k+1, \infty} (\R^{\nu + 1}, H^s) $. 

\item
The Whitney-Sobolev norm  of $ u $ in \eqref{def norm Lip Stein uniform} is 
equivalently given by 
\begin{equation} \label{def -equiv}
\| u \|_{s,F}^\kug := \| u \|_s^\kug
= \max_{|j| \leq k} 
\Big\{ \g^{|j|} \sup_{\lambda \in F}  \| u^{(j)}(\lm) \|_s, \g^{k+1} 
\sup_{\lambda \neq \lambda_0}  \frac{\| R_j(\lm, \lm_0) \|_s}{|\lm - \lm_0|^{k+1 - |j|}}  \Big\} \, .
\end{equation}
%\item
%The exponent of $\g$ in \eqref{17 Stein uniform} gives the number of ``derivatives'' of $u$ that are involved in the Taylor expansion (taking into account that 
%in the remainder there is one derivative more than in the Taylor polynomial);
%on the other hand the exponent of $|\lm-\lm_0|$ gives the order of the Taylor expansion of $u^{(j)}$ with respect to $\lm$. This is the reason for the difference of $|j|$ between the two exponents. 
%The factor $\g$ is normalized by the rescaling \eqref{resca}.
\end{enumerate}

% The following properties of the Whitney-Sobolev functions are proved, 
% thanks to the Whitney extension Theorem \ref{thm:WET}, 
% by the corresponding  ones valid in $ W^{k+1, \infty} (\R^{\nu + 1}, H^s) $. 
Theorem \ref{thm:WET} and \eqref{Wg} 
provide an extension operator which associates 
to an element $u \in \Lip(k+1,F,s,\g)$ 
an extension $\tilde u \in \Lip(k+1,\R^{\nu+1},s,\g) $. 
As already observed, the space $\Lip(k+1,\R^{\nu+1},s,\g) $ 
coincides with $ W^{k+1, \infty} (\R^{\nu + 1}, H^s)$, 
with equivalence of the norms (see \eqref{0203.1})
$$
\| u \|_{s, F}^\kug  \sim_{\nu, k} 
\| \tilde u \|_{W^{k+1,\infty, \gamma}(\R^{\nu+1}, H^s)} := 
\sum_{|\a| \leq k+1} \g^{|\a|} \| \pa_\lm^\a \tilde u \|_{L^\infty(\R^{\nu+1}, H^s)}  \,.
$$
By Lemma \ref{lemma:2702.1}, 
the extension $\tilde u$ is independent of the Sobolev space $H^s$.

We can identify any element $u \in \Lip(k+1, F, s, \g)$
(which is a collection $u = \{ u^{(j)} : |j| \leq k\}$) 
with the equivalence class of functions $f \in W^{k+1, \infty}(\R^{\nu+1}, H^s) / \! \! \sim$  
with respect to the equivalence relation $f \sim g$ when $\pa_\lm^j f(\lm) = \pa_\lm^j g(\lm)$
for all $\lm \in F$, for all $|j| \leq k+1$.

\smallskip

For any $N>0$, we introduce the smoothing operators
\begin{equation}\label{def:smoothings}
(\Pi_N u)(\ph,x) := \sum_{\la \ell,j \ra \leq N} u_{\ell j} e^{\ii (\ell\cdot\ph + jx)} \qquad
\Pi^\perp_N := {\rm Id} - \Pi_N.
\end{equation}

\begin{lemma} {\bf (Smoothing)} \label{lemma:smoothing}
Consider the space $\Lip(k+1,F,s,\g)$ defined in Definition \ref{def:Lip F uniform}. 
The smoothing operators $\Pi_N, \Pi_N^\perp$ satisfy the estimates
\begin{align}
\| \Pi_N u \|_{s}^\kug 
& \leq N^\alpha \| u \|_{s-\alpha}^\kug\, , \quad 0 \leq \a \leq s, 
\label{p2-proi} \\
\| \Pi_N^\bot u \|_{s}^\kug 
& \leq N^{-\alpha} \| u \|_{s + \alpha}^\kug\, , \quad  \a \geq 0.
\label{p3-proi}
\end{align}
\end{lemma}

\begin{proof}
See Appendix \ref{sec:U}.
\end{proof}

\begin{lemma} {\bf (Interpolation)}\label{lemma:interpolation}
Consider the space $\Lip(k+1,F,s,\g)$ defined in Definition \ref{def:Lip F uniform}. 

(i) Let $s_1<s_2$. Then for any $\theta\in(0,1)$ one has
\begin{equation}\label{2202.3}
\| u \|_s^{k+1, \gamma} \leq (\| u \|_{s_1}^{k+1, \gamma})^\theta 
(\| u \|_{s_2}^{k+1, \gamma})^{1 - \theta}\,, \quad s:=\theta s_1 + (1-\theta) s_2\,. 
\end{equation}

(ii) Let $a_0, b_0 \geq0$ and $p,q>0$. For all $\epsilon>0$, there exists a constant $C(\epsilon):= C(\epsilon,p,q)>0$, which satisfies $C(1)<1$, such that
\begin{equation}\label{2202.2}
\| u \|^{k+1,\g}_{a_0 + p} \| v \|^{k+1,\g}_{b_0 + q} 
\leq \epsilon \| u \|^{k+1,\g}_{a_0 + p +q} \| v \|^{k+1,\g}_{b_0} 
+ C(\epsilon) \| u \|^{k+1,\g}_{a_0} \| v \|^{k+1,\g}_{b_0 + p +q}\ .
\end{equation}
\end{lemma}

\begin{proof}
See Appendix \ref{sec:U}.
\end{proof}

\begin{lemma}{\bf (Product and composition)}
\label{lemma:LS norms}
Consider the space $\Lip(k+1,F,s,\g)$ defined in Definition \ref{def:Lip F uniform}.
For all $ s \geq s_0 > (\nu + 1)/2 $, we have
\begin{align}
\| uv \|_{s}^\kug
& \leq C(s, k) \| u \|_{s}^\kug \| v \|_{s_0}^\kug 
+ C(s_0, k) \| u \|_{s_0}^\kug \| v \|_{s}^\kug\,. 
\label{p1-pr}
\end{align}
Let $ \| \b \|_{2s_0+1}^\kug \leq \d (s_0, k) $ small enough. 
Then the composition operator 
$$
\mB : u \mapsto \mB u, \quad  
(\mB u)(\ph,x) := u(\ph, x + \b (\ph,x)) \, , 
$$
satisfies the following tame estimates: for all $ s \geq s_0$, 
\be\label{pr-comp1}
\| {\cal B} u \|_{s}^\kug \lesssim_{s, k} \| u \|_{s+k+1}^\kug 
+ \| \b \|_{s}^\kug \| u \|_{s_0+k+2}^\kug \, .
\ee
Let $ \| \b \|_{2s_0+k+2}^\kug \leq \d (s_0, k) $ small enough.  The function $ \breve \beta $ defined by 
the inverse diffeomorphism 
$ y = x + \beta (\vphi, x) $ if and only if $ x = y + \breve \b ( \vphi, y ) $,  
satisfies 
\be\label{p1-diffeo-inv}
\| \breve \beta \|_{s}^\kug \lesssim_{s, k}  \| \b \|_{s+k+1}^\kug \, . 
\ee
\end{lemma}

\begin{proof} See Appendix \ref{sec:U}. % and Lemma 2.30 in \cite{BertiMontalto}.  
\end{proof}

If  $ \om $ belongs to the set of Diophantine  vectors $ \mathtt{DC}(\gamma, \tau) $, where
\begin{equation}\label{DC tau0 gamma0}
\mathtt{DC}(\gamma, \tau) := \Big\{ \omega \in \mathtt \R^\nu : |\omega \cdot \ell| \geq \frac{\gamma}{|\ell|^{\tau}} \ \  \forall \ell \in \Z^\nu \setminus  \{ 0 \} \Big\} \, ,
\end{equation}
the equation $\ompaph v = u$, where $u(\ph,x)$ has zero average with respect to $ \vphi $, 
has the periodic solution 
\begin{equation}\label{def:ompaph}
(\om \cdot \pa_\vphi )^{-1} u := \sum_{\ell \in \Z^{\nu} \setminus \{0\}, j \in \Z} 
\frac{ u_{\ell, j} }{\ii \om \cdot \ell }e^{\ii (\ell \cdot \vphi + j x )} \,.
% \quad u = \sum_{\ell \in \Z^{\nu} \setminus \{0\}, j \in \Z} u_{\ell, j} e^{\ii (\ell \cdot \vphi + j x )} \, . 
\end{equation}
For all $\om \in \R^\nu$ we define its extension
\begin{equation} \label{def ompaph-1 ext}
(\ompaph)^{-1}_{ext} u(\ph,x) := \sum_{(\ell, j) \in \Z^{\nu+1}} 
\frac{\chi( \om \cdot \ell \g^{-1} \langle \ell \rangle^{\t}) }{\ii \om \cdot \ell}\,
u_{\ell,j} \, e^{\ii (\ell \cdot \ph + jx)},
\end{equation}
where $\chi \in \mC^\infty(\R,\R)$ is an even and positive cut-off function such that 
\begin{equation}\label{cut off simboli 1}
\chi(\xi) = \begin{cases}
0 & \quad \text{if } \quad |\xi| \leq \frac13 \\
1 & \quad \text{if} \quad  \ |\xi| \geq \frac23\,,
\end{cases} \qquad 
\partial_\xi \chi(\xi) > 0 \quad \forall \xi \in \Big(\frac13, \frac23 \Big) \, . 
\end{equation}
Note that $(\ompaph)^{-1}_{ext} u = (\ompaph)^{-1} u$
for all $\om \in \mathtt{DC}(\gamma, \tau)$.

\begin{lemma} { \bf (Diophantine equation)}
\label{lemma:WD}
For all $u \in W^{k+1,\infty,\g}(\R^{\nu+1}, H^{s+\mu})$, we have 
\be \label{2802.2}
\| (\om \cdot \pa_\vphi )^{-1}_{ext} u \|_{s,\R^{\nu+1}}^\kug
\leq C(k) \g^{-1} \| u \|_{s+\mu, \R^{\nu+1}}^\kug, 
\qquad \mu := k+1 +  \t(k+2). 
\ee
Moreover, for $F \subseteq \mathtt{DC}(\g,\t) \times \R$ one has 
\be \label{Diophantine-1}
\| (\om \cdot \pa_\vphi )^{-1} u \|_{s,F}^\kug
\leq C(k) \g^{-1}\| u \|_{s+\mu, F}^\kug \,.  
\ee
\end{lemma}

\begin{proof} See Appendix \ref{sec:U}. 
\end{proof}

We finally state a standard Moser tame estimate for the nonlinear 
composition operator
$$
u(\vphi, x) \mapsto {\mathtt f}(u)(\vphi, x) := f(\vphi, x, u(\vphi, x)) \, . 
$$ 
Since the variables $ (\vphi, x) := y $ have the same role, 
we state it for a  generic Sobolev space  $ H^s (\T^d ) $.  

\begin{lemma}{\bf (Composition operator)} \label{Moser norme pesate}
Let $ f \in {\cal C}^{\infty}(\T^d \times \R, \C )$ and $C_0 > 0$. 
Consider the space $\Lip(k+1,F,s,\g)$ given in Definition \ref{def:Lip F uniform}. 
If $u(\lambda) \in H^s(\T^d, \R)$, $\lambda \in F$ is a family of Sobolev functions
satisfying $\| u \|_{s_0,F}^{k+1, \gamma} \leq C_0$, then, for all $ s \geq s_0 > (d + 1)/2 $, 
\begin{equation} \label{0811.10}
\| {\mathtt f}(u) \|_{s, F}^{k+1, \gamma} \leq C(s, k, f, C_0 ) ( 1 + \| u \|_{s, F}^{k+1, \gamma}) \,.
\end{equation} 
The constant $C(s,k,f, C_0)$ depends on $s$, $k$ and linearly on $\| f \|_{\mC^m(\T^d \times B)}$, where $m$ is an integer larger than $s+k+1$, and $B \subset \R$ is a bounded interval such that $u(\lm,y) \in B$ for all $\lm \in F$, $y \in \T^d$,
for all $\| u \|_{s_0,F}^{k+1, \gamma} \leq C_0$.
% If $\| u \|_{s_0} > 1$, then the first inequality in \eqref{0811.10} holds with a constant 
% $ C(s, f, \| u \|_{s_0})$ also depending on $\| u \|_{s_0}$. 
\end{lemma}

\begin{proof} See Appendix \ref{sec:U}. 
\end{proof}

\subsection{Linear operators}

Along the paper we consider $ \vphi $-dependent families of  linear operators
 $ A : \T^\nu \mapsto {\cal L}( L^2(\T_x))  $, 
$ \vphi \mapsto A(\vphi) $
acting on functions $ u(x) $ of the space variable $ x $, i.e.\ 
on subspaces of $ L^2 (\T_x) $, either real or complex valued. 
We also regard $ A $ as an operator (which for simplicity we denote by $A $ as well)
that acts on functions $ u(\varphi , x) $ of space-time, i.e.\  
we consider the corresponding operator 
$ A \in {\cal L}(L^2(\T^\nu \times \T ) )$ defined by
\begin{equation}\label{def azione toplitz u vphi x}
( A u) (\varphi , x) := (A(\varphi) u(\varphi, \cdot ))(x) \, .  
\end{equation}
We say that an operator $ A $ is {\it real} if 
it maps real valued functions into real valued functions. 

We represent a real operator acting on $ (\eta, \psi) \in L^2(\T^{\nu + 1}, \R^2) $ by a matrix 
\begin{equation}\label{cal R eta psi}
{\cal R} \begin{pmatrix}
\eta \\
\psi
\end{pmatrix} = \begin{pmatrix}
A & B \\
C & D
\end{pmatrix}
 \begin{pmatrix}
\eta \\
\psi
\end{pmatrix} 
\end{equation}
where $A, B, C, D$ are real operators acting on  the scalar valued components 
$ \eta, \psi \in L^2(\T^{\nu + 1}, \R)$.

The action of an operator $ A$ as in \eqref{def azione toplitz u vphi x} on a scalar function 
$ u  := u(\vphi, x ) \in L^2 (\T^\nu \times \T, \C) $,
that we expand in Fourier series as
\be\label{u-Fourier-expa}
u(\vphi, x ) =   \sum_{j \in \Z} u_{j} (\vphi) e^{\ii j x } =
\sum_{\ell \in \Z^\nu, j \in \Z}  u_{\ell,j}  e^{\ii (\ell \cdot \vphi + j x)}\,,
\ee
is 
\begin{equation}\label{matrice operatori Toplitz}
A u (\vphi, x) = \sum_{j , j' \in \Z} A_j^{j'}(\vphi) u_{j'}(\vphi) e^{\ii j x} = 
 \sum_{\ell \in \Z^\nu, j \in \Z} \, \sum_{\ell' \in \Z^\nu, j' \in \Z} A_j^{j'}(\ell - \ell') u_{\ell', j'} e^{\ii (\ell \cdot \vphi + j x)} \, . 
\ee
We shall identify an operator $ A $ with the matrix $ \big( A^{j'}_j (\ell- \ell') \big)_{j, j' \in \Z, \ell, \ell' \in \Z^\nu } $, which is T\"oplitz with respect to the index $\ell$. In this paper we always consider T\"oplitz operators as in 
\eqref{def azione toplitz u vphi x}, \eqref{matrice operatori Toplitz}. 
% Some of the abstract properties that we will use below hold also for more general linear operators. 

The matrix entries $A^{j'}_j (\ell- \ell')$ of a bounded operator  $A : H^s \to H^s$ (as in \eqref{matrice operatori Toplitz}) satisfy
\begin{equation} \label{1701.1}
\sum_{\ell, j} |A_j^{j'}(\ell - \ell')|^2 \langle \ell, j\rangle^{2s} 
\leq \| A \|_{\mL(H^s)}^2 \langle \ell', j' \rangle^{2s} \,, 
\ \quad \forall (\ell', j') \in \Z^{\nu+1} \, , 
\end{equation}
where $\| A \|_{\mL(H^s)} := \sup \{ \| Ah \|_s : \| h \|_s = 1 \}$ is  the operator norm 
(consider $h = e^{\ii (\ell', j') \cdot (\ph,x)}$).   

\begin{definition}\label{def:maj} 
Given a linear operator $ A $ as in \eqref{matrice operatori Toplitz} we define the operator
\begin{enumerate}
\item $ | A | $  {\bf (majorant operator)}  
whose matrix elements are $  | A_j^{j'}(\ell - \ell')| $,  
\item $ \Pi_N A $, $ N \in \N  $ {\bf (smoothed operator)}  
whose matrix elements are
\be\label{proiettore-oper}
 (\Pi_N A)^{j'}_j (\ell- \ell') := 
 \begin{cases}
 A^{j'}_j (\ell- \ell') \quad {\rm if} \quad   \langle \ell - \ell' ,j - j' \rangle \leq N \\
 0  \qquad \qquad \quad {\rm  otherwise} \, .
 \end{cases} 
\ee
We also denote $ \Pi_N^\bot := {\rm Id} - \Pi_N $, 
\item $ \langle \pa_{\vphi, x} \rangle^{b}  A $, $ b \in \R $,   % the operator 
whose matrix elements are  $  \langle \ell - \ell', j - j' \rangle^b A_j^{j'}(\ell - \ell') $.
\item $\partial_{\vphi_m} A(\vphi) = [\partial_{\vphi_m}, A] = \partial_{\vphi_m} \circ A - A \circ \partial_{\vphi_m}$ {\bf (differentiated operator)} whose matrix elements are $\ii (\ell_m - \ell_m') A_{j}^{j'}(\ell - \ell')$. 
\end{enumerate}
\end{definition}
Similarly the commutator $ [\pa_x, A ]$ is represented by the matrix 
with entries $ \ii (j - j') A^{j'}_j (\ell- \ell') $.

Given linear operators $ A $, $ B$ as in \eqref{matrice operatori Toplitz} we have that (see Lemma 2.4 in \cite{BertiMontalto}) 
\begin{align}\label{disuguaglianza importante moduli} 
\| \, | A + B| u \|_s \leq \| \, | A | \, \norma u \norma \, \|_s 
+ \| \, | B| \, \norma u \norma \, \|_s  \, , \quad
\| \, | A B| u \|_s \leq \| \,| A| | B| \, \norma u \norma \, \|_s \, , 
\end{align}
where, for a given a function  
$ u(\varphi , x) $ % \in L^2 (\T^\nu \times \T, \C) $ 
expanded in Fourier series as in \eqref{u-Fourier-expa},
we define the majorant function
\begin{equation}\label{funzioni modulo fourier}
\norma u \norma (\vphi, x) :=  \sum_{\ell \in \Z^\nu, j \in \Z} |u_{\ell, j}| e^{\ii (\ell \cdot \vphi + j x)} \, .
\end{equation}
Note that  the Sobolev norms of $ u $ and $ \norma u \norma $ are the same, 
i.e. 
\begin{equation}\label{Soboequals}
 \| u \|_s = \| \norma u \norma \|_s.
 \end{equation}

\subsection{Pseudo-differential operators}\label{sec:pseudo}

In this section we recall the main properties of pseudo-differential operators on the torus that we shall use in the paper, similarly to \cite{Alaz-Bal}, \cite{BertiMontalto}. 
Pseudo-differential operators on the torus may be seen as a particular case
of  the theory on $ \R^n $, as developed for example in \cite{Ho1}. 
% They can also be directly expressed through Fourier series, for which we refer to \cite{SV}.  

\begin{definition} \label{def:Ps2} {\bf ($\Psi {\rm DO}$)}
A linear operator $ A $ is called a \emph{pseudo-differential} operator 
of order $m$ % $ \leq m $ 
if its symbol $ a (x,j) $ is 
the restriction to $ \R \times \Z $ of a function $ a (x, \xi ) $ which is $ {\cal C}^\infty $-smooth on $ \R \times \R $,
 $ 2 \pi $-periodic in $ x $, and satisfies the inequalities
\be\label{symbol-pseudo2}
\big| \pa_x^\a \pa_\xi^\b a (x,\xi ) \big| \leq C_{\a,\b} \langle \xi \rangle^{m - \b} \, , \quad \forall \a, \b \in \N \, .
\ee
We call $ a(x, \xi ) $ the symbol of the operator $ A $, which we denote 
$$ 
A = {\rm Op} (a) = a(x, D) \, , \quad D := D_x := \frac{1}{\ii} \pa_x \, . 
$$ 
We denote by $ S^m $ 
the class of all the symbols $ a(x, \xi ) $ satisfying \eqref{symbol-pseudo2}, and by $ OPS^m $ 
the associated set of pseudo-differential  operators of order $ m $. 
We set $ OPS^{-\infty} := \cap_{m \in \R} OPS^{m} $.

For a matrix of pseudo differential operators 
\begin{equation}\label{operatori matriciali sezione pseudo diff}
A = \begin{pmatrix}
A_1  & A_2 \\
A_3 & A_4
\end{pmatrix}, \quad A_i \in OPS^m, \quad i =1, \ldots , 4
\end{equation}
we say that  $A \in OPS^m$.   
\end{definition}

When the symbol $ a (x) $ is independent of $ j $, the operator $ A = {\rm Op} (a) $ is 
the multiplication operator by the function $ a(x)$, i.e.\  $ A : u (x) \mapsto a ( x) u(x )$. 
In such a case we shall also denote $ A = {\rm Op} (a)  = a (x) $. 

We underline that we regard any operator $ {\rm Op}(a) $ as an operator 
acting only on $ 2 \pi $-periodic functions 
$ u (x)  = \sum_{j \in \Z} u_j e^{\ii j x }$ as 
$$
(Au) (x)  := {\rm Op}(a) [u] (x) 
:= {\mathop \sum}_{j \in \Z} a(x,j) u_j e^{\ii j x }  \, . 
$$
Along the paper we consider $ \vphi $-dependent pseudo-differential operators
$ (A u)(\vphi, x) = \sum_{j \in \Z} a(\vphi, x, j) u_j(\vphi) e^{\ii j x } $
where the symbol $ a(\vphi, x, \xi ) $ is $ {\cal C}^\infty $-smooth also in $ \vphi $. 
We still denote $ A := A(\vphi) = {\rm Op} (a(\vphi, \cdot) ) = {\rm Op} (a ) $. 

Moreover we consider pseudo-differential operators  $ A(\lambda) := {\rm Op}(a(\lambda, \vphi, x, \xi)) $ that are ${k_0}$ times differentiable with respect to a parameter
$ \lambda := (\omega, \h)$ in an open subset $ \Lambda_0 \subseteq \R^\nu \times [\h_1, \h_2] $. 
The regularity constant $ k_0 \in \N $ 
is fixed once and for all in Section \ref{sec:degenerate KAM}. 
Note that 
$ \partial_{\lambda}^k A = {\rm Op}(\partial_{\lambda}^k a) $, $ \forall k \in \N^{\nu + 1} $.

We shall use 
the following notation, used also in \cite{Alaz-Bal}, \cite{BertiMontalto}.
For any $m \in \R \setminus \{ 0\}$, we set
\begin{equation}\label{definizione Dm}
|D|^m := {\rm Op} \big( \chi(\xi) |\xi|^m \big)\,,
\end{equation}
where $\chi$ is the even, positive $\mC^\infty$ cut-off defined in \eqref{cut off simboli 1}.
We also identify the Hilbert transform ${\cal H}$, acting on the $2 \pi$-periodic functions, defined by
\be\label{Hilbert-transf}
\mH ( e^{\ii jx} ) := - \ii \, \sign(j) e^{\ii jx} \, , \    \forall j \neq 0 \, ,  \quad \mH (1) := 0\,, 
\ee
with the Fourier multiplier $ {\rm Op}\big(- \ii\, \sign(\xi) \chi(\xi) \big)$, i.e. 
$ {\cal H} \equiv {\rm Op}\big(- \ii \,\sign(\xi) \chi(\xi) \big) $. 
 
We shall identify the projector $\pi_0 $, defined on the $ 2 \pi $-periodic functions  as
\begin{equation}\label{def pi0}
\pi_0 u := \frac{1}{2\pi} \int_\T u(x)\, d x\, , 
\end{equation}
with the Fourier multiplier $ {\rm Op}\big( 1 - \chi(\xi) \big)$, i.e. 
$ \pi_0 \equiv {\rm Op}\big( 1 -  \chi(\xi) \big) $,  
where the cut-off $\chi(\xi)$ is defined in \eqref{cut off simboli 1}. 
We also define the Fourier multiplier $ \langle D \rangle^m $, $m \in \R \setminus \{ 0 \}$, 
as
\begin{equation} \label{def langle xi rangle}
\langle D \rangle^m := \pi_0 + |D|^m
:= \Op \big( ( 1 - \chi(\xi)) + \chi(\xi) |\xi|^m \big), \quad \xi \in \R \, .
\end{equation}
We now recall the  pseudo-differential norm introduced in Definition 2.11 in \cite{BertiMontalto} 
(inspired by M\'etivier \cite{Met}, chapter 5),
which controls the regularity in $ (\vphi, x)$,  and the decay in $ \xi $,  of the symbol $ a (\vphi, x, \xi) \in S^m $, 
together with its derivatives 
$ \pa_\xi^\b a \in S^{m - \b}$, $ 0 \leq \b \leq \a $, in the Sobolev norm $ \| \ \|_s $.  

\begin{definition}\label{def:pseudo-norm} {\bf (Weighted $\Psi DO$ norm)}
Let $ A(\lambda) := a(\lambda, \vphi, x, D) \in OPS^m $ 
be a family of pseudo-differential operators with symbol $ a(\lambda, \vphi, x, \xi) \in S^m $, $ m \in \R $, which are 
$k_0$ times differentiable with respect to $ \lambda \in \Lambda_0 \subset \R^{\nu + 1} $. 
For $ \g \in (0,1) $, $ \a \in \N $, $ s \geq 0 $, we define  the weighted norm 
\begin{equation}\label{norm1 parameter}
\norma A \norma_{m, s, \alpha}^{k_0, \gamma} := \sum_{|k| \leq k_0} \gamma^{|k|} 
\sup_{\lambda \in {\Lambda}_0}\norma\partial_\lambda^k A(\lambda) \norma_{m, s, \alpha}  
\end{equation}
where 
%we use the multi-index notation $ k = ( k_1, \ldots , k_{\nu + 1}) \in \N^{\nu + 1} $ with  $ | k  | := | k_1| + \ldots + | k_{\nu + 1} | $, and  
\be\label{norm1}
\norma A(\lambda)  \norma_{m, s, \a} := 
\max_{0 \leq \beta  \leq \a} \, \sup_{\xi \in \R} \|  \partial_\xi^\beta 
a(\lambda, \cdot, \cdot, \xi )  \|_{s} \langle \xi \rangle^{-m + \beta} \, . 
\ee
For a matrix of pseudo differential operators $A \in OPS^m$ as in \eqref{operatori matriciali sezione pseudo diff}, we define its pseudo differential norm  
$$
\norma A \norma_{m, s, \alpha}^{k_0, \gamma} 
:= \max_{i = 1, \ldots, 4} \norma A_i \norma_{m, s, \alpha}^{k_0, \gamma}\,. 
$$
\end{definition}
For each $ k_0, \gamma, m  $ fixed, the norm \eqref{norm1 parameter} is 
non-decreasing both in $ s $ and $ \a $, 
namely  
\be\label{norm-increa}
\forall s \leq s', \, \a \leq  \a' \, ,  \qquad 
\norma \  \norma_{m, s, \a}^{k_0, \gamma} \leq \norma \ \norma_{m, s', \a}^{k_0, \gamma} \, ,  \ \ 
\norma \  \norma_{m, s, \a}^{k_0, \gamma}  \leq 
\norma \ \norma_{m, s, \a'}^{k_0, \gamma} \, ,
\ee
and it is non-increasing in  $ m $, i.e. 
\be\label{crescente-m-neg}
\forall m \leq m',  % \qquad  \Longrightarrow 
\qquad \norma \  \norma_{m', s, \a}^{k_0, \gamma}  
\leq \norma \ \norma_{m, s, \a}^{k_0, \gamma} \, . 
\ee
Given a function $ a(\lambda, \vphi, x ) $ that is ${\cal C}^\infty $ in $(\ph,x)$ 
and ${k_0}$ times differentiable in $\lambda$, 
the ``weighted $\Psi$DO norm'' of the corresponding multiplication operator $\Op(a)$ is 
\be \label{norma a moltiplicazione}
\norma \Op (a)  \norma_{0, s, \a}^{k_0, \gamma} 
= \sum_{| k | \leq k_0} \gamma^{|k |} \sup_{\lambda \in \Lambda_0} 
\| \partial_{\lambda}^k a (\lambda) \|_s 
= \| a \|_{W^{k_0,\infty,\g}(\Lambda_0, H^s)}
\sim_{k_0} \| a \|_s^{k_0,\gamma} \, , \quad  \forall \a \in \N  \,, 
\ee
see \eqref{0203.1}.
% where $ \| a \|_s^{k_0, \gamma}$ is the Whitney-Sobolev norm defined 
% in Definition \ref{def:Lip F uniform}. 
For a Fourier multiplier $   g(\lambda, D)  $ with symbol $ g \in S^m $, we simply have 
\be\label{Norm Fourier multiplier}
\norma  {\rm Op}(g)  \norma_{m, s, \a}^{k_0, \gamma} 
= \norma  {\rm Op}(g)  \norma_{m, 0, \a}^{k_0, \gamma}
\leq C(m, \a, g, k_0) \, , \quad \forall s \geq 0 \, . 
\ee
Given a symbol $a(\lambda, \vphi, x, \xi) \in S^m$, we define its averages 
%space-time average as
$$
 \langle a \rangle_{\vphi}(\lambda, x, \xi) 
:= \frac{1}{(2 \pi)^\nu} \int_{\T^{\nu}} a(\lambda, \vphi, x, \xi)\,d \vphi\, \, , 
 \quad 
\langle a \rangle_{\vphi, x}(\lambda, \xi) := \frac{1}{(2 \pi)^{\nu + 1}} \int_{\T^{\nu + 1}} a(\lambda, \vphi, x, \xi)\,d \vphi\, d x\,.
$$
One has that $ \langle a \rangle_{\vphi} $ and $ \langle a \rangle_{\vphi, x} $ 
are symbols in $ S^m $ that satisfy
\be\label{stima astratta simbolo mediato}
% \label{stima astratta simbolo mediato-in-vphi}
\norma {\rm Op}( \langle a \rangle_{\vphi} )  \norma_{m, s, \alpha}^{k_0, \gamma} 
\lesssim \norma {\rm Op}(a) \norma_{m, s, \alpha}^{k_0, \gamma} \, , \quad 
\norma {\rm Op}( \langle a \rangle_{\vphi, x} )  \norma_{m, s, \alpha}^{k_0, \gamma} \lesssim \norma {\rm Op}(a) \norma_{m, 0, \alpha}^{k_0, \gamma} \, , \quad \forall s \geq 0\,. 
\ee
The norm $ \norma \ \norma_{0,s,0}$  controls the action of a pseudo-differential 
operator on the Sobolev spaces $ H^s $, see Lemma \ref{lemma: action Sobolev}. 
The norm $  \norma \  \norma_{m, s, \alpha}^{k_0, \gamma}  $ 
is closed under composition and  satisfies tame estimates. 

\medskip

\noindent
{\bf Composition.}
If $ A = a( x, D) \in OPS^{m} $, $ B = b(x, D) \in  OPS^{m'} $ % $ m , m' \in \R $,  
% are pseudo-differential operators, 
then the composition operator 
$ A B := A \circ B = \sigma_{AB} ( x, D) $ is a pseudo-differential operator in $OPS^{m+m'}$ 
whose symbol $\sigma_{AB} $ has the following 
asymptotic expansion:
for all $N \geq 1$,  
\be\label{expansion symbol}
\s_{AB} (x, \xi) = \sum_{\b =0}^{N-1} \frac{1}{\ii^\b \b !}  \pa_\xi^\b a (x, \xi) \, \pa_x^\b b (x, \xi) + r_N (x, \xi)
\qquad {\rm where} \qquad r_N := r_{N, AB} \in S^{m + m' - N }  \, ,
\ee
and the remainder $ r_N $ has the explicit formula 
\be\label{rNTaylor}
r_N (x, \xi) := r_{N, AB}(x, \xi) := \frac{1}{\ii^N (N-1)!} \int_0^1 (1- \t )^{N-1}  
\sum_{j \in \Z} (\pa_\xi^N a)(x, \xi + \t j ) \widehat {(\pa_x^N b)} (j , \xi) 
e^{\ii j x } \, d \tau \, .
\ee
We remind the following composition estimate 
proved in Lemma 2.13 in \cite{BertiMontalto}.

\begin{lemma} \label{lemma stime Ck parametri} 
{\bf (Composition)} 
Let $ A = a(\lambda, \vphi, x, D) $, $ B = b(\lambda, \vphi, x, D) $ be pseudo-differential operators
with symbols $ a (\lambda, \vphi, x, \xi) \in S^m $, $ b (\lambda, \vphi, x, \xi ) \in S^{m'} $, $ m , m' \in \R $. Then $ A(\lambda) \circ B(\lambda) \in OPS^{m + m'} $
satisfies,   for all $ \a \in \N $, $ s \geq s_0 $, 
\begin{equation}\label{estimate composition parameters}
\norma A B \norma_{m + m', s, \alpha}^{k_0, \gamma} 
\lesssim_{m,  \alpha, k_0} C(s) \norma A \norma_{m, s, \alpha}^{k_0, \gamma} 
\norma B \norma_{m', s_0 + \alpha + |m|, \alpha}^{k_0, \gamma}  
+ C(s_0) \norma A  \norma_{m, s_0, \alpha}^{k_0, \gamma}  
\norma B \norma_{m', s + \alpha + |m|, \alpha}^{k_0, \gamma} \, . 
\end{equation}
Moreover, for any integer $ N \geq 1  $,  
the remainder $ R_N := {\rm Op}(r_N) $ in \eqref{expansion symbol} satisfies
\be
\begin{aligned}
\norma R_N \norma_{m+ m'- N, s, \alpha}^{k_0, \gamma} 
\lesssim_{m, N,  \alpha, k_0} &
C(s) \norma A \norma_{m, s, N + \alpha}^{k_0, \gamma} 
\norma B  \norma_{m', s_0 + 2 N + |m| + \alpha, \alpha }^{k_0, \gamma}  
\\
% \qquad \qquad  \ \ \, 
+ \, &  C(s_0)\norma A \norma_{m, s_0 , N + \alpha}^{k_0, \gamma}
\norma B  \norma_{m', s + 2 N + |m| + \alpha, \alpha }^{k_0, \gamma}.
\label{stima RN derivate xi parametri} 
\end{aligned}
\ee
Both  \eqref{estimate composition parameters}-\eqref{stima RN derivate xi parametri} hold  
with the constant $ C(s_0) $ 
interchanged with $ C(s) $. 

Analogous estimates hold if $A$ and $B$ are matrix operators of the form \eqref{operatori matriciali sezione pseudo diff}. 
\end{lemma}

%\begin{proof}
%See Lemma 2.13 in \cite{BertiMontalto}. 
% \end{proof}

For a Fourier multiplier $  g(\lambda, D) $ with symbol $ g \in S^{m'}  $ % (independent of $ \lambda $)
we have the simpler estimate
\be\label{lemma composizione multiplier}
\norma A \circ g(D) \norma_{m+m', s, \a}^{k_0, \gamma} 
\lesssim_{k_0,  \alpha} \norma A  \norma_{m, s,\a}^{k_0, \gamma} 
\norma {\rm Op}(g)  \norma_{m', 0,\a}^{k_0, \gamma}  
\lesssim_{k_0,  \alpha, m'} \norma A  \norma_{m, s,\a}^{k_0, \gamma}   \, .
\ee
By \eqref{expansion symbol} the  commutator between two pseudo-differential operators $ A = a(x, D) 
\in OPS^{m} $ and $ B = b (x, D)  \in OPS^{m'} $ 
is a pseudo-differential operator $ [A, B] \in OPS^{m + m' - 1} $ 
with symbol  $ a \star b $, 
namely 
\be\label{symbol commutator}
[A, B] = {\rm Op} ( a \star b ) \, .  
\ee
By \eqref{expansion symbol} the symbol  $ a \star b  \in S^{m + m' - 1} $ admits the expansion 
\be\label{Expansion Moyal bracket}
a \star b = - \ii \{ a, b \} + {\mathtt r_{\mathtt 2}} (a, b )  \qquad {\rm where}
\qquad \{a, b\} := \pa_\xi a \, \pa_x b - \pa_x a \, \pa_\xi b   \in S^{m + m' - 1} 
\ee
is the Poisson bracket between $ a(x, \xi) $ and $ b(x, \xi)  $, and  
\be\label{def:r2ab}
{\mathtt r_{\mathtt 2}} (a, b ) := r_{2, AB} - r_{2, BA} \in S^{m + m' - 2} \, . 
\ee
By Lemma \ref{lemma stime Ck parametri} we deduce the following  corollary.

\begin{lemma}{\bf (Commutator)} \label{lemma tame norma commutatore}
If $ A = a(\lambda, \vphi, x, D) \in OPS^m $ and  $ B = b(\lambda, \vphi, x, D) \in OPS^{m'} $,  
%be pseudo-differential operators
%with symbols $ a (\lambda, \vphi, x, \xi) \in S^m $, $ b (\lambda, \vphi, x, \xi ) \in S^{m'} $, 
$ m , m' \in \R $, then  
the commutator $ [A, B] := AB - B A \in OPS^{m + m' - 1} $   satisfies 
\be
\begin{aligned}
\norma [A,B]  \norma_{m+m' -1, s, \alpha}^{k_0, \gamma} & 
\lesssim_{m,m',   \alpha, k_0} 
C(s) \norma A \norma_{m, s + 2 + |m' | + \alpha, \alpha + 1}^{k_0, \gamma} 
\norma B  \norma_{m', s_0 + 2 + |m| + \alpha , \alpha + 1}^{k_0, \gamma}  \\
& \qquad \qquad  
+ C(s_0) \norma A  
\norma_{m, s_0 + 2 + |m' | + \alpha, \alpha + 1}^{k_0, \gamma} 
\norma B \norma_{m', s + 2 + |m| + \alpha , \alpha + 1}^{k_0, \gamma}.  
\label{stima commutator parte astratta}
\end{aligned}
\ee
\end{lemma}

\begin{proof}
Use the expansion in \eqref{expansion symbol} with $N=1$ for both $AB$ and $BA$,
then use \eqref{stima RN derivate xi parametri} and \eqref{norm-increa}.
% See Lemma 2.15 in \cite{BertiMontalto}. 
\end{proof}
Given two linear operators $A$ and $B$, we define inductively the operators ${\rm Ad}^n_{A}(B)$, 
$ n \in \N $ in the following way: $ {\rm Ad}_{A}(B) := [A, B] $ and $ {\rm Ad}_A^{n + 1}(B) := [A, {\rm Ad}_A^n (B)] $, $  n \in \N $. 
Iterating the estimate \eqref{stima commutator parte astratta}, one deduces  
\begin{align}
\norma {\rm Ad}^n_A(B) \norma_{n m + m' - n, s, \alpha}^{k_0, \gamma} & \lesssim_{m, m', s, \alpha, k_0} (\norma A\norma^{k_0, \gamma}_{m, s_0 + {\mathtt c}_{n }(m , m', \alpha), \alpha + n})^n \norma B \norma_{m', s + {\mathtt c}_{n }(m , m', \alpha), \alpha + n}^{k_0, \gamma} \label{stima Ad pseudo diff astratta}\\
& \quad + (\norma A\norma^{k_0, \gamma}_{m, s_0 + {\mathtt c}_{n }(m , m', \alpha), \alpha + n})^{n - 1} \norma A\norma^{k_0, \gamma}_{m, s + {\mathtt c}_{n }(m , m', \alpha), \alpha + n}\norma B \norma_{m', s_0 + {\mathtt c}_{n }(m , m', \alpha), \alpha + n}^{k_0, \gamma}   \nonumber
\end{align}
for suitable constants $ {\mathtt c}_{n }(m , m', \alpha) > 0$. 

\smallskip

We remind the following estimate for the adjoint operator proved in Lemma 2.16 in \cite{BertiMontalto}.  

\begin{lemma}{\bf (Adjoint)} \label{stima pseudo diff aggiunto}
{\bf }
Let $A= a(\lambda, \vphi, x, D)$ be a pseudo-differential operator with symbol 
$ a(\lambda, \vphi, x, \xi) \in S^m,  m \in \R $. Then the $ L^2$-adjoint 
$A^*  \in OPS^m $ 
satisfies 
$$
\norma A^* \norma_{m, s, 0}^{k_0, \gamma} \lesssim_{m} \norma A \norma_{m, s + s_0 + |m|, 0}^{k_0, \gamma} \, . 
$$
The same estimate holds if $A$ is a matrix operator of the form \eqref{operatori matriciali sezione pseudo diff}. 
\end{lemma}

Finally we report  a lemma about inverse of pseudo-differential operators.  

\begin{lemma} {\bf (Invertibility)} \label{Neumann pseudo diff}
Let  $ \Phi :=  {\rm Id} + A $ where $ A := {\rm Op}(a(\lambda, \vphi, x, \xi )) \in OPS^0 $.  
There exist constants  $ C(s_0, \a, k_0) $, $ C(s, \a, k_0) \geq 1 $, $ s \geq s_0 $,  such that,  if 
\be\label{hyp:lemma-2.14}
C(s_0,\a, k_0)\norma A \norma_{0, s_0 + \alpha, \alpha}^{k_0, \gamma}   \leq 1/2  \, , 
\ee
then, for all $ \lambda  $,  the operator 
$ \Phi $ is invertible, $ \Phi^{- 1}  \in OPS^0 $ and, for all $ s \geq s_0 $,    
\be\label{con:lemma-2.14}
\norma \Phi^{- 1} - {\rm Id}  \norma_{0,s,\alpha}^{k_0, \gamma} \leq C(s,\a, k_0)  \norma A \norma_{0,s + \alpha, \alpha}^{k_0, \gamma} \, . 
\ee
The same estimate holds for a matrix operator $\Phi = {\mathbb I}_2 + A$ where ${\mathbb I}_2 = \begin{pmatrix}
{\rm Id} & 0 \\
0 & {\rm Id}
\end{pmatrix}$ and $A$ has the form \eqref{operatori matriciali sezione pseudo diff}.  
\end{lemma}

\begin{proof}
By a Neumann series argument. See Lemma 2.17 in \cite{BertiMontalto}.
\end{proof}

\subsection{Integral operators and Hilbert transform}\label{sub:integral-op}

In this section we consider integral operators with a $ {\cal C}^\infty $ kernel, 
which are the operators in $ OPS^{-\infty}$. As in the previous section, 
they 
% we deal with families of such operators that 
are ${k_0}$ times differentiable with respect to % a parameter
$ \lambda := (\omega, \h)$ in an open set $ \Lambda_0 \subseteq \R^{\nu+1} $.

\begin{lemma} \label{lem:Int}
Let $K := K( \lambda, \cdot ) \in {\cal C}^\infty(\T^\nu \times \T \times \T)$. Then the integral operator 
\be\label{integral operator}
({\cal R} u ) ( \vphi, x) := \int_\T K(\lambda, \vphi, x, y) u(\vphi, y)\,d y 
\ee
is in $ OPS^{- \infty}$ and, for all $ m, s,  \a \in \N $,  
$ \norma {\cal R}  \norma_{-m, s, \a}^{k_0, \gamma} \leq 
C(m, s, \alpha, k_0)  \| K \|_{{\cal C}^{s + m + \alpha}}^{k_0, \gamma} $.
\end{lemma}

\begin{proof}
See Lemma 2.32 in \cite{BertiMontalto}.
\end{proof}

An integral operator transforms into another integral operator under a change of variables
\begin{equation}\label{diffeo p integral operator}
P u(\vphi, x) := u(\vphi, x + p(\vphi, x)) \, . 
\end{equation}
 
\begin{lemma} \label{lemma cio}
Let $K(\lambda, \cdot ) \in {\cal C}^\infty(\T^\nu \times \T \times \T)$ 
and $p(\lambda, \cdot ) \in {\cal C}^\infty(\T^\nu \times \T, \R) $. 
There exists $\delta := \delta(s_0, k_0) > 0$ such that if $\| p\|_{2 s_0 + k_0 + 1}^{k_0, \gamma} \leq \delta$, then 
the integral operator $ {\cal R }$ in \eqref{integral operator} transforms into the integral operator
$ \big( P^{- 1} {\cal R} P \big) u(\vphi,  x) = \int_\T \breve K(\lambda, \vphi, x, y) u(\vphi, y)\,dy  $
with a  $ {\cal C}^\infty $ kernel 
$$
 \breve K(\lambda, \vphi, x, z) := \big( 1 + \partial_z q (\lambda, \vphi, z) \big) K(\lambda, \vphi, x + q(\lambda, \vphi, x), 
 z + q(\lambda, \vphi, z)) ,
$$
 where $z \mapsto z + q(\lambda, \vphi, z)$ is the inverse diffeomorphism of $x \mapsto x + p(\lambda, \vphi, x)$.
The function $\breve K$ satisfies
$$
\| \breve K\|_{s}^{k_0, \gamma} \leq C(s, k_0)  \big( \| K \|_{s + k_0}^{k_0, \gamma} 
+ \| p \|_{{s + k_0 + 1}}^{k_0, \gamma} \| K \|_{s_0 + k_0 + 1}^{k_0, \gamma}\big) \qquad \forall s \geq s_0\,.
$$
\end{lemma}

\begin{proof}
See Lemma 2.34 in \cite{BertiMontalto}.
\end{proof}

We now recall some properties of the Hilbert transform $ \mH $ 
defined as a Fourier multiplier in \eqref{Hilbert-transf}. 
%The Hilbert transform also admits an integral representation. 
%Given a $ 2\pi $-periodic function $u$, its Hilbert transform is  
%$$
%(\mH u)(x) 
%:= \frac{1}{2\pi} \, {\rm  p. v. }  \int_\T \frac{u(y)}{\tan (\frac12 \,(x-y)) }\,dy 
%:= \lim_{\e \to 0} \frac{1}{2\pi} \Big\{ \int_{x-\pi}^{x-\e} + \int_{x+\e}^{x+\pi} \Big\} \frac{u(y)}{\tan (\frac12 \,(x-y)) }\,dy .
%$$
The commutator between  $ \mH $ and the multiplication operator by a smooth function $ a $ is a 
regularizing operator in $ OPS^{-\infty } $, as stated  in 
Lemma 2.35 in \cite{BertiMontalto} 
(see also Lemma B.5 in \cite{Baldi Benj-Ono}, Appendices H and I in \cite{IPT}). 

\begin{lemma}\label{lem: commutator aH}
Let $ a( \lambda, \cdot, \cdot ) \in {\cal C}^{\infty} (\T^\nu \times \T, \R)$. Then the commutator $[a, \mH ] $ is in 
$ OPS^{-\infty }$ and satisfies, for all $ m,  s,  \a \in \N $, 
$$
\norma  [a, \mH  ] \norma_{-m, s, \a}^{k_0, \gamma} \leq C(m, s, \alpha, k_0) \| a \|_{{s + s_0 + 1+ m + \alpha}}^{k_0, \gamma} \, . 
$$
\end{lemma}

We also report the following classical lemma, see e.g. 
Lemma 2.36 in \cite{BertiMontalto} and Lemma B.5 in \cite{Baldi Benj-Ono}
(and Appendices H and I in \cite{IPT} for similar statements). 

\begin{lemma} \label{coniugio Hilbert}
Let $p = p(\lambda, \cdot)$ be  in $  {\cal C}^\infty(\T^{\nu + 1})$ and 
$P := P(\lambda,\cdot)$ be the associated change of variable 
defined in \eqref{diffeo p integral operator}. There exists $\delta(s_0, k_0) > 0$ such that, if $\|p \|_{2 s_0 + k_0 + 1}^{k_0, \gamma} \leq \delta(s_0, k_0)$, then 
the operator $P^{- 1} \mH P - \mH$ is an integral operator of the form
$$
( P^{-1} \mH P - \mH) u (\vphi, x) 
= \int_\T  \,K(\lambda, \vphi, x,z) u(\ph, z)\,dz
$$
where $ K = K(\lambda, \cdot) \in {\cal C}^\infty (\T^\nu \times \T \times \T ) $ is given by 
$ K(\lambda, \vphi, x, z) := - \frac{1}{\pi} \partial_z \log(1 + g(\lambda, \vphi, x, z)) $
with 
$$
g(\lambda, \vphi, x, z)  := \cos \Big( \frac{q(\lambda, \vphi, x) - q(\lambda, \vphi, z)}{ 2}  \Big) - 1 + 
\cos \Big( \frac{x-z}{2}  \Big) 
\frac{\sin(\frac12 ( q(\lambda,  \vphi, x) - q(\lambda,  \vphi, z)))}
{\sin(\frac12(x-z))}\,
$$
where $z \mapsto q(\lambda, \vphi, z)$ is the inverse diffeomorphism of $x \mapsto x + p(\lambda, \vphi, x)$.
The kernel $K$ satisfies the estimate
$$
\| K \|_{s}^{k_0, \gamma} \leq C(s, k_0) \| p \|_{s +  k_0 + 2}^{k_0, \gamma} \,, \quad \forall s \geq s_0\,.
$$
\end{lemma}

We finally provide a simple estimate for the integral kernel of a family of Fourier multipliers in $ OPS^{-\infty }$. 

\begin{lemma}\label{lemma nucleo Fourier multiplier OPS - infty}
Let $ g(\lambda, \vphi,  \xi)$ be a family of Fourier multipliers 
with $\partial_\lambda^k g(\lambda, \vphi,  \cdot) \in S^{- \infty}$, 
for all $k \in \N^{\nu + 1}$, $|k| \leq k_0$. 
Then the operator ${\rm Op}(g)$ admits the integral representation 
\begin{equation}\label{espressione Kernel operatori OPS - infty}
\big[{\rm Op}(g) u\big](\vphi, x) = \int_\T K_g(\lambda, \vphi, x, y) u(\vphi, y)\, d y\,, \qquad K_g(\lambda, \vphi,  x, y) := \frac{1}{2\pi}\, \sum_{j \in \Z} 
g(\lambda, \vphi, j) e^{\ii j (x - y)}\, ,
\end{equation}
and the kernel $K_g$ 
satisfies, for all $s\in\N$, the estimate
\be\label{Fourier-integrale}
\| K_g\|_{{\cal C}^s}^{k_0, \gamma} 
\lesssim  \norma {\rm Op}(g) \norma_{- 1, s + s_0, 0}^{k_0, \gamma} + \norma {\rm Op}(g) \norma_{- s -s_0- 1, 0, 0}^{k_0, \gamma}\,.
\ee
\end{lemma}

\begin{proof}
 The lemma follows by differentiating the explicit expression of the integral Kernel $K_g$ in \eqref{espressione Kernel operatori OPS - infty}.     
\end{proof}

\subsection{Reversible, Even,  Real operators}\label{sezione operatori reversibili e even}

We introduce now some algebraic properties that have a key role in the proof. 

\begin{definition}\label{def:even}
{\bf (Even operator)}  A linear operator $ A := A(\vphi) $ as in \eqref{matrice operatori Toplitz} 
is {\sc even} if each $ A(\vphi) $,  $ \vphi \in \T^\nu $,  leaves invariant  the space of  functions even in $  x $. 
%Similarly a matrix of operators $\begin{pmatrix} A & B \\
%C & D
%\end{pmatrix}$ is even if $A, B,C, D$ are even operators. 
\end{definition}

Since the Fourier coefficients of an even function satisfy $ u_{- j}  = u_j $ for all $j \in \Z $, we have that 
\be\label{even operators Fourier}
A \ \text{is even} \quad \Longleftrightarrow \quad   
A_j^{j'}(\vphi) + A_j^{-j'}(\ph) = A_{-j}^{j'}(\vphi) + A_{-j}^{- j'}(\vphi) \, , 
\quad \forall j, j' \in \Z, \ \ph \in \T^\nu.
\ee
%We also notice that, if $ A $ is an even operator, then the operator 
%\be\label{A-tilde}
%\tilde A := \frac{1}{2} \big(  A \circ \tau + \tau \circ A \big)  \qquad {\rm where} \qquad (\tau u)(x) := u(- x )
%\ee
%coincides, on the subspace $ E := \{ u(-x) = u(x) \} $ of the  functions even in $ x $, with $ A $, namely  
%$ \tilde A_{|E} = A_{|E} $.

\begin{definition} {\bf (Reversibility) }  
An operator ${\cal R}$ as in \eqref{cal R eta psi} is
\begin{enumerate}
\item {\sc reversible} if 
$ {\cal R}(- \vphi ) \circ \rho = - \rho \circ {\cal R}(\vphi ) $  for all $\vphi \in \T^\nu $,   
where the involution $ \rho $ is defined in \eqref{defS}, 
\item
{\sc reversibility preserving} if  $ {\cal R}(- \vphi ) \circ \rho =  \rho \circ {\cal R}(\vphi ) $ 
for all $\vphi \in \T^\nu $. 
\end{enumerate}
%where $ \rho $ is the involution defined in \eqref{defS}. 
\end{definition}
The composition of a reversible operator with a  reversibility preserving operator is reversible. 
%
%Recalling the definition of $ \rho $ in \eqref{defS}
 It turns out that an operator $ {\cal R} $ as in \eqref{cal R eta psi} is
\begin{enumerate}
\item
 reversible if and only if $ \vphi \mapsto A (\vphi), D (\vphi) $ are odd and $  \vphi \mapsto B(\vphi), C(\vphi) $ are even,
\item 
reversibility preserving if and only if
$  \vphi \mapsto A (\vphi), D (\vphi) $ are even and $  \vphi \mapsto B(\vphi), C(\vphi) $ are odd. 
\end{enumerate}

We shall say that a linear operator of the form $ {\cal L} := {\om \cdot \pa_\vphi} + A(\vphi) $ 
is reversible, respectively even, if $A(\ph)$ is reversible, respectively even.
Conjugating the linear operator $ {\cal L} := {\om \cdot \pa_\vphi} + A(\vphi) $ by a family of invertible linear maps $ \Phi(\vphi) $ we get
the transformed operator 
$$
\begin{aligned}
& \qquad {\cal L}_+ :=  \Phi^{-1}(\vphi) {\cal L} \Phi (\vphi) = {\om \cdot \pa_\vphi} + A_+ (\vphi) \, , \\
& A_+ (\vphi) := 
 \Phi^{-1}(\vphi) (\om \cdot \pa_\vphi \Phi (\vphi) ) + \Phi^{-1}(\vphi) A (\vphi) \Phi (\vphi) \,.  
 \end{aligned}
$$
It results that the conjugation of an even and reversible operator with an operator 
$ \Phi (\vphi) $ that is even and reversibility preserving is even and  reversible.

\begin{lemma}\label{even:pseudo}
Let $ A := {\rm Op} (a) $ be a pseudo-differential operator. Then the following holds:
\begin{enumerate}
\item 
If the symbol $a$ satisfies $ a(-x, - \xi) = a(x,  \xi) $, then $A$ is even. 
\item \label{item:due}
If $ A = {\rm Op} (a)  $ is  even,  % operator according to Definition \ref{def:even}. 
then the pseudo-differential operator $ {\rm Op} ( \tilde a ) $ with symbol 
\be\label{simbolo:tilde}
\tilde a(x, \xi) := \frac12 \big(  a(x, \xi) + a (-x, - \xi) \big) 
\ee 
coincides with $ {\rm Op} (a) $ on the subspace $ E := \{ u(-x) = u(x) \} $ of the functions even in $ x $, namely
$ {\rm Op} ( \tilde a)_{|E} = {\rm Op} (a)_{|E} $. 
\item 
$A$ is real, i.e.\ it maps real functions into real functions,  
if and only if  the symbol  $ \overline{a(x, - \xi)} = a(x,  \xi) $. 
\item
Let $g(\xi)$ be a Fourier multiplier satisfying $g(\xi) = g(- \xi)$. 
If $A = {\rm Op}(a)$ is even, then the operator 
${\rm Op}(a(x, \xi) g(\xi)) = {\rm Op}(a) \circ {\rm Op}(g)$ is an even operator.
More generally, the composition of even operators is an even operator.
\end{enumerate}
\end{lemma}

We shall use the following remark. 

\begin{remark}\label{rem:change}
By item \ref{item:due}, we can replace an even 
pseudo-differential operator $ {\rm Op} (a) $ acting 
on the subspace of functions even in $ x $,
with the  operator $ {\rm Op} (\tilde a) $  where the symbol $ \tilde a (x, \xi)$ defined in \eqref{simbolo:tilde}
satisfies $ \tilde a (-x, -\xi) = \tilde a(x, \xi) $. 
The pseudo-differential norms of $ {\rm Op} (  a) $ and $  {\rm Op} ( \tilde a) $ are equivalent. 
Moreover, the space average
$$ 
\langle \tilde a \rangle_x (\xi) := \frac{1}{2 \pi} \int_{\T} \tilde a (x, \xi) \, d x \qquad
{\rm satisfies} \qquad
\langle \tilde a \rangle_x (-\xi) = \langle \tilde a \rangle_x (\xi) \, ,
$$ 
and, therefore, the Fourier multiplier $ \langle \tilde a \rangle_x (D) $ is even. 
\end{remark}

It is convenient to consider a
real operator $ {\cal R} = \begin{pmatrix} A & B \\ C & D \end{pmatrix} $ as in \eqref{cal R eta psi}, which acts on the real variables $ (\eta, \psi) \in \R^2 $, 
as a linear operator acting on the complex variables $ (u, \bar u ) $
introduced by the linear change of coordinates $(\eta,\psi) = \mC (u,\bar u)$, 
where
\begin{equation}\label{def:mC}
\mC := \frac12 \begin{pmatrix} 1 & 1 \\ -\ii & \ii \end{pmatrix}, \quad \ 
\mC^{-1} = \begin{pmatrix} 1 & \ii \\ 1 & -\ii \end{pmatrix}.
\end{equation}
We get that the {\it real} operator $\mR$ acting in the complex coordinates $ (u, \bar u) = \mC^{-1} (\eta,\psi) $ takes the form  
\be
\begin{aligned}\label{operatori in coordinate complesse}
& \qquad \qquad \qquad  {\bf R} = \mC^{-1} \mR \mC := \begin{pmatrix}
{\cal R}_1 & {\cal R}_2 \\
\overline{\cal R}_2 & \overline{\cal R}_1
\end{pmatrix}, \\
& {\cal R}_1 := \frac12 \big\{(A + D) - \ii (B - C) \big\} , \quad  {\cal R}_2 := \frac12 \big\{ (A - D) + \ii(B + C) \big\} 
\end{aligned}
\ee 
where the \emph{conjugate} operator $ \overline{A} $ is defined by 
\be\label{def:barA} 
\overline{A}(u) :=  \overline{A( \bar u)} \, .
\ee
We say that a matrix operator acting on the complex variables $(u, \bar u )$  is {\sc real} if it has the structure in \eqref{operatori in coordinate complesse} and it  is {\sc even}  if both ${\cal R}_1$, ${\cal R}_2$ are even.  The composition of two real (resp. even) operators is a real (resp. even) operator.

The following properties of the conjugated operator hold: 
\begin{enumerate}
\item 
$ \overline{AB} = \overline{A} \ \overline{B} \, .  $ 
\item 
If $ (A_j^{j'}) $ is the matrix of $A$, then the matrix entries of $\overline{A}$ are
$ ( \overline{A} \,)_j^{j'} = \overline{ A_{-j}^{-j'} }$.
\item 
If $A = {\rm Op} (a(x,\xi))$ is a pseudo-differential operator, then its conjugate is %  $\overline A$ is 
$ \overline A = {\rm Op} (\overline{a(x, -\xi)}) $. The pseudo differential norms of $A$ and $\overline A$ are equal, namely $\norma A \norma_{m, s, \alpha}^{k_0, \gamma} = \norma \overline A \norma_{m, s, \alpha}^{k_0, \gamma}$.   
\end{enumerate}

In the complex coordinates $ (u, \bar u) = \mC^{-1} (\eta,\psi) $
the involution $ \rho $ defined in \eqref{defS} reads as the map 
$ u \mapsto \bar u  $.

 \begin{lemma}
Let $ {\bf R} $ be a real operator as in  \eqref{operatori in coordinate complesse}. One has 
\begin{enumerate}
 \item  $ {\bf R} $ is reversible if and only if $ {\cal R}_i (- \vphi) = - \overline{{\cal R}_i} ( \vphi) $ 
for all $\vphi \in \T^\nu $, $i = 1,2$, or equivalently
\begin{equation} \label{2601.4}
(\mR_i)_j^{j'}(-\ph) = - \overline{ (\mR_i)_{-j}^{-j'}(\ph) } \quad \forall \ph \in \T^\nu \, , 
\quad \text{i.e.} \ \ 
(\mR_i)_j^{j'}(\ell) = - \overline{ (\mR_i)_{-j}^{-j'}(\ell) } \quad \forall \ell \in \Z^\nu \, .
\end{equation}
 \item  $ {\bf R} $ is reversibility preserving if and only if 
 $ {\cal R}_i (- \vphi) = \overline{{\cal R}_i} ( \vphi) $ 
for all $\vphi \in \T^\nu $, $i = 1,2$, or equivalently
\begin{equation} \label{2601.5}
(\mR_i)_j^{j'}(-\ph) = \overline{ (\mR_i)_{-j}^{-j'}(\ph) } \quad \forall \ph \in \T^\nu \, , 
\quad \text{i.e.} \ \ 
(\mR_i)_j^{j'}(\ell) = \overline{ (\mR_i)_{-j}^{-j'}(\ell) } \quad \forall \ell \in \Z^\nu \, . 
\end{equation}
\end{enumerate}
\end{lemma}

\subsection{$ {\cal D}^{k_0}$-tame and modulo-tame operators}

In this section we recall  the notion and the main properties of 
$ {\cal D}^{k_0}$-tame and modulo-tame operators that will be used in the paper.
%in the reducibility scheme of the linearized operator in Section \ref{sec: reducibility}.
For the proofs we refer to  Section 2.2 of \cite{BertiMontalto} where this notion was introduced. 

Let $ A := A(\lambda) $ be a family of  linear operators as in \eqref{matrice operatori Toplitz}, 
$k_0$ times differentiable with respect to % the parameter 
$ \lambda$ in an open set $\Lambda_0  \subset  \R^{\nu + 1} $. 

\begin{definition} { \bf ($ {\cal D}^{k_0} $-$ \s $-tame)}\label{def:Ck0}
Let $\s\geq0$. A linear operator  $ A := A(\lambda)  $ as in \eqref{def azione toplitz u vphi x} is 
$ {\cal D}^{k_0}$-$ \s $-tame  if there exists a non-decreasing function $[s_0, S] \to [0, + \infty)$, $s \mapsto {\mathfrak M}_A(s)$, possibly with $S = + \infty$, such that for all $ s_0 \leq s \leq S $, for all $u \in H^{s+\s} $ 
\be\label{CK0-sigma-tame}
\sup_{|k| \leq k_0} \sup_{\lambda \in \Lambda_0} \gamma^{ |k|}
\| (\partial_\lambda^k A(\lambda)) u \|_s \leq  {\mathfrak M}_A(s_0) \| u \|_{s+\s} + 
{\mathfrak M}_A (s) \| u  \|_{s_0+\s} \,. 
\ee
%where  the functions $ s \mapsto  {\mathfrak M}_A(s) \geq 0  $ are non-decreasing in $ s $. 
We say that $ {\mathfrak M}_{A}(s) $ is a {\sc tame constant} of the operator $ A $.
The constant $ {\mathfrak M}_A(s) := {\mathfrak M}_A (k_0, \s, s) $
may also depend on $ k_0, \s $ but,
since $ k_0, \s $ are considered in this paper absolute constants, 
we shall often omit to write them. 

When the ``loss of derivatives"  $ \sigma $ is zero, 
we simply write $ {\cal D}^{k_0} $-tame instead of $ {\cal D}^{k_0} $-$ 0 $-tame. 

For a real matrix operator (as in  \eqref{operatori in coordinate complesse})
\be\label{operatori in coordinate complesse-bis}
A = 
\begin{pmatrix}
A_1 & A_2 \\
\overline{A}_2 & \overline{A}_1
\end{pmatrix},
\ee
we denote the tame constant  
$ {\mathfrak M}_{A}(s) := \max \{ {\mathfrak M}_{A_1}(s), {\mathfrak M}_{A_2}(s)  \} $.
\end{definition}

Note that the tame constants $ {\frak M}_A(s)$ are not uniquely determined. Moreover, 
if $S< + \infty$, every linear operator $A$
 that is uniformly bounded in $\lm$ (together with its derivatives $\pa_\lm^k A$) as an operator from $H^{s+\s}$ to $H^s$
 is $ {\cal D}^{k_0} $-$ \s $-tame. The relevance of Definition \ref{def:Ck0} is that, for the remainder 
 operators which we shall obtain along the reducibility of the linearized operator in Sections 
\ref{linearizzato siti normali}-\ref{sec: reducibility},
 we are able to prove bounds of the tame constants ${\frak M}_A(s)$ better than the trivial operator norm. 

\begin{remark}
In Sections 
\ref{linearizzato siti normali}-\ref{sec: reducibility} we work with 
$ {\cal D}^{k_0}$-$\sigma$-tame operators with  a finite $ S < + \infty$,  whose tame constants ${\frak M}_A(s)$ may depend also on $S$, for instance  
$ {\frak M}_A(s) \leq C(S) (1 + \| \fracchi_0 \|_{s + \mu}^{k_0, \gamma}) $, 
for all $s_0 \leq s \leq S $.
\end{remark}

An immediate consequence of \eqref{CK0-sigma-tame} (with $ k = 0 $, $ s = s_0 $) is that 
%\begin{equation}\label{norma operatoriale costante tame}
$ \| A \|_{{\cal L}(H^{s_0 + \sigma}, H^{s_0})} \leq 2 {\mathfrak M}_A(s_0) $. 

%\end{equation}
Also note that representing the operator $ A $ by its matrix elements 
$ \big(A_j^{j'} (\ell - \ell') \big)_{\ell, \ell' \in \Z^\nu, j, j' \in \Z} $ as in \eqref{matrice operatori Toplitz} we have, for all
$ |k| \leq k_0 $, $ j' \in \Z $, $ \ell' \in \Z^\nu $,  
\be\label{tame-coeff}
\gamma^{2 |k|} {\mathop \sum}_{\ell , j} \langle \ell, j \rangle^{2 s} |\partial_\lambda^k A_j^{j'}(\ell - \ell')|^2 
\leq 2 \big({\mathfrak M}_A(s_0) \big)^2 \langle \ell', j' \rangle^{2 (s+\s)} + 2 \big({\mathfrak M}_A(s) \big)^2 \langle \ell', j' \rangle^{2 (s_0+\s)} 	\, . 
\ee
The class of $ {\cal D}^{k_0} $-$ \s $-tame operators is closed under composition. 

\begin{lemma}\label{composizione operatori tame AB} {\bf (Composition)}
Let $ A, B $ be respectively $ {\cal D}^{k_0} $-$\sigma_A$-tame and 
$ {\cal D}^{k_0} $-$\sigma_B$-tame operators with tame 
constants respectively $ {\mathfrak M}_A (s) $ and $ {\mathfrak M}_B (s) $. 
Then the composition % ed operator 
$ A \circ B $ is $ {\cal D}^{k_0} $-$(\sigma_A + \sigma_B)$-tame with a tame constant satisfying 
$$
 {\mathfrak M}_{A B} (s) \leq  C(k_0) \big( {\mathfrak M}_{A}(s) 
 {\mathfrak M}_{B} (s_0 + \sigma_A) + {\mathfrak M}_{A} (s_0) 
{\mathfrak M}_{B} (s + \sigma_A) \big)\,.
$$
The same estimate holds if $A,B$ are matrix operators as in \eqref{operatori in coordinate complesse-bis}. 
\end{lemma}

\begin{proof}
See Lemma 2.20 in \cite{BertiMontalto}. 
\end{proof}

We now discuss the action of a $ {\cal D}^{k_0} $-$ \s $-tame operator  $ A(\lambda) $ on a family of  
Sobolev functions  $ u (\lambda) \in H^s $.
% which are  ${k_0}$ times differentiable with respect to  $ \lambda  \in  \Lambda_0 \subset \R^{\nu + 1} $. 

\begin{lemma}\label{lemma operatore e funzioni dipendenti da parametro}
{\bf (Action on $ H^s $)}
Let $ A := A(\lambda) $ be a $ {\cal D}^{k_0} $-$ \s $-tame operator. 
Then, $ \forall s \geq s_0 $, for any family of Sobolev functions $ u := u(\lambda) \in H^{s+\s} $ 
which is $k_0$ times differentiable with respect to $ \lambda $,  we have 
$$
\| A u \|_s^{k_0, \gamma} \lesssim_{k_0}  {\mathfrak M}_A(s_0) \| u \|_{s + \sigma}^{k_0, \gamma} 
+ {\mathfrak M}_A(s) \| u \|_{s_0 + \sigma}^{k_0, \gamma}  \,.
$$
The same estimate holds if $A$ is a matrix operator as in \eqref{operatori in coordinate complesse-bis}. 
\end{lemma}

\begin{proof}
See Lemma 2.22 in \cite{BertiMontalto}. 
\end{proof}

Pseudo-differential operators are tame operators. We shall use in particular the following lemma.

\begin{lemma}\label{lemma: action Sobolev}
Let $ A = a(\lambda, \vphi, x, D) \in OPS^0 $  be a family of pseudo-differential operators
that are ${k_0}$ times differentiable with respect to  $\lambda $. 
If $ \norma A \norma_{0, s, 0}^{k_0, \gamma} < + \infty $, $ s \geq s_0 $, then  $ A $ is ${\cal D}^{k_0}$-tame with a tame constant satisfying 
\begin{equation}\label{interpolazione parametri operatore funzioni}
{\mathfrak M}_A(s) \leq C(s)  \norma A \norma_{0, s, 0}^{k_0, \gamma}\,.
\end{equation}
As a consequence 
\begin{equation}\label{interpolazione parametri operatore funzioni (2)}
\| A h \|_s^{k_0, \gamma} \leq C(s_0, k_0) \norma A \norma_{0, s_0, 0}^{k_0, \gamma} \| h \|_{s}^{k_0, \gamma} + C(s, k_0) \norma A \norma_{0, s, 0}^{k_0, \gamma} \| h \|_{s_0}^{k_0, \gamma}\,.
\end{equation}
The same statement holds if $ A $ is a matrix operator of the form \eqref{operatori in coordinate complesse-bis}. 
\end{lemma}

\begin{proof}
See Lemma 2.21 in \cite{BertiMontalto} for the proof of \eqref{interpolazione parametri operatore funzioni}, then apply Lemma \ref{lemma operatore e funzioni dipendenti da parametro} 
to deduce \eqref{interpolazione parametri operatore funzioni (2)}. 
\end{proof}

In view of the KAM reducibility scheme of Section  \ref{sec: reducibility}, we also consider
the stronger notion of $ {\cal D}^{k_0} $-modulo-tame operator, which we  need  only for 
operators with loss of derivatives $ \s = 0 $.

\begin{definition}\label{def:op-tame} 
{ \bf ($ {\cal D}^{k_0} $-modulo-tame)} A  linear operator $ A := A(\lambda) $ as in \eqref{def azione toplitz u vphi x} is  
$ {\cal D}^{k_0}$-modulo-tame  if there exists a non-decreasing function $[s_0, S] \to [0, + \infty)$, $s \mapsto {\mathfrak M}_{A}^\sharp (s)$, such that for all $ k \in \N^{\nu + 1} $, $ |k| \leq k_0 $,  the majorant operators $  | \partial_\lambda^k A | $ 
(Definition \ref{def:maj})
satisfy the following weighted tame estimates:   for all $ s_0 \leq s \leq S $, $ u \in H^{s} $,  
\be\label{CK0-tame}
\sup_{|k| \leq k_0} \sup_{\lambda \in \Lambda_0}\gamma^{ |k|}
\| \, | \partial_\lambda^k A | u \|_s \leq  
{\mathfrak M}_{A}^\sharp (s_0) \| u \|_{s} +
{\mathfrak M}_{A}^\sharp (s) \| u \|_{s_0} \,.
\ee
The constant $ {\mathfrak M}_A^\sharp (s) $ is called a {\sc modulo-tame constant} of the operator $ A $. 

For a matrix operator as in  \eqref{operatori in coordinate complesse-bis}
we denote the modulo tame constant  
$ {\mathfrak M}^\sharp_{A}(s) := \max \{ {\mathfrak M}^\sharp_{A_1}(s), {\mathfrak M}^\sharp_{A_2}(s)  \} $.
\end{definition}

If $ A $, $B$ are $ {\cal D}^{k_0}$-modulo-tame operators, with
$ | A_j^{j'} (\ell) | \leq | B_j^{j'} (\ell) | $, then $ {\mathfrak M}^\sharp_{A}(s) \leq {\mathfrak M}^\sharp_{B}(s) $. 

\begin{lemma}\label{A versus |A|}
An operator $ A $ that is $ {\cal D}^{k_0}$-modulo-tame is also $ {\cal D}^{k_0}$-tame and
$ {\mathfrak M}_A (s) \leq  {\mathfrak M}_A^\sharp (s) $. 
The same holds if $A$ is a matrix operator as in \eqref{operatori in coordinate complesse-bis}. 
\end{lemma}

\begin{proof}
See Lemma 2.24 in \cite{BertiMontalto}. 
%For all $k\in\N^{\nu+1}$ with $|k| \leq k_0$ and for all $u\in H^s$, one has
%$$
%\| (\pa_\lm^k A) u \|_s^2
%= \sum_{\ell,j} \langle \ell,j \rangle^{2s} \big| \sum_{\ell',j'} \pa_\lm^k A_j^{j'} (\ell -\ell') u_{\ell',j'} \big|^2
%\leq \sum_{\ell,j} \langle \ell,j \rangle^{2s} \big( \sum_{\ell',j'}  |\pa_\lm^k A_j^{j'} (\ell -\ell')| |u_{\ell',j'}| \big)^2
%= \| \, |\pa_\lm^k A| (\norma u \norma) \|_s^2 \, .
%$$
%Then the thesis follows by \eqref{Soboequals} and by Definitions \ref{def:Ck0} and \ref{def:op-tame}.
\end{proof}

The class of operators which are $ {\cal D}^{k_0} $-modulo-tame is closed under
sum and composition.

\begin{lemma} \label{interpolazione moduli parametri} {\bf (Sum and composition)}
Let $ A, B $ be $ {\cal D}^{k_0} $-modulo-tame operators with modulo-tame constants respectively 
$ {\mathfrak M}_A^\sharp(s) $ and $ {\mathfrak M}_B^\sharp(s) $. Then 
$ A+ B $ is $ {\cal D}^{k_0} $-modulo-tame with a modulo-tame constant satisfying 
\be\label{modulo-tame-A+B}
{\mathfrak M}_{A + B}^\sharp (s) \leq {\mathfrak M}_A^\sharp (s)  + {\mathfrak M}_B^\sharp (s)  \,.
\ee
The composed operator 
$  A  \circ B $ is $ {\cal D}^{k_0} $-modulo-tame with a modulo-tame constant satisfying 
\begin{equation}\label{modulo tame constant for composition}
 {\mathfrak M}_{A B}^\sharp (s) \leq  C(k_0) \big( {\mathfrak M}_{A}^\sharp(s) 
 {\mathfrak M}_{B}^\sharp (s_0) + {\mathfrak M}_{A}^\sharp (s_0) 
{\mathfrak M}_{B}^\sharp (s) \big)\,.
\end{equation}
Assume in addition that $ \langle \partial_{\vphi, x} \rangle^{\mathtt b} A $, 
$ \langle \partial_{\vphi, x} \rangle^{\mathtt b}  B $ (see Definition \ref{def:maj}) are $ {\cal D}^{k_0}$-modulo-tame with a modulo-tame constant 
respectively $ {\mathfrak M}_{\langle \partial_{\vphi, x} \rangle^{\mathtt b} A}^\sharp (s) $ and  
 $ {\mathfrak M}_{\langle \partial_{\vphi, x} \rangle^{\mathtt b} B}^\sharp (s) $. 
Then $ \langle \partial_{\vphi, x} \rangle^{\mathtt b} (A  B) $ is $ {\cal D}^{k_0}$-modulo-tame 
with a modulo-tame constant satisfying
\be
\begin{aligned}\label{K cal A cal B}
{\mathfrak M}_{\langle \partial_{\vphi, x} \rangle^{\mathtt b} (A  B)}^\sharp (s) & \leq
 C({\mathtt b}) C(k_0)\Big( 
{\mathfrak M}_{\langle \partial_{\vphi, x} \rangle^{\mathtt b} A}^\sharp (s) 
{\mathfrak M}_{B}^\sharp (s_0) + 
{\mathfrak M}_{\langle \partial_{\vphi, x} \rangle^{\mathtt b} A }^\sharp (s_0) 
{\mathfrak M}_{B}^\sharp (s)  \\ 
& \qquad \qquad \qquad \quad + {\mathfrak M}_{A}^\sharp (s) {\mathfrak M}_{ \langle\partial_{\vphi, x} \rangle^{\mathtt b} B}^\sharp (s_0) 
+ {\mathfrak M}_{A}^\sharp (s_0) {\mathfrak M}_{ \langle \partial_{\vphi, x} \rangle^{\mathtt b} B}^\sharp (s)\Big) 
\end{aligned}
\ee
for some constants $ C(k_0) , C( {\mathtt b} ) \geq 1 $. The same statement holds if $A$ and $B$ are matrix operators as in \eqref{operatori in coordinate complesse-bis}. 
\end{lemma}
\begin{proof}
\noindent
 The estimates \eqref{modulo-tame-A+B}, \eqref{modulo tame constant for composition} are proved in Lemma 2.25 of \cite{BertiMontalto}. The bound \eqref{K cal A cal B} is proved as  the estimate (2.76) of Lemma 2.25 in \cite{BertiMontalto}, replacing $\langle \partial_\vphi \rangle^{\mathtt b}$ (cf. Definition 2.3 in \cite{BertiMontalto}) with $\langle \partial_{\vphi, x} \rangle^{\mathtt b}$. 
\end{proof}

%By \eqref{modulo tame constant for composition}, if $ A $ is $ {\cal D}^{k_0}$-modulo-tame, then, 
%for all $n \geq 1$, each  $ A^n  $ is $ {\cal D}^{k_0}$-modulo-tame and % there exists $ C (k_0)  \geq 1  $ such that 
%\begin{equation}\label{M Psi n}
%{\mathfrak M}_{A^n}^\sharp (s) \leq \big( 2 C (k_0) {\mathfrak M}_{A}^\sharp (s_0) \big)^{n - 1} {\mathfrak M}_{A}^\sharp(s)\,.
%\end{equation}
%Moreover, by \eqref{K cal A cal B} and \eqref{M Psi n},  if $ \langle \pa_{\vphi,x}  \rangle^{\mathtt b} A $ is $ {\cal D}^{k_0} $-modulo-tame, then,  
%for all $ n \geq 2 $, each $ \langle \partial_{\vphi,x} \rangle^{\mathtt b}  A^n $ 
%is $ {\cal D}^{k_0}$-modulo-tame with
%\begin{equation}\label{K Psi n}
%{\mathfrak M}_{\langle\partial_{\vphi, x} \rangle^{ \mathtt b} A^n}^\sharp (s) \leq  (4 C(\mathtt b) C(k_0))^{n - 1} \Big( 
%{\mathfrak M}^\sharp_{\langle \partial_{\vphi, x} \rangle^{\mathtt b} A}(s) \big[ {\mathfrak M}^\sharp_A(s_0) \big]^{n - 1}  
% + {\mathfrak M}^\sharp_{\langle \partial_{\vphi, x} \rangle^{\mathtt b} A}(s_0) {\mathfrak M}_A^\sharp(s) 
% \big[ {\mathfrak M}_A^\sharp(s_0) \big]^{n - 2} \Big)\,.
%\end{equation}
%Estimates \eqref{M Psi n}, \eqref{K Psi n} hold also in the case when $A$ is a matrix operator of the form \eqref{operatori in coordinate complesse-bis}. 

Iterating \eqref{modulo tame constant for composition}-\eqref{K cal A cal B}, one estimates ${\mathfrak M}_{\langle\partial_{\vphi, x} \rangle^{ \mathtt b} A^n}^\sharp (s)$, and arguing as in Lemma 2.26 of \cite{BertiMontalto} we 
deduce the following lemma.

\begin{lemma}{\bf (Invertibility)}\label{serie di neumann per maggiorantia}
Let $\Phi := {\rm Id} + A $, where $ A $ and 
$ \langle \partial_{\vphi, x} \rangle^{\mathtt b}  A $ are $ {\cal D}^{k_0}$-modulo-tame.
% ${\cal D}^{k_0}$-modulo-tame with modulo-tame constant $ {\mathfrak M}_{A}^\sharp (s)  $. 
Assume  the smallness condition 
\begin{equation}\label{piccolezza neumann tamea}
4 C(\mathtt b) C(k_0)  {\mathfrak M}_{A}^\sharp (s_0)  \leq 1/ 2\,.
\end{equation}
Then the operator $ \Phi $ is invertible,  
$\check A :=   \Phi^{- 1} - {\rm Id}  $  is 
${\cal D}^{k_0}$-modulo-tame, as well as $ \langle \partial_{\vphi, x} \rangle^{\mathtt b}  \check A $,
and they admit modulo-tame constants satisfying 
$$
 {\mathfrak M}_{\check A}^\sharp (s) \leq  2 {\mathfrak M}_A^\sharp (s) \, , \quad 
{\mathfrak M}_{\langle \partial_{\vphi, x} \rangle^{\mathtt b}  \check A}^\sharp (s)  \leq 
2 {\mathfrak M}_{ \langle \partial_{\vphi, x} \rangle^{\mathtt b}A}^\sharp (s)  + 
8  C(\mathtt b) C(k_0)  {\mathfrak M}_{ \langle \partial_{\vphi, x} \rangle^{\mathtt b}A}^\sharp (s_0)\, {\mathfrak M}_A^\sharp(s) \, .  
$$
The same statement holds if $A$ is a matrix operator of the form \eqref{operatori in coordinate complesse-bis}. 
\end{lemma}

\begin{corollary}\label{serie di neumann per maggioranti}
Let $m \in \R$, $\Phi := {\rm Id} + A $ where 
$ \langle D \rangle^m A \langle D \rangle^{- m}$ and 
$ \langle \partial_{\vphi, x} \rangle^{\mathtt b}  \langle D \rangle^m A \langle D \rangle^{- m}$ 
are ${\cal D}^{k_0}$-modulo-tame. 
Assume the smallness condition 
\begin{equation}\label{piccolezza neumann tame}
4 C(\mathtt b) C(k_0)  {\mathfrak M}_{\langle D \rangle^m A \langle D \rangle^{- m}}^\sharp (s_0)  \leq 1/ 2\,.
\end{equation}
Let $\check A :=   \Phi^{- 1} - {\rm Id}$. 
Then the operators $\langle D \rangle^m \check A \langle D \rangle^{- m}$ and  
$ \langle \partial_{\vphi, x} \rangle^{\mathtt b} \langle D \rangle^m \check A \langle D \rangle^{- m}$ are ${\cal D}^{k_0}$-modulo-tame and they admit modulo-tame constants satisfying 
\begin{align*}
& \qquad \qquad \qquad  \qquad \qquad \qquad \quad  {\mathfrak M}_{\langle D \rangle^m \check A \langle D \rangle^{- m}}^\sharp (s) \leq  2 {\mathfrak M}_{\langle D \rangle^m A \langle D \rangle^{- m}}^\sharp (s) \,, 
\\
& {\mathfrak M}_{\langle \partial_{\vphi, x} \rangle^{\mathtt b}   \langle D \rangle^m \check A \langle D \rangle^{- m} }^\sharp (s)  \! \leq \!
2 {\mathfrak M}_{ \langle \partial_{\vphi, x} \rangle^{\mathtt b}\langle D \rangle^m A \langle D \rangle^{- m}}^\sharp (s) 
\!  + \!
8  C(\mathtt b) C(k_0)  {\mathfrak M}_{ \langle \partial_{\vphi, x} \rangle^{\mathtt b}\langle D \rangle^m A \langle D \rangle^{- m}}^\sharp (s_0)\, {\mathfrak M}_{\langle D \rangle^m A \langle D \rangle^{- m}}^\sharp(s) \, .  
\end{align*}
The same statement holds if $A$ is a matrix operator of the form \eqref{operatori in coordinate complesse-bis}. 
\end{corollary}

\begin{proof}
Let us write 
$ \Phi_m := \langle D \rangle^m \Phi \langle D \rangle^{- m} = {\rm Id} + A_m $ with 
$ A_m := \langle D \rangle^m A \langle D \rangle^{- m} $. 
The corollary follows by Lemma \ref{serie di neumann per maggiorantia}, since the smallness condition
 \eqref{piccolezza neumann tame} is \eqref{piccolezza neumann tamea} with $A = A_m$, and
 $\Phi_m^{-1} =  {\rm Id} + \langle D \rangle^m \check A \langle D \rangle^{- m} $.  
\end{proof}

\begin{lemma} \label{lemma:smoothing-tame} {\bf (Smoothing)}
Suppose that $ \langle \pa_{\vphi,x} \rangle^{\mathtt b} A $, $ {\mathtt b} \geq  0 $, is $ {\cal D}^{k_0} $-modulo-tame. Then 
the operator $ \Pi_N^\bot A $ (see Definition \ref{def:maj}) is $ {\cal D}^{k_0} $-modulo-tame with a modulo-tame constant satisfying 
\be\label{proprieta tame proiettori moduli}
{\mathfrak M}_{\Pi_N^\bot A}^\sharp (s) \leq N^{- {\mathtt b} }{\mathfrak M}_{ \langle \pa_{\vphi, x} \rangle^{\mathtt b} A}^\sharp (s) \, ,
\quad
{\mathfrak M}_{\Pi_N^\bot A}^\sharp (s) \leq  {\mathfrak M}_{ A}^\sharp (s) \, . 
\ee
The same estimate holds when $A$ is a matrix operator of the form \eqref{operatori in coordinate complesse-bis}. 
\end{lemma}

\begin{proof}
As in Lemma 2.27 in \cite{BertiMontalto}, replacing $\langle \partial_\vphi \rangle^{\mathtt b}$ (cf. Definition 2.3 in \cite{BertiMontalto}) with $\langle \partial_{\vphi, x} \rangle^{\mathtt b}$. 
\end{proof}

In order to verify that an operator is modulo-tame, we shall use the following Lemma. Notice that the right hand side of \eqref{tame op modulo tame op lemma astratto} below contains 
 tame constants
(not modulo-tame) of operators which control more space and time derivatives than $\langle \partial_{\vphi, x} \rangle^{\mathtt b}\langle D \rangle^\fm A \langle D \rangle^\fm$. 

\begin{lemma}\label{lem: Initialization astratto} Let $\mathtt b, \fm \geq 0$. Then 
  \begin{equation}\label{tame op modulo tame op lemma astratto}
  {\mathfrak M}_{\langle \partial_{\vphi, x} \rangle^{\mathtt b}\langle D \rangle^\fm A \langle D \rangle^\fm}^\sharp ( s) \lesssim_{s_0, \mathtt b}  {\mathfrak M}_{ \langle D \rangle^{\fm + \mathtt b} A \langle D \rangle^{\fm + \mathtt b +  1} }(s) + \max_{i =1, \ldots,  \nu} \big\{
{\mathfrak M}_{ \partial_{\vphi_i}^{s_0 +  {\mathtt b}} \langle D \rangle^{\fm + \mathtt b} A \langle D \rangle^{\fm + \mathtt b +  1} }(s)  \big\}   \,. 
\end{equation}
\end{lemma}

\begin{proof}
We denote by $ {\mathbb M} (s, \mathtt b)  $ the right hand side in \eqref{tame op modulo tame op lemma astratto}.
For any $\a, \b \in \N $, the matrix elements of the operator 
$\pa_{\ph_i}^\a \langle D \rangle^\b A \langle D \rangle^{\b + 1}$
are  $\ii^\a (\ell_i - \ell'_i)^\a \langle j \rangle^\b A_j^{j'}(\ell - \ell') 
\langle j' \rangle^{\b+1} $. 
Then, by \eqref{tame-coeff} with $ \s = 0 $, applied to 
the operators $ \langle D \rangle^{\fm + \mathtt b} A \langle D \rangle^{\fm + \mathtt b +  1} $ 
and 
$ \partial_{\vphi_i}^{s_0 +  {\mathtt b}} \langle D \rangle^{\fm + \mathtt b} A \langle D \rangle^{\fm + \mathtt b +  1} $, 
we get, using the inequality 
$ \langle \ell - \ell' \rangle^{2 (s_0+{\mathtt b})}   \lesssim_{\mathtt b} 
1 + \max_{i = 1, \ldots, \nu}|\ell_i - \ell'_i|^{2 (s_0+{\mathtt b})} $, the bound
%and recalling the definition of ${\mathbb M}(s,\ttb)$, 
% estimates \eqref{monfalcone 5}-\eqref{monfalcone 4} imply
%Then, by \eqref{tame-coeff} with $ \s = 0 $, 
%we have that
% $ \forall |k | \leq k_0 $, $ s_0 \leq s \leq S $, $\ell' \in \Z^\nu,  j' \in \Z$, 
%\begin{align}
%& \gamma^{2 |k|} {\mathop \sum}_{\ell , j} \langle \ell, j \rangle^{2 s} \langle j \rangle^{2 (\fm + \mathtt b)} |\partial_\lambda^k A_j^{j'}(\ell - \ell')|^2 \langle j' \rangle^{2 (\fm + \mathtt b +  1)} 
%\leq 2 {\mathbb M}^2 (s_0, \mathtt b) \langle \ell', j' \rangle^{2 s} + 2  {\mathbb M}^2(s, \mathtt b) \langle \ell', j' \rangle^{2 s_0} \,, 
%\label{monfalcone 5} 
%\\
%& \gamma^{2 |k|} {\mathop \sum}_{\ell, j} \langle \ell, j \rangle^{2 s} |\ell_i - \ell'_i|^{2 (s_0 + \mathtt b)} \langle j \rangle^{2 (\fm + \mathtt b)} 
%|\partial_\lambda^k A_j^{j'}(\ell - \ell')|^2 \langle j' \rangle^{2 (\fm + \mathtt b + 1)} 
%\notag \\ & \qquad 
%\leq 2  {\mathbb M}^2(s_0, \mathtt b) \langle \ell', j' \rangle^{2 s} 
%+ 2  {\mathbb M}^2(s, \mathtt b) \langle \ell', j' \rangle^{2 s_0}\,. 
%\label{monfalcone 4} 
%\end{align}
%Using the inequality $ \langle \ell - \ell' \rangle^{2 \a}   \lesssim_{\a} 
%1 + \max_{i = 1, \ldots, \nu}|\ell_i - \ell'_i|^{2 \a} $
%for 
%%$ \a = s_0 $ and 
%$ \a = s_0+{\mathtt b} $, 
%%and recalling the definition of ${\mathbb M}(s,\ttb)$,  
%estimates \eqref{monfalcone 5}-\eqref{monfalcone 4} imply
\begin{align}
%&  \!\! \! \gamma^{2|k|} {\mathop \sum}_{\ell, j} \langle \ell, j \rangle^{2 s} \langle \ell - \ell' \rangle^{2 s_0}
% \langle j \rangle^{2 \fm} |\partial_\lambda^k {R}_j^{j'}(\ell - \ell')|^2  \langle j' \rangle^{2 (\fm + 1)}
% \lesssim_{\mathtt b} {\mathfrak M}_0^2(s_0, {\mathtt b}) \langle \ell', j'\rangle^{2 s} + 
%{\mathfrak M}_0^2(s, {\mathtt b}) \langle \ell', j' \rangle^{2 s_0}  
%\label{carletto0} \\
& \!\! \! \gamma^{2|k|}  {\mathop \sum}_{\ell, j} \langle \ell, j \rangle^{2 s} 
 \langle \ell - \ell' \rangle^{2 ({s_0+{\mathtt b}})}  \langle j \rangle^{2 (\fm + \mathtt b)} 
  | \partial_\lambda^k {A}_j^{j'}(\ell - \ell')|^2 \langle j' \rangle^{2 (\fm + \mathtt b + 1)}  
\notag \\ 
& \qquad 
\lesssim_{\mathtt b}  {\mathbb M}^2(s_0, {\mathtt b}) \langle \ell', j'\rangle^{2 s} + 
{\mathbb M}^2(s, {\mathtt b}) \langle \ell', j' \rangle^{2 s_0}  \, .  \label{carletto}
\end{align}
For all $|k | \leq k_0 $, by Cauchy-Schwarz inequality and using that 
\begin{equation} \label{1309.1}
\langle\ell - \ell', j - j' \rangle^{\mathtt b} 
\lesssim_\ttb \langle \ell - \ell' \rangle^\mathtt b \langle j - j' \rangle^{\mathtt b} 
\lesssim_\ttb \langle \ell - \ell' \rangle^{\mathtt b }(\langle j  \rangle^{\mathtt b} + \langle j'  \rangle^{\mathtt b} ) 
\lesssim_\ttb \langle \ell - \ell' \rangle^{\mathtt b } \langle j \rangle^{\mathtt b} \langle j' \rangle^{\mathtt b}
\end{equation} 
we get
\begin{align}
\| |\langle \partial_{\vphi, x} \rangle^{\mathtt b} \langle D \rangle^\fm 
& \partial_\lambda^k A \langle D \rangle^\fm | h \|_s^2 
\lesssim_{\mathtt b} 
{\mathop \sum}_{\ell, j} \langle \ell, j \rangle^{2 s} 
\Big( {\mathop \sum}_{\ell', j'} | \langle \ell - \ell' \rangle^{\mathtt b} \langle j \rangle^{\fm + \mathtt b} \partial_\lambda^k   A_j^{j'}(\ell - \ell') \langle j' \rangle^{\fm + \mathtt b}| |h_{\ell', j'}| \Big)^2 
\nonumber \\
& \lesssim_{\mathtt b} {\mathop \sum}_{\ell, j} \langle \ell, j \rangle^{2 s} 
\Big(  {\mathop \sum}_{\ell', j'} \langle \ell - \ell' \rangle^{s_0 + \mathtt b} \langle  j \rangle^{\fm + \mathtt b}
|\partial_\lambda^k A_j^{j'}(\ell - \ell')| \langle j' \rangle^{\fm + \mathtt b + 1} |h_{\ell', j'}| \frac{1}{\langle \ell - \ell' \rangle^{s_0} \langle j' \rangle } \Big)^2 \nonumber \\
& \lesssim_{s_0, \mathtt b} {\mathop \sum}_{\ell, j} \langle \ell, j \rangle^{2 s} 
{\mathop \sum}_{\ell', j'} \langle \ell - \ell' \rangle^{2 (s_0 + \mathtt b)} 
\langle j \rangle^{2 (\fm + \mathtt b)} |\partial_\lambda^k A_j^{j'}(\ell - \ell')|^2 \langle j' \rangle^{2 (\fm + \mathtt b + 1)} |h_{\ell', j'}|^2 \nonumber \\
& \lesssim_{s_0, \mathtt b}  {\mathop \sum}_{\ell' , j'} |h_{\ell', j'}|^2 
{\mathop \sum}_{\ell, j} \langle \ell, j \rangle^{2 s} \langle \ell - \ell' \rangle^{2 (s_0 + \mathtt b)} \langle j \rangle^{2 (\fm + \mathtt b)} |\partial_\lambda^k A_j^{j'}(\ell - \ell')|^2 \langle j' \rangle^{2 (\fm + \mathtt b + 1)} \nonumber \\
& \stackrel{ \eqref{carletto}}{{\lesssim}_{s_0, \mathtt b}}  \gamma^{- 2 |k|}
{\mathop \sum}_{\ell', j'} |h_{\ell', j'}|^2 
\big( {\mathbb M}^2(s_0, \mathtt b) \langle \ell', j' \rangle^{2 s} + {\mathbb M}^2(s, \mathtt b) \langle \ell', j' \rangle^{2 s_0}    \big)
\nonumber   \\
 & \lesssim_{s_0, \mathtt b}  \gamma^{- 2|k|} \big( {\mathbb M}^2(s_0, \mathtt b) \| h \|_{s}^2 + {\mathbb M}^2(s, \mathtt b) \| h \|_{s_0}^2 \big)   \label{tame-inizio}
\end{align}
using \eqref{Soboequals}, whence the claimed statement follows.  
%Therefore (recall 
% \eqref{CK0-tame}) one has    
% $ {\mathfrak M}_{ \langle \pa_{\vphi, x} \rangle^{\mathtt b} \langle D \rangle^\fm {R} \langle D \rangle^\fm }^\sharp (s) \lesssim_{s_0, \mathtt b} {\mathfrak M}_0(s, \mathtt b) $. 
%%Since \eqref{tame-inizio} holds for both $R = {R}_1^{(0)}$ and $R = R_2^{(0)}$, we have proved that 
%%$$ 
%%  {\mathfrak M}_{ \langle \pa_{\vphi, x} \rangle^{\mathtt b} \langle D \rangle^\fm {\cal R}_0  \langle D \rangle^\fm  }^\sharp (s) 
%% \lesssim_{s_0, \mathtt b}  {\mathfrak M}_0 (s, {\mathtt b})  \, . 
% % $$ 
%  The inequality   $ {\mathfrak M}_{  \langle D \rangle^\fm {\cal R}_0  \langle D \rangle^\fm  }^\sharp (s)  \lesssim  {\mathfrak M}_0 (s, {\mathtt b})  $ follows similarly. %  by \eqref{carletto0}. 
\end{proof}

\begin{lemma}\label{lemma coniugazione proiettore pi 0}
Let $\pi_0$ be the projector defined in \eqref{def pi0} by $\pi_0 u := \frac{1}{2\pi} \int_\T u(x)\, d x$. 
Let $A,B$ be $\vphi$-dependent families of operators as in \eqref{matrice operatori Toplitz} that, 
together with their adjoints $A^*, B^*$ with respect to the $L^2_x$ scalar product,
are ${\cal D}^{k_0}$-$\sigma$-tame. 
Let $m_1, m_2 \geq 0$, $\beta_0 \in \N$. 
Then for any $\beta \in \N^\nu$, $|\beta| \leq \beta_0$, the operator 
$\langle D \rangle^{m_1} \big( \partial_\vphi^\beta ( A \pi_0 B - \pi_0 ) \big) \langle D \rangle^{m_2}$ is ${\cal D}^{k_0}$-tame with a tame constant satisfying, for all $s \geq s_0$, 
\be
\begin{aligned}\label{tame-c-finale}
\mathfrak M_{\langle D \rangle^{m_1} ( \partial_\vphi^\beta ( A \pi_0 B - \pi_0 ) ) 
\langle D \rangle^{m_2}}(s) 
& \lesssim_{m, s, \beta_0, k_0} 
\mathfrak M_{A - \Id}(s + \beta_0 + m_1) \big( 1 + \mathfrak M_{B^* - \Id}(s_0 + m_2) \big) 
\\ & \qquad \quad
+ \mathfrak M_{B^* - \Id}(s + \beta_0 + m_2) \big( 1 + \mathfrak M_{A - \Id}(s_0 + m_1) \big).
\end{aligned}
\ee
The same estimate holds if $A, B$ are matrix operators of the form \eqref{operatori in coordinate complesse-bis} 
and $\pi_0$ is replaced by the matrix operator 
$\Pi_0$ defined in \eqref{def pauli matrix cal R4}.
\end{lemma}

\begin{proof}
A direct calculation shows that $\langle D \rangle^{m_1} \big( A \pi_0 B - \pi_0\big) \langle D \rangle^{m_2} [h] 
= g_1 ( h, g_2 )_{L^2_x} + (h, g_3)_{L^2_x}$ where $ g_1, g_2, g_3 $ are the functions defined by
$$
g_1 := \frac{1}{2\pi}\, \langle D \rangle^{m_1}(A - {\rm Id})[1]\,, \quad 
g_2 := \langle D \rangle^{m_2} B^*[1]\,, \quad 
g_3 := \frac{1}{2\pi}\, \langle D \rangle^{m_2}(B^* - {\rm Id})[1] \,. 
$$
The estimate \eqref{tame-c-finale} then follows by computing for any $\beta \in \N^\nu$, $k \in \N^{\nu + 1}$ with $|\beta| \leq \beta_0$, $|k| \leq k_0$, the operator $\partial_\lambda^k \partial_\vphi^\beta \big(\langle D \rangle^{m_1} ( A \pi_0 B - \pi_0 ) \langle D \rangle^{m_2} \big)$. 
\end{proof}

\subsection{Tame estimates for the flow of pseudo-PDEs}\label{AppendiceA} 

We report in this section several 
results 
concerning  tame estimates  for the flow $ \Phi^\tau $ of the pseudo-PDE
\begin{equation}\label{pseudo PDE}
\begin{cases}
\partial_\tau u = \ii a(\vphi, x) |D|^{\frac12}  u  \\
u(0, x) = u_0(\vphi, x) \, , 
\end{cases}\qquad \vphi \in \T^\nu\,, \quad  x \in \T\,,
\end{equation}
where $a(\vphi, x) = a(\lambda, \vphi, x) $ is a real valued function that is $ {\cal C}^\infty$ 
with respect to the variables $(\vphi, x)$ and $ k_0 $ times differentiable with respect to the parameters 
$\lambda =(\omega, \h ) $. 
The function $ a := a(i) $  may depend also on the ``approximate" torus $ i (\vphi )$.  
Most of these results  have been obtained  in the Appendix of \cite{BertiMontalto}. 

The flow operator 
$ \Phi^\tau := \Phi (\tau) := \Phi (\lambda, \vphi, \tau) $ 
satisfies the equation 
\begin{equation}\label{flow-propagator-beta-1}
\begin{cases}
\partial_\tau \Phi(\tau) = \ii a(\vphi, x) |D|^{\frac12} \Phi(\tau) \\
\Phi(0) = {\rm Id}\,.
\end{cases}
\end{equation}
Since the function $ a (\vphi, x)  $ is real valued, usual energy estimates imply that the flow $ \Phi (\tau) $ is a 
bounded operator mapping $  H^s_x $ to $ H^s_x $.
In the Appendix of \cite{BertiMontalto} it is proved that 
the flow $ \Phi (\tau) $ satisfies also tame estimates in $ H^s_{\vphi, x} $, see Proposition \ref{proposition 2.40 unificata} below.
Moreover, since \eqref{pseudo PDE} is an autonomous equation, its flow $ \Phi (\vphi, \tau) $  satisfies the group property 
\be\label{group-flow}
\Phi (\vphi, \tau_1 + \tau_2 ) =  \Phi (\vphi, \tau_1  ) \circ  \Phi (\vphi, \tau_2 ) \, ,  \qquad 
\Phi (\vphi, \tau )^{-1} = \Phi (\vphi, - \tau ) \, , 
\ee
and, 
since $ a  (\lambda, \cdot) $ is $ {k_0} $ times differentiable  with respect to the parameter $ \lambda $, then
 $ \Phi (\lambda, \vphi, \tau) $   is  $ k_0 $ times differentiable with 
respect to $ \lambda $ as well. 
Also notice that  $  \Phi^{-1} (\tau) = \Phi( - \tau ) = \overline \Phi ( \tau ) $, because these operators solve the same Cauchy problem.  
Moreover, if $ a(\vphi, x)   $ is $ \odd(\vphi)\even(x) $, 
then, recalling Section \ref{sezione operatori reversibili e even}, the real operator 
$$ 
{\bf \Phi} (\vphi, \tau ) := \begin{pmatrix} \Phi (\vphi, \tau) & 0 \\ 0 & \overline\Phi (\vphi, \tau) \end{pmatrix} 
$$ 
is even and reversibility preserving.

The operator  $ \partial_\lambda^k \pa_{\vphi}^{\beta} \Phi $ 
loses $ | D_x |^{\frac{|\b| + |k|}{2}} $ derivatives, 
which,  in \eqref{copenaghen B omega} below,  % and \eqref{stima derivate flusso generalissima derivate i}
are compensated  by $ \langle D \rangle^{- m_1} $ on the left hand side and $\langle D \rangle^{- m_2}$ on the right hand side, 
with $m_1, m_2 \in \R$ satisfying $m_1 + m_2= \frac{|\b| + |k|}{2}$. 
The following proposition provides tame estimates in the Sobolev spaces
$ H^{s}_{\vphi, x} $.  

\begin{proposition} \label{proposition 2.40 unificata}
Let $ \beta_0, k_0 \in \N$. 
For any $\beta, k \in \N^\nu$ with $|\b| \leq \b_0$, $|k| \leq k_0$,
for any $m_1, m_2 \in \R$ with $m_1 + m_2 = \frac{|\beta| + |k|}{2}$,
for any $s \geq s_0$, 
there exist constants 
$ \sigma(|\beta|, |k|, m_1, m_2) > 0$, $\delta(s, m_1) > 0$ such that if
\begin{equation}\label{piccolezza a partial vphi beta k D beta k}
\| a \|_{ 2s_0 + |m_1| + 2} \leq \delta (s, m_1) \,, \quad  
\| a \|_{s_0 + \sigma(\b_0, k_0, m_1, m_2)}^{k_0, \gamma} \leq 1\,, 
%\quad \s := \s(\b_0, k_0, m_1, m_2) 
\end{equation}
%with $\delta(s) > 0$ small enough, $s \geq s_0$, 
then the flow $ \Phi(\tau) := \Phi(\lambda, \vphi, \tau)$ of \eqref{pseudo PDE} satisfies
%for all $s \geq s_0$:
\begin{align}
& \sup_{\tau \in [0, 1]} \| \langle D \rangle^{- m_1} \partial_\lambda^k \partial_\vphi^\beta \Phi (\tau)  \langle D \rangle^{- m_2} h \|_s 
\lesssim_{s,\b_0, k_0, m_1,m_2} \gamma^{- |k|} 
\Big( \| h \|_s + \| a \|_{s + \sigma( |\beta|, |k|, m_1, m_2)}^{k_0, \gamma} \| h \|_{s_0} \Big)\label{copenaghen B omega} \\
& 
\sup_{\tau \in [0, 1]} \| \partial_\lambda^k (\Phi(\tau )- {\rm I d}) h \|_s   \lesssim_s \gamma^{- |k|}\big( \| a \|_{s_0}^{k_0, \gamma} \|  h\|_{s + \frac{|k| + 1}{2}} + \| a \|_{s + s_0 + k_0 + \frac32}^{k_0, \gamma} \| h \|_{s_0 + \frac{|k| + 1}{2}}  \big) \, .  
\label{arido 2}
\end{align}
\end{proposition}

\begin{proof}
The proof is similar to Propositions A.7, A.10 and A.11 in \cite{BertiMontalto} with, in addition, the presence of $\langle D \rangle^{- m_1}$ and $\langle D \rangle^{- m_2}$ in \eqref{copenaghen B omega}.
\end{proof}

We consider also the dependence of the flow $\Phi$  with respect to the torus $ i := i(\vphi ) $ and the estimates for the adjoint operator $\Phi^*$.  

\begin{lemma} \label{lemma 2.42 unificato}
Let $s_1 > s_0$,  $ \beta_0 \in \N$. 
For any $\beta \in \N^\nu$, $|\beta| \leq \beta_0$, for any $m_1, m_2 \in \R$ satisfying $m_1 + m_2 = \frac{|\beta| + 1}{2}$ there exists 
a constant $\sigma(|\beta|) = \sigma(|\beta|, m_1, m_2) > 0$ such that if
$ \| a \|_{s_1 + \sigma(\beta_0)} \leq \delta (s) $
with $\delta(s) > 0$ small enough, then the following estimate holds:  
\begin{align}
\sup_{\tau \in [0, 1]} \| \langle D \rangle^{- m_1}  \partial_\vphi^\beta \Delta_{12}\Phi (\tau)  \langle D \rangle^{- m_2} h \|_{s_1} 
\lesssim_{s_1} \| \Delta_{12} a \|_{s_1 + \sigma(|\beta|)} \| h \|_{s_1} \, ,
\label{stima derivate flusso generalissima derivate i}
\end{align}
where $\Delta_{12}\Phi:= \Phi(i_2) - \Phi(i_1)$ and $\Delta_{12}a:= a(i_2) - a(i_1)$. Moreover, for any 
$ k \in \N^{\nu + 1}$, $|k| \leq k_0$, for all $s \geq s_0 $, 
\begin{align*}
 \| (\partial_\lambda^k\Phi^* )  h \|_s & \lesssim_s \gamma^{- |k|} \big( \| h \|_{s + \frac{|k|}{2}} + \| a \|_{s + s_0 + |k| + \frac32}^{k_0, \gamma} \| h \|_{s_0 + \frac{|k|}{2}}\big)  \\
 \| \partial_\lambda^k(\Phi^* - {\rm Id}) h \|_s & \lesssim_s \gamma^{- |k|} \big(\| a \|_{s_0}^{k_0, \gamma} \| h \|_{s + \frac{|k| + 1}{2}} + \| a \|_{s + s_0 + |k| + 2}^{k_0, \gamma} \| h \|_{s_0 + \frac{|k| + 1}{2}} \big) \, . 
\end{align*}
Finally, for all $s \in [s_0, s_1] $,
$$
\| \Delta_{12} \Phi^* h \|_s 
\lesssim_s \| \Delta_{12} a \|_{s + s_0 + \frac12} \| h \|_{s + \frac12} \,.
$$
\end{lemma}

\begin{proof}
The proof is similar to Propositions A.13, A.14, A.17 and A.18 of \cite{BertiMontalto}.
\end{proof}

\section{Degenerate KAM theory}\label{sec:degenerate KAM}

In this section we extend the degenerate KAM theory approach  of \cite{BaBM} and \cite{BertiMontalto}.

\begin{definition}\label{def:non-deg}
A function $ f := (f_1, \ldots, f_N ) :  [ {\mathtt h}_1, {\mathtt h}_2] \to \R^N $ is called non-degenerate 
if, for any vector $ c := (c_1, \ldots, c_N ) \in \R^N \setminus \{0\}$, 
the function $ f \cdot c = f_1 c_1 + \ldots + f_N c_N $ 
is not identically zero on the whole interval $ [ {\mathtt h}_1, {\mathtt h}_2]  $. 
\end{definition}

From a geometric point of view, $ f $ non-degenerate means that the image of the curve $ f([{\mathtt h}_1,{\mathtt h}_2]) \subset \R^N $ is not contained in any  hyperplane of $ \R^N $. For such a reason a curve $ f $ which satisfies the non-degeneracy property of Definition \ref{def:non-deg} is also referred to as an {\it essentially non-planar} curve, or a curve with {\it full torsion}.  
Given $ \Splus \subset \N^+ $ we denote the unperturbed tangential and normal frequency vectors by 
\be\label{tangential-normal-frequencies}
{\vec \om} ( {\mathtt h} ) := ( \om_j ( {\mathtt h} ) )_{j \in \Splus} \, , \quad 
{\vec \Omega} ({\mathtt h} ) := ( \Om_j ( {\mathtt h} ) )_{j \in \N^+ \setminus \Splus} := 
( \om_j ( {\mathtt h} ) )_{j \in \N^+ \setminus \Splus} \, ,  
\ee
where $ \om_j ( {\mathtt h}) = \sqrt{j \tanh ( {\mathtt h} j)} $ are defined in \eqref{LIN:fre}.

\begin{lemma}\label{non degenerazione frequenze imperturbate}
{\bf (Non-degeneracy)}  The frequency vectors 
$ {\vec \om} ({\mathtt h}) \in \R^\nu $, $ ({\vec \om} ({\mathtt h} ), 1) \in \R^{\nu+1} $ and 
$$
({\vec \om} ({\mathtt h} ),  \Omega_j({\mathtt h} )) \in \R^{\nu+1}\,,
\quad ( {\vec \om} ( {\mathtt h} ),  \Omega_j( {\mathtt h} ),  \Omega_{j'}( {\mathtt h} )) \in \R^{\nu+2} \,, 	\ \forall j, j' \in 
\N^+ \setminus \Splus \,, \ j \neq j'\,,
$$
are non-degenerate.
\end{lemma}

\begin{proof}
We first prove that for any $ N $, for any 
$ \omega_{j_1}( {\mathtt h}), \ldots, \omega_{j_N}( {\mathtt h} ) $ with 
$ 1 \leq j_1 < j_2 < \ldots < j_N $ 
the function $  [{\mathtt h}_1, {\mathtt h}_2] \ni {\mathtt h} \mapsto (\omega_{j_1}({\mathtt h} ), \ldots, \omega_{j_N}({\mathtt h} )) \in \R^N $
is non-degenerate according to Definition \ref{def:non-deg}, namely that, for all $ c \in \R^N \setminus \{0\} $, the function    
$ {\mathtt h} \mapsto c_1 \omega_{j_1}({\mathtt h} ) + \ldots + c_N \omega_{j_N}({\mathtt h}) $ is not identically zero
on the interval  $  [{\mathtt h}_1, {\mathtt h}_2] $. 
We shall prove, equivalently, that the function 
$$ 
{\mathtt h} \mapsto c_1 \omega_{j_1}( {\mathtt h}^4 ) + \ldots + c_N \omega_{j_N}( {\mathtt h}^4) 
$$
is not identically zero on the interval $ [{\mathtt h}_1^4, {\mathtt h}_2^4] $. 
The advantage of replacing $\h$ with $\h^4$ is that each function  % (see \eqref{LIN:fre})  
$$ 
\h \mapsto \om_j ({\mathtt h}^4) = \sqrt{j \tanh ({\mathtt h}^4 j)} 
$$ 
is {\it analytic also in a neighborhood of $ {\mathtt h} = 0 $}, unlike the function 
$ \om_j ( {\mathtt h} ) = \sqrt{j \tanh ( {\mathtt h}  j)}  $. 
Clearly, % First we observe that 
the function $g_1(\h) := \sqrt{ \tanh (\h^4) }$ is analytic 
in a neighborhood % ball around $\h$, for every
of any  $\h \in \R \setminus \{ 0 \}$,  
because % , on such a ball, 
$g_1$ is the composition of analytic functions. 
Let us prove that it has an analytic continuation at $ \h = 0 $. 
The Taylor series at $ z = 0 $ of the hyperbolic tangent has the form 
$$
\tanh (z) = \sum_{n=0}^\infty T_{n} z^{2n+1}  = 
z- \frac{z^3}{3} + \frac{2}{15} z^5 + \ldots \, ,  
$$
and it is convergent for $ |z| < \pi / 2 $ 
(the poles of $ \tanh z $ closest  to $ z = 0 $ are $ \pm \ii \pi / 2 $). 
Then the power series 
$$
\tanh ( z^4 ) = \sum_{n=0}^\infty  T_n z^{4(2n+1)}  = z^4 \Big(1+  \sum_{n\geq 1}  T_n z^{8n} \Big) 
=  z^4- \frac{z^{12}}{3} + \frac{2}{15} z^{20} + \ldots  
$$
is convergent in $ | z | < (\pi/2)^{1/4} $.
Moreover $|\sum_{n\geq 1}  T_n z^{8n}| < 1$ in a ball $|z| < \d$, 
for some positive $\d$ sufficiently small. 
As a consequence, also the real function 
\be\label{def:gh}
g_1 ({\mathtt h}) := 
\om_1 ( {\mathtt h}^4) = \sqrt{\tanh ({\mathtt h}^4) } = 
{\mathtt h}^2 \Big(1+  \sum_{n\geq 1}  T_n {\mathtt h}^{8n} \Big)^{1/2} =
\sum_{n=0}^{+\infty} b_n \frac{{\mathtt h}^{8n+2}}{(8n+2)!} = {\mathtt h}^2 - \frac{{\mathtt h}^{10}}{6} + \ldots 
\ee
is analytic in the ball $|z| < \d$. 
Thus % implies that 
$g_1$ is analytic on the whole real axis.  The Taylor coefficients $b_n$ are computable. 
We expand in Taylor series at $ {\mathtt h} = 0 $ 
also each function, for $ j \geq 1 $,   
\be\label{def:gh-j}
g_j ({\mathtt h} ) := \om_j ( {\mathtt h}^4 ) = \sqrt{j} \sqrt{\tanh ( {\mathtt h}^4 j ) } =  
\sqrt{j} \, g_1 ( j^{1/4} {\mathtt h} ) =
\sum_{n=0}^{+\infty} b_n  j^{2n+1} \frac{{\mathtt h}^{8n+2}}{(8n+2)!} \,,
\ee
which is analytic on the whole $\R$, similarly as $g_1$. 

Now fix $ N $ integers $ 1 \leq j_1 <  j_2 < \ldots < j_N $. 
We prove that for all $ c \in \R^N \setminus \{0\} $, the analytic function 
$ c_1 g_{j_1}({\mathtt h}) + \ldots + c_N g_{j_N}({\mathtt h}) $ is not identically zero.
Suppose, by contradiction, that there exists $ c \in \R^N \setminus \{0\} $ such that 
\be\label{la-identita}
c_1 g_{j_1}({\mathtt h}) + \ldots + c_N g_{j_N}({\mathtt h})  = 0 
\quad \forall {\mathtt h} \in\mathbb R . 
\ee
The real analytic function $ g_1({\mathtt h}) $ defined in \eqref{def:gh} is not a polynomial 
(to see this, observe its limit as $\h \to \infty$).
Hence there exist $ N $ Taylor coefficients $ b_n \neq 0 $ of $ g_1 $,
say $ b_{n_1}, \ldots, b_{n_N} $ with $ n_1 < n_2 < \ldots < n_N $. % (they could be also computed). 
We differentiate with respect to $ {\mathtt h} $ the identity in \eqref{la-identita}  
 and we find
 $$
 \begin{cases}
 c_1 \big(D_{\mathtt h}^{(8n_1+2)} g_{j_1}\big)({\mathtt h}) + \ldots + c_N \big(D_{\mathtt h}^{(8n_1+2)} g_{j_N}\big)({\mathtt h})  = 0 \cr 
 c_1 \big(D_{\mathtt h}^{(8n_2+2)} g_{j_1}\big)({\mathtt h}) + \ldots + c_N \big(D_{\mathtt h}^{(8n_2+2)} g_{j_N}\big)({\mathtt h})  = 0 \cr
 \ldots  \ldots \ldots \cr
 c_1 \big(D_{\mathtt h}^{(8n_N+2)} g_{j_1}\big)({\mathtt h}) + \ldots + c_N \big(D_{\mathtt h}^{(8n_N+2)} g_{j_N}\big)({\mathtt h})  = 0 \, .
 \end{cases}
$$
As a consequence the $ N \times N $-matrix 
\be\label{matrixAN}
{\cal A}({\mathtt h}) := \begin{pmatrix}
\big(D_{\mathtt h}^{(8n_1+2)} g_{j_1}\big)({\mathtt h}) & 
 \dots & \big(D_h^{(8n_1+2)} g_{j_N}\big)({\mathtt h}) 
\\
\big(D_{\mathtt h}^{(8n_2+2)} g_{j_1}\big)({\mathtt h}) &  
\dots & \big(D_h^{(8n_2+2)} g_{j_N}\big)({\mathtt h} ) \\
\vdots & \ddots & \vdots \\
\big(D_{\mathtt h}^{(8n_N+2)} g_{j_1}\big)({\mathtt h}) & 
\dots & \big(D_{\mathtt h}^{(8n_N+2)} g_{j_N}\big)({\mathtt h}) 
\end{pmatrix}
\ee
is singular for all $ {\mathtt h} \in \R $, and so the analytic function 
\be\label{detAh0} 
\det {\cal A}({\mathtt h}) = 0 \quad  \forall {\mathtt h} \in \mathbb R 
\ee
is identically zero. 
In particular at $ {\mathtt h} = 0 $ we have  $ \det {\cal A}(0) = 0 $. 
On the other hand, by \eqref{def:gh-j} and the multi-linearity of the determinant we compute 
$$
\det {\cal A}(0) := \det \begin{pmatrix}
b_{n_1} j_1^{2n_1+1} & \dots & b_{n_1} j_N^{2n_1+1}   \\
b_{n_2} j_1^{2n_2+1}  &  \dots & b_{n_2} j_N^{2n_2+1}  \\
\vdots & \ddots & \vdots \\
b_{n_N} j_1^{2n_N+1} &  \dots & b_{n_N} j_N^{2n_N+1} 
\end{pmatrix} =
b_{n_1} \ldots b_{n_N}
\det \begin{pmatrix}
 j_1^{2n_1+1} &   \dots &  j_N^{2n_1+1} \\
 j_1^{2n_2+1}  &  \dots &  j_N^{2n_2+1} \\
\vdots & \ddots & \vdots   \\
j_1^{2n_N + 1} &  \dots &  j_N^{2 n_N + 1} 
\end{pmatrix} \, . 
$$
This is a generalized Vandermonde determinant. 
We use the following result. % (see e.g.\ \cite{RS}). % , \cite{SWZ}).

\begin{lemma}\label{lemma-VDM} 
Let $x_1, \ldots, x_N, \a_1, \ldots, \a_N$ be real numbers, with 
$ 0 < x_1 <  \ldots < x_N  $ and $ \a_1 < \ldots < \a_N $.  
Then 
$$
\det \begin{pmatrix}
 x_1^{\a_1}  &   \dots &   x_N^{\a_1}   \\
\vdots & \ddots & \vdots  \\
x_1^{\a_N} &  \dots &  x_N^{\a_N} 
\end{pmatrix} >  0 \, . 
$$
\end{lemma}

\begin{proof}
The lemma is proved in \cite{RS}. 
\end{proof}

Since $ 1 \leq j_1 < j_2 < \ldots < j_N $ and the exponents $ \a_j := 2 n_j + 1 $ are increasing 
$ \a_1 < \ldots < \a_N  $,  Lemma \ref{lemma-VDM} implies that $\det {\cal A}(0) \neq 0 $
(recall that $ b_{n_1},  \ldots,  b_{n_N} \neq 0 $). 
This is a contradiction with \eqref{detAh0}.  

\smallskip

In order to conclude the proof of Lemma \ref{non degenerazione frequenze imperturbate} 
we have to prove that, for any $ N $, for any 
%$ \omega_{j_1}({\mathtt h}), \ldots, \omega_{j_N}({\mathtt h}) $ with 
$ 1 \leq j_1 < j_2 < \ldots < j_N $, 
the function $  [{\mathtt h}_1, {\mathtt h}_2] \ni {\mathtt h}  \mapsto (1, \omega_{j_1}({\mathtt h} ), \ldots, \omega_{j_N}({\mathtt h} )) \in \R^{N+1} $
is non-degenerate according to Definition \ref{def:non-deg},
namely that, for all $ c = (c_0, c_1, \ldots, c_N ) 
\in \R^{N+1} \setminus \{0\}$, the function    
$ {\mathtt h}  \mapsto c_0 + c_1 \omega_{j_1}({\mathtt h} ) + \ldots + c_N \omega_{j_N}( {\mathtt h}) $ is not identically zero
on the interval  $  [{\mathtt h}_1, {\mathtt h}_2] $. 
We shall prove, equivalently, that the real analytic function 
$ {\mathtt h} \mapsto c_0 + c_1 \omega_{j_1}({\mathtt h}^4) + \ldots + c_N \omega_{j_N}({\mathtt h}^4) $
is not identically zero on $\mathbb R$.  

Suppose, by contradiction, that there exists 
$ c = (c_0, c_1, \ldots, c_N )  \in \R^{N+1} \setminus \{0\} $ such that 
\be\label{la-identita+1}
c_0 + c_1 g_{j_1}({\mathtt h}) + \ldots + c_N g_{j_N}({\mathtt h} )  = 0 
\quad \forall \h \in \R. 
\ee
As above, 
we differentiate with respect to $ {\mathtt h}  $ the identity \eqref{la-identita+1},  
and we find that 
the $ (N+1) \times (N+1) $-matrix 
\be\label{defB}
{\cal B}({\mathtt h}) := \begin{pmatrix}
1 & g_{j_1}({\mathtt h}) &  \ldots & g_{j_N}({\mathtt h}) \\
0 & (D_{\mathtt h}^{(8n_1+2)} g_{j_1})({\mathtt h}) & 
 \dots & (D_{\mathtt h}^{(8n_1+2)} g_{j_N})({\mathtt h}) \\
0 & \vdots & \ddots & \vdots  \\
0 & (D_{\mathtt h}^{(8n_N+2)} g_{j_1})({\mathtt h} ) & 
 \dots & (D_{\mathtt h}^{(8n_N+2)} g_{j_N})({\mathtt h} )
\end{pmatrix}
\ee
is singular for all $ {\mathtt h} \in \R $, 
and so the analytic function $\det \mB(\h) = 0$ for all $\h \in \R$. 
By expanding the determinant of the matrix in \eqref{defB} along the first column by Laplace we get 
$ \det {\cal B}({\mathtt h}) = \det {\cal A}({\mathtt h} ) $,
where the matrix $ {\cal A}({\mathtt h}) $ is  defined in \eqref{matrixAN}. 
We have already proved that $ \det {\cal A}(0) \neq 0 $, 
and this gives a contradiction. 
\end{proof}

In the next proposition 
we deduce the quantitative bounds \eqref{0 Melnikov}-\eqref{2 Melnikov+} 
from the qualitative non-degeneracy condition of Lemma 
\ref{non degenerazione frequenze imperturbate}, 
the analyticity of the linear frequencies $\om_j$ in \eqref{LIN:fre}, 
and their asymptotics \eqref{espansione asintotica degli autovalori}.

\begin{proposition}\label{Lemma: degenerate KAM}
{\bf (Transversality)}
There exist $ \barka \in \N $, $ \rho_0 > 0$ such that, 
for any $\h \in [ {\mathtt h}_1, {\mathtt h}_2] $,
\begin{align}\label{0 Melnikov}
\max_{k \leq \barka} 
|\partial_{\mathtt h}^{k}  \{{\vec \om} ({\mathtt h}) \cdot \ell   \} | & 
\geq \rho_0 \langle \ell \rangle\,, 
\quad \forall \ell  \in \Z^\nu \setminus \{ 0 \},   
\\
\label{1 Melnikov}
\max_{k \leq \barka}
|\partial_{\mathtt h}^{k}  \{{\vec \om} ({\mathtt h}) \cdot \ell  +  \Omega_j ({\mathtt h} ) \} | 
& \geq \rho_0 \langle \ell  \rangle\,, 
\quad \forall \ell  \in \Z^\nu, \  j \in \N^+ \setminus \Splus,  
\\
\label{2 Melnikov-}
\max_{k \leq \barka}
|\partial_{\mathtt h}^{k}  \{{\vec \om} ({\mathtt h} ) \cdot \ell  
 + \Omega_j ({\mathtt h} ) - \Omega_{j'}( {\mathtt h} ) \} | 
& \geq \rho_0 \langle \ell  \rangle\,, 
\quad \forall \ell \in \Z^\nu \setminus \{0\}, \  j, j' \in \N^+ \setminus \Splus,   
\\
\label{2 Melnikov+}
\max_{k \leq \barka} 
|\partial_{\mathtt h}^{k}  \{{\vec \om} ({\mathtt h} ) \cdot \ell 
+  \Omega_j ({\mathtt h} ) +  \Omega_{j'}({\mathtt h} ) \} | 
& \geq \rho_0 \langle \ell  \rangle\,, 
\quad \forall \ell  \in \Z^{\nu}, \  j, j' \in \N^+ \setminus \Splus 
\end{align}
where $\vec\om(\h)$ and $\Omega_j(\h)$ are defined in \eqref{tangential-normal-frequencies}.
We recall the notation $\la \ell \ra := \max \{ 1, |\ell| \}$.
We call (following \cite{Ru1}) $ \rho_0 $  the ``\emph{amount of non-degeneracy}'' 
and $ \barka $ the ``\emph{index of non-degeneracy}''.
\end{proposition}

Note that  in  \eqref{2 Melnikov-} we exclude the index $ \ell = 0 $. 
In this case we directly have that, for all $ {\mathtt h} \in [{\mathtt h}_1, {\mathtt h}_2]$
\be\label{caso-ell=0}
|  \Omega_j ({\mathtt h}) - \Omega_{j'}({\mathtt h} ) | \geq 
c_1 | \sqrt{j} -\sqrt{j'}| = c_1 \frac{| j - j' |}{\sqrt{j} + \sqrt{j'}}  
\quad \forall j, j' \in \N^+, 
\qquad \text{where} \ c_1 := \sqrt{\tanh ( {\mathtt h}_1) } \,.
\ee

\begin{proof} All the inequalities  \eqref{0 Melnikov}-\eqref{2 Melnikov+} are proved by contradiction.

{\sc Proof of \eqref{0 Melnikov}}.
Suppose that for all $\barka \in \N$, for all $\rho_0 > 0$ there exist 
$ \ell \in \Z^\nu \setminus \{0\} $, $ {\mathtt h} \in [ {\mathtt h}_1, {\mathtt h}_2]$  such that 
$ \max_{k \leq \barka} |\partial_{\mathtt h}^k\{ {\vec \om}({\mathtt h}) \cdot \ell  \} | < \rho_0 \langle \ell  \rangle $. 
This implies that for all $ m \in \N $, taking $ \barka = m$, $\rho_0 = \frac{1}{1+m }$,
there exist $ \ell_m \in \Z^\nu \setminus \{0\} $, $ {\mathtt h}_m \in [ {\mathtt h}_1, {\mathtt h}_2]$ such that 
$$
\max_{k \leq m} |  \partial_{\mathtt h}^k \{ {\vec \om} ( {\mathtt h}_m) \cdot \ell _m \}| <  \frac{1}{1 + m} 
 \langle \ell_m  \rangle
$$
and therefore 
\begin{equation}\label{rick 1}
\forall k \in \N, \quad \forall m \geq k \, , \quad 
\Big| \partial_{\mathtt h}^k {\vec \om} ( {\mathtt h}_m) \cdot \frac{\ell_m}{\langle \ell_m \rangle} \Big| < \frac{1}{1 + m}\,.
\end{equation}
The sequences $({\mathtt h}_m)_{m \in \N} \subset [ {\mathtt h}_1, {\mathtt h}_2] $
and $( \ell_m / \langle \ell_m \rangle)_{m \in \N} \subset \R^\nu \setminus \{0\} $ are bounded.
By compactness  there exists a sequence $ m_n \to + \infty $ such that 
$ {\mathtt h}_{m_n} \to \bar {\mathtt h} \in [ {\mathtt h}_1, {\mathtt h}_2], $ $ \ell_{m_n} / \langle \ell_{m_n} \rangle \to \bar c \neq 0 $. 
Passing to the limit in \eqref{rick 1} for $m_n \to + \infty$ we deduce that 
$ \partial_{\mathtt h}^k {\vec \om}(\bar {\mathtt h}) \cdot \bar c = 0 $ for all $ k \in \N $. 
We conclude that the analytic function $ {\mathtt h} \mapsto {\vec \om} ({\mathtt h} )\cdot \bar c $ is identically zero.  
Since $ \bar c \neq 0 $, 
this is in contradiction with Lemma \ref{non degenerazione frequenze imperturbate}.

\medskip

{\sc Proof of \eqref{1 Melnikov}}. 
First of all note that for all ${\mathtt h} \in [ {\mathtt h}_1, {\mathtt h}_2] $, we have 
$ |{\vec \om} ( {\mathtt h}  ) \cdot \ell  +  \Omega_j ({\mathtt h} )  | \geq $ 
$ \Omega_j ( {\mathtt h} )   - |{\vec \om} ({\mathtt h} )  \cdot \ell  | \geq $ 
$ c_1 j^{1/2} - C | \ell  | \geq  |\ell| $
if $ j^{1/2} \geq C_0 |\ell  | $ for some $ C_0 > 0 $. 
Therefore in \eqref{1 Melnikov} we can restrict to the indices 
$ (\ell  , j ) \in \Z^\nu \times (\N^+ \setminus \Splus) $ satisfying  
\be\label{prima restrizione}
j^{\frac12} < C_0 | \ell  | \, . 
\ee
Arguing by contradiction (as for proving \eqref{0 Melnikov}), we suppose that 
for all $m \in \N$ there exist  $ \ell_m \in \Z^\nu$, $ j_m \in \N^+ \setminus \Splus  $ and 
$ {\mathtt h}_m \in [ {\mathtt h}_1, {\mathtt h}_2]$, such that 
$$
\max_{k \leq m}\Big| 
\partial_{\mathtt h}^k \Big\{ {\vec \om}({\mathtt h}_m) \cdot \frac{\ell_m}{\langle \ell_m \rangle} + 
\frac{ \Omega_{j_m}({\mathtt h}_m)}{\langle \ell_m \rangle} \Big\} \Big| 
< \frac{1}{1 + m} 
$$
and therefore
\begin{equation}\label{rick 2}
\forall k \in \N, \quad \forall m \geq k \, , \quad 
\Big| 
\partial_{\mathtt h}^k \Big\{ {\vec \om}({\mathtt h}_m) \cdot \frac{\ell_m}{\langle \ell_m \rangle} + 
\frac{ \Omega_{j_m}({\mathtt h}_m)}{\langle \ell_m \rangle} \Big\} \Big| 
< \frac{1}{1 + m} \, . 
\end{equation}
Since the sequences $({\mathtt h}_m)_{m \in \N} \subset [{\mathtt h}_1, {\mathtt h}_2] $ 
and $(\ell_{m} / \langle \ell_m \rangle)_{m \in \N} \in \R^\nu $ are bounded, 
there exists a sequence $m_n \to + \infty$ such that 
\be\label{up to subsequence}
{\mathtt h}_{m_n} \to \bar {\mathtt h} \in [ {\mathtt h}_1, {\mathtt h}_2]\,,\quad \frac{\ell_{m_n}}{\langle \ell_{m_n} \rangle} \to \bar c \in \R^\nu \,.
\ee
We now distinguish two cases.

\emph{Case 1: $ (\ell_{m_n}) \subset \Z^\nu $ is bounded}. 
In this case, up to a subsequence,  $\ell_{m_n} \to \bar \ell \in \Z^\nu $, and since $|j_m| \leq C 
| \ell_m|^{2}$ for all $m$ (see \eqref{prima restrizione}), 
we have $j_{m_n} \to \bar \jmath $. Passing to the limit for $m_n \to + \infty$ in \eqref{rick 2} we deduce, by \eqref{up to subsequence}, that 
$$
\partial_{\mathtt h}^k \big\{ {\vec \om}(\bar {\mathtt h}) \cdot \bar c + 
 \Omega_{\bar \jmath} (\bar {\mathtt h})  \langle \bar \ell \rangle^{-1} \big\} = 0\,,\quad \forall k \in \N .
$$
Therefore the analytic function
$ {\mathtt h} \mapsto {\vec \om}(  {\mathtt h} ) \cdot \bar c + {\langle \bar \ell \rangle}^{-1} \Omega_{\bar \jmath}( {\mathtt h} )  $
is identically zero. Since
$ ( \bar c, \langle \bar \ell \rangle^{-1}) \neq 0 $  this is in contradiction with Lemma \ref{non degenerazione frequenze imperturbate}.

\emph{Case 2: $ (\ell_{m_n})$ is unbounded}. Up to a subsequence, $ | \ell_{m_n}| \to + \infty $.
In this case the constant $ \bar c $ in \eqref{up to subsequence} is nonzero. 
Moreover, by \eqref{prima restrizione}, we also have that, up to a subsequence, 
\be\label{rapporto tende bar d}
j_{m_n}^{\frac12}  \langle \ell_{m_n} \rangle^{-1} \to \bar d \in \R.
\ee
By \eqref{espansione asintotica degli autovalori},  \eqref{up to subsequence}, 
\eqref{rapporto tende bar d}, we get
\be\label{convergenza Meln1}
\frac{ \Omega_{j_{m_n}}( {\mathtt h}_{m_n})}{\langle \ell_{m_n} \rangle} = 
\frac{ j_{m_n}^{\frac12}}{ \langle \ell_{m_n} \rangle } + 
 \frac{ r(j_{m_n} \,, {\mathtt h}_{m_n}) }{ \langle \ell_{m_n} \rangle }  
\to \bar d \,, 
\quad 
\partial_{\mathtt h}^k \frac{ \Omega_{j_{m_n}} ({\mathtt h}_{m_n})}{\langle \ell_{m_n} \rangle}  =
\partial_{\mathtt h}^k \frac{ r( j_{m_n} \,, {\mathtt h}_{m_n}) }{ \langle \ell_{m_n} \rangle }  
\to 0 
\quad \forall k \geq 1
\ee
as $ m_n \to + \infty $. 
Passing to the limit in \eqref{rick 2}, by \eqref{convergenza Meln1}, \eqref{up to subsequence}
we deduce that 
$ \partial_{\mathtt h}^k \big\{ {\vec \om}(\bar {\mathtt h}) \cdot \bar c + \bar d  \big\} = 0 $, for all $k \in \N $. 
Therefore the analytic function 
$ {\mathtt h} \mapsto {\vec \om}( {\mathtt h}) \cdot \bar c + \bar d = 0 $ is identically zero. 
Since $ (\bar c, \bar d ) \neq 0 $
this is in contradiction with  Lemma \ref{non degenerazione frequenze imperturbate}.

\medskip

{\sc Proof of \eqref{2 Melnikov-}}. 
For all $ {\mathtt h} \in [{\mathtt h}_1, {\mathtt h}_2] $, by \eqref{caso-ell=0} and \eqref{LIN:fre},
we have
\[ 
| {\vec \om} ({\mathtt h} ) \cdot \ell  + \Omega_j ({\mathtt h} ) -  \Omega_{j'} ({\mathtt h} ) | 
\geq |  \Omega_j ({\mathtt h} ) -  \Omega_{j'} ({\mathtt h} ) | 
- | {\vec \om} ({\mathtt h} ) | |\ell | 
\geq c_1 |j^{\frac12} - j'^{\frac12}|  - C |\ell| \geq \langle \ell \rangle 
\]
provided $ | j^{\frac12} - j'^{\frac12}|  \geq C_1  \langle \ell  \rangle $, for some $ C_1 > 0 $.
Therefore in  \eqref{2 Melnikov-} we can restrict to the indices such that 
\begin{equation}\label{rick 4}
|j^{\frac12} - j'^{\frac12}| < C_1 \langle \ell \rangle\,.
\end{equation}
Moreover in \eqref{2 Melnikov-} we can also assume that $ j \neq j' $, 
otherwise \eqref{2 Melnikov-} reduces to \eqref{0 Melnikov}, which is already proved. 
If, by contradiction, \eqref{2 Melnikov-} is false, we deduce, arguing as in the previous cases, 
that, for all $m \in \N$, there exist $ \ell_m \in \Z^\nu\setminus \{0\} $, 
$j_m, j'_m \in \N^+ \setminus \Splus $, $ j_m \neq j'_m $, 
$ {\mathtt h}_m \in [ {\mathtt h}_1, {\mathtt h}_2]$, such that 
\begin{equation}\label{rick 5}
\forall k \in \N \, , \quad \forall m \geq k \, , \quad 
\Big| \partial_{\mathtt h}^k \Big\{ {\vec \om}({\mathtt h}_m) \cdot 
\frac{\ell_m}{\langle \ell_m \rangle} + \frac{ \Omega_{j_m}({\mathtt h}_m)}{\langle \ell_m \rangle} - 
\frac{  \Omega_{j'_m}({\mathtt h}_m)}{\langle \ell_m \rangle} \Big\} \Big| 
<  \frac{1}{1 + m}\,.
\end{equation}
As in the previous cases, since the sequences 
$ ({\mathtt h}_m)_{m \in \N}$, $(\ell_{m} / \langle \ell_{m}\rangle)_{m \in \N} $ 
are bounded, there exists $ m_n \to + \infty $ such that 
\be\label{conver2}
{\mathtt h}_{m_n} \to \bar {\mathtt h} \in [ {\mathtt h}_1, {\mathtt h}_2]\,,\quad \ell_{m_n} / \langle \ell_{m_n} \rangle \to \bar c 
\in \R^\nu \setminus \{0\} \,.
\ee
We distinguish again two cases.

\emph{Case 1 : $ (\ell_{m_n} )$ is unbounded}.
Using \eqref{rick 4} we deduce that, up to a subsequence,  
\be\label{convjj'}
|j_m^{\frac12} - j_m'^{\frac12}| \langle \ell_m \rangle^{-1} \to \bar d \in \R \, . 
\ee
Hence passing to the limit in \eqref{rick 5} for $ m_n \to + \infty $, 
we deduce by \eqref{conver2}, \eqref{convjj'},
\eqref{espansione asintotica degli autovalori} that 
$$
\partial_{\mathtt h}^k \{ {\vec \om}(\bar {\mathtt h}) \cdot \bar c + \bar d \} = 0 
\quad \forall k \in \N. 
$$
Therefore the analytic function $ {\mathtt h}  \mapsto  {\vec \om}( {\mathtt h} ) \cdot \bar c + \bar d $ is identically zero.  
This in contradiction with  Lemma \ref{non degenerazione frequenze imperturbate}.

\emph{Case 2 : $ (\ell_{m_n}) $ is bounded}. 
By \eqref{rick 4},  we have that
$ | \sqrt{j_m} - \sqrt{j_m'} | \leq C $ and so, up to a subsequence, 
only the following two subcases are possible: 
\begin{enumerate}
\item[$(i)$] $ j_m, j_m' \leq C $. 
Up to a subsequence, 
$ j_{m_n} \to \bar \jmath $, $ j'_{m_n} \to \bar \jmath' $,  
$ \ell_{m_n} \to \bar \ell \neq 0 $ and $\h_{m_n} \to \bar \h$.
Hence passing to the limit in \eqref{rick 5} we deduce that 
$$ 
\partial_{\mathtt h}^k  
\Big\{ {\vec \om}( \bar {\mathtt h} ) \cdot \bar c + 
\frac{ \Omega_{\bar \jmath}(\bar {\mathtt h}) - 
\Omega_{\bar \jmath'}( \bar {\mathtt h} )}{\langle \bar \ell \rangle} \Big\} = 0 
\quad  \forall k  \in \N  \, . 
$$ 
Hence the analytic function 
$ {\mathtt h}  \mapsto {\vec \om}( \bar {\mathtt h}) \cdot \bar c 
+ ({  \Omega_{\bar \jmath}(\bar {\mathtt h}) -  \Omega_{\bar \jmath'}( \bar {\mathtt h} )}) {\langle \bar \ell \rangle}^{-1} $
is identically zero, which is a contradiction  
with Lemma \ref{non degenerazione frequenze imperturbate}.

\item[$(ii)$] $ j_m, j_m' \to + \infty $. 
By \eqref{convjj'} and \eqref{espansione asintotica degli autovalori}, we deduce, 
passing to the limit in \eqref{rick 5}, that 
$$
\partial_{\mathtt h}^k  
\big\{ {\vec \om}( {\mathtt h} ) \cdot \bar c + \bar d \big\} = 0 
\quad  \forall k  \in \N  \, . 
$$ 
Hence the analytic function 
$ {\mathtt h}  \mapsto {\vec \om}( {\mathtt h} ) \cdot \bar c 
+ \bar d $
is identically zero, which 
contradicts  Lemma \ref{non degenerazione frequenze imperturbate}.
\end{enumerate}

{\sc Proof of \eqref{2 Melnikov+}.} 
The proof is similar to \eqref{1 Melnikov}. 
First of all note that for all ${\mathtt h}  \in [{\mathtt h}_1, {\mathtt h}_2] $, we have 
$$ 
| {\vec \om} ({\mathtt h} ) \cdot \ell  + \Omega_j ( {\mathtt h} ) +  \Omega_{j'} ({\mathtt h} )  | \geq  
\Omega_j ({\mathtt h} ) + \Omega_{j'} ( {\mathtt h} )
 - |{\vec \om} ({\mathtt h}  )  \cdot \ell  | \geq  c_1 \sqrt{j}+  c_1\sqrt{j'} - C | \ell  | \geq  |\ell  | 
 $$ 
if $ \sqrt{j}+  \sqrt{j'} \geq C_0 |\ell  | $ for some $ C_0 > 0 $. 
Therefore in \eqref{1 Melnikov} we can restrict the analysis to the indices 
$ (\ell  , j, j' ) \in \Z^\nu \times (\N^+ \setminus \Splus)^2 $ satisfying  
\be\label{prima restrizione2+}
\sqrt{j}+  \sqrt{j'}  < C_0 | \ell  | \, . 
\ee
Arguing by contradiction as above, we suppose that 
for all $m \in \N$ there exist  $ \ell_m \in \Z^\nu$, $ j_m \in \N^+ \setminus \Splus  $ and 
$ {\mathtt h}_m \in [ {\mathtt h}_1, {\mathtt h}_2]$ such that 
\begin{equation}\label{rick 22+}
\forall k \in \N, \quad \forall m \geq k \, , \quad 
\Big| 
\partial_{\mathtt h}^k \Big\{ {\vec \om}({\mathtt h}_m) \cdot \frac{\ell_m}{\langle \ell_m \rangle} + 
\frac{ \Omega_{j_m}({\mathtt h}_m)}{\langle \ell_m \rangle} + 
\frac{ \Omega_{j_m'}({\mathtt h}_m)}{\langle \ell_m \rangle} \Big\} \Big| 
< \frac{1}{1 + m} \, . 
\end{equation}
Since the sequences $({\mathtt h}_m)_{m \in \N} \subset [{\mathtt h}_1, {\mathtt h}_2] $ 
and $(\ell_{m} / \langle \ell_m \rangle)_{m \in \N} \in \R^\nu $ are bounded, there exist $m_n \to + \infty$ such that 
\be\label{up to subsequence2+}
{\mathtt h}_{m_n} \to \bar {\mathtt h} \in [ {\mathtt h}_1, {\mathtt h}_2]\,,\quad \frac{\ell_{m_n}}{\langle \ell_{m_n} \rangle} \to \bar c \in \R^\nu \,.
\ee
We now distinguish two cases.

\emph{Case 1: $ (\ell_{m_n}) \subset \Z^\nu $ is bounded}. 
Up to a subsequence,  $\ell_{m_n} \to \bar \ell \in \Z^\nu $, 
and since, by \eqref{prima restrizione2+}, also $ j_m, j_m' \leq C $ for all $m$, 
we have $j_{m_n} \to \bar \jmath $, $j_{m_n}' \to \bar \jmath' $. 
Passing to the limit for $m_n \to + \infty$ in \eqref{rick 22+} we deduce, by \eqref{up to subsequence2+}, that 
$$
\partial_{\mathtt h}^k \big\{ {\vec \om}(\bar {\mathtt h}) \cdot \bar c + 
\Omega_{\bar \jmath} (\bar {\mathtt h})  \langle \bar \ell \rangle^{-1}+
 \Omega_{\bar \jmath'} (\bar {\mathtt h})  \langle \bar \ell \rangle^{-1}  \big\} = 0 
 \quad \forall k \in \N \, .
$$
Therefore the analytic function
$ {\mathtt h} \mapsto {\vec \om}(  {\mathtt h} ) \cdot \bar c + {\langle \bar \ell \rangle}^{-1}  \Omega_{\bar \jmath}( \h) +
\langle \bar \ell \rangle^{-1}  \Omega_{\bar \jmath'} ({\mathtt h} ) $
is identically zero. 
This is in contradiction with Lemma \ref{non degenerazione frequenze imperturbate}.

\emph{Case 2: $ (\ell_{m_n})$ is unbounded}. 
Up to a subsequence, $ | \ell_{m_n}| \to + \infty $.
In this case the constant $ \bar c$ in \eqref{up to subsequence2+} is nonzero. 
Moreover, by \eqref{prima restrizione2+}, we also have that, up to a subsequence, 
\be\label{rapporto tende}
( j_{m_n}^{\frac12} + j_{m_n}'^{\frac12} ) \langle \ell_{m_n} \rangle^{-1} \to \bar d \in \R \, .
\ee
By \eqref{espansione asintotica degli autovalori},  
\eqref{up to subsequence2+}, \eqref{rapporto tende}, 
passing to the limit as $ m_n \to + \infty $ in \eqref{rick 22+}
we deduce that 
$ \partial_{\mathtt h}^k \big\{ {\vec \om}(\bar {\mathtt h}) \cdot \bar c + \bar d  \big\} = 0 $
for all $k \in \N $. 
Therefore the analytic function 
$ {\mathtt h} \mapsto {\vec \om}( {\mathtt h} ) \cdot \bar c + \bar d = 0 $ is identically zero. 
Since $ (\bar c, \bar d ) \neq 0 $,
this is in contradiction with Lemma \ref{non degenerazione frequenze imperturbate}.
\end{proof}

\section{Nash-Moser theorem and measure estimates}\label{sec:functional}

Rescaling $ u \mapsto \e u $, we 
% We rescale the variable $ u = \e \tilde u $ with $ \tilde u = O (1) $, 
write \eqref{WW0} % (after dropping the tilde) 
as the Hamiltonian system 
%\be\label{WW-riscalato}
%\pa_t u = J \Omega u + \e X_{P_\e} (u)  
%\ee
%where $ J \Omega $ is the linearized Hamiltonian vector field in \eqref{definizione Omega} and
%$$
%X_{P_\e} (u, {\mathtt h}) := X_{P_\e} (u) :=  \begin{pmatrix} 
% \e^{-1} (G(\e \eta, {\mathtt h})- G(0, {\mathtt h}))\psi \\
%- \frac{1}{2} \psi_x^2  
%+\frac{1}{2} \frac{ \big(G(\e \eta, {\mathtt h})\psi + \e \eta_x \psi_x \bigr)^2}{1+ (\e \eta_x)^2} \\
%\end{pmatrix} \, .  
%$$
%System \eqref{WW-riscalato} is the Hamiltonian system 
generated by the Hamiltonian 
$$
{\cal H}_\e (u) := \e^{-2} H(\e u ) = H_L (u) + \e P_\e (u) 
$$
where $ H $ is the water waves Hamiltonian \eqref{Hamiltonian} (with $ g = 1 $ and depth $ \mathtt h $),  
$ H_L $ is defined in \eqref{Hamiltonian linear} and
\begin{equation}\label{definizione P epsilon}
P_\e (u, {\mathtt h}) := P_\e (u) := 
\frac{1}{2 \e} \int_\T  \psi\,\big( G(\e \eta, {\mathtt h}) - G(0, {\mathtt h}) \big) \psi \,dx  \, .
\end{equation}
We decompose the phase space 
\begin{equation}\label{phase space 2017}
H_{0, {\rm even}}^1 := \Big\{ u := (\eta, \psi) \in H_0^1(\T_x) \times  \dot{H}^1(\T_x)\,, \quad u(x) = u(- x)\Big\}
= H_{\Splus} \oplus  H_{\Splus}^\bot 
\end{equation} 
as the direct sum of the symplectic subspaces
%\begin{equation}\label{H S + H S + bot}
%H_{0, {\rm even}}^1 = H_{\Splus} \oplus  H_{\Splus}^\bot 
%\end{equation}
% where 
$H_{\Splus}$ and $H_{\Splus}^\bot$  defined in \eqref{splitting S-S-bot}, 
%$$
%H_{\Splus} := \Big\{ v := 
%\sum_{j \in \Splus} \begin{pmatrix} 
% \eta_j  \\
%\psi_j \\
%\end{pmatrix} \cos (jx)  \Big\} \, , \quad 
%H_{\Splus}^\bot := \Big\{ z := \begin{pmatrix} 
% \eta \\
%\psi  \\
%\end{pmatrix} = \sum_{j \in \N^+ \setminus \Splus} \begin{pmatrix} 
% \eta_j  \\
%\psi_j \\
%\end{pmatrix} \cos (jx)  \Big\} \, .
%$$
 we  introduce action-angle variables on the tangential sites as in \eqref{AA}, 
%by setting
%\begin{equation}\label{action angle variables}
%f\eta_j := \sqrt{\frac{2}{\pi}} \, \om_j^{1/2} \sqrt{\xi_j + I_j}  \cos (\theta_j ) , \  \
%\psi_j := \sqrt{\frac{2}{\pi}} \, \om_j^{- 1/2} \sqrt{\xi_j + I_j} \, \sin (\theta_j ) \, ,  \ \  j \in \Splus \, ,
%\end{equation}
%where $ \xi_j > 0 $, $ j = 1, \ldots, \nu $, the variables $I_j$ satisfy $ | I_j | < \xi_j $, 
and we  leave unchanged  the  normal component $ z $. 
The symplectic $ 2 $-form in \eqref{2form tutto} reads % (for simplicity of notation we denote it in the same way)
\begin{equation}\label{2form}
{\cal W} := \Big( {\mathop \sum}_{j \in \Splus} d \theta_j \wedge d I_j \Big)  \oplus 
{\cal W}_{|H_{\Splus}^\bot} = d \Lambda ,
\end{equation}
where $ \Lambda $ is the Liouville  $1$-form
\begin{equation}\label{Lambda 1 form}
\Lambda_{(\theta, I, z)}[\widehat \theta, \widehat I, \widehat z] := - 
\sum_{j \in \Splus} I_j  \widehat \theta_j - \frac12 \big( J z\,,\,\widehat z \big)_{L^2}  \, . 
\end{equation}
 Hence the Hamiltonian system generated by $ {\cal H}_\e $
  % \eqref{WW-riscalato} 
transforms into the one 
%th  Hamiltonian system 
%$$
%\dot \theta = \partial_I H_\e ( \theta, I, z ) \, , \ 
%\dot I = - \partial_\theta H_\e ( \theta, I, z ) \, , \quad 
%z_t = J \nabla_z H_\e ( \theta, I, z ) 
%$$
generated by the Hamiltonian
\be\label{new Hamiltonian} 
H_\e :=  {\cal H}_\e \circ A = \e^{-2} H \circ \e A 
\ee
where
\begin{equation}\label{definizione A}
A (\theta, I, z) := v (\theta, I) +  z := 
\sum_{j \in \Splus} 
\sqrt{\frac{2}{\pi}}
\begin{pmatrix} 
\omega_j^{1/2} \sqrt{\xi_j + I_j} \, \cos (\theta_j ) \\
- \omega_j^{- 1/2} \sqrt{\xi_j + I_j} \, \sin (\theta_j )  \\
\end{pmatrix} \cos (jx) + z \, . 
\end{equation}
We denote by $X_{H_\e} := (\partial_I H_\e,  - \partial_\theta H_\e, J \nabla_z H_\e)$  % ( \theta, I, z ) %
the Hamiltonian vector field
in the variables $(\theta, I, z ) \in \T^\nu \times \R^\nu \times H_{\Splus}^\bot $. 
The involution $ \rho $ in \eqref{defS} becomes 
\be\label{involuzione tilde rho}
\tilde \rho : (\theta, I, z ) \mapsto (- \theta, I, \rho z ) \, .
\ee
By \eqref{Hamiltonian} and \eqref{new Hamiltonian}  the Hamiltonian $ H_\e  $ reads (up to a constant) 
\begin{equation}\label{definizione cal N P}
H_\e = {\cal N} + \e P \, , \quad 
{\cal N} := H_L \circ A  =    {\vec \om} (\h) \cdot  I  + 
\frac12 (z, \Om z)_{L^2} \, , \quad P :=  P_\e \circ A\,, 
\end{equation}
where $ {\vec \om} (\h) $ is defined in \eqref{tangential-normal-frequencies}
 and  $ \Om $  in \eqref{definizione Omega}. We look for an embedded invariant torus 
$$
i : \T^\nu \to \T^\nu \times \R^\nu \times H_{\Splus}^\bot, \quad  
\vphi \mapsto i (\vphi) := ( \theta (\vphi), I (\vphi), z (\vphi)) 
$$
of the Hamiltonian vector field $ X_{H_\e}  $ 
filled by quasi-periodic solutions with Diophantine frequency $ \om \in \R^\nu $
(and which satisfies also  first and second order Melnikov non-resonance conditions
as in \eqref{Cantor set infinito riccardo}). 

\subsection{Nash-Moser theorem of hypothetical conjugation}

For $ \a \in \R^\nu $, we consider the modified Hamiltonian
\begin{equation}\label{H alpha}
H_\a := {\cal N}_\a + \e P \, , \quad {\cal N}_\a :=  \a \cdot I + \frac12 (z, \Om z)_{L^2} \, .
\end{equation}
We look for zeros of the nonlinear operator
\begin{align}
 {\cal F} (i, \a ) 
& :=   {\cal F} (i, \a, \om, {\mathtt h}, \e )  := \Dom i (\vphi) - X_{H_\a} ( i (\vphi))
=  \Dom i (\vphi) -  (X_{{\cal N}_\a}  +  \e X_{P})  (i(\vphi) )  \label{operatorF}  \\
& \nonumber  :=  \left(
\begin{array}{c}
\Dom \theta (\vphi) -  \a - \e \partial_I P ( i(\vphi)   )   \\
\Dom I (\vphi)  +  \e \partial_\teta P( i(\vphi)  )  \\
\Dom z (\vphi) -  J (\Om z(\vphi) + \e \nabla_z P ( i(\vphi) ))  
\end{array}
\right) 
\end{align}
where $ \Theta(\ph) := \teta (\vphi) - \vphi $ is $ (2 \pi)^\nu $-periodic. Thus 
$ \vphi \mapsto i (\vphi) $ is an embedded torus, invariant for the Hamiltonian vector field $ X_{H_\a } $  
and  filled by quasi-periodic solutions with frequency $ \om $.

Each Hamiltonian $ H_\a $ in \eqref{H alpha} is reversible, i.e. $  H_\a \circ \tilde \rho = H_\a $
where the involution $ \tilde \rho $ is defined in \eqref{involuzione tilde rho}. 
We look for reversible solutions of $ {\cal F}(i, \a) = 0 $,  namely satisfying $ {\tilde \rho} i (\vphi ) = i (- \vphi) $ 
(see \eqref{involuzione tilde rho}), i.e.  
\begin{equation}\label{parity solution}
\theta(-\vphi) = - \theta (\vphi) \, , \quad 
I(-\vphi) = I(\vphi) \, , \quad 
z (- \vphi ) = ( \rho z)(\vphi) \, . 
\end{equation}
The norm of the periodic component of the embedded torus 
\begin{equation}\label{componente periodica}
{\mathfrak I}(\vphi)  := i (\vphi) - (\vphi,0,0) := ( {\Theta} (\ph), I(\ph), z(\ph))\,, \quad \Theta(\ph) := \teta (\vphi) - \vphi \, , 
\end{equation}
is 
\be\label{def:norma-cp}
\|  {\mathfrak I}  \|_s^{k_0,\g} 
:= \| \Theta \|_{H^s_\vphi}^{k_0,\g} +  \| I  \|_{H^s_\vphi}^{k_0,\g} 
+  \| z \|_s^{k_0,\g}\,, 
\ee
where $ \| z \|_s^{k_0,\g} =  \| \eta \|_s^{k_0,\g} + \| \psi \|_s^{k_0,\g} $.
We define 
\begin{equation} \label{def k0}
k_0 := \barka + 2,	
\end{equation}
where $\barka$ is the index of non-degeneracy provided by Proposition \ref{Lemma: degenerate KAM}, 
which only depends on the linear unperturbed frequencies. 
Thus $k_0$ is considered as an absolute constant,
and we will often omit to explicitly write the dependence of the various constants 
with respect to $ k_0 $.
We look for quasi-periodic solutions with frequency $ \om $  
belonging to a $ \d $-neighborhood (independent of $ \e $)
\be\label{unperturbed-frequencies} 
\tOm := \Big\{  \om \in \R^\nu \, : \, {\rm dist} \big(\om,  {\vec \omega}[{\mathtt h}_1, {\mathtt h}_2]\big) < \d \Big\}, \quad \d > 0
\ee
of the unperturbed linear frequencies $ {\vec \omega}[{\mathtt h}_1, {\mathtt h}_2] $ defined  in \eqref{tangential-normal-frequencies}.
 
\begin{theorem} \label{main theorem}
{\bf (Nash-Moser theorem)}
Fix finitely many tangential sites $ \Splus \subset \N^+ $ and let $ \nu := |\Splus | $. 
Let $\t \geq 1$. 
There exist positive constants $ a_0, \e_0, \kappa_1, C $ 
depending on $\Splus, k_0, \t$ such that, 
for all $ \g = \e^a $, $ 0 < a < a_0 $, for all $ \e \in (0, \e_0) $,
there exist a $ k_0 $ times differentiable function
\be\label{mappa aep}
\a_\infty : \R^\nu \times [{\mathtt h}_1, {\mathtt h}_2] \mapsto \R^\nu \, , 
\quad \a_\infty (\om, {\mathtt h} ) = \om + r_\e (\om, {\mathtt h}) \, , 
\quad {with} \quad 
|r_\e|^{k_0, \gamma}  \leq C \e \g^{-1} \, , 
\ee
a family of embedded tori  $ i_\infty  $ defined for all 
$ (\om, {\mathtt h})  \in \R^\nu \times   [{\mathtt h}_1, {\mathtt h}_2 ] $ 
satisfying % the reversibility property 
\eqref{parity solution} and 
\be\label{stima toro finale}
\|  i_\infty (\vphi) -  (\vphi,0,0) \|_{s_0}^{k_0, \g} \leq C \e \g^{-1}  \, , 
\ee
a sequence of $ {k_0} $ times differentiable functions 
$ \mu_j^\infty : \R^\nu \times [{\mathtt h}_1, {\mathtt h}_2 ]  \to \R $, 
$  j \in \N^+ \setminus \Splus $, of the form
\begin{equation}\label{autovalori infiniti}
\mu_j^\infty (\omega, \h ) = 
\mathtt m_{\frac12}^\infty (\om, \h ) ( j \tanh ( \h j ) )^{\frac12}   
+ \rin_j^\infty (\omega, \h )
\end{equation} 
satisfying 
\begin{equation}\label{stime autovalori infiniti}
| \mathtt m_{\frac12}^\infty  - 1|^{k_0, \gamma} \leq C \e \g^{-1} \,, \qquad 
 \sup_{j \in \N^+ \setminus \Splus} j^{\frac12} |\rin_j^\infty |^{k_0, \gamma} 
 \leq C \e \gamma^{-\kappa_1}
\end{equation}
such that  for all $ (\om, {\mathtt h}) $ in the  Cantor like set
\begin{align}\label{Cantor set infinito riccardo}
{\cal C}_\infty^{\gamma} & := \Big\{ ( \omega, {\mathtt h} ) \in \tOm \times [ {\mathtt h}_1, {\mathtt h}_2 ] \, : 
 \, |\om \cdot \ell  | \geq 8 \g \langle \ell \rangle^{-\tau}, \, \forall \ell \in \Z^{\nu} \setminus \{ 0 \} \, ,  
   \\
&  \qquad  |\omega \cdot \ell  + \mu_j^\infty (\om, {\mathtt h})  | \geq 4 \gamma j^{\frac12} \langle \ell  \rangle^{- \tau}, \, 
 \forall \ell   \in \Z^\nu, \, j \in \N^+ \setminus \Splus, \ % 0 < \gamma \leq \gamma_0 /2 \, , \ \tau \geq \tau_0 \, ,  
  \nonumber \\
  & \qquad |\omega \cdot \ell  + 
 \mu_j^\infty (\om, {\mathtt h}) + \mu_{j'}^\infty (\om, {\mathtt h} ) | \geq 
 4 \gamma (j^{\frac12} + j'^{\frac12} ) \langle \ell  \rangle^{-\tau}, \,\,
 \forall \ell   \in \Z^\nu ,\,\,j, j' \in \N^+ \setminus \Splus \,, \nonumber  \\ 
& \qquad |\omega \cdot \ell  + 
 \mu_j^\infty (\om, {\mathtt h}) -  \mu_{j'}^\infty (\om, {\mathtt h}) | \geq 
 4 \gamma j^{-\perd} j'^{-\perd} \langle \ell  \rangle^{-\tau}, \,\,
 \forall \ell   \in \Z^\nu ,\,\,j, j' \in \N^+ \setminus \Splus , \, (\ell,j,j') \neq (0,j,j) 
 \Big\} \nonumber
\end{align}
the function $ i_\infty (\vphi) := i_\infty (\omega, {\mathtt h}, \e)(\vphi) $ is a solution of 
$ {\cal F}( i_\infty, \a_\infty (\om, {\mathtt h}) , \om, {\mathtt h}, \e)  = 0 $. As a consequence the embedded torus 
$ \vphi \mapsto i_\infty (\vphi) $ is invariant for the Hamiltonian vector field $ X_{H_{\a_\infty (\om, {\mathtt h})}  } $  
and it is filled by quasi-periodic solutions with frequency $ \om $.
\end{theorem}

Theorem \ref{main theorem} is proved in Section \ref{sec:NM}.
The very weak second Melnikov non-resonance conditions in \eqref{Cantor set infinito riccardo}
can be verified for most parameters if $ \perd $ is large enough, 
i.e.\ $ \perd  > \frac34 \, \barka $, see Theorem \ref{Teorema stima in misura} below.

\subsection{Measure estimates}\label{sec:measure}

The aim is now to deduce Theorem \ref{thm:main0} from Theorem \ref{main theorem}. 

By  \eqref{mappa aep}  the function $ \a_\infty ( \cdot, {\mathtt h}) $ from $ \tOm $ into the image $\a_\infty (\tOm, {\mathtt h}) $ is invertible:
\be\label{a-b}
\beta = \a_\infty (\om, {\mathtt h}) = \om + r_\e (\om, {\mathtt h})  
\quad \Longleftrightarrow \quad 
\om  = \a_\infty^{-1}(\b , {\mathtt h}) 
= \beta + \breve r_\e (\beta, {\mathtt h}) \quad \text{with}  \quad
| \breve r_\e |^{k_0, \gamma} \leq C \e \g^{-1} \, . 
\ee
We underline that the function $\alpha_\infty^{- 1}(\cdot, {\mathtt h})$ is the inverse of $\alpha_\infty(\cdot, {\mathtt h})$, at any fixed value of $ {\mathtt h} $ in $[ {\mathtt h}_1, {\mathtt h}_2]$.
Then, for any $ \b \in \a_\infty ({\cal C}_\infty^\gamma ) $, 
Theorem \ref{main theorem} proves the existence of an embedded invariant torus  
filled by quasi-periodic solutions with Diophantine frequency 
$ \om =  \a_\infty^{-1}(\b, {\mathtt h} )  $ for the  Hamiltonian 
$$
H_\b =  \beta \cdot I + \frac12 ( z, \Om z)_{L^2} + \e P \, . 
$$
Consider the curve of the unperturbed tangential  frequencies  
$ [{\mathtt h}_1, {\mathtt h}_2] \ni {\mathtt h} \mapsto \vec \om ({\mathtt h} ) := 
( \sqrt{j \tanh ({\mathtt h} j) } )_{j \in \Splus} $ in \eqref{tang-vec}. 
In Theorem \ref{Teorema stima in misura} below we 
prove that for ``most" values of $ {\mathtt h} \in [{\mathtt h}_1, {\mathtt h}_2] $ 
the vector $(\alpha_\infty^{- 1}(\vec \om ({\mathtt h} ), {\mathtt h}), {\mathtt h} )$ is in ${\cal C}^\gamma_\infty$. 
Hence, for such 
values of $ {\mathtt h} $ we have found 
 an embedded invariant torus for the Hamiltonian 
$ H_\e $ in \eqref{definizione cal N P}, 
filled by quasi-periodic solutions with Diophantine frequency $ \om =  \a_\infty^{-1}( \vec \om ({\mathtt h} ), {\mathtt h} ) $.

This implies Theorem \ref{thm:main0} together with the  following measure estimate.

\begin{theorem}\label{Teorema stima in misura} {\bf (Measure estimates)} 
Let
\begin{equation}\label{relazione tau k0}
\gamma = \e^a \,, \quad \ 0 < a < \min\{a_0, 1 / (k_0 + \kappa_1) \} \, , \quad 
\tau > \barka ( \nu + 4) \, , \quad
\perd > \frac{3 \barka}{4} \,,
\end{equation} 
where $\barka$ is the index of non-degeneracy given by Proposition \ref{Lemma: degenerate KAM}
and $k_0 = \barka + 2$.
Then the set 
\be\label{defG-ep}
{\cal G}_\e 
:= \big\{ {\mathtt h} \in [ {\mathtt h}_1, {\mathtt h}_2] :  
\big( \a_\infty^{-1}({\vec \om}({\mathtt h}) , {\mathtt h}), {\mathtt h} \big)  \in  {\cal C}^\gamma_\infty \big\} 
\ee
has a measure satisfying  $ |{\cal G}_\e| \to {\mathtt h}_2 - {\mathtt h}_1 $ as $ \e \to 0 $.
\end{theorem}

The rest of this section is devoted to the proof of Theorem \ref{Teorema stima in misura}. 
By \eqref{a-b} the vector
\begin{equation}\label{omega epsilon kappa}
\om_\e ({\mathtt h}) :=    \a_\infty^{-1}( {\vec \om} ({\mathtt h} ), {\mathtt h}  ) 
= {\vec \om} ({\mathtt h}) + {\mathtt r}_\e ( {\mathtt h} ) \, , \quad   
{\mathtt r}_\e ( {\mathtt h} ) := {\breve r}_\e ({\vec \om} ({\mathtt h}), {\mathtt h} ) \, , 
\ee
satisfies
\be\label{omega epsilon kappa-1}
| \pa_{\mathtt h}^k {\mathtt r}_\e  ({\mathtt h}) | \leq C \e \g^{-k-1} 
\quad \forall 0 \leq k \leq k_0 \, .
\end{equation}
We also denote, with a small abuse of notation, 
for all $ j \in \N^+ \setminus \Splus $, 
\begin{equation}\label{mu j infty kappa}
\mu_j^\infty ( \h ) 
:= \mu_j^\infty ( \om_\e(\h), \h) 
:= \mathtt m_{\frac12}^\infty (\h) (j \tanh (\h j))^{\frac12} 
+ \rin_j^\infty (\h), 
\end{equation}
where 
\begin{equation}\label{autovalori in kappa}
\mathtt m_{\frac12}^\infty (\h) := \mathtt m_{\frac12}^\infty (\omega_\e(\h), \h)\,, \qquad 
\rin_j^\infty(\h) := \rin_j^\infty(\omega_\e(\h), \h).
\end{equation}
By \eqref{stime autovalori infiniti}, \eqref{autovalori in kappa} and 
\eqref{omega epsilon kappa}-\eqref{omega epsilon kappa-1}, 
using that $ \e \gamma^{- k_0 - 1} \leq 1 $ 
(which by \eqref{relazione tau k0} is satisfied for $ \e $ small), we get
\begin{equation}\label{stime coefficienti autovalori in kappa}
| \partial_{\mathtt h}^{k} ( \mathtt m_{\frac12}^\infty({\mathtt h}) - 1) |
\leq C \e \g^{-1-k} \,, \quad 
\sup_{j \in \N^+ \setminus \Splus} 
j^{\frac12} | \pa_{\mathtt h}^k  \rin_j^\infty ({\mathtt h})| \leq C \e \gamma^{- \kappa_1 - k} 
\qquad \forall 0 \leq k \leq k_0\,.
\end{equation}
By \eqref{Cantor set infinito riccardo}, \eqref{omega epsilon kappa}, 
\eqref{mu j infty kappa}, the Cantor set $ {\cal G}_\e $  in \eqref{defG-ep} becomes 
\begin{align}
{\cal G}_\e  & = \Big\{ {\mathtt h} \in [{\mathtt h}_1, {\mathtt h}_2]  \!: \! 
 \, |\om_\e({\mathtt h}) \cdot \ell  | \geq 8 \g \langle \ell \rangle^{-\tau}, \, \forall \ell \in \Z^{\nu} \setminus \{ 0 \},  
\notag \\
&  \qquad  | \om_\e({\mathtt h}) \cdot \ell  + \mu_j^\infty ({\mathtt h})  | \geq 4 \gamma j^{\frac12} \langle \ell  \rangle^{- \tau}, \, 
 \forall \ell   \in \Z^\nu, \, j \in \N^+ \setminus \Splus,  
\nonumber \\ 
& \qquad |\om_\e({\mathtt h}) \cdot \ell  + 
 \mu_j^\infty ({\mathtt h}) + \mu_{j'}^\infty ({\mathtt h}) | \geq 
 4 \gamma (j^{\frac12} + j'^{\frac12} ) \langle \ell  \rangle^{-\tau}, \,\,
 \forall \ell   \in \Z^\nu ,\,\,j, j' \in \N^+ \setminus \Splus \,, 
\nonumber \\ 
& \qquad | \om_\e({\mathtt h}) \cdot \ell  + 
 \mu_j^\infty ({\mathtt h}) -  \mu_{j'}^\infty ({\mathtt h}) | \geq 
  \frac{4 \gamma \langle \ell  \rangle^{-\tau} }{ j^{\perd} j'^{\perd}}  , 
 \forall \ell   \in \Z^\nu, j, j' \in \N^+ \setminus \Splus, (\ell, j, j') \neq (0, j, j)
 \Big\} \, . 
\label{Cantor set infinito riccardo1}
\end{align}
We estimate the measure of the complementary set  
\begin{equation}\label{complementare insieme di cantor}
{\cal G}_\e^c  := [{\mathtt h}_1, {\mathtt h}_2] \setminus {\cal G}_\e 
:= \Big( \bigcup_{\ell \neq 0} R_{\ell}^{(0)} \Big) 
\cup \Big( \bigcup_{\ell , j} R_{\ell, j}^{(I)} \Big)
\cup \Big(\bigcup_{\ell, j, j'} Q_{\ell j j'}^{(II)} \Big)
\cup \Big(\bigcup_{(\ell, j, j') \neq (0,j,j)} R_{\ell j j'}^{(II)} \Big) 
\end{equation}
where the ``resonant sets" are 
\begin{align}
& R_\ell^{(0)} := \big\{ {\mathtt h} \in [{\mathtt h}_1, {\mathtt h}_2] : 
 |\om_\e({\mathtt h}) \cdot \ell  | < 8 \g \langle \ell \rangle^{-\tau} \big\} 
\label{reso1} \\
& R_{\ell j}^{(I)} := \big\{ {\mathtt h} \in [{\mathtt h}_1, {\mathtt h}_2] : 
 | \om_\e({\mathtt h}) \cdot \ell  + \mu_j^\infty ({\mathtt h})  | < 4 \gamma j^{\frac12} \langle \ell  \rangle^{- \tau} \big\} \label{reso2} \\
 & Q_{\ell j j'}^{(II)}:= \big\{ {\mathtt h} \in [{\mathtt h}_1, {\mathtt h}_2]: 
|\om_\e({\mathtt h}) \cdot \ell  + 
 \mu_j^\infty ({\mathtt h}) + \mu_{j'}^\infty ({\mathtt h}) | < 
 4 \gamma (j^{\frac12} + j'^{\frac12} ) \langle \ell  \rangle^{-\tau} \big\}  \label{reso3} \\
 & R_{\ell j j'}^{(II)}  := \Big\{ {\mathtt h} \in [{\mathtt h}_1, {\mathtt h}_2]: 
| \om_\e({\mathtt h}) \cdot \ell  + 
 \mu_j^\infty ({\mathtt h}) -  \mu_{j'}^\infty ({\mathtt h}) | < 
  \frac{4 \gamma \langle \ell  \rangle^{-\tau} }{ j^{\perd} j'^{\perd}} \Big\} \label{reso4}    
\end{align}
with $j,j' \in \N^+ \setminus \Splus$.
We first note that some of these sets are empty. 

\begin{lemma}\label{restrizione indici risonanti} 
For $ \e $, $ \gamma \in (0, \g_0 ) $ small, we have that 
\begin{enumerate}
\item 
If $R_{\ell j}^{(I)} \neq \emptyset$ then $j^{\frac12} \leq C \langle \ell \rangle$.
\item 
If $ R_{\ell j j'}^{(II)} \neq \emptyset$ then  $|j^{\frac12} - j'^{\frac12}| \leq C \langle \ell \rangle$. Moreover,
%$ ( \ell, j, j') \neq (0,j,j) $ (namely)
$ R_{0 j j'}^{(II)} = \emptyset $,  for all $ j \neq j' $.
\item  
 If $Q_{\ell j j'}^{(II)} \neq \emptyset$ then $j^{\frac12} + j'^{\frac12} \leq C \langle \ell \rangle $. 
 \end{enumerate}
\end{lemma}

\begin{proof}
Let us consider the case of $ R_{\ell j j'}^{(II)} $. If $ R_{\ell j j'}^{(II)} \neq \emptyset $ 
there is $ {\mathtt h} \in [{\mathtt h}_1, {\mathtt h}_2] $ such that 
\be\label{up1}
|  
 \mu_j^\infty ({\mathtt h}) -  \mu_{j'}^\infty ({\mathtt h}) | < 
  \frac{4 \gamma \langle \ell  \rangle^{-\tau} }{ j^{\perd} j'^{\perd}} + |\om_\e({\mathtt h}) \cdot \ell | 
  \leq C \langle \ell \rangle \, . 
\ee
On the other hand, \eqref{mu j infty kappa}, \eqref{stime coefficienti autovalori in kappa}, 
and \eqref{caso-ell=0} imply  
\be\label{lo1}
| \mu_j^\infty ({\mathtt h}) -  \mu_{j'}^\infty ({\mathtt h}) | \geq 
\mathtt m_{\frac12}^\infty c | \sqrt{j} - \sqrt{j'} | - C \e \gamma^{- \kappa_1}  
\geq \frac{c}{2} | \sqrt{j} - \sqrt{j'} | - 1 \, . 
\ee
Combining \eqref{up1} and \eqref{lo1} we deduce $|j^{\frac12} - j'^{\frac12}| \leq C \langle \ell \rangle $.

Next we prove that $ R_{0 j j'}^{(II)} = \emptyset $, $ \forall  j \neq j' $.
Recalling \eqref{mu j infty kappa}, \eqref{stime coefficienti autovalori in kappa}, 
and the definition $\Om_j(\h) = \sqrt{j \tanh ( {\mathtt h} j)} $, 
we have 
\begin{align}
| \mu_j^\infty ({\mathtt h}) -  \mu_{j'}^\infty ({\mathtt h}) | 
& \geq \mathtt m_{\frac12}^\infty ({\mathtt h}) 
| \Om_j  ({\mathtt h}) -  \Om_{j'}  ({\mathtt h}) |  -  \frac{C\e \gamma^{- \kappa_1}}{j^{\frac12}} -  \frac{C\e \gamma^{-\kappa_1}}{(j')^{\frac12}} 	\nonumber \\
& \stackrel{\eqref{caso-ell=0}} \geq \label{low-b1}
\frac{c}{2} | \sqrt{j}  -  \sqrt{j'} |  -  \frac{C\e \gamma^{- \kappa_1}}{j^{\frac12}} 
- \frac{C\e \gamma^{- \kappa_1}}{(j')^{\frac12}} \,.
\end{align}
Now we observe that, 
for any fixed $j \in \N^+$, the minimum of $|\sqrt{j} -  \sqrt{j'}|$ 
over all $j' \in \N^+ \setminus \{ j \}$ is attained at $j' = j+1$. 
By symmetry, this implies that $ |\sqrt{j} -  \sqrt{j'}| $ is greater or equal than both 
$(\sqrt{j+1} + \sqrt{j})^{-1}$ and $(\sqrt{j'+1} + \sqrt{j'})^{-1}$.
Hence, with $c_0 := 1/(1 + \sqrt{2})$, one has
\be\label{el-sqrt}
|\sqrt{j} -  \sqrt{j'}| 
\geq c_0 \, \max \Big\{ \frac{1}{\sqrt{j}} \, , \frac{1}{\sqrt{j'}} \Big\}
\geq \frac{c_0}{2}\, \Big( \frac{1}{\sqrt{j}} \, + \frac{1}{\sqrt{j'}} \Big)
\geq \frac{c_0}{j^{\frac14} (j')^{\frac14}}
\qquad \forall j, j' \in \N^+, \ j \neq j'.
\ee  
As a consequence of \eqref{low-b1} and of the three inequalities in \eqref{el-sqrt}, 
for $\e \gamma^{- \kappa_1}$ small enough, we get for all $ j \neq j' $ 
$$
| \mu_j^\infty ({\mathtt h}) -  \mu_{j'}^\infty ({\mathtt h}) |  
\geq \frac{c}{8} | \sqrt{j}  -  \sqrt{j'} | 
\geq \frac{4 \gamma}{j^\perd j'^{\perd}} \, , 
$$
for $ \gamma $ small, since $ \perd \geq 1/4 $. 
This proves that  $ R_{0 j j'}^{(II)} = \emptyset $, for all $   j \neq j' $. 

The statement for $R_{\ell j}^{(I)}$ and $Q_{\ell j j'}^{(II)}$ is elementary. 
\end{proof}

By Lemma \ref{restrizione indici risonanti}, 
the last union in \eqref{complementare insieme di cantor} becomes 
\begin{equation} \label{7.marzo.1}
\bigcup_{(\ell, j, j') \neq (0,j,j)} R_{\ell j j'}^{(II)}
= \bigcup_{\begin{subarray}{c} \ell	\neq 0 \\ |\sqrt{j} - \sqrt{j'}| \leq C \langle \ell \rangle
\end{subarray}} R_{\ell j j'}^{(II)} .
\end{equation}
In order to estimate the measure of the sets 
\eqref{reso1}-\eqref{reso4} that are nonempty,
the key point is to prove that the perturbed frequencies satisfy estimates  similar to 
\eqref{0 Melnikov}-\eqref{2 Melnikov+}  in Proposition \ref{Lemma: degenerate KAM}. 

\begin{lemma} \label{Lemma: degenerate KAM perturbato}
{\bf (Perturbed transversality)}
For $ \e $ small enough, for all $ {\mathtt h} \in [{\mathtt h}_1, {\mathtt h}_2] $, 
\begin{align}\label{0 Melnikov perturbate}
\max_{k \leq \barka} |\partial_{\mathtt h}^{k}  \{\omega_\e ({\mathtt h}) \cdot \ell   \} | 
& \geq \frac{\rho_0}{2}\, \langle \ell \rangle 
\quad \forall \ell  \in \Z^\nu \setminus \{ 0 \}, 
\\
\label{1 Melnikov perturbate}
\max_{k \leq \barka} |\partial_{\mathtt h}^{k}  \{\omega_\e ({\mathtt h}) \cdot \ell  + \mu_j^\infty ({\mathtt h}) \} | 
& \geq \frac{\rho_0}{2}\, \langle \ell \rangle 
\quad \forall \ell  \in \Z^\nu, \, j \in  \N^+ \setminus \Splus : j^{\frac12}\leq C  \langle \ell \rangle,   
\\
\label{2 Melnikov- perturbate}
\max_{k \leq \barka} |\partial_{\mathtt h}^{k} \{ \omega_\e ({\mathtt h}) \cdot \ell  
+ \mu_j^\infty ({\mathtt h}) - \mu_{j'}^\infty({\mathtt h}) \} | 
& \geq \frac{\rho_0}{2}\, \langle \ell \rangle 
\quad \forall \ell \in \Z^\nu \setminus \{0\}, \  
j, j' \in \N^+ \setminus \Splus : |j^{\frac12} - j'^{\frac12}| \leq C  \langle \ell \rangle,
\\
\label{2 Melnikov+ perturbate}
\max_{k \leq \barka} |\partial_{\mathtt h}^{k}  \{\omega_\e ({\mathtt h}) \cdot \ell 
+ \mu_j^\infty ({\mathtt h}) + \mu_{j'}^\infty({\mathtt h}) \} | 
& \geq \frac{\rho_0}{2}\, \langle \ell \rangle 
\quad \forall \ell  \in \Z^{\nu}, \ j, j' \in \N^+ \setminus \Splus : j^{\frac12} + j'^{\frac12} \leq C  \langle \ell \rangle,
\end{align}
where $\barka$ is the index of non-degeneracy given by Proposition \ref{Lemma: degenerate KAM}.
\end{lemma}

\begin{proof}
The most delicate estimate is \eqref{2 Melnikov- perturbate}. 
We split 
$$
\mu_j^\infty ({\mathtt h}) = { \Omega}_j({\mathtt h}) + (\mu_j^\infty - { \Omega}_j)({\mathtt h})
$$
where $ { \Omega}_j({\mathtt h}) := j^\frac12 ( \tanh (j {\mathtt h} ))^\frac12 $. 
A direct calculation using \eqref{espansione asintotica degli autovalori} and \eqref{el-sqrt} shows that, for $\h \in [\h_1, \h_2]$, 
\begin{equation}\label{derivate Omega j - Omega j'}
| \pa_{\mathtt h}^k \{ { \Omega}_j({\mathtt h}) - { \Omega}_{j'}({\mathtt h}) \}| 
\leq C_k |j^{\frac12} - j'^{\frac12}| \quad \forall \, k \geq 0 \, . 
\end{equation}
Then, using \eqref{stime coefficienti autovalori in kappa}, 
one has, for all $0 \leq k \leq k_0 $,  
\begin{align}
|\pa_{\mathtt h}^k 
\{ (\mu_j^\infty - \mu_{j'}^\infty)({\mathtt h}) - ( { \Omega}_j - { \Omega}_{j'})({\mathtt h}) \}| & \leq 
|\pa_{\mathtt h}^k \{ (\mathtt m_{\frac12}^\infty({\mathtt h}) - 1) ({ \Omega}_j({\mathtt h}) - { \Omega}_{j'}({\mathtt h}) ) \} | \nonumber  + 
| \pa_{\mathtt h}^k \rin_j^\infty({\mathtt h})| + | \pa_{\mathtt h}^k \rin_{j'}^\infty({\mathtt h})| \nonumber\\
& \stackrel{\eqref{derivate Omega j - Omega j'}} 
\leq C_{k_0} \{ \e \g^{-1-k} |j^{\frac12} - j'^{\frac12}| 
+ \e \gamma^{- \kappa_1 - k} ( j^{-\frac12} + (j')^{-\frac12} ) \} \nonumber \\
& \stackrel{\eqref{el-sqrt}} \leq
C_{k_0}' \e \gamma^{- \kappa_1 - k } |j^{\frac12} - j'^{\frac12}|\,. 
\label{mu j - mu j' infty}
\end{align}
Recall that $k_0 = \barka + 2$ (see \eqref{def k0}). 
By \eqref{omega epsilon kappa-1} and \eqref{mu j - mu j' infty}, using $|j^{\frac12} - j'^{\frac12}| \leq C \langle \ell \rangle$, we get 
\begin{align}
\max_{k \leq \barka} | \pa_{\mathtt h}^k \{\omega_\e ({\mathtt h}) \cdot \ell + \mu_j^\infty ({\mathtt h}) - \mu_{j'}^\infty({\mathtt h})\}| 
& \geq \max_{k \leq \barka} | \pa_{\mathtt h}^k \{{\vec \om}({\mathtt h}) \cdot \ell 
+ \Omega_j({\mathtt h}) - \Omega_{j'}({\mathtt h}) \}| 
- C \e \gamma^{- (1 + \barka)} |\ell| \nonumber\\
& \quad - C \e \gamma^{- (\barka + \kappa_1)} |j^{\frac12} - j'^{\frac12}| 
\nonumber\\
& \geq \max_{k \leq \barka} | \pa_{\mathtt h}^k \{{\vec \om}({\mathtt h}) \cdot \ell + 
\Omega_j({\mathtt h}) - \Omega_{j'}({\mathtt h}) \}| 
- C \e \gamma^{- (\barka + \kappa_1)} \langle \ell \rangle 
\nonumber\\
& \stackrel{\eqref{2 Melnikov-}}{\geq} 
\rho_0 \langle \ell \rangle - C \e \gamma^{- (\barka + \kappa_1 )} 
\langle \ell \rangle 
\geq \rho_0 \langle \ell \rangle / 2 \nonumber 
\end{align}
provided $ \e \gamma^{- (\barka + \kappa_1)} \leq \rho_0 \slash (2 C) $, 
which, by  \eqref{relazione tau k0}, is satisfied for $ \e $ small enough. 
\end{proof}

As an application of R\"ussmann Theorem 17.1 in \cite{Ru1} we deduce the following 

\begin{lemma}{\bf (Estimates of the resonant sets)} \label{stima risonanti Russman}
The measure of the sets in \eqref{reso1}-\eqref{reso4} satisfies 
\begin{align*}
& |R_{\ell}^{(0)}| \lesssim 
\big( \gamma \langle \ell \rangle^{- (\tau + 1)}\big)^{\frac{1}{\barka}} 
\quad \forall \ell \neq 0 \, , 
\qquad   \, 
|R_{\ell j}^{(I)}| \lesssim 
\big( \gamma j^{\frac12} \langle \ell \rangle^{- (\tau + 1)}\big)^{\frac{1}{\barka}} \, , 
\\
&  
|R_{\ell j j'}^{(II)}| \lesssim \Big(  \gamma 
\frac{\langle \ell \rangle^{-(\tau + 1)} }{j^{\perd} j'^{\perd}} \Big)^{\frac{1}{\barka}} 
\quad \forall \ell \neq 0 , \qquad 
|Q_{\ell j j'}^{(II)}| \lesssim   
\big(  \gamma (j^{\frac12} + j'^{\frac12}) \langle \ell \rangle^{- (\tau + 1)}\big)^{\frac{1}{\barka}} \, . 
\end{align*}
\end{lemma}

\begin{proof}
We prove the estimate of $ R_{\ell j j'}^{(II)} $ in \eqref{reso4}. 
The other cases are simpler. 
We write 
$$
R_{\ell j j'}^{(II)} = \Big\{ {\mathtt h} \in [{\mathtt h}_1, {\mathtt h}_2] : 
| f_{\ell j j'}({\mathtt h})| <  \frac{4 \gamma }{ \langle \ell \rangle^{\tau + 1} j^{\perd} j'^{\perd}}
\Big\}
$$
where
$ f_{\ell j j'}(\h) := ( \omega_\e({\mathtt h}) \cdot \ell + \mu_j^\infty({\mathtt h}) - \mu_{j'}^\infty({\mathtt h}) ) \langle \ell\rangle^{-1} $. 
By \eqref{7.marzo.1}, we restrict to the case 
$|j^{\frac12} - j'^{\frac12}| \leq C \langle \ell \rangle$ and $\ell \neq 0$.
By \eqref{2 Melnikov- perturbate},
$$
\max_{k \leq \barka} | \pa_{\mathtt h}^k 
f_{\ell j j'}({\mathtt h})| \geq \rho_0 / 2 \,, \quad \forall {\mathtt h} \in [{\mathtt h}_1, {\mathtt h}_2]\,.
$$
In addition, \eqref{omega epsilon kappa}-\eqref{stime coefficienti autovalori in kappa} 
and Lemma \ref{restrizione indici risonanti} imply that
$ \max_{k \leq k_0} | \pa_{\mathtt h}^k f_{\ell j j'}({\mathtt h}) | \leq C $ 
for all ${\mathtt h} \in [{\mathtt h}_1, {\mathtt h}_2] $, 
provided $\e \gamma^{- (k_0 + \kappa_1)}$ is small enough,
namely, by \eqref{relazione tau k0}, $\e$ is small enough.
In particular, $f_{\ell j j'}$ is of class $\mC^{k_0 - 1} = \mC^{\barka + 1}$. 
Thus Theorem 17.1 in \cite{Ru1} applies,
whence the lemma follows. 
\end{proof}

\noindent 
{\sc Proof of Theorem \ref{Teorema stima in misura} completed.}
By Lemma \ref{restrizione indici risonanti} 
(in particular, recalling that $R_{\ell jj'}^{(II)}$ is empty for $\ell = 0$ and $j \neq j'$, 
see \eqref{7.marzo.1})
and Lemma \ref{stima risonanti Russman}, 
the measure of the set ${\cal G}_\e^c$ in \eqref{complementare insieme di cantor} is estimated by
\begin{align}
|{\cal G}_\e^c| & \leq 
\sum_{\ell \neq 0} |R_\ell^{(0)}| 
+ \sum_{\ell, j} |R_{\ell j}^{(I)}| 
+ \sum_{(\ell , j, j') \neq (0,j,j)} |R_{\ell j j'}^{(II)}| 
+ \sum_{\ell , j, j'} |Q_{\ell j j'}^{(II)}| 
\nonumber \\
& \leq 
\sum_{\ell \neq 0 } |R_\ell^{(0)}| 
+ \sum_{j \leq C \langle \ell \rangle^2} |R_{\ell j}^{(I)}| 
+ \sum_{ \begin{subarray}{c} \ell \neq 0 \\ 
| \sqrt{j} - \sqrt{j'} | \leq C \langle \ell \rangle \end{subarray}} 
|R_{\ell j j'}^{(II)}| 
+ \sum_{ j, j' \leq C \langle \ell \rangle^2} |Q_{\ell j j'}^{(II)}| 
\nonumber\\
& % \stackrel{ \text{Lemma \ref{stima risonanti Russman}} }{\lesssim} 
\lesssim 
\sum_{\ell} \Big( \frac{\gamma}{\langle \ell \rangle^{\tau + 1} } \Big)^{\frac{1}{\barka}} 
+ \sum_{j \leq C \langle \ell \rangle^2} 
\Big( \frac{\gamma j^{\frac12} }{ \langle \ell \rangle^{\tau + 1} } \Big)^{\frac{1}{\barka}} 
+ \sum_{| \sqrt{j} - \sqrt{j'} | \leq C \langle \ell \rangle } 
\Big(  \frac{ \gamma }{ \langle \ell \rangle^{\tau + 1} j^{\perd}j'^{\perd}}  \Big)^{\frac{1}{\barka}} 
+ \sum_{ j, j' \leq C \langle \ell \rangle^{2}} 
\Big( \frac{\gamma (j^{\frac12} + j'^{\frac12})}{ \langle \ell \rangle^{\tau + 1} } \Big)^{\frac{1}{\barka}} 
\nonumber \\
& \leq C \gamma^{\frac{1}{\barka}} \Big\{ 
\sum_{\ell \in \Z^\nu} \frac{1}{ \langle \ell \rangle^{\frac{\tau}{\barka} - 4} } \, 
+ \sum_{| \sqrt{j} - \sqrt{j'} | \leq C \langle \ell \rangle } 
\frac{1}{ \langle \ell \rangle^{\frac{\tau + 1}{\barka}} j^{\frac{\perd}{\barka}} j'^{\frac{\perd}{\barka}} } \Big\} \, .
\label{stima-complementary-set}
\end{align}
The first series in \eqref{stima-complementary-set} converges because 
$\frac{\tau}{\barka} - 4 > \nu$ by \eqref{relazione tau k0}. 
For the second series in \eqref{stima-complementary-set}, we observe that 
the sum is symmetric in $(j,j')$ and, 
for $j \leq j'$, the bound $| \sqrt{j} - \sqrt{j'} | \leq C \langle \ell \rangle$ implies 
that $j \leq j' \leq j + C^2 \langle \ell \rangle^2 + 2C \sqrt{j} \langle \ell \rangle$. 
Since 
\[
\forall \ell, j, \quad \ 
\sum_{j' = j}^{j+p} \frac{1}{j'^{\frac{\perd}{\barka}}} 
\leq 
\sum_{j' = j}^{j+p} \frac{1}{j^{\frac{\perd}{\barka}}} 
= \frac{p+1}{j^{\frac{\perd}{\barka}}} \,, \quad 
p := C^2 \langle \ell \rangle^2 + 2C \sqrt{j} \langle \ell \rangle,
\]
the second series in \eqref{stima-complementary-set} converges because 
$\frac{\t+1}{\barka} - 2 > \nu$ and $2 \frac{\perd}{\barka} - \frac12 > 1$ 
by \eqref{relazione tau k0}. 
By \eqref{stima-complementary-set} we get
$$
|{\cal G}_\e^c | \leq C \gamma^{\frac{1}{\barka}}  \, . 
$$
In conclusion, for  $ \gamma = \e^a $,  
we find  $ |{\cal G}_\e| \geq {\mathtt h}_2 - {\mathtt h}_1 - C \e^{a / \barka } $ 
and the proof of Theorem \ref{Teorema stima in misura} is concluded.

\section{Approximate inverse}\label{sezione approximate inverse}
In order to implement a convergent Nash-Moser scheme that leads to a solution of 
$ \mF(i, \alpha) = 0 $ we construct an \emph{almost-approximate right inverse} (see Theorem \ref{thm:stima inverso approssimato})
of the linearized operator 
\be\label{diaF}
d_{i, \alpha} {\cal F}(i_0, \alpha_0 )[\widehat \imath \,, \widehat \alpha ] =
\Dom \widehat \imath - d_i X_{H_\alpha} ( i_0 (\vphi) ) [\widehat \imath ] - (\widehat \alpha,0, 0 )
 \, . 
\ee
Note that 
$ d_{i, \alpha} {\cal F}(i_0, \alpha_0 ) = d_{i, \alpha} {\cal F}(i_0 ) $ is independent of $ \alpha_0 $, see \eqref{operatorF} and recall 
that the perturbation  $P$ does not depend on $\alpha $. 

Since  the linearized operator $ d_i X_{H_\alpha} ( i_0 (\vphi) ) $ 
has the $ (\theta, I, z ) $-components which are all coupled,  it is particularly intricate to invert the operator \eqref{diaF}.  
Then we  implement the approach in \cite{BB13}, \cite{BBM-auto},   \cite{BertiMontalto}  to reduce it,
approximately, to a triangular form. 
We outline the steps of this strategy. 
The first observation is that, close to an invariant torus, there exists
symplectic coordinates in which the linearized equations are a triangular system as in \eqref{linear-torus-new coordinates}. 
We implement quantitatively this observation for any torus, which, in general, is non invariant.
Thus we define  the ``error function'' 
\begin{equation} \label{def Zetone}
Z(\vphi) :=  (Z_1, Z_2, Z_3) (\vphi) := {\cal F}(i_0, \alpha_0) (\vphi) =
\om \cdot \pa_\vphi i_0(\vphi) - X_{H_{\alpha_0}}(i_0(\vphi)) \, .
\end{equation}
If $ Z = 0 $ then the torus $ i_0 $ is invariant for $ X_{H_{\alpha_0}} $; 
in general, we say that $ i_0 $ is ``approximately invariant'', up to order $O(Z)$. 
Given a torus $i_0(\vphi) = (\theta_0(\vphi), I_0(\vphi), z_0(\vphi))$ satisfying \eqref{ansatz 0}
(condition which is satisfied by the approximate solutions obtained by the Nash-Moser iteration of Section \ref{sec:NM}), 
we first construct an isotropic torus $i_\delta(\vphi) = (\theta_0(\vphi), I_\delta(\vphi), z_0(\vphi))$ which is close to $i_0$,
see Lemma \ref{toro isotropico modificato}.  
Note that, by  \eqref{stima toro modificato}, 
$ {\cal F}(i_\delta, \alpha_0) $ is also $ O(Z) $. 
Since $i_\d$ is isotropic, the diffeomorphism $ (\phi, y, w) \mapsto G_\delta(\phi, y, w)$ 
defined in \eqref{trasformazione modificata simplettica} is symplectic.  
In these coordinates, the torus $i_\delta$ reads $(\phi, 0, 0)$,  
and the transformed Hamiltonian system becomes \eqref{sistema dopo trasformazione inverso approssimato}, where, by Lemma \ref{coefficienti nuovi}, 
the terms $ \pa_\phi K_{00}, K_{10} - \omega, K_{01}$ are $O(Z)$. 
Thus, neglecting such terms, 
the problem of finding an approximate inverse of the linearized operator $d_{i, \alpha} {\cal F}(i_0, \alpha_0 )$ is reduced to the task of inverting the operator ${\mathbb D}$ in \eqref{operatore inverso approssimato}.
We solve system  \eqref{operatore inverso approssimato proiettato}   in a {\it triangular} way. 
First we solve the equation for the $y$-component of system \eqref{operatore inverso approssimato proiettato}, simply by inverting the differential operator $\omega \cdot \partial_\vphi$, see \eqref{soleta}  and 
recall that $\omega$ is Diophantine. 
Then in \eqref{normalw} we solve the equation for the $w$-component, 
thanks to the almost invertibility of the operator $\mL_\omega$ in \eqref{Lomega def}, 
which is proved in Theorem \ref{inversione parziale cal L omega} 
and stated in this section as assumption \eqref{inversion assumption}-\eqref{tame inverse}. 
Finally the equation \eqref{equazione psi hat} for the $\phi$-component is solved in \eqref{sol psi}, by modifying the counterterms according to \eqref{sol alpha} and by inverting $\omega \cdot \partial_\vphi$. In conclusion, in Theorem \ref{thm:stima inverso approssimato} we estimate quantitatively how 
the conjugation of ${\mathbb D}$ with the differential of $G_\delta$ (see \eqref{definizione T}) is an almost approximate inverse of the linearized operator $d_{i, \alpha}{\cal F}(i_0, \alpha_0)$. 

\smallskip

First of all, we state some preliminary estimates for the composition operator induced by 
the Hamiltonian vector field $ X_P = (  \pa_I P, - \pa_\theta P, J \nabla_z P) $ in \eqref{operatorF}. 
\begin{lemma}{\bf (Estimates of the perturbation $P$)}\label{lemma quantitativo forma normale}
Let $ \fracchi(\ph) $ in \eqref{componente periodica} satisfy
$ \| {\mathfrak I} \|_{3 s_0 + 2 k_0 + 5}^{k_0, \gamma}  \leq 1$.
Then 
the following estimates hold:
\begin{equation}\label{stime XP}
\| X_P(i)\|_s^{k_0, \gamma}  \lesssim_s 1 + \| {\mathfrak I}\|_{s + 2 s_0 + 2 k_0 + 3}^{k_0, \gamma}\,, 
\end{equation}
and for all $\widehat \imath := (\widehat \theta, \widehat I, \widehat z)$ 
\begin{align}\label{stima derivata XP}
 \| d_i X_P(i)[\widehat \imath]\|_s^{k_0, \gamma} & \lesssim_s \| \widehat \imath \|_{s + 1}^{k_0, \gamma} + \| \mathfrak I\|_{s + 2 s_0 + 2 k_0 + 4}^{k_0, \gamma} \| \widehat \imath \|_{s_0 + 1}^{k_0, \gamma}\,, \\
 \label{stima derivata seconda XP}
\| d^2_i X_P(i)[\widehat \imath, \widehat \imath]\|_s^{k_0, \gamma} & \lesssim_s \| \widehat \imath\|_{s + 1}^{k_0, \gamma} \| \widehat \imath \|_{s_0 + 1}^{k_0, \gamma} + \| \mathfrak I\|_{s + 2 s_0 + 2 k_0 + 5}^{k_0, \gamma} (\| \widehat \imath \|_{s_0 + 1}^{k_0, \gamma})^2\,.
\end{align}
\end{lemma}
\begin{proof}
The proof is the same as the one of Lemma 5.1 in \cite{BertiMontalto}, using also the estimates on the Dirichlet Neumann operator in Proposition \ref{lemma dirichlet Neumann}. 
\end{proof}

Along this section we assume the following hypothesis, 
which is verified by the approximate solutions obtained at each step of the Nash-Moser 
Theorem \ref{iterazione-non-lineare}. 

\begin{itemize}
\item {\sc Ansatz.} 
The map $(\omega, \h) \mapsto \fracchi_0(\omega, \h) :=  i_0(\ph; \om, \h) - (\ph,0,0) $ 
is ${k_0} $ times differentiable with respect 
to the parameters $(\omega, \h) \in \R^\nu \times [\h_1, \h_2] $, 
and for some $ \mu := \mu (\t, \nu) >  0 $,  $\gamma \in (0, 1)$,   
\begin{equation}\label{ansatz 0}
\| {\mathfrak I}_0  \|_{s_0+\mu}^{k_0, \gamma} +  |\alpha_0 - \omega|^{k_0, \gamma} \leq C\e \gamma^{- 1} \,.
\end{equation}
For some $\kappa : = \kappa( \tau, \nu) > 0$, we shall always assume the smallness condition $ \e \g^{- \kappa} \ll 1  $.
\end{itemize}

We now   implement the symplectic procedure to 
reduce $ d_{i, \alpha} {\cal F}(i_0, \alpha_0 ) $ approximately to a triangular form. 
 An invariant torus $ i_0 $ with Diophantine flow 
is isotropic (see \cite{HF},\cite{BB13}), 
namely the pull-back $ 1$-form $ i_0^* \Lambda $ is closed, 
where $ \Lambda $ is the 1-form in \eqref{Lambda 1 form}. 
This is equivalent to say that the 2-form 
$ i_0^* {\cal W} =  i_0^* d \Lambda  = d i_0^* \Lambda = 0 $. 
For an approximately invariant torus $ i_0 $ the 1-form $ i_0^* \Lambda $ 
is only ``approximately closed":
we consider
\begin{equation}\label{coefficienti pull back di Lambda}
i_0^* \Lambda = {\mathop \sum}_{k = 1}^\nu a_k (\vphi) d \vphi_k \,,\quad 
a_k(\vphi) := - \big( [\pa_\ph \teta_0 (\vphi)]^T I_0 (\vphi)  \big)_k 
- \frac12 ( \partial_{\vphi_k} z_0(\ph), J z_0(\ph) )_{L^2(\T_x)}
\end{equation}
and we show that 
\begin{equation} \label{def Akj} 
i_0^* {\cal W} = d \, i_0^* \Lambda = {\mathop\sum}_{1 \leq k < j \leq \nu} A_{k j}(\vphi) d \vphi_k \wedge d \vphi_j\,,\quad A_{k j} (\vphi) := 
\partial_{\vphi_k} a_j(\ph) - \partial_{\vphi_j} a_k(\ph) \, , 
\end{equation} 
is of order $O(Z)$, see Lemma \ref{lemma:aprile}. 
By \eqref{operatorF}, \eqref{stime XP}, \eqref{ansatz 0}, 
the error function $Z$ defined in \eqref{def Zetone} is estimated in terms of the approximate torus as
\begin{equation}\label{Lemma Z piccolezza} 
\| Z\|_{s}^{k_0, \gamma} \lesssim_s \e \gamma^{-  1} + \| \fracchi_0\|_{s + 2}^{k_0, \gamma}.
\end{equation}

\begin{lemma} \label{lemma:aprile}
Assume that $ \om $ belongs to $ \mathtt {DC} (\gamma, \tau)$ defined in \eqref{DC tau0 gamma0}. 
Then the coefficients $ A_{kj}  $ in \eqref{def Akj} satisfy 
\begin{equation}\label{stima A ij}
\| A_{k j} \|_s^{k_0, \gamma} \lesssim_s \gamma^{-1}
\big(\| Z \|_{s+\t(k_0 + 1) + k_0+1}^{k_0, \gamma} + \| Z \|_{s_0+1}^{k_0, \gamma} \|  {\mathfrak I}_0 \|_{s+\t(k_0 + 1) + k_0+1}^{k_0, \gamma} \big)\,.
\end{equation}
\end{lemma}

\begin{proof}
The  $ A_{kj}  $  satisfy the identity 
$  \Dom A_{k j} 
= $ $ {\cal W}\big( \pa_\ph Z(\vphi) \underline{e}_k ,  \pa_\ph i_0(\vphi)  \underline{e}_j \big) 
+ $ $ {\cal W} \big(\pa_\ph i_0(\vphi) \underline{e}_k , \pa_\ph Z(\vphi) \underline{e}_j \big)  $
where  $ \underline{e}_k  $ denotes the $ k $-th versor of $ \R^\nu $, see \cite{BB13}, Lemma 5. 
Then \eqref{stima A ij} follows by \eqref{ansatz 0} and Lemma \ref{lemma:WD}. 
\end{proof}

As in \cite{BB13}, \cite{BBM-auto} we first modify the approximate torus $ i_0 $ to obtain an isotropic torus $ i_\d $ which is 
still approximately invariant. We denote the Laplacian  $ \Delta_\vphi := \sum_{k=1}^\nu \partial_{\vphi_k}^2 $.

\begin{lemma}\label{toro isotropico modificato} {\bf (Isotropic torus)} 
The torus $ i_\delta(\vphi) := (\theta_0(\vphi), I_\delta(\vphi), z_0(\vphi) ) $ defined by 
\begin{equation}\label{y 0 - y delta}
I_\d := I_0 +  [\pa_\ph \theta_0(\vphi)]^{- T}  \rho(\vphi) \, , \qquad 
\rho_j(\vphi) := \Delta_\vphi^{-1} {\mathop\sum}_{ k = 1}^\nu \partial_{\vphi_j} A_{k j}(\vphi) 
\end{equation}
 is {\it isotropic}. 
There is $ \s := \s(\nu,\t, k_0) $ such that
\begin{align} \label{2015-2}
\| I_\delta - I_0 \|_s^{k_0, \gamma} & \leq \| I_0 \|_{s+1}^{k_0,\gamma} \\
\label{stima y - y delta}
\| I_\delta - I_0 \|_s^{k_0, \gamma} 
& \lesssim_s  \gamma^{-1} \big(\| Z \|_{s + \s}^{k_0, \gamma} + 
\| Z \|_{s_0 + \s}^{k_0, \gamma} \|  {\mathfrak I}_0 \|_{s + \s}^{k_0, \gamma} \big) \,,
\\
\label{stima toro modificato}
\| {\cal F}(i_\delta, \alpha_0) \|_s^{k_0, \gamma}
& \lesssim_s  \| Z \|_{s + \s}^{k_0, \gamma}  +  \| Z \|_{s_0 + \s}^{k_0, \gamma} \|  {\mathfrak I}_0 \|_{s + \s}^{k_0, \gamma} \\
\label{derivata i delta}
\| d_i [ i_\d][ \widehat \imath ] \|_s^{k_0, \gamma} & \lesssim_s \| \widehat \imath \|_s^{k_0, \gamma} +  
\| {\mathfrak I}_0\|_{s + \s}^{k_0, \gamma} \| \widehat \imath  \|_{s_0}^{k_0, \gamma} \, .
\end{align}
 \end{lemma}
%In the paper we denote equivalently the differential by $ \partial_i $ or  $ d_i $. 
We denote 
by $ \s := \s(\nu, \tau, k_0 ) $ possibly different (larger) ``loss of derivatives"  constants. 

\begin{proof}
The Lemma follows  as in \cite{BBM-auto}
by \eqref{stima derivata XP} and 
\eqref{coefficienti pull back di Lambda}-\eqref{stima A ij}. 
\end{proof}

In order to find an approximate inverse of the linearized operator $d_{i, \alpha} {\cal F}(i_\delta )$, 
we introduce  the symplectic diffeomorpshim 
$ G_\delta : (\phi, y, w) \to (\theta, I, z)$ of the phase space $\T^\nu \times \R^\nu \times H_{\Splus}^\bot$ defined by
\begin{equation}\label{trasformazione modificata simplettica}
\begin{pmatrix}
\theta \\
I \\
z
\end{pmatrix} := G_\delta \begin{pmatrix}
\phi \\
y \\
w
\end{pmatrix} := 
\begin{pmatrix}
\!\!\!\!\!\!\!\!\!\!\!\!\!\!\!\!\!\!\!\!\!\!\!\!\!\!\!\!\!\!\!\!\!
\!\!\!\!\!\!\!\!\!\!\!\!\!\!\!\!\!\!\!\!\!\!\!\!\!\!\!\!\!\!\!\!\!
\!\!\!\!\!\!\!\!\!\!\!\!\!\!\!\!\!\!\!\!\!\!\!\!\!\!\!\!\!\!\!\! \theta_0(\phi) \\
I_\delta (\phi) + [\pa_\phi \theta_0(\phi)]^{-T} y - \big[ (\pa_\teta \tilde{z}_0) (\theta_0(\phi)) \big]^T J w \\
\!\!\!\!\!\!\!\!\!\!\!\!\!\!\!\!\!\!\!\!\!\!\!\!\!\!\!\!\!
\!\!\!\!\!\!\!\!\!\!\!\!\!\!\!\!\!\!\!\!\!\!\!\!\!\!\!\!\!
\!\!\!\!\!\!\!\!\!\!\!\!\!\!\!\!\!\!\!\!\!\!\!\!\!\!\!\!\!  z_0(\phi) + w
\end{pmatrix} 
\end{equation}
where $ \tilde{z}_0 (\theta) := z_0 (\theta_0^{-1} (\theta))$. 
It is proved in \cite{BB13} that $ G_\delta $ is symplectic, because  the torus $ i_\d $ is isotropic 
(Lemma \ref{toro isotropico modificato}).
In the new coordinates,  $ i_\delta $ is the trivial embedded torus
$ (\phi , y , w ) = (\phi , 0, 0 ) $.  Under the symplectic change of variables $ G_\d $ the Hamiltonian 
vector field $ X_{H_\a} $ (the Hamiltonian $  H_\a $ is defined in \eqref{H alpha}) changes into 
\be\label{new-Hamilt-K}
X_{K_\a} = (D G_\d)^{-1} X_{H_\a} \circ G_\d \qquad {\rm where} \qquad K_\alpha := H_{\alpha} \circ G_\d  \, .
\ee
By \eqref{parity solution} the transformation $ G_\d $ is also reversibility preserving and so $ K_\alpha $ is reversible, 
$ K_\a \circ \tilde \rho = K_\a $.  

The Taylor expansion of $ K_\a $ at the trivial torus $ (\phi , 0, 0 ) $ is 
\begin{align} 
K_\alpha (\phi, y , w)
& =  K_{00}(\phi, \alpha) + K_{10}(\phi, \alpha) \cdot y + (K_{0 1}(\phi, \alpha), w)_{L^2(\T_x)} + 
\frac12 K_{2 0}(\phi) y \cdot y 
\nonumber \\ & 
\quad +  \big( K_{11}(\phi) y , w \big)_{L^2(\T_x)} 
+ \frac12 \big(K_{02}(\phi) w , w \big)_{L^2(\T_x)} + K_{\geq 3}(\phi, y, w)  
\label{KHG}
\end{align}
where $ K_{\geq 3} $ collects the terms at least cubic in the variables $ (y, w )$.
The Taylor coefficient $K_{00}(\phi, \alpha) \in \R $,  
$K_{10}(\phi, \alpha) \in \R^\nu $,  
$K_{01}(\phi, \alpha) \in H_{\Splus}^\bot$, % (it is a function of $ x  \in \T $),
$K_{20}(\phi) $ is a $\nu \times \nu$ real matrix, 
$K_{02}(\phi)$ is a linear self-adjoint operator of $ H_{\Splus}^\bot $ and 
$K_{11}(\phi) \in {\cal L}(\R^\nu, H_{\Splus}^\bot )$. 
Note that, by \eqref{H alpha} and \eqref{trasformazione modificata simplettica}, 
the only Taylor coefficients that depend on $ \a $ are
$ K_{00} $, $ K_{10} $, $ K_{01} $.

The Hamilton equations associated to \eqref{KHG}  are 
\begin{equation}\label{sistema dopo trasformazione inverso approssimato}
\begin{cases}
\dot \phi \hspace{-30pt} & = K_{10}(\phi, \alpha) +  K_{20}(\phi) y + 
K_{11}^T (\phi) w + \partial_{y} K_{\geq 3}(\phi, y, w)
\\
\dot y \hspace{-30pt} & = 
\partial_\phi K_{00}(\phi, \alpha) - [\partial_{\phi}K_{10}(\phi, \alpha)]^T  y - 
[\partial_{\phi} K_{01}(\phi, \alpha)]^T  w  
\\
& \quad -
\partial_\phi \big( \frac12 K_{2 0}(\phi) y \cdot y + ( K_{11}(\phi) y , w )_{L^2(\T_x)} + 
\frac12 ( K_{02}(\phi) w , w )_{L^2(\T_x)} + K_{\geq 3}(\phi, y, w) \big)
\\
\dot w \hspace{-30pt} & = J \big( K_{01}(\phi, \alpha) + 
K_{11}(\phi) y +  K_{0 2}(\phi) w + \nabla_w K_{\geq 3}(\phi, y, w) \big) 
\end{cases} 
\end{equation}
where $ \partial_{\phi}K_{10}^T $ is the $ \nu \times \nu $ transposed matrix and 
$ \partial_{\phi}K_{01}^T $,  $ K_{11}^T  : {H_{\Splus}^\bot \to \R^\nu} $ are defined by the 
duality relation $ ( \partial_{\phi} K_{01} [\hat \phi ],  w)_{L^2_x}  = \hat \phi \cdot [\partial_{\phi}K_{01}]^T w  $,
$ \forall \hat \phi \in \R^\nu, w \in H_{\Splus}^\bot $, 
and similarly for $ K_{11} $. 
Explicitly, for all  $ w \in H_{\Splus}^\bot $, 
and denoting by $\underline{e}_k$ the $k$-th versor of $\R^\nu$, 
\begin{equation} \label{K11 tras}
K_{11}^T(\phi) w =  {\mathop \sum}_{k=1}^\nu \big(K_{11}^T(\phi) w \cdot \underline{e}_k\big) \underline{e}_k   =
{\mathop \sum}_{k=1}^\nu  
\big( w, K_{11}(\phi) \underline{e}_k  \big)_{L^2(\T_x)}  \underline{e}_k  \, \in \R^\nu \, .  
\end{equation}
The coefficients $ K_{00} $, $ K_{10} $, $K_{01} $  in the Taylor expansion \eqref{KHG} vanish 
on an exact solution (i.e. $ Z  = 0 $).

\begin{lemma} \label{coefficienti nuovi} 
We have  
\begin{align}\label{K 00 10 01}
& \|  \partial_\phi K_{00}(\cdot, \alpha_0) \|_s^{k_0, \gamma} 
+ \| K_{10}(\cdot, \alpha_0) - \om  \|_s^{k_0, \gamma} +  \| K_{0 1}(\cdot, \alpha_0) \|_s^{k_0, \gamma} 
\lesssim_s  \| Z \|_{s + \s}^{k_0, \gamma} +  \| Z \|_{s_0 + \s}^{k_0, \gamma} \| {\mathfrak I}_0 \|_{s + \s}^{k_0, \gamma} \, . \\
%\end{lemma}
%
%\begin{proof}
%Use Lemma 8 of \cite{BB13} or Lemma 6.4 of \cite{BBM-auto} and \eqref{ansatz 0},  \eqref{stima y - y delta}, \eqref{stima toro modificato}.
%the following identities are proved 
%\begin{align*}
%\partial_\phi K_{00}(\phi, \alpha_0) & =  
%- [ \pa_\phi \teta_0 (\phi) ]^T \big( - Z_{2, \d} - 
%[ \pa_\phi I_\d] [ \pa_\phi \teta_0]^{-1} Z_{1, \d}   
%- [ (\pa_\theta {\tilde z}_0)( \teta_0 (\phi)) ]^T J Z_{3,\d} 
%\\ 
%&  \quad - [ (\pa_\theta {\tilde z}_0)(\teta_0 (\phi)) ]^T J \partial_\phi z_0 (\phi) [ \pa_\phi \teta_0 (\phi)]^{-1} Z_{1,\d} \big) \, ,  
%\\
%K_{10}(\phi, \alpha_0) & 
%=  \omega -   [ \pa_\phi \theta_0(\phi)]^{-1} Z_{1,\d}(\phi) \,, 
%\\
%K_{01}(\phi, \alpha_0) 
%& = J Z_{3,\d} - J \pa_\phi z_0(\phi) [\pa_\phi \theta_0(\phi)]^{-1} Z_{1,\d}(\phi) \, ,
%\end{align*}
%where $  Z_\d = (Z_{1,\d}, Z_{2,\d}, Z_{3,\d}) := {\cal F}(i_\d, \alpha_0) $. 
%Then \eqref{ansatz 0},  \eqref{stima y - y delta}, \eqref{stima toro modificato}  imply \eqref{K 00 10 01}.
%\end{proof}
%
%%The norm of $K_{20}$ is the sum of the norms of its matrix entries.  
%We now estimate  the variation of the coefficients
%$ K_{00} $, $ K_{10} $, $K_{01} $ with respect to $ \a $.
%Note, in particular, that $ \partial_\alpha K_{10} \approx {\rm Id} $ says that the tangential frequencies vary with 
%$ \a \in \R^\nu $. We also estimate  $ K_{20}$ and $ K_{11} $. 
%
%\begin{lemma} \label{lemma:Kapponi vari} We have 
\notag
& \| \partial_\alpha K_{00}\|_s^{k_0, \gamma} + \| \partial_\alpha K_{10} - {\rm Id} \|_s^{k_0, \gamma} + \| \partial_\alpha K_{0 1}\|_s^{k_0, \gamma} \lesssim_s   \| \fracchi_0\|_{s + \sigma}^{k_0, \gamma} \, , 
\quad   \|K_{20}  \|_s^{k_0, \gamma} \lesssim_s \e \big( 1 + \| \fracchi_0\|_{s + \s}^{k_0, \gamma} \big) \, ,  \\ 
\notag
& \| K_{11} y \|_s^{k_0, \gamma} 
\lesssim_s \e \big(\| y \|_s^{k_0, \gamma}
+ \| \fracchi_0 \|_{s + \sigma}^{k_0, \gamma}  
\| y \|_{s_0}^{k_0, \gamma} \big) \, , \quad  \| K_{11}^T w \|_s^{k_0, \gamma}
\lesssim_s \e \big(\| w \|_{s + 2}^{k_0, \gamma}
+  \| \fracchi_0 \|_{s + \sigma}^{k_0, \gamma}
\| w \|_{s_0 + 2}^{k_0, \gamma} \big)\, . 
\end{align}
\end{lemma}

\begin{proof}
The lemma follows as in \cite{BB13}, \cite{BBM-auto}, \cite{BertiMontalto} 
by \eqref{stime XP}, \eqref{ansatz 0}, \eqref{2015-2}, \eqref{stima y - y delta}, \eqref{stima toro modificato}, \eqref{K11 tras}.
\end{proof}

Under the linear change of variables 
\begin{equation}\label{DGdelta}
D G_\delta(\vphi, 0, 0) 
\begin{pmatrix}
\widehat \phi \, \\
\widehat y \\
\widehat w
\end{pmatrix} 
:= 
\begin{pmatrix}
\pa_\phi \theta_0(\vphi) & 0 & 0 \\
\pa_\phi I_\delta(\vphi) & [\pa_\phi \theta_0(\vphi)]^{-T} & 
- [(\pa_\theta \tilde{z}_0)(\theta_0(\vphi))]^T J \\
\pa_\phi z_0(\vphi) & 0 & I
\end{pmatrix}
\begin{pmatrix}
\widehat \phi \, \\
\widehat y \\
\widehat w
\end{pmatrix} 
\end{equation}
the linearized operator  $d_{i, \alpha}{\cal F}(i_\delta )$ is approximately transformed
(see the proof of Theorem \ref{thm:stima inverso approssimato}) 
into the one  obtained when one linearizes 
the Hamiltonian system \eqref{sistema dopo trasformazione inverso approssimato} at $(\phi, y , w ) = (\vphi, 0, 0 )$,
differentiating also in $ \a $ at $ \a_0 $, and changing $ \partial_t \rightsquigarrow \Dom $, 
namely 
\begin{equation}\label{lin idelta}
\begin{pmatrix}
\widehat \phi  \\
\widehat y    \\ 
\widehat w \\
\widehat \a 
\end{pmatrix} \mapsto
%\hspace{-5pt}
\begin{pmatrix}
\Dom \widehat \phi - \partial_\phi K_{10}(\vphi)[\widehat \phi \, ] - \partial_\alpha K_{10}(\vphi)[\widehat \alpha] - 
K_{2 0}(\vphi)\widehat y - K_{11}^T (\vphi) \widehat w \\
 \Dom  \widehat y + \partial_{\phi\phi} K_{00}(\vphi)[\widehat \phi] + 
 \partial_\phi \partial_\alpha  K_{00}(\vphi)[\widehat \alpha] + 
[\partial_\phi K_{10}(\vphi)]^T \widehat y + 
[\partial_\phi  K_{01}(\vphi)]^T \widehat w   \\ 
\Dom  \widehat w - J 
\{ \partial_\phi K_{01}(\vphi)[\widehat \phi] + \partial_\alpha K_{01}(\vphi)[\widehat \alpha] + K_{11}(\vphi) \widehat y + K_{02}(\vphi) \widehat w \}
\end{pmatrix} \! .  \hspace{-5pt}
\end{equation}
 As in \cite{BBM-auto},  
 by \eqref{DGdelta}, \eqref{ansatz 0}, \eqref{2015-2}, 
 the induced composition operator satisfies:   
for all $ \widehat \imath := (\widehat \phi, \widehat y, \widehat w) $
\begin{gather} \label{DG delta}
\|DG_\delta(\vphi,0,0) [\widehat \imath] \|_s^{k_0, \gamma} + \|DG_\delta(\vphi,0,0)^{-1} [\widehat \imath] \|_s^{k_0, \gamma} 
\lesssim_s \| \widehat \imath \|_{s}^{k_0, \gamma} +  \| {\mathfrak I}_0 \|_{s + \s}^{k_0, \gamma}  \| \widehat \imath \|_{s_0}^{k_0, \gamma}\,,
\\ 
\| D^2 G_\delta(\vphi,0,0)[\widehat \imath_1, \widehat \imath_2] \|_s^{k_0, \gamma} 
\lesssim_s  \| \widehat \imath_1\|_s^{k_0, \gamma}  \| \widehat \imath_2 \|_{s_0}^{k_0, \gamma} 
+ \| \widehat \imath_1\|_{s_0}^{k_0, \gamma}  \| \widehat \imath_2 \|_{s}^{k_0, \gamma} 
+  \| {\mathfrak I}_0  \|_{s + \s}^{k_0, \gamma} \|\widehat \imath_1 \|_{s_0}^{k_0, \gamma}  \| \widehat \imath_2\|_{s_0}^{k_0, \gamma} \, .  \label{DG2 delta} 
\end{gather}
In order to construct an ``almost-approximate" inverse of \eqref{lin idelta} we need  that 
\be\label{Lomega def}
{\cal L}_\omega := \Pi_{\Splus}^\bot \big(\Dom   - J K_{02}(\vphi) \big)_{|{H_{\Splus}^\bot}} 
\ee
is ``almost-invertible" up to remainders of size $O(N_{n - 1}^{- \mathtt a})$ (see precisely \eqref{stima R omega corsivo}) where 
 \be\label{NnKn}
 N_n := K_n^p \, , \quad \forall n \geq 0\,,  
 \ee
 and  
\be\label{definizione Kn}
K_n := K_0^{\chi^{n}} \, , \quad \chi := 3/ 2 
\ee
are the scales used in the nonlinear Nash-Moser iteration in Section \ref{sec:NM}.  
The almost invertibility of $\mL_\om$ is proved in Theorem \ref{inversione parziale cal L omega} 
as the conclusion of the analysis 
of Sections \ref{linearizzato siti normali}-\ref{sec: reducibility}, 
and it is stated here as an assumption (to avoid the involved definition of the 
set $\tLm_o$).
Let $H^s_\bot(\T^{\nu + 1}) := H^s(\T^{\nu + 1}) \cap H_{\Splus}^\bot$ 
and recall that the phase space contains only functions even in $x$, see \eqref{phase space 2017}.
\begin{itemize}
\item {\bf Almost-invertibility of ${\cal L}_\omega$.} 
There exists a subset $ \tLm_o  \subset \mathtt{DC}(\gamma, \tau) \times [\h_1, \h_2] $ such that, 
for all $ (\omega, \h) \in  \tLm_o  $ the operator $ {\cal L}_\omega $ in \eqref{Lomega def} may be decomposed as
\be\label{inversion assumption}
{\cal L}_\omega  
= {\cal L}_\omega^< + {\cal R}_\omega + {\cal R}_\omega^\bot 
\ee
where $ {\cal L}_\omega^< $ is invertible. More precisely, there exist constants 
$  K_0, M, \sigma, \mu(\mathtt b), \mathtt a, p > 0$ such that for any $s_0 \leq s \leq S$,  the operators ${\cal R}_\omega$, ${\cal R}_\omega^\bot $ satisfy the estimates  
\begin{align}  \label{stima R omega corsivo}
\|{\cal R}_\omega h \|_s^{k_0, \gamma} & 
\lesssim_S  \e \gamma^{- 2(M+1)} N_{n - 1}^{- {\mathtt a}}\big( \|  h \|_{s + \sigma}^{k_0, \gamma} +  \| \fracchi_0 \|_ {s + \mu (\mathtt b)  + \sigma }^{k_0, \gamma} \| h \|_{s_0 + \sigma}^{k_0, \gamma} \big)\, , \\
\label{stima R omega bot corsivo bassa}
\| {\cal R}_\omega^\bot h \|_{s_0}^{k_0, \gamma} & \lesssim_S
K_n^{- b} \big( \| h \|_{s_0 + b 
+ \sigma}^{k_0, \gamma} + 
\| \fracchi_0 \|_ { s_0  + \mu (\mathtt b)  + \sigma +b }^{k_0, \gamma}  \|  h \|_{s_0 + \sigma}^{k_0, \gamma}\big)\,, 
\qquad \forall b > 0\,, \\
\label{stima R omega bot corsivo alta}
\| {\cal R}_\omega^\bot h \|_s^{k_0, \gamma} & \lesssim_S 
 \|  h \|_{s + \sigma}^{k_0, \gamma} + \| \fracchi_0 \|_ {s  + \mu (\mathtt b) + \sigma}^{k_0, \gamma} \| h \|_{s_0 + \sigma}^{k_0, \gamma} \,.
\end{align}
Moreover, for every function
$ g \in H^{s+\sigma}_{\bot}(\T^{\nu + 1}, \R^2) $ and such that $ g(-\vphi) = - \rho g( \vphi)$, for every $(\omega, \mathtt h) \in \Lambda_o$, there 
is a solution $ h :=  ({\cal L}_\om^<)^{- 1} g  \in H^{s}_{\bot}(\T^{\nu + 1}, \R^2) $ such that
$ h (-\vphi) =  \rho h ( \vphi) $,  
of the linear equation $ {\cal L}_\om^< h = g $. The operator $({\cal L}_\om^<)^{- 1}$ satisfies for all $s_0 \leq s \leq S$ the tame estimate 
\begin{equation}\label{tame inverse}
\| ({\cal L}_\om^<)^{- 1} g \|_s^{k_0, \gamma} \lesssim_S  \g^{-1} 
\big(  \| g \|_{s + \sigma}^{k_0, \gamma} + 
 \| {\mathfrak I}_0 \|_{s + \mu({\mathtt b}) + \sigma}^{k_0, \gamma}  \|g \|_{s_0 + \sigma}^{k_0, \gamma}  \big) \,.
\end{equation}
\end{itemize}

In order to find an almost-approximate inverse of the linear operator in \eqref{lin idelta} 
(and so of $ d_{i, \a} {\cal F}(i_\d) $),
it is sufficient to invert the operator
\begin{equation}\label{operatore inverso approssimato} 
{\mathbb D} [\widehat \phi, \widehat y, \widehat w, \widehat \alpha ] := 
  \begin{pmatrix}
\Dom \widehat \phi - \partial_\alpha K_{10}(\vphi)[\widehat \alpha] - 
K_{20}(\vphi) \widehat y  - K_{11}^T(\vphi) \widehat w\\
\Dom  \widehat y + \partial_\phi \partial_\alpha  K_{00}(\vphi)[\widehat \alpha] \\
(\mL_\omega^{<}) \widehat w  - J \partial_\alpha K_{01}(\vphi)[\widehat \alpha] -J K_{11}(\vphi)\widehat y  
\end{pmatrix}
\end{equation}
%The operator ${\mathbb D}$ in \eqref{operatore inverso approssimato} is 
obtained by neglecting in \eqref{lin idelta} 
the terms $ \partial_\phi K_{10} $, $ \partial_{\phi \phi} K_{00} $, $ \partial_\phi K_{00} $, 
$ \partial_\phi K_{01} $, which are $O(Z)$ by Lemma \ref{coefficienti nuovi}, 
and the small remainders ${\cal R}_\omega $, ${\cal R}_\omega^\bot $ 
appearing in \eqref{inversion assumption}. 
We look for an inverse of ${\mathbb D}$ % defined in \eqref{operatore inverso approssimato} 
by solving the system 
\begin{equation}\label{operatore inverso approssimato proiettato}
{\mathbb D} [\widehat \phi, \widehat y, \widehat w, \widehat \alpha] 
%= \begin{pmatrix}
%\omega \cdot \partial_\vphi \widehat \phi - \partial_\alpha K_{10}(\vphi)[\widehat \alpha] - 
%K_{20}(\vphi) \widehat y  - K_{11}^T(\vphi) \widehat w\\
%\omega \cdot \partial_\vphi \widehat y + \partial_\alpha \partial_\phi K_{00}(\vphi)[\widehat \alpha] \\
%({\cal L}_\omega^<) \widehat w  - J \partial_\alpha K_{01}(\vphi)[\widehat \alpha] -J K_{11}(\vphi)\widehat y  
%\end{pmatrix} 
= 
\begin{pmatrix}
g_1  \\
g_2  \\
g_3 
\end{pmatrix}
\end{equation}
where $(g_1, g_2, g_3)$ satisfy the reversibility property 
\begin{equation}\label{parita g1 g2 g3}
g_1(\vphi) = g_1(- \vphi)\,,\quad g_2(\vphi) = - g_2(- \vphi)\,,\quad g_3(\vphi) = - (\rho g_3)(- \vphi)\,.
\end{equation}
We first consider the second equation in \eqref{operatore inverso approssimato proiettato}, namely 
$ \omega \cdot \partial_\vphi  \widehat y  = g_2  -  \partial_\alpha \partial_\phi K_{00}(\vphi)[\widehat \alpha] $. 
By reversibility, the $\vphi$-average of the right hand side of this equation is zero, and so its solution is  
%the operator ${\cal D}_\omega^{(n)}$ is invertible, the only solution with zero average is 
\begin{equation}\label{soleta}
\widehat y := ( \omega \cdot \partial_\vphi )^{-1} \big(
g_2 - \partial_\alpha \partial_\phi K_{00}(\vphi)[\widehat \alpha] \big)\,.  
\end{equation}
Then we consider the third equation
$ ({\cal L}_\om^<) \widehat w = g_3 + J K_{11}(\vphi) \widehat y + J \partial_\alpha K_{0 1}(\vphi)[\widehat \alpha] $,
which, by the inversion assumption \eqref{tame inverse}, has a solution 
\begin{equation}\label{normalw}
\widehat w := ({\cal L}_\om^<)^{-1} \big( g_3 +J K_{11}(\vphi) \widehat y + J \partial_\alpha K_{0 1}(\vphi)[\widehat \alpha] \big) \, .  
\end{equation}
Finally, we solve the first equation in \eqref{operatore inverso approssimato proiettato}, 
which, substituting \eqref{soleta}, \eqref{normalw}, becomes
\begin{equation}\label{equazione psi hat}
\omega \cdot \partial_\vphi \widehat \phi  = 
g_1 +  M_1(\vphi)[\widehat \alpha] + M_2(\vphi) g_2 + M_3(\vphi) g_3\,,
\end{equation}
where
\begin{align}\label{M1}
& \qquad \qquad M_1(\vphi) :=  \partial_\alpha K_{10}(\vphi) - M_2(\vphi)\partial_\alpha \partial_\phi K_{00}(\vphi)  + M_3(\vphi) J \partial_\alpha K_{01}(\vphi)\,, \\
& \label{cal M2}
M_2(\vphi) :=  K_{20}(\vphi) [\omega \cdot \partial_\vphi]^{-1} + K_{11}^T(\vphi)({\cal L}_\om^<)^{- 1} J K_{11}(\vphi)[\omega \cdot \partial_\vphi]^{-1} \, , \quad 
M_3(\vphi) :=  K_{11}^T (\vphi) ({\cal L}_\om^<)^{-1} \, .  
\end{align}
In order to solve equation \eqref{equazione psi hat} we have 
to choose $ \widehat \alpha $ such that the right hand side  has zero average.  
By Lemma \ref{coefficienti nuovi}, \eqref{ansatz 0}, the $\ph$-averaged matrix is
$ \langle M_1 \rangle = {\rm Id} + O( \e \gamma^{-1 }) $.  
Therefore, for $ \e \gamma^{- 1} $ small enough,  
$ \langle M_1 \rangle$ is invertible and $\langle M_1 \rangle^{-1} = {\rm Id} 
+ O(\e \gamma^{- 1})$. Thus we define 
\begin{equation}\label{sol alpha}
\widehat \alpha  := - \langle M_1 \rangle^{-1} 
( \langle g_1 \rangle + \langle M_2 g_2 \rangle + \langle M_3 g_3 \rangle ) \, .
\end{equation}
With this choice of $ \widehat \alpha$,
equation \eqref{equazione psi hat} has the solution
\begin{equation}\label{sol psi}
\widehat \phi :=
(\omega \cdot \partial_\vphi )^{-1} \big( g_1 + M_1(\vphi)[\widehat \alpha] + M_2(\vphi) g_2 + M_3(\vphi) g_3 \big) \, . 
\end{equation}
In conclusion, we have obtained
a solution  $(\widehat \phi, \widehat y, \widehat w, \widehat \alpha)$ of the linear system \eqref{operatore inverso approssimato proiettato}. 

\begin{proposition}\label{prop: ai}
Assume \eqref{ansatz 0} (with $\mu = \mu(\mathtt b) + \sigma$) and \eqref{tame inverse}. 
Then, for all $(\om, \h) \in \tLm_o $, for all $ g := (g_1, g_2, g_3) $ even in $x$ and satisfying \eqref{parita g1 g2 g3},
 system \eqref{operatore inverso approssimato proiettato} has a solution 
$ {\mathbb D}^{-1} g := (\widehat \phi, \widehat y, \widehat w, \widehat \alpha ) $,
where $(\widehat \phi, \widehat y, \widehat w, \widehat \alpha)$ are defined in 
\eqref{sol psi}, \eqref{soleta},  \eqref{normalw}, \eqref{sol alpha}, which satisfies \eqref{parity solution} and for any $s_0 \leq s \leq S$
\begin{equation} \label{stima T 0 b}
\| {\mathbb D}^{-1} g \|_s^{k_0, \gamma}
\lesssim_S \gamma^{-1} \big( \| g \|_{s + \sigma }^{k_0, \gamma} 
+  \| {\mathfrak I}_0  \|_{s + \mu(\mathtt b) + \sigma}^{k_0, \gamma}
 \| g \|_{s_0 + \sigma}^{k_0, \gamma}  \big).
\end{equation}
\end{proposition}

\begin{proof}
The lemma follows by \eqref{normalw}, \eqref{M1}, \eqref{cal M2}, \eqref{sol alpha}, \eqref{sol psi}, Lemma \ref{coefficienti nuovi}, \eqref{tame inverse}, \eqref{ansatz 0}.
\end{proof}
Finally we prove that the operator 
\begin{equation}\label{definizione T} 
{\bf T}_0 := {\bf T}_0(i_0) := (D { \widetilde G}_\delta)(\vphi,0,0) \circ {\mathbb D}^{-1} \circ (D G_\delta) (\vphi,0,0)^{-1}
\end{equation}
is an almost-approximate right  inverse for $d_{i,\alpha} {\cal F}(i_0 )$ where
$ \widetilde{G}_\delta (\phi, y, w, \alpha) := $  $ \big( G_\delta (\phi, y, w), \alpha \big) $ 
 is the identity on the $ \alpha $-component. 
We denote the norm $ \| (\phi, y, w, \alpha) \|_s^{k_0, \gamma} := $ $  \max \{  \| (\phi, y, w) \|_s^{k_0, \gamma}, 
$ $ | \alpha |^{k_0, \gamma}  \} $.

\begin{theorem}  \label{thm:stima inverso approssimato}
{\bf (Almost-approximate inverse)}
Assume the inversion assumption  \eqref{inversion assumption}-\eqref{tame inverse}.
%and that $\e \gamma^{- 2M - 3}$ is small enough. 
Then, there exists $ \bar \sigma := \bar \sigma(\tau, \nu, k_0) > 0 $ such that, 
if \eqref{ansatz 0} holds with $\mu = \mu(\mathtt b) + \bar \sigma $, then for all $ (\om, \h) \in \tLm_o $, 
for all $ g := (g_1, g_2, g_3) $ even in $x$ and satisfying \eqref{parita g1 g2 g3},  
the operator $ {\bf T}_0 $ defined in \eqref{definizione T} satisfies,  for all $s_0 \leq s \leq S $, 
\begin{equation}\label{stima inverso approssimato 1}
\| {\bf T}_0 g \|_{s}^{k_0, \gamma} 
\lesssim_S  \gamma^{-1}  \big(\| g \|_{s + \bar \sigma}^{k_0, \gamma}  
+  \| {\mathfrak I}_0 \|_{s + \mu({\mathtt b}) +  \bar \sigma }^{k_0, \gamma}
\| g \|_{s_0 + \bar \sigma}^{k_0, \gamma}  \big)\, .
\end{equation}
 Moreover  ${\bf T}_0 $ is an almost-approximate inverse of $d_{i, \alpha} {\cal F}(i_0 )$, namely 
\begin{equation}\label{splitting per approximate inverse}
d_{i, \alpha} {\cal F}(i_0) \circ {\bf T}_0 - {\rm Id} = {\cal P} (i_0) + {\cal P}_\omega (i_0) + {\cal P}_\omega^\bot  (i_0)
\end{equation}
where, for all $ s_0 \leq s \leq S $, 
\begin{align}
 \| {\cal P} g \|_s^{k_0, \gamma} &
 \lesssim_S  \g^{-1} \Big( \| {\cal F}(i_0, \alpha_0) \|_{s_0 + \bar \sigma}^{k_0, \gamma} \| g \|_{s + \bar \sigma}^{k_0, \gamma}  
\nonumber\\
& \quad + \big\{ \| {\cal F}(i_0, \alpha_0) \|_{s + \bar \sigma}^{k_0, \gamma}
+\| {\cal F}(i_0, \alpha_0) \|_{s_0  + \bar \sigma}^{k_0, \gamma} \| {\mathfrak I}_0 \|_{s + \mu({\mathtt b}) + \bar \sigma}^{k_0, \gamma} \big\} \| g \|_{s_0 + \bar \sigma}^{k_0, \gamma} \Big) ,
\label{stima inverso approssimato 2} \\
\| {\cal P}_\omega g \|_s^{k_0, \gamma} & 
\lesssim_S  \e \g^{-2M - 3} N_{n - 1}^{- {\mathtt a}} \big( \| g \|_{s + \bar \sigma}^{k_0, \gamma} +\| \fracchi_0 \|_{s  + \mu({\mathtt b}) + \bar \sigma}^{k_0, \gamma}  \|  g \|_{s_0 + \bar \sigma}^{k_0, \gamma} \big)\,, \label{stima cal G omega} \\
\| {\cal P}_\omega^\bot g\|_{s_0}^{k_0, \gamma} & \lesssim_{S, b} 
\gamma^{- 1} K_n^{- b } \big( \| g \|_{s_0 + \bar \sigma + b }^{k_0, \gamma} +
\| \fracchi_0 \|_{s_0 +  \mu({\mathtt b}) + \bar \sigma  +b    }^{k_0, \gamma} \big \| g \|_{s_0 + \bar \sigma}^{k_0, \gamma} \big)\,,\,\,
\quad\forall b > 0 \, ,   \label{stima cal G omega bot bassa} \\
\| {\cal P}_\omega^\bot g\|_s^{k_0, \gamma} &  \lesssim_S \gamma^{- 1}  
\big(\| g \|_{s + \bar \sigma}^{k_0, \gamma} + \| \fracchi_0 \|_{s + \mu({\mathtt b}) + \bar \sigma}^{k_0, \gamma}  \| g \|_{s_0 + \bar \sigma}^{k_0, \gamma} \big)\, . \label{stima cal G omega bot alta} 
\end{align}
\end{theorem}
%Note that ${\cal P}$ vanishes if ${\cal F}(i_0, \alp)$
\begin{proof}
Bound \eqref{stima inverso approssimato 1} 
follows from \eqref{definizione T}, \eqref{stima T 0 b},  \eqref{DG delta}.
By \eqref{operatorF}, since $ X_\mN $ does not depend on $ I $,  
and $ i_\d $ differs by $ i_0 $ only in the $ I $ component (see \eqref{y 0 - y delta}),  
we have \,
\be \label{verona 0}
{\cal E}_0 := d_{i, \alpha} {\cal F}(i_0 )   - d_{i, \alpha} {\cal F}(i_\delta ) 
= \e \int_0^1 \partial_I d_i X_P (\theta_0, I_\d + s (I_0 - I_\d), z_0) [I_0-I_\d,   \Pi [ \, \cdot \, ] \,  ] ds 
\ee
where $ \Pi $ is the projection $ (\widehat \imath, \widehat \a ) \mapsto \widehat \imath $. Denote by $  {\mathtt u} := (\phi, y, w) $  the symplectic coordinates induced by $ G_\d $ in \eqref{trasformazione modificata simplettica}. 
Under the symplectic map $G_\delta $, the nonlinear operator ${\cal F}$ in \eqref{operatorF} is transformed into 
\be \label{trasfo imp}
{\cal F}(G_\delta(  {\mathtt u} (\vphi) ), \alpha ) 
= D G_\delta( {\mathtt u}  (\vphi) ) \big(  {\cal D}_\om {\mathtt u} (\vphi) - X_{K_\alpha} ( {\mathtt u} (\vphi), \alpha)  \big) 
\ee
where $ K_{\alpha} = H_{\alpha} \circ G_\delta $, see \eqref{new-Hamilt-K} and 
\eqref{sistema dopo trasformazione inverso approssimato}. 
Differentiating  \eqref{trasfo imp} at the trivial torus 
$ {\mathtt u}_\delta (\vphi) = G_\delta^{-1}(i_\delta) (\vphi) = (\ph, 0 , 0 ) $, 
at  $ \alpha = \alpha_0 $, 
we get
\begin{align} \label{verona 2}
d_{i , \alpha} {\cal F}(i_\delta ) 
=  & D G_\delta( {\mathtt u}_\delta) 
\big( \Dom 
- d_{\mathtt u, \alpha} X_{K_\alpha}( {\mathtt u}_\delta, \alpha_0) 
\big) D \widetilde G_\d ( {\mathtt u}_\d)^{-1}
+ {\cal E}_1 \,,
\\
{\cal E}_1 
:=  & 
D^2 G_\delta( {\mathtt u}_\delta) \big[ D G_\delta( {\mathtt u}_\delta)^{-1} {\cal F}(i_\delta, \alpha_0), \,  D G_\d({\mathtt u}_\d)^{-1} 
 \Pi [ \, \cdot \, ] \,  \big] \,   \label{verona 2 0} 
\end{align}
In expanded form $ \Dom  - d_{\mathtt u, \alpha} X_{K_\a}( {\mathtt u}_\delta, \alpha_0) $ is provided by \eqref{lin idelta}.
By \eqref{operatore inverso approssimato}, \eqref{Lomega def}, 
\eqref{inversion assumption} and Lemma \ref{coefficienti nuovi} we split  
\begin{equation}\label{splitting linearizzato nuove coordinate}
\om \! \cdot \! \pa_\vphi 
- d_{\mathtt u, \alpha} X_K( {\mathtt u}_\delta, \alpha_0) 
= \mathbb{D}  
+ R_Z   + {\mathbb R}_\omega+ 
{\mathbb R}_\omega^\bot  
\end{equation}
where  
$$
R_Z [  \widehat \phi, \widehat y, \widehat w, \widehat \alpha]
:= \begin{pmatrix}
 - \partial_\phi K_{10}(\vphi, \alpha_0) [\widehat \phi ] \\
 \partial_{\phi \phi} K_{00} (\vphi, \alpha_0) [ \widehat \phi ] + 
 [\partial_\phi K_{10}(\vphi, \alpha_0)]^T \widehat y + 
 [\partial_\phi K_{01}(\vphi, \alpha_0)]^T \widehat w  \\
 - J \{ \partial_{\phi} K_{01}(\vphi, \alpha_0)[ \widehat \phi ] \}
 \end{pmatrix}\,,
$$
and 
$$
{\mathbb R}_\omega[\widehat \phi, \widehat y, \widehat w, \widehat \alpha] := \begin{pmatrix}
0 \\
0 \\
{\cal R}_\omega [\widehat w]
\end{pmatrix}\,,\qquad {\mathbb R}_\omega^\bot[\widehat \phi, \widehat y , \widehat w, \widehat \alpha] := \begin{pmatrix}
0 \\
0 \\
{\cal R}_\omega^\bot[\widehat w]
\end{pmatrix}\,.
$$
By \eqref{verona 0}, \eqref{verona 2}, \eqref{verona 2 0}, \eqref{splitting linearizzato nuove coordinate} we get the decomposition
\begin{align} 
 d_{i, \alpha} {\cal F}(i_0 ) 
& = D G_\delta({\mathtt u}_\delta) \circ {\mathbb D} \circ D {\widetilde G}_\delta ({\mathtt u}_\delta)^{-1} + {\cal E} + {\cal E}_\omega + {\cal E}_\omega^\bot  \label{E2}
 \end{align}
where 
\begin{equation}\label{cal E (n) omega}
{\cal E} := {\cal E}_0 + {\cal E}_1 + D G_\delta ( {\mathtt u}_\delta)R_Z D {\widetilde G}_\delta ({\mathtt u}_\delta)^{-1}\,, \quad {\cal E}_\omega := D G_\delta ( {\mathtt u}_\delta)   {\mathbb R}_\omega    D {\widetilde G}_\delta ({\mathtt u}_\delta)^{-1}\,,
\end{equation}
\begin{equation}\label{cal E omega bot}
   {\cal E}_\omega^\bot :=   D G_\delta( {\mathtt u}_\delta)  {\mathbb R}_\omega^\bot  
D {\widetilde G}_\delta ({\mathtt u}_\delta)^{-1} \, . 
\end{equation}
Applying $ {\bf T}_0 $ defined in \eqref{definizione T} to the right hand side in \eqref{E2} (recall that $ {\mathtt u}_\delta (\vphi) := (\vphi, 0, 0 ) $), 
since $ {\mathbb D} \circ  {\mathbb D}^{-1} = {\rm Id} $ (Proposition \ref{prop: ai}), 
we get 
\begin{align*}
& \qquad \qquad \quad d_{i, \alpha} {\cal F}(i_0 ) \circ {\bf T}_0  - {\rm Id} 
={\cal P} + {\cal P}_\omega + {\cal P}_\omega^\bot\,, \quad {\cal P} := {\cal E} \circ {\bf T}_0, \quad
 {\cal P}_\omega := \mE_\omega \circ {\bf T}_0 \, , \quad {\cal P}_\omega^\bot :=  \mE_\omega^\bot \circ {\bf T}_0 \, . 
\end{align*}
By \eqref{ansatz 0}, \eqref{K 00 10 01}, \eqref{2015-2},  \eqref{stima y - y delta}, \eqref{stima toro modificato}, 
 \eqref{DG delta}-\eqref{DG2 delta} we get the estimate 
\begin{equation}\label{stima parte trascurata 2}
\| {\cal E} [\, \widehat \imath, \widehat \alpha \, ] \|_s^{k_0, \gamma} \lesssim_s 
 \| Z \|_{s_0 + \s}^{k_0, \gamma} \| \widehat \imath \|_{s + \s}^{k_0, \gamma} +  
\| Z \|_{s + \s}^{k_0, \gamma} \| \widehat \imath \|_{s_0 + \s}^{k_0, \gamma} + 
\| Z \|_{s_0 + \s}^{k_0, \gamma} \| \widehat \imath \|_{s_0 + \s}^{k_0, \gamma} \| \fracchi_0 \|_{s+\s}^{k_0, \gamma} \, ,
\end{equation}
where $ Z := \mF(i_0, \alpha_0)$, recall \eqref{def Zetone}. 
Then \eqref{stima inverso approssimato 2} follows from 
\eqref{stima inverso approssimato 1}, \eqref{stima parte trascurata 2}, 
\eqref{ansatz 0}. Estimates \eqref{stima cal G omega}, 
\eqref{stima cal G omega bot bassa}, \eqref{stima cal G omega bot alta}
 follow by \eqref{stima R omega corsivo}-\eqref{stima R omega bot corsivo alta}, 
\eqref{stima inverso approssimato 1},  \eqref{DG delta}, \eqref{2015-2}, \eqref{ansatz 0}.  
 \end{proof}

\section{The linearized operator in the normal directions}\label{linearizzato siti normali}

In order to write an explicit  expression of the linear operator 
$\mL_\om$ defined in \eqref{Lomega def}
we have to express the operator $ K_{02}(\phi) $ 
in terms of the original water waves Hamiltonian vector field.  

\begin{lemma} \label{thm:Lin+FBR}
The operator  $ K_{02}(\phi) $ is 
\begin{equation}\label{K 02}
K_{02}(\phi) = \Pi_{\Splus}^\bot \partial_u \nabla_u H(T_\delta(\phi)) + \e R(\phi) 
\end{equation}
where $ H $ is the water waves Hamiltonian defined in \eqref{Hamiltonian} (with gravity constant $ g =1 $ and depth 
$ h $ replaced by $ \mathtt h $), evaluated at the torus 
\begin{equation}\label{T delta}
T_\delta(\phi) := \e A(i_\delta(\phi)) = \e A(\theta_0(\phi), I_\delta(\phi), z_0(\phi) ) =
\e  v (\theta_0 (\phi), I_\d(\phi)) +  \e z_0 (\phi) 
\end{equation}
with $ A (\theta, I, z ) $, $ v (\teta, I )$ defined in \eqref{definizione A}.
The operator $ K_{02}(\phi) $ is even and reversible.  
The remainder $ R(\phi) $ has the ``finite dimensional" form 
\begin{equation}\label{forma buona resto}
R(\phi)[h] = {\mathop \sum}_{j \in \Splus} \big(h\,,\,g_j \big)_{L^2_x} \chi_j\,, \quad \forall h \in H_{\Splus}^\bot \, ,  
\end{equation}
for functions $ g_j, \chi_j \in H_{\Splus}^\bot  $ which satisfy the tame estimates: 
for some $ \sigma:= \sigma(\tau, \nu) > 0 $, 
$ \forall s \geq s_0 $, 
\begin{equation}\label{stime gj chij}
\| g_j\|_s^{k_0, \gamma} +\| \chi_j\|_s^{k_0, \gamma} \lesssim_s 1 + \| \fracchi_\delta\|_{s + \s}^{k_0, \gamma}\,,\quad \| d_i g_j[\widehat \imath]\|_s +\| d_i \chi_j
[\widehat \imath]\|_s \lesssim_s \| \widehat \imath \|_{s + \s}+ \| \fracchi_\delta\|_{s + \s} \| \widehat \imath\|_{s_0 + \s}\, . 
\end{equation}
\end{lemma}

\begin{proof}
The lemma follows as in Lemma 6.1 in \cite{BertiMontalto}.
\end{proof}

By Lemma \ref{thm:Lin+FBR} the linear operator $ {\cal L}_\om $ defined in \eqref{Lomega def} has the form 
\be\label{representation Lom}
{\cal L}_\om =  \Pi_{\Splus}^\bot ( {\cal L} + \e R)_{| H_{\Splus}^\bot}  \qquad {\rm where} \qquad {\cal L} := 
\Dom   - J \partial_u  \nabla_u 
H (T_\delta(\vphi))
\ee
is obtained linearizing the original water waves system \eqref{WW0}, \eqref{HS} at the torus
$ u = (\eta, \psi) = T_\d(\vphi) $ 
defined in \eqref{T delta}, changing $ \pa_t \rightsquigarrow \Dom $.
The function $ \eta (\vphi, x) $ is $ \even (\vphi) \even (x) $ and $ \psi (\vphi, x) $ is $ \odd (\vphi) \even (x) $.  

In order to compute the linearization of the Dirichlet-Neumann operator, we recall the ``shape derivative'' formula, given for instance in \cite{Lannes}, \cite{LannesLivre},
\begin{equation} \label{formula shape der}
G'(\eta) [\hat {\eta} ] \psi 
= \lim_{\ep \rightarrow 0} \frac{1}{\ep} \{ G (\eta+\ep \hat \eta ) \psi - G(\eta) \psi \}
= - G(\eta) (B \hat \eta ) -\partial_x (V \hat \eta )
\end{equation}
where 
\begin{equation} \label{def B V}
\B := \B(\eta,\psi) := \frac{\eta_x \psi_x + G(\eta)\psi }{ 1 + \eta_x^2 }\,, 
\qquad 
V := V(\eta,\psi) := \psi_x - B \eta_x \, .
\end{equation}
It turns out that $ (V,B) = \nabla_{x,y} \Phi $ is the velocity field 
evaluated at the free surface  $ (x,  \eta (x)) $. 
Using \eqref{formula shape der}, 
the linearized operator of \eqref{WW0} is represented by the $2 \times 2$ operator matrix
\begin{equation}  \label{linearized vero}
\mL := \om \cdot \pa_\vphi + \begin{pmatrix} 
\partial_x V + G(\eta) B & - G(\eta) \\
(1 + B V_x) + B G(\eta) B \  &  V \partial_x - B G(\eta) \end{pmatrix} .
\end{equation}
Since the operator $G(\eta)$ is even according to Definition \ref{def:even}, the function $ B $ is $ \odd (\vphi) \even (x) $ and
$ V $ is $ \odd (\vphi) \odd (x) $. 
The operator $ {\cal L}$ acts on $ H^1(\T) \times H^1(\T)$.

The operators $ {\cal L}_\om $ and $ {\cal L} $ are real, even and reversible. We are going to make several transformations, whose aim is to conjugate the linearized operator to a constant coefficients operator, up to a remainder that is small in size and regularizing at a conveniently high order.

\begin{remark}\label{remark riccardo inverse}
 It is convenient to first ignore the projection $\Pi_{\mathbb S_+}^\bot$ and consider the linearized operator ${\cal L}$ acting on the whole space $H^1(\T) \times H^1(\T)$. At the end of the conjugation procedure, we shall restrict ourselves to the phase space $H^1_0(\T) \times \dot{H}^1(\T)$ and perform the projection on the normal subspace $ H_{\Splus}^\bot $, see Section \ref{coniugio cal L omega}. 
The finite dimensional remainder $ \e  R $
transforms under conjugation into an operator of the same form 
and therefore it will be dealt with only once at the end of Section \ref{coniugio cal L omega}. 
\end{remark}

For the sequel we will always assume the following ansatz (satisfied by the approximate solutions
obtained along the nonlinear Nash-Moser iteration of Section \ref{sec:NM}): 
for some constant $ \mu_0 := \mu_0 (\tau, \nu) > 0$, $\gamma \in (0, 1) $,  
\begin{equation}\label{ansatz I delta}
\| \fracchi_0 \|_{s_0 + \mu_0}^{k_0, \gamma} \leq 1   \ \, , 
\qquad \text{and  so, by } \eqref{2015-2}, \ \| \fracchi_\d \|_{s_0 + \mu_0}^{k_0, \gamma} \leq 2 \, . 
\end{equation}
In order to estimate the  variation of the eigenvalues with respect to the approximate invariant torus, 
we need also to estimate the derivatives (or the variation) with respect to the torus $i(\vphi)$ in another low norm $\| \ \|_{s_1}$, for all the Sobolev indices $s_1$ such that  
\begin{equation}\label{vincolo s1 derivate i}
s_1 + \sigma_0 \leq s_0 + \mu_0 \,, \quad \text{for \,\,some} \quad \sigma_0 := \sigma_0(\tau, \nu)> 0\,.
\end{equation}
Thus by \eqref{ansatz I delta} we have 
\be\label{estimate:low norm for der}
\|\fracchi_0 \|_{s_1 + \sigma_0}^{k_0, \gamma} \leq 1 
\qquad \text{and  so, by } \eqref{2015-2}, \ 
\|\fracchi_\d \|_{s_1 + \sigma_0}^{k_0, \gamma} \leq 2 \,.
\ee
 The constants $\mu_0$ and $\sigma_0$ represent the {\it loss of derivatives} 
 accumulated along the reduction procedure 
of  Sections \ref{sec: change-transport equation}-\ref{sezione descent method}.
What is important is that they are independent of the Sobolev index $ s $. Along Sections \ref{linearizzato siti normali}-\ref{sezione descent method}, we shall denote by $\sigma := \sigma(k_0, \tau, \nu) > 0$ a constant (which possibly increases from lemma to lemma) representing the loss of derivatives along the finitely many steps of the reduction procedure. 

As a consequence of Moser composition Lemma \ref{Moser norme pesate}, 
the Sobolev norm of the function $ u = T_\d $ defined in \eqref{T delta} satisfies, $ \forall s \geq s_0 $,  
\be\label{tame Tdelta}
\| u \|_s^{k_0,\gamma} = \| \eta \|_s^{k_0,\gamma}  + 
\| \psi \|_s^{k_0,\gamma} \leq \e C(s)  \big(  1 + \| \fracchi_0 \|_{s}^{k_0, \gamma})  
\ee
(the function $ A $  defined in \eqref{definizione A} is smooth). Similarly 
\begin{equation}\label{derivata i T delta}
\| \Delta_{12} u \|_{s_1} \lesssim_{s_1} \e \| i_2 - i_1 \|_{s_1} 
\end{equation}
where we denote $ \Delta_{12} u:= u(i_2) - u(i_1)$; we will systematically use this notation.

In the next sections we shall also assume
that, for some $ \kappa := \kappa (\tau, \nu) > 0 $, we have  
$$
\e \gamma^{- \kappa} \leq \delta (S)  \, ,
$$
where $ \delta (S) > 0 $ is a constant small enough and  $ S $ will be  fixed in 
\eqref{costanti nash moser 2}.
We recall that $ \fracchi_0 := \fracchi_0 (\om, \h )  $ 
is defined for all  $ (\om, \h) \in \R^\nu \times [\h_1, \h_2] $ 
and that  the functions $ B, V $ 
appearing  in  $ {\cal L } $ in \eqref{linearized vero} 
are $ {\cal C}^\infty $ in $ (\vphi, x) $ as 
the approximate torus 
$ u = (\eta, \psi) =  T_\d (\vphi)  $. 
This enables to  use directly pseudo-differential operator theory as reminded in Section \ref{sec:pseudo}.

Starting from here,  until the end of Section \ref{coniugio cal L omega}, 
our goal is to prove Proposition \ref{prop: sintesi linearized}.

\subsection{Linearized  good unknown of Alinhac}\label{sez:Alinhac}

Following \cite{Alaz-Bal}, \cite{BertiMontalto} we conjugate  the linearized operator $ \mL $ in \eqref{linearized vero} by the multiplication operator
\be\label{mapZ}
\mZ := \begin{pmatrix}  1 & 0 \\ B & 1  \end{pmatrix},
\quad 
\mZ^{-1} = \begin{pmatrix} 1 & 0 \\ - B & 1  \end{pmatrix} \, , 
\ee
where $ B = B (\vphi, x)$ is the function defined in \eqref{def B V}, 
obtaining  
\begin{equation} \label{mZ mL0}
\mL_0 := \mZ^{-1} \mL \mZ 
= \om \cdot \pa_\vphi + \begin{pmatrix}
\partial_x V \  & \ - G(\eta) \\ 
a \ &  V \partial_x 
\end{pmatrix}
\end{equation} 
where  $a$ is the function
\begin{equation}  \label{a}
a := a(\vphi, x) := 1 + (\omega \cdot \partial_\vphi B) + V B_x \, .
\end{equation} 
All $a,B,V$ are real valued periodic functions of $(\ph,x)$ --- variable coefficients --- and
satisfy 
\[
B = \odd(\ph) \even(x), \quad 
V = \odd(\ph) \odd(x), \quad 
a = \even(\ph) \even(x) \, . 
\]
The matrix $ \mZ $ in \eqref{mapZ} amounts to introduce, as in Lannes \cite{Lannes}-\cite{LannesLivre},  
a linearized version of the \emph{good unknown of Alinhac}, working 
with the variables  $ (\eta, \varsigma ) $ with $ \varsigma := \psi - B \eta $, instead of $ (\eta, \psi) $.

\begin{lemma}  \label{lemma:remainder mR0}
The maps $\mZ^{\pm 1} -  {\rm Id} $ are even, reversibility preserving and ${\cal D}^{k_0}$-tame with tame constants satisfying, for all $s \geq s_0$,  
\begin{equation}  \label{est Z-Id}
{\mathfrak M}_{\mZ^{\pm 1} - {\rm Id}} (s)\,,\,  {\mathfrak M}_{(\mZ^{\pm 1} - {\rm Id})^*} (s)
\lesssim_s \e \big( 1+ \| \fracchi_0 \|_{s + \sigma}^{k_0, \gamma} \big) \, . 
\end{equation}
The operator $ {\cal L}_0 $ is even and reversible. 
There is $\sigma := \sigma(\tau, \nu) > 0 $ such that    the functions 
\begin{equation}\label{stima V B a c}
\|a - 1 \|_s^{k_0, \gamma} + \| V \|_s^{k_0, \gamma} + \| B\|_s^{k_0, \gamma} \lesssim_s \e \big(1 + \| \fracchi_0 \|_{s + \sigma}^{k_0, \gamma} \big)\,.
\end{equation}
Moreover 
\begin{align}\label{stima derivate i primo step}
& \| \Delta_{12} a  \|_{s_1} + \| \Delta_{12} V  \|_{s_1} + \| \Delta_{12} B  \|_{s_1} \lesssim_{s_1}  \e \| i_1 - i_2 \|_{s_1 + \sigma} \\
& \label{derivate in i cal Z}
\| \Delta_{12} ({\cal Z}^{\pm 1} ) h \|_{s_1}\,,\,\| \Delta_{12} ({\cal Z}^{\pm 1})^* h \|_{s_1} \lesssim_{s_1}  \e \| i_1 - i_2\|_{s_1 + \sigma} \|  h \|_{s_1 }\,.
\end{align}
\end{lemma}

\begin{proof}
The proof is the same as the one of Lemma 6.3 in \cite{BertiMontalto}. 
\end{proof}

We  expand $ {\cal L}_0 $ in \eqref{mZ mL0} as 
\begin{equation} \label{mZ mL0-new}
\mL_0 
= \om \cdot \pa_\vphi + \begin{pmatrix}
V \pa_x  \  & 0 \\ 
0 &  V \partial_x 
\end{pmatrix} +
\begin{pmatrix}
V_x \  & \ - G(\eta) \\ 
a \ &  0
\end{pmatrix}  \, . 
\end{equation} 
In the next section we deal with the first order operator  $\ompaph + V  \pa_x $.

\section{Straightening the first order vector field}\label{sec: change-transport equation}

The aim of this section 
is to conjugate the variable coefficients operator $\ompaph + V(\ph,x) \pa_x $ to the constant coefficients vector field $\ompaph$, namely to find a change of variable $ {\cal B} $ such that 
\be\label{goal:strait}
{\cal B}^{-1} \big( \ompaph + V(\ph,x) \pa_x \big) {\cal B} = \ompaph \, . 
\ee
{\bf Quasi-periodic transport equation.}
We consider a $ \vphi $-dependent family of diffeomorphisms of $ \T_x $ of the space variable $y = x +\b(\ph,x)$ 
where the function $\b : \T^{\nu}_{\vphi} \times \T_x \to \R$ is odd in $x$, even in $\ph$, 
and $\| \b_x \|_{L^\infty}<1/2$. 
We denote by $\mB$ the corresponding composition operator, namely $(\mB h)(\ph,x) := h(\ph, x + \b (\ph,x)) $. 
The conjugated operator in the left hand side in \eqref{goal:strait} is 
\begin{equation} \label{2702.2} 
\mB^{-1} \big( \ompaph + V(\ph,x) \partial_x \big) \mB  
= \ompaph + c(\ph,y) \, \partial_y
\end{equation}
where 
\begin{equation} \label{2702.3} 
c(\ph,y) := \mB^{-1} \big(\ompaph \b + V (1 + \b_x ) \big) (\ph,y) \, .
\end{equation}
In view of \eqref{2702.2}-\eqref{2702.3}
we obtain \eqref{goal:strait} if  
$ \b(\ph,x) $  solves  the equation
\be\label{transport}
\om \cdot \partial_\vphi \beta (\ph,x) + V(\vphi,x) ( 1 + \beta_x (\ph,x) ) = 0 \, ,
\ee
which can be interpreted as a {\it quasi-periodic transport} equation. 
\\[1mm]
{\bf Quasi-periodic characteristic equation.}
Instead of solving directly \eqref{transport} we solve the equation satisfied 
by the inverse diffeomorphism 
\be \label{1202.10}
x + \b(\ph,x) = y \quad \iff \quad x = y + \breve\b (\ph,y) \, , \
\quad  \forall x,y \in \R, \ \ph \in \T^\nu \, .
\ee
It turns out that equation \eqref{transport} for $\beta(\ph,x)$ is equivalent to the following equation 
for $\breve\b(\ph,y)$:
\be \label{1202.11}
\ompaph \breve\b(\ph,y) = V(\ph, y + \breve\b(\ph,y)) 
\ee
which is a quasi-periodic version of the {\it characteristic} equation $ \dot x = V(\omega t, x ) $. 

\begin{remark}\label{rem:geo}
We can give a geometric interpretation of equation \eqref{1202.11} 
in terms of conjugation of vector fields on the torus $ \T^{\nu} \times \T $.
Under the diffeomorphism of $ \T^{\nu} \times \T $ defined by 
$$
\begin{pmatrix} 
\vphi \\ 
x
\end{pmatrix} = \begin{pmatrix} 
\psi \\ 
y + \breve \b (\psi, y)
\end{pmatrix} \, , \quad 
\text{the system}  \quad 
\frac{d}{dt}   \begin{pmatrix} 
 \vphi \\ 
x
\end{pmatrix} = 
\begin{pmatrix} 
 \omega \\ 
V(\vphi, x)
\end{pmatrix}
$$
transforms into  
$$
\frac{d}{dt}   \begin{pmatrix} 
 \psi \\ 
y
\end{pmatrix} = 
\begin{pmatrix} 
 \omega \\ 
 \big\{ - \ompaph \breve\b(\psi,y) + V(\ph, y + \breve\b(\psi,y)) \big\} \big( 1 + \breve \beta_y (\psi, y) \big)^{-1} 
\end{pmatrix} \, .
$$
The vector field  in the new coordinates 
reduces to $ (\omega, 0 ) $ if and only if 
 \eqref{1202.11} holds. 
 In the new variables  the solutions are simply given by $ y(t) = c  $, $ c \in \R $, and 
  {\it all} the solutions of the scalar quasi-periodically forced differential equation 
 $  \dot x = V(\om t , x )  $ are time quasi-periodic 
of  the form $ x(t) = c + \breve \beta (\omega t, c ) $. 
\end{remark}

In Theorem \ref{thm:nonlinear transport} we solve equation \eqref{1202.11}, 
for $ V(\vphi, x) $ small  and   $ \omega $ 
Diophantine, 
%n  the set of Diophantine vectors $ \mathtt{DC}(\gamma, \tau)$ defined in \eqref{DC tau0 gamma0}, 
 by  applying the Nash-Moser-H\"ormander implicit function theorem in Appendix \ref{sec:NMH}.  
Rename $\breve\beta \to u$, $y \to x$, and write \eqref{1202.11} as 
\be \label{1202.12}
F(u)(\ph,x) := \ompaph u(\ph,x) - V(\ph,x+u(\ph,x)) = 0 \, .
\ee
The linearized operator at a given function $ u (\vphi, x ) $ is 
\be \label{1202.16}
F'(u)h:= \ompaph h - q(\ph,x) h, \quad 
q(\ph,x) := V_x(\ph, x + u(\ph,x)) \, .
\ee
In the next lemma we solve the linear problem $ F'(u)h = f $.

\begin{lemma}{\bf (Linearized quasi-periodic characteristic equation)}
\label{lemma:inv transport}
Let $\varsigma := 3k_0 + 2\t(k_0+1) +2 = 2\mu + k_0 + 2$, where $\mu$ is the loss in \eqref{Diophantine-1} (with $k+1 = k_0$), 
and let $\omega \in \mathtt{DC}(2 \gamma, \tau)$. 
Assume that  the periodic function $u $ is $ \even(\ph) \odd(x)$, that 
$ V $ is $ \odd(\ph) \odd(x) $, 
and 
\begin{equation}\label{2202.1}
\|u \|_{s_0 + \varsigma}^{k_0, \gamma} + \g^{-1} \|V \|_{s_0+\varsigma}^{k_0, \gamma} \leq \d_0
\end{equation}
with $\d_0$ small enough. 
Then, given a periodic function $f $ which is $ \odd(\ph) \odd(x)$, the linearized equation
\be \label{1202.13}
F'(u)h = f 
\ee
has a unique periodic solution $h(\ph,x)$ which is $\even(\ph) \odd(x)$ having zero average in $\ph$, i.e. \begin{equation}\label{trasporto media zero}
\langle h \rangle_\ph (x) := \frac{1}{(2\pi)^\nu} \int_{\T^\nu} h(\ph,x)\, d\ph = 0  
\quad \forall x\in\T.
\end{equation}
This defines a right inverse of the linearized operator $ F' (u) $, which we denote by $h = F'(u)^{-1} f$. 
%The right inverse $F'(u)^{-1}$
It satisfies 
\be \label{1202.14}
\| F'(u)^{-1} f  \|_s^{k_0, \gamma} \lesssim_s 
\g^{-1} \big( \|f \|_{s+\varsigma}^{k_0, \gamma} + \g^{-1} 
(\|V \|_{s+\varsigma}^{k_0, \gamma} + \|u \|_{s+\varsigma}^{k_0, \gamma} \|V \|_{s_0+\varsigma}^{k_0, \gamma} ) \|f \|_{s_0}^{k_0 , \gamma}  \big)
\ee
for all $s \geq s_0$, where $\| \cdot  \|_s^{k_0, \gamma}$ denotes the norm of 
$\Lip(k_0,\mathtt{DC}(2 \gamma, \tau),s,\g)$. 
\end{lemma}

\begin{proof}
Given $f$, we have to solve the linear equation 
$\ompaph h - q h = f$, 
where $q$ is the function defined  in \eqref{1202.16}. 
From the parity of $u,V$ it follows that $q$ is $\odd(\ph) \even(x)$.
By variation of constants, we look for solutions of the form $h = w e^v$, and we find (recalling  \eqref{def:ompaph})
$$
v := (\ompaph)^{-1} q, \quad 
w := w_0 + g, \quad
w_0:= (\ompaph)^{-1} (e^{-v} f), \quad
g=g(x) := - \frac{\langle w_0 e^v \rangle_\ph}{\langle e^v \rangle_\ph} \,.
$$
This choice of $g$, and hence of $w$, is the only one matching the zero average requirement \eqref{trasporto media zero}; this gives uniqueness of the solution.
Moreover $v = \even(\ph) \even(x) \, , w_0 = \even(\ph)\odd(x) \, ,  g = \odd(x)$,  
whence 
$h $ is $ \even(\ph) \odd(x)$. 
Using \eqref{p1-pr}, \eqref{pr-comp1}, \eqref{Diophantine-1}, \eqref{0811.10}, \eqref{2202.1}, and \eqref{2202.2} the proof of \eqref{1202.14} is complete.
\end{proof}

We now prove the existence of a solution of equation \eqref{1202.12}.

\begin{theorem}{\bf (Solution of the quasi-periodic characteristic equation \eqref{1202.12})}
\label{thm:nonlinear transport}
Let $\varsigma$ be the constant defined in Lemma \ref{lemma:inv transport}, 
and let $s_2 := 2s_0 + 3\varsigma + 1$, $p := 3\varsigma + 2$.
Assume that $V $ is $ \odd(\ph) \odd(x)$. 
There exist $\d \in (0,1), C > 0$ depending on $\varsigma,s_0$ such that, 
for all $\omega \in \mathtt{DC}(2 \gamma, \tau)$, 
if $V \in \Lip(k_0,\mathtt{DC}(2 \gamma, \tau), s_2 + p, \g)$ 
satisfies 
\begin{equation} \label{0811.2}
\g^{-1} \| V \|_{s_2 + p}^{k_0,\g} \leq \d,
\end{equation}
then there exists a solution $u \in \Lip(k_0,\mathtt{DC}(2 \gamma, \tau), s_2, \g)$ 
of $F(u) = 0$. 
The solution $u$ is $\even(\ph)\odd(x)$, it has zero average in $\ph$, and satisfies
\begin{equation} \label{0811.1}
\| u \|_{s_2}^{k_0,\g} 
%\leq C \g^{-1} \| V \|_{s_2 + \s}^{k_0,\g}.  
\leq C \g^{-1} \| V \|_{s_2 + p}^{k_0,\g}.
\end{equation}
If, in addition, $V \in \Lip(k_0,\mathtt{DC}(2 \gamma, \tau), s + p, \g)$ for $s > s_2$, 
then $u \in \Lip(k_0,\mathtt{DC}(2 \gamma, \tau), s, \g)$, with 
\begin{equation} \label{0811.3}
\| u \|_s^{k_0,\g} \leq C_s \g^{-1} \| V \|_{s + p}^{k_0,\g} 
\end{equation}
for some constant $C_s$ depending on $s,\varsigma,s_0$, independent of $V,\g$.
\end{theorem}

\begin{proof}
We apply Theorem \ref{thm:NMH} of Appendix \ref{sec:NMH}. 
For $a,b \geq 0$, we define
\begin{alignat}{2} \label{0411.1}
E_a & := \big\{ u \in \Lip(k_0,\mathtt{DC}(2 \gamma, \tau), 2s_0 + a, \g) : 
u = \even(\ph) \odd(x), \ \langle u \rangle_\ph (x) = 0 \big\}, \quad & 
\| u \|_{E_a} & := \| u \|_{2s_0 + a}^{k_0,\g} ,
\\
F_b & := \big\{ g \in \Lip(k_0,\mathtt{DC}(2 \gamma, \tau), 2s_0 + b, \g) : 
g = \odd(\ph) \odd(x) \big\}, \quad &
\| g \|_{F_b} & := \| g \|_{2s_0 + b}^{k_0,\g}
\label{0411.2}
\end{alignat} 
($s_0$ is in the last term of \eqref{1202.14}, 
while $2s_0$ appears in the composition estimate \eqref{pr-comp1}).
We consider Fourier truncations at powers of $2$ as smoothing operators, namely
\begin{equation} \label{0411.3}
S_n \ : \ u(\ph,x) = \sum_{(\ell, j) \in \Z^{\nu+1}}  u_{\ell j} e^{\ii (\ell \cdot \ph + j x)} 
\quad \mapsto \quad 
(S_n u)(\ph,x) := \sum_{\langle \ell, j \rangle \leq 2^n}  u_{\ell j} e^{\ii (\ell \cdot \ph + j x)}
\end{equation}
on both spaces $E_a$ and $F_b$. 
Hence both $E_a$ and $F_b$ satisfy \eqref{S0}-\eqref{S4}, and the operators $R_n$ defined in \eqref{new.24} give the dyadic decomposition $2^n < \langle \ell, j \rangle \leq 2^{n+1}$.  
Since $S_n$ in \eqref{0411.3} are ``crude'' Fourier truncations, 
\eqref{2705.4} holds with ``$=$'' instead of ``$\leq$'' and $C=1$. 
As a consequence, every $g \in F_\b$ satisfies the first inequality in \eqref{2705.1} with $A=1$ 
(it becomes, in fact, an equality), and, similarly, if $g \in F_{\b+c}$ then \eqref{0406.1} holds with $A_c = 1$ (and ``$=$''). 

We denote by $\mV$ the composition operator $\mV(u)(\ph,x) := V(\ph, x + u(\ph,x))$,  
and define $\Phi(u) := \ompaph u - \mV(u)$, 
namely we take the nonlinear operator $F$ in \eqref{1202.12} 
as the operator $\Phi$ of Theorem \ref{thm:NMH}. 
By Lemma \ref{lemma:LS norms}, if $\| u \|_{2s_0+1}^{k_0,\g} \leq \d_{\ref{lemma:LS norms}}$ 
(where we denote by $\d_{\ref{lemma:LS norms}}$ the constant $\d$ of Lemma \ref{lemma:LS norms}), 
then $\mV(u)$ satisfies \eqref{pr-comp1}, namely for all $s \geq s_0$
\begin{equation} \label{0611.1}
\| \mV(u) \|_s^{k_0,\g} 
\lesssim_s \| V \|_{s+k_0}^{k_0,\g} + \| u \|_s^{k_0,\g} \| V \|_{s_0 + k_0 + 1}^{k_0,\g} ,
\end{equation}
and its second derivative $\mV''(u)[v,w] = V_{xx}(\ph,x+u(\ph,x)) v w$ satisfies 
\begin{align} 
\| \mV''(u)[v,w] \|_s^{k_0,\g} 
& \lesssim_s \| V \|_{s_0+k_0+3}^{k_0,\g} 
\Big( \| v \|_{s}^{k_0,\g} \| w \|_{s_0}^{k_0,\g} + \| v \|_{s_0}^{k_0,\g} \| w \|_{s}^{k_0,\g} \Big)
\notag \\ & \quad \ \ 
+ \big\{ \| V \|_{s_0+k_0+3}^{k_0,\g} \| u \|_s^{k_0,\g} + \| V \|_{s+k_0+2}^{k_0,\g} \big\} 
\| v \|_{s_0}^{k_0,\g} \| w \|_{s_0}^{k_0,\g}.
\label{0611.2}
\end{align}
We fix $\mu, U$ of Theorem \ref{thm:NMH} as 
$\mu := 1$, $U := \{ u \in E_1 : \| u \|_{E_1} \leq \d_{\ref{lemma:LS norms}} \}$. 
Thus $\Phi$ maps $U \to F_0$ and $U \cap E_{a+\mu} \to F_a$ for all $a \in [0,a_2 - 1]$, 
provided that $\| V \|_{2s_0 + a_2 - 1 + k_0}^{k_0,\g} < \infty$ ($a_2$ will be fixed below in \eqref{2102.1}). 
Moreover, for all $a \in [0,a_2-1]$, $\Phi$ is of class $C^2(U \cap E_{a+\mu}, F_a)$
and it satisfies \eqref{Phi sec} with $a_0 := 0$, 
\begin{equation} \label{0711.1}
M_1(a) := C(a) \| V \|_{s_0 + k_0 + 3}^{k_0,\g}, \quad 
M_2(a) := M_1(a), \quad 
M_3(a) := C(a) \| V \|_{2 s_0 + k_0 + 2 + a}^{k_0,\g}.
\end{equation}
We fix $a_1, \d_1$ of Theorem \ref{thm:NMH} as 
$a_1 := \varsigma$, % $a_1 := \max \{ 1, \s - s_0 \}$, 
where $\varsigma = 3k_0 + 2\t(k_0+1) +2$ is the constant appearing in Lemma \ref{lemma:inv transport},
and $\d_1 := \frac12 \d_{\ref{lemma:inv transport}}$, 
where $\d_{\ref{lemma:inv transport}}$ is the constant $\d_0$ of Lemma \ref{lemma:inv transport}.
If $\g^{-1} \| V \|_{s_0 + \varsigma}^{k_0,\g} \leq \d_1$ and $\| v \|_{E_{a_1}} \leq \d_1$, 
then, by Lemma \ref{lemma:inv transport}, the right inverse $\Psi(v) := F'(v)^{-1}$ is well defined,
and it satisfies 
\begin{equation} \label{0711.2}
\| \Psi(v)g \|_{E_a} \leq L_1(a) \| g \|_{F_{a+\varsigma}} 
+ (L_2(a) \| v \|_{E_{a+\varsigma}} + L_3(a)) \| g \|_{F_0} 
\end{equation}
where
\begin{equation} \label{0711.3}
L_1(a) := C(a) \g^{-1}, \quad 
L_2(a) := C(a) \g^{-2} \| V \|_{s_0 + \varsigma}^{k_0, \g}, \quad 
L_3(a) := C(a) \g^{-2} \| V \|_{2s_0 + a + \varsigma}^{k_0, \g}.
\end{equation}
We fix $\a, \b, a_2$ of Theorem \ref{thm:NMH} as
\begin{equation}\label{2102.1}
\b := 4\varsigma+1, \quad \a := 3\varsigma + 1, \quad a_2 := 5\varsigma + 3,
\end{equation}
so that \eqref{ineq 2016} is satisfied. Bound \eqref{0711.2} implies \eqref{tame in NM} for all $a \in [a_1, a_2]$ 
provided that $\| V \|_{2s_0 + a_2 + \varsigma}^{k_0,\g} < \infty$. 

All the hypotheses of the first part of Theorem \ref{thm:NMH} are satisfied.  
As a consequence, there exists a constant $\d_{\ref{qui.02}}$ (given by \eqref{qui.02} with $A = 1$) 
such that, if $\| g \|_{F_\b} \leq \d_{\ref{qui.02}}$,
then the equation $\Phi(u) = \Phi(0) + g$ has a solution $u \in E_\a$, with bound \eqref{qui.01}.
In particular, the result applies to $g = V$, in which case the equation $\Phi(u) = \Phi(0) + g$ becomes $\Phi(u) = 0$. 
We have to verify the smallness condition $\| g \|_{F_\b} \leq \d_{\ref{qui.02}}$. Using \eqref{0711.1}, \eqref{0711.3}, \eqref{0811.2}, we verify that $\d_{\ref{qui.02}} \geq C \g$. Thus, the smallness condition $\| g \|_{F_\b} \leq \d_{\ref{qui.02}}$ is satisfied if $\| V \|_{2s_0 + a_2 + \varsigma}^{k_0,\g} \g^{-1}$ is smaller 
than some $\d$ depending on $\varsigma, s_0$. This is assumption \eqref{0811.2}, since $2s_0 + a_2 + \varsigma = s_2 + p$.
Then \eqref{qui.01}, recalling \eqref{2102.1}, gives $\| u \|_{s_2}^{k_0,\g} \leq C \g^{-1} \| V \|_{s_2 + \varsigma}^{k_0,\g}$,
which implies \eqref{0811.1} since $p \geq \varsigma$.

We finally prove estimate \eqref{0811.3}. Let $c>0$. If, in addition, $\| V \|_{2s_0 + a_2 + c + \varsigma}^{k_0,\g} < \infty$,
then all the assumptions of the second part of Theorem \ref{thm:NMH} are satisfied. 
By \eqref{0711.1}, \eqref{0711.3} and \eqref{0811.2}, we estimate the constants defined in \eqref{qui.04}-\eqref{qui.03} as
\[
\mG_1 \leq C_c \g^{-2} \| V \|_{2s_0 + a_2 + c + \varsigma}^{k_0, \gamma}, \quad 
\mG_2 \leq C_c \g^{-1}, \quad 
z \leq C_c
\]
for some constant $C_c$ depending on $c$. Bound \eqref{0211.10} implies \eqref{0811.3} 
with $s = s_2 + c$
(the highest norm of $V$ in \eqref{0811.3} does not come from the term $\| V \|_{F_{\b+c}}$ of \eqref{0211.10}, but from the factor $\mG_1$). 
The proof is complete.
\end{proof}

The next lemma deals with the dependence of the solution $u$ of \eqref{1202.12} on $V$ (actually it would be enough to estimate this Lipschitz dependence only in the ``low'' norm $s_1$ introduced in \eqref{vincolo s1 derivate i}). 

\begin{lemma} {\bf (Lipschitz dependence of $u$ on $V$)}
\label{lemma:uV Lip}
Let $\varsigma,s_2,p$ be as defined in Theorem \ref{thm:nonlinear transport}.
Let $V_1, V_2$ satisfy \eqref{0811.2}, 
and let $u_1, u_2$ be the solutions of 
$$
\omega \cdot \partial_\vphi u_i - V_i(\vphi, x + u_i(\vphi, x)) = 0, \quad i = 1,2,
$$
given by Theorem \ref{thm:nonlinear transport}. Then for all $s \geq s_2-\mu$ (where $\mu$ is the constant defined in \eqref{Diophantine-1})
\begin{equation} \label{0811.8}
\| u_1 - u_2 \|_s^{k_0,\g} 
\lesssim_s \g^{-1} \| V_1 - V_2 \|_{s + \mu + k_0}^{k_0,\g} 
+ \g^{-2} \max_{i=1,2} \| V_i \|_{s + 2\mu + p}^{k_0,\g} 
\| V_1 - V_2 \|_{s_2 + k_0}^{k_0,\g}.
\end{equation}
\end{lemma}

\begin{proof}
The difference $h := u_1 -u_2$ is $\even(\ph) \odd(x)$, it has zero average in $\ph$ 
and it solves $\omega \cdot \partial_\vphi h - a h = b$, where 
$$
a (\vphi, x) := \int_0^1 (\partial_x V_1)(\vphi, x + t u_1 +(1 - t) u_2)\, d t \,, \quad 
b(\vphi, x) := (V_1 - V_2)(\vphi, x + u_2)\, .
$$
The function 
$a $ is $  \odd(\ph) \even(x)$ and $b $ is $ \odd(\ph) \odd(x)$. 
Then, by variation of constants and uniqueness, $h = w e^v$, where (as in Lemma \ref{lemma:inv transport})
$$
v := (\omega \cdot \partial_\vphi)^{- 1} a, \quad 
w := w_0 + g, \quad 
w_0 := (\omega \cdot \partial_\vphi)^{- 1} (e^{-v} b), \quad 
g = g(x) := - \frac{\langle w_0 e^v \rangle_\ph}{\langle e^v \rangle_\ph}\,.
$$
Then \eqref{0811.8} follows by \eqref{pr-comp1}, \eqref{0811.2}, \eqref{0811.1}, \eqref{0811.3}, \eqref{Diophantine-1} and \eqref{0811.10}.
%
%*******************************************
%
%By \eqref{pr-comp1}, \eqref{0811.2}, \eqref{0811.1}, \eqref{0811.3}, one has 
%\[
%\| a \|_s^{k_0,\g} \lesssim_s \| V_{1,2} \|_{s+p}^{k_0,\g}, \quad 
%\| b \|_s^{k_0,\g} \lesssim_s \| V_1 - V_2 \|_{s+k_0}^{k_0,\g} 
%+ \g^{-1} \| V_2 \|_{s+p}^{k_0,\g} \| V_1 - V_2 \|_{s_0+k_0+1}^{k_0,\g}  
%\quad \forall s \geq s_2,
%\]
%where $\| V_{1,2} \|_s^{k_0,\g} := \max_{i=1,2} \| V_i \|_s^{k_0,\g}$, 
%and, like in Theorem \ref{thm:nonlinear transport}, $s_2 := 2s_0 + 3\varsigma + 1$, $p := 3\varsigma + 2$. 
%By \eqref{Diophantine-1} and \eqref{0811.10}, 
%\[
%\| v \|_s^{k_0,\g} \lesssim_s \g^{-1} \| V_{1,2} \|_{s+\mu+p}^{k_0,\g}, \quad 
%\| e^v \|_s^{k_0,\g} \lesssim_s 1 + \g^{-1} \| V_{1,2} \|_{s+\mu+p}^{k_0,\g} 
%\quad \forall s \geq s_2 - \mu,
%\]
%where $\mu$ is defined in \eqref{Diophantine-1}. Then  
%\[
%\| w_0 \|_s^{k_0,\g} \lesssim_s \g^{-1} \| V_1 - V_2 \|_{s + \mu + k_0}^{k_0,\g} 
%+ \g^{-2} \| V_{1,2} \|_{s + 2\mu + p}^{k_0,\g} \| V_1 - V_2 \|_{s_2 + k_0}^{k_0,\g}  
%\]
%for all $s \geq s_2 - \mu$, and $w_0 e^v, g, h$ satisfy the same bound.
\end{proof}

In Theorem \ref{thm:nonlinear transport}, 
for any $\lm = (\om,\h) \in \mathtt{DC}(2 \g,\t) \times [\h_1, \h_2]$ 
we have constructed a periodic  function $u = \breve \b$ that solves \eqref{1202.12}, namely 
the quasi-periodic characteristic  equation \eqref{1202.11}, 
so that the periodic function $\b$, defined by the inverse diffeomorphism in \eqref{1202.10}, solves 
the quasi-periodic transport equation \eqref{transport}.
%Let $y = x + \b(\ph,x)$ be the inverse diffeomorphism of $x = y + \breve \beta(\vphi, y)$, 
%where $\breve \beta := u$ is the solution of equation \eqref{1202.11}, namely \eqref{1202.12}, 
%given by Theorem \ref{thm:nonlinear transport}.
%Then $\b(\ph,x)$ solves \eqref{transport}.

By Theorem \ref{thm:WET} we define an extension $\mE_k(u) = \mE_k(\breve\b) =: \breve\b_{ext}$ 
(with $k+1 = k_0$) 
to the whole parameter space $\R^\nu \times [\h_1, \h_2]$.
By the linearity of the extension operator $\mE_k$ 
and by the norm equivalence \eqref{basso Stein}, 
the difference of the extended functions $\mE_k(u_1) - \mE_k(u_2)$ 
also satisfies the same estimate \eqref{0811.8} as $u_1 - u_2$. 

We define an extension $\b_{ext}$ of $\b$ to the whole space $\lm \in \R^\nu \times [\h_1, \h_2]$
by
\[ % \begin{equation} \label{2802.1}
y = x + \b_{ext} (\ph,x) \quad \Leftrightarrow \quad 
x = y + \breve\b_{ext} (\ph,y) \quad \forall x,y \in \T, \ \ph \in \T^\nu
\] % \end{equation}
(note that, in general, $\b_{ext}$ and $\mE_k(\b)$ are two different extensions of $\b$ 
outside $\mathtt{DC}(\g,\t) \times [\h_1, \h_2]$).
The extended functions $\b_{ext}, \breve\b_{ext}$ induce the operators $\mB_{ext}, \mB_{ext}^{-1}$ by 
\[
(\mB_{ext} h) (\ph,x) := h(\ph, x + \b_{ext}(\ph,x)), \qquad 
(\mB_{ext}^{-1} h) (\ph,y) := h(\ph, y + \breve\b_{ext}(\ph,y)), \qquad 
\mB_{ext} \circ \mB_{ext}^{-1} = \mathrm{Id},
\]
and they are defined for $\lm \in \R^\nu \times [\h_1, \h_2]$. 

\emph{Notation:} for simplicity, in the sequel we will drop the subscript ``$ext$'' and we rename 
\begin{equation} \label{def ext}
\b_{ext} := \b, \quad 
\breve\b_{ext} := \breve\b, \quad  
\mB_{ext} := \mB, \quad  
\mB_{ext}^{-1} := \mB^{-1}.
\end{equation} 
We have the following estimates on the transformations ${\cal B}$ and ${\cal B}^{-1}$. 

\begin{lemma} 
\label{stime trasformazione cal B}
Let $\b, \breve\b$ be defined in \eqref{def ext}.
There exists $\s:= \s(\t,\nu,k_0)$ such that, if \eqref{ansatz I delta} holds with $\mu_0 \geq \sigma$, then 
for any $s \geq s_2$, 
\begin{equation}\label{beta beta tilde stima}
\| \beta\|_s^{k_0, \gamma}, \| \breve \beta\|_s^{k_0, \gamma} \lesssim_{s} \e \g^{-1} 
\big(1 + \| \fracchi_0 \|_{s + \sigma}^{k_0, \gamma} \big)\,.
\end{equation}
The operators $ A = {\cal B}^{\pm 1} - {\rm Id}, ({\cal B}^{\pm 1} - {\rm Id})^*$ satisfy the estimates  
\begin{equation}\label{est BAPQ (1)}
\| A h \|_s^{k_0, \gamma} \lesssim_{s} \e \g^{-1} \big( \| h \|_{s + k_0 + 1}^{k_0, \gamma} + \| \fracchi_0\|_{s + \sigma}^{k_0, \gamma} \| h \|_{s_0 + k_0 + 2}^{k_0, \gamma} \big) \quad\forall s\geq s_2 \, .
\end{equation}
Let $i_1, i_2$ be two given embedded tori. Then, denoting $\Delta_{12} \b = \b(i_2) - \b(i_1)$ and similarly for the other quantities, we have
\begin{align}
\| \Delta_{12} \beta \|_{s_1 }, \| \Delta_{12} \breve \beta \|_{s_1 } 
& \lesssim_{s_1} \e\g^{-1} \| i_1 - i_2 \|_{s_1 + \sigma} \,,
\label{stima differenza beta beta tilde}
\\
\| (\Delta_{12} A)[h]\|_{s_1} & \lesssim_{s_1} \e\g^{-1} \| i_1 - i_2\|_{s_1 + \sigma} \| h \|_{s_1 + 1}\,, \quad A \in \{ {\cal B}^{\pm 1}, ({\cal B}^{\pm 1})^* \}\,,
\label{Delta 12 cal B}
\end{align}
where $s_1$ is introduced in \eqref{vincolo s1 derivate i}.
\end{lemma}

\begin{proof}
Bound \eqref{beta beta tilde stima} for $\breve \beta$ follows, recalling that $\breve \beta = u$,
by  \eqref{0811.3} and  \eqref{stima V B a c}.   
Estimate \eqref{beta beta tilde stima} 
for $\beta$ follows by that for $\breve \beta$, applying \eqref{p1-diffeo-inv}.
We now prove estimate \eqref{est BAPQ (1)} for ${\cal B} - {\rm Id}$. We have 
$$
({\cal B} - {\rm Id}) h = \beta\, \int_0^1{\cal B}_\tau [h_x]\, d \tau\,, \qquad 
{\cal B}_\tau [f](\vphi, x) : = f(\vphi, x + \tau \beta(\vphi, x))\,.
$$
Then \eqref{est BAPQ (1)} follows by applying \eqref{pr-comp1} to the operator ${\cal B}_\tau$, using the estimates on $\beta$, ansatz \eqref{ansatz I delta} and  \eqref{p1-pr}. 
The estimate for ${\cal B}^{- 1} - \mathrm{Id}$ is obtained similarly. 
The estimate on the adjoint operators follows because 
$$
{\cal B}^*  h(\vphi, y) = (1 + \breve \beta(\vphi, y)) h(\vphi, y + \breve \beta(\vphi, y)), \quad ({\cal B}^{- 1})^* h(\vphi, x) =   (1 +  \beta(\vphi, x)) h(\vphi, x +  \beta(\vphi, x))\,.
$$
Estimates  \eqref{stima differenza beta beta tilde}, \eqref{Delta 12 cal B} 
%\eqref{stima differenza beta beta tilde} for $\Delta_{12} \breve \beta$ 
follow by Lemma \ref{lemma:uV Lip}, and by \eqref{stima V B a c}-\eqref{stima derivate i primo step}. 
\end{proof}

We now conjugate the whole operator $\mL_0$ in \eqref{mZ mL0} by the diffeomorphism $ {\cal B} $.

\begin{lemma} \label{lem:stime coefficienti cal L1}
Let $\b, \breve\b, \mB, \mB^{-1}$ be defined in \eqref{def ext}.
For all $\lm \in \mathtt{DC}(\g,\t) \times [\h_1, \h_2]$, the transformation $\mB$ 
conjugates the operator $\mL_0$ defined in \eqref{mZ mL0} to % (i.e. \eqref{mZ mL0-new}) to
\begin{align}
\mL_1 & := \mB^{-1} \mL_0 \mB
= \ompaph + \begin{pmatrix} 
  a_1 &  - a_2 \pa_y \mH T_\h + \mR_1 \\ 
a_3 & 0 
\end{pmatrix},
\label{mL1}
\\ 
& \quad T_\h  := \tanh( \h |D_y|) := \Op\big( \tanh(\h \chi(\xi) |\xi|) \big),
\label{def Th}
\end{align}
where $ a_1, a_2, a_3 $ are the functions 
\be \label{defa2a3}
a_1(\ph,y)  := (\mB^{-1} V_x)(\ph,y), \quad 
a_2(\ph,y) := 1 + (\mB^{-1} \b_x)(\ph,y) \, , \quad 
a_3(\ph,y)  := (\mB^{-1} a)(\ph,y),
\ee
and $\mR_1$ is a pseudo-differential operator of order $OPS^{-\infty}$. 
Formula \eqref{defa2a3} defines the functions $a_1, a_2, a_3$ on the whole parameter space 
$\R^\nu \times [\h_1, \h_2]$. 
The operator $\mR_1$ admits an extension to $\R^\nu \times [\h_1, \h_2]$ as well, 
which we also denote by $\mR_1$. 
The real valued functions $\b, a_1, a_2, a_3$ have parity
\begin{equation} \label{rem:eps a1 b1 c1}
\beta = \even(\ph)\odd(x); \qquad
a_1 = \odd(\ph)\even(y); \qquad 
a_2, a_3 = \even(\ph)\even(y).
\end{equation}
There exists $\sigma = \sigma (\tau, \nu, k_0) > 0$ such that for any $m, \alpha \geq 0$, assuming \eqref{ansatz I delta} with $\mu_0 \geq \sigma + m + \alpha$, for any $s \geq s_0$, 
on $\R^\nu \times [\h_1, \h_2]$ the following estimates hold:
\begin{align}
& \| a_1 \|_{s}^{k_0, \gamma} + \| a_2 - 1 \|_s^{k_0, \gamma} + \| a_3 - 1  \|_s^{k_0, \gamma}  
\lesssim_s \e \gamma^{-1}  \big(1 + \| \fracchi_0 \|_{s+\s}^{k_0, \gamma} \big), 
\label{stime coefficienti cal L1}\\
& \norma \mR_1  \norma_{- m,s,\alpha}^{k_0, \gamma} \lesssim_{m,s,\a} \e \gamma^{-1} 
\big(1 + \| \fracchi_0 \|_{s+\s + m + \alpha}^{k_0, \gamma} \big) \, . \label{stima resto cal L1}
\end{align}
Finally, given   two tori  $i_1, i_2 $, we have  
\begin{align}
& \| \Delta_{12} a_1 \|_{s_1} + \| \Delta_{12}a_2  \|_{s_1} + \| \Delta_{12} a_3   \|_{s_1}  
\lesssim_{s_1} \e \gamma^{-1} \| \Delta_{12} i\|_{s_1 + \sigma} \, , \label{stime coefficienti cal L1 Delta}\\
& \norma \Delta_{12} \mR_1  \norma_{- m,s_1,\alpha} \lesssim_{m,s_1,\a} \e \gamma^{-1} \| \Delta_{12} i\|_{s_1 + \sigma + m + \alpha} \, . 
\label{stima resto cal L1 Delta}
\end{align}
\end{lemma}

\begin{proof} 
By \eqref{mZ mL0-new} and  \eqref{2702.2}-\eqref{transport} we have that
\begin{equation}\label{first-tra1}
\mL_1  := \mB^{-1} \mL_0 \mB 
= \om \cdot \pa_\vphi +
\begin{pmatrix}
a_1 & \ - \mB^{-1}  G(\eta) \mB \\ 
a_3 &  0
\end{pmatrix} 
\end{equation} 
where the functions $ a_1 $ and $ a_3 $ are defined in \eqref{defa2a3}. 
We now conjugate the Dirichlet-Neumann operator $ G(\eta )  $
 under the diffeomorphism $ \mB $. 
Following Proposition \ref{lemma dirichlet Neumann}, we write 
\be\label{expa:Geta}
G(\eta ) = |D_x| \tanh(\h |D_x|) + \mR_G = \pa_x \mH T_\h + \mR_G \, , \ \qquad  T_{\h} := \tanh(\h |D_x|) \, ,
\ee
where $ \mR_G $ is an integral operator in $ OPS^{-\infty } $.
We decompose
\begin{equation} \label{tangente iperbolica espansione} 
\tanh(\mathtt h |D_x|) = {\rm Id} + {\rm Op}(r_{\mathtt h}), \quad 
r_{\mathtt h}(\xi) := \, - \frac{2}{1 + e^{2 \mathtt h |\xi| \chi(\xi)}} \in S^{- \infty},
\end{equation}
and,
%Recalling the decomposition \eqref{tangente iperbolica espansione},
since $\mB^{-1} \, \pa_x \, \mB = a_2 \pa_y$ where the function $ a_2 $ is defined in \eqref{defa2a3}, we have
\begin{align}
\mB^{-1} \pa_x \mH T_\h \mB 
& = (\mB^{-1} \pa_x \mB) (\mB^{-1} \mH \mB) (\mB^{-1} T_\h \mB) = a_2 \pa_y \{ \mH + (\mB^{-1} \mH \mB - \mH) \} (\mB^{-1} T_\h \mB)
\nonumber\\ & 
= a_2 \pa_y \mH T_\h 
+ a_2 \pa_y \mH [\mB^{-1} {\rm Op}(r_\h) \mB - {\rm Op}(r_\h)]
+ a_2 \pa_y (\mB^{-1} \mH \mB - \mH) (\mB^{-1} T_\h \mB) \, . \label{bla bla bla bla}
\end{align}
Therefore by \eqref{expa:Geta}-\eqref{bla bla bla bla} we get 
\begin{equation} \label{def mR1}
- {\cal B}^{- 1} G(\eta) {\cal B}  = - a_2 \pa_y \mH T_\h  + {\cal R}_1\,,
\end{equation}
where  
$ {\cal R}_1 $ 
is the operator in $ OPS^{-\infty} $ defined by 
\be
\begin{aligned}
 {\cal R}_1 := & {\cal R}_1^{(1)} + {\cal R}_1^{(2)} + {\cal R}_1^{(3)}  \qquad
& {\cal R}_1^{(1)} := & - \mB^{-1} \mR_G \mB \, , \\ 
 {\cal R}_1^{(2)}  := & -  a_2 \pa_y \mH [\mB^{-1} {\rm Op}(r_\h) \mB - {\rm Op}(r_\h)] \, , \qquad
& {\cal R}_1^{(3)} := & - a_2 \pa_y (\mB^{-1} \mH \mB - \mH) \mB^{-1} T_\h \mB \, .
\label{2702.1}
\end{aligned}
\ee
Notice that $ \mB^{-1} \mR_G \mB $ and $ \mB^{-1} {\rm Op}(r_\h) \mB $ are in $OPS^{-\infty}$
since $ \mR_G $ and $ {\rm Op}(r_\h) $, defined in \eqref{expa:Geta} and in \eqref{tangente iperbolica espansione},  are in $OPS^{-\infty}$.
The operator  $\mB^{-1} \mH \mB - \mH$ is in $OPS^{-\infty}$  by Lemma \ref{coniugio Hilbert}.

In conclusion,  \eqref{first-tra1} and \eqref{def mR1} imply  \eqref{mL1}-\eqref{defa2a3}, for 
all $ \lambda $ in the Cantor set  $\mathtt{DC}(\g,\t) \times [\h_1, \h_2]$.
By formulas \eqref{2702.1}, $\mR_1$ is defined on the whole parameter space 
$\R^\nu \times [\h_1, \h_2]$. 

Estimates \eqref{stime coefficienti cal L1}, \eqref{stime coefficienti cal L1 Delta} 
for $a_1, a_2, a_3$ on $\R^\nu \times [\h_1, \h_2]$ 
follow by \eqref{stima V B a c}, \eqref{stima derivate i primo step} 
and  Lemma \ref{stime trasformazione cal B}. Estimates \eqref{stima resto cal L1}, \eqref{stima resto cal L1 Delta} follow applying Lemmata \ref{lemma cio} and \ref{coniugio Hilbert} and Proposition \ref{lemma dirichlet Neumann}, and by using Lemma \ref{stime trasformazione cal B}.
\qedhere
\end{proof}

\begin{remark} 
We stress that the conjugation 
identity \eqref{mL1} holds only on the Cantor set $\mathtt{DC}(\g,\t) \times [\h_1, \h_2]$.
It is technically convenient to consider the extension of $a_1, a_2, a_3, \mR_1$ 
to the whole parameter space $\R^\nu \times [\h_1, \h_2]$, 
in order to directly use the results of Section \ref{sec:pseudo}
expressed by means of classical derivatives with respect to the parameter $\lm$. 
Formulas \eqref{defa2a3} and \eqref{2702.1} 
define $a_1, a_2, a_3, \mR_1$ on the whole parameter space $\R^\nu \times [\h_1, \h_2]$. 
Note that the resulting extended operator $\mL_1$ in the right hand side of \eqref{mL1} 
is defined on $\R^\nu \times [\h_1, \h_2]$, 
and in general it is different from $\mB^{-1} \mL_0 \mB$ 
outside $\mathtt{DC}(\g,\t) \times [\h_1, \h_2]$. 
\end{remark}

In the sequel we rename in \eqref{mL1}-\eqref{rem:eps a1 b1 c1} the space variable $ y $ by $ x $.

\section{Change of the space variable}\label{sezione nuova cambio di variabile}

We consider a $\vphi$-independent diffeomorphism of the torus $\T$ of the form 
\begin{equation}\label{diffeo solo x nuovo}
y =  x + \alpha(x) \qquad {\rm with  \ inverse } \qquad x =  y + \breve \alpha(y)
\end{equation}
where $\alpha $ is a $ {\cal C}^\infty(\T_x)$ real valued function, independent of $ \vphi $, 
satisfying $ \| \alpha_x \|_{L^\infty} \leq 1/2$. We also make the following ansatz on 
$\alpha$ that will be verified when we choose it in Section \ref{sec:primo semi-FIO}, see formula \eqref{equazione omologica per alpha}: the function $\alpha$ is ${\rm odd}(x)$ and  $\alpha = \alpha(\lambda) =  \alpha(\lambda, i_0 (\lambda))$, $\lambda \in \R^{\nu + 1}$ is $k_0$ times differentiable with respect to the parameter $\lambda \in \R^{\nu + 1}$ with $\partial_\lambda^k \alpha \in {\cal C}^\infty(\T)$ for any $k \in \N^{\nu + 1}$, $|k| \leq k_0 $, and it satisfies the estimate 
\begin{equation}\label{stime alpha nuove ansatz}
\begin{aligned}
& \| \alpha\|_s^{k_0, \gamma} \lesssim_s \e \gamma^{- 1} 
\big(1 + \| \fracchi_0\|_{s + \sigma}^{k_0, \gamma} \big)\,, \   \forall s \geq s_0\,, \quad   \| \Delta_{12} \alpha\|_{s_1} \lesssim_{s_1} \e \gamma^{- 1} \| \Delta_{12} i \|_{s_1 + \sigma}\,,
\end{aligned}
\end{equation}
for some $\sigma = \sigma(k_0, \tau, \nu) > 0$. 
By \eqref{stime alpha nuove ansatz} and Lemma \ref{lemma:LS norms}, 
arguing as in the proof of Lemma \ref{stime trasformazione cal B}, one gets 
\begin{equation}\label{stime alpha breve nuovo dopo ansatz}
 \| \breve \alpha\|_s^{k_0, \gamma} \lesssim_s \e \gamma^{- 1}
\big(1 + \| \fracchi_0\|_{s + \sigma}^{k_0, \gamma} \big)\,, \  \forall s \geq s_0\,,  \quad
\| \Delta_{12} \breve \alpha\|_{s_1} \lesssim_{s_1} \e \gamma^{- 1} \| \Delta_{12} i \|_{s_1 + \sigma}\,,
\end{equation}
for some $\sigma = \sigma(k_0, \tau, \nu) > 0$. 
Furthermore, the function $\breve \alpha (y) $ is $\odd(y)$. 

 We conjugate the operator $ {\cal L}_1 $ in \eqref{mL1} by the composition operator
\begin{equation}\label{operatore cambio di variabile x nuovo}
(\mA u)(\ph,x) := u(\ph,x + \a(x)), \quad 
(\mA^{-1} u)(\ph,y) := u(\ph, y + \breve\a(y)) \, . 
\end{equation}
By \eqref{mL1}, using that the operator ${\cal A}$ is $\vphi$-independent, 
recalling expansion \eqref{tangente iperbolica espansione}
and arguing as in \eqref{bla bla bla bla} to compute the conjugation ${\cal A}^{- 1} \big( - a_2 \partial_x {\cal H} T_{\mathtt h}\big) {\cal A} $, one has 
\begin{equation}\label{nuovo cal L2 prima}
{\cal L}_2 := {\cal A}^{- 1} {\cal L}_1 {\cal A} = \omega \cdot \partial_\vphi + \begin{pmatrix}
a_4 & - a_5 \partial_y {\cal H} T_{\mathtt h} + {\cal R}_2 \\
a_6 & 0 
\end{pmatrix}\,,
\end{equation}
where  $ a_4, a_5, a_6 $ are the functions 
\begin{align}
& a_4 (\vphi, y):= ({\cal A}^{- 1} a_1)(\vphi, y) = a_1(\vphi, y + \breve \alpha(y))\,,  
\label{coefficienti cal L2 nuovo-a4} \\
&  a_5(\vphi, y) := \big( {\cal A}^{- 1} ( a_2 (1 + \alpha_x) ) \big)(\vphi, y) 
= \{ a_2(\vphi, x) (1 + \alpha_x(x)) \} |_{x = y + \breve \alpha(y)} \label{coefficienti cal L2 nuovo} \\
& a_6(\vphi, y) := ({\cal A}^{- 1} a_3)(\vphi, y) = a_3(\vphi, y + \breve \alpha(y)) \label{coefficienti cal L2 nuovo-a6} 
\end{align}
and $ {\cal R}_2 $  is the operator in $ OPS^{-\infty} $ given by
\be
{\cal R}_2 := - a_5 \pa_y \mH \big[ \mA^{-1} {\rm Op}(r_\h) \mA - {\rm Op}(r_\h) \big]
-  a_5 \pa_y (\mA^{-1} \mH \mA - \mH) (\mA^{-1} T_\h \mA) + {\cal A}^{- 1}{\cal R}_1 {\cal A}\,.
\label{coefficienti cal L2 nuovo-R2}
\ee

\begin{lemma}\label{stime trasformazione nuova cal A}
There exists a constant $\sigma = \sigma(k_0, \tau, \nu) > 0$ such that,
if \eqref{ansatz I delta} holds with $\mu_0 \geq \sigma$, then the following holds: the operators $A \in \{ {\cal A}^{\pm 1} - {\rm Id}, ({\cal A}^{\pm 1} - {\rm Id})^* \}$ are even and reversibility preserving and satisfy
\begin{equation}\label{stime nuovo cal A inv}
\begin{aligned}
& \| A h\|_s^{k_0, \gamma} \lesssim_s \e\gamma^{- 1} \big(  \| h \|_{s + k_0 + 1}^{k_0, \gamma} + \| \fracchi_0\|_{s + \sigma}^{k_0, \gamma} \| h \|_{s_0 + k_0 + 2}^{k_0, \gamma}\big)\,, \quad \forall s \geq s_0\,, \\
& \| (\Delta_{1 2} A) h \|_{s_1 } \lesssim_{s_1} \e \gamma^{- 1} \| \Delta_{12} i \|_{s_1 + \sigma} \| h \|_{s_1 + 1}\,.
\end{aligned}
\end{equation}
The real valued 
functions $a_4, a_5, a_6$ in \eqref{coefficienti cal L2 nuovo-a4}-\eqref{coefficienti cal L2 nuovo-a6} satisfy 
\be\label{parity a4 a5 a6}
a_4 = \odd(\ph)\even(y), \qquad 
a_5, a_6 = \even(\ph)\even(y) \, , 
\ee
and 
\begin{equation}\label{stime a4 a5 a6 nuovo}
\begin{aligned}
& \| a_4\|_s^{k_0, \gamma}\,, \| a_5 - 1\|_s^{k_0, \gamma}\,, \| a_6 - 1\|_s^{k_0, \gamma} \lesssim_s \e \gamma^{- 1} 
\big(1 + \| \fracchi_0\|_{s + \sigma}^{k_0, \gamma} \big) \\
& \| \Delta_{12} a_4\|_{s_1}\,, \| \Delta_{12} a_5\|_{s_1}\,, \| \Delta_{12} a_6\|_{s_1} \lesssim_{s_1} \e \gamma^{- 1} \| \Delta_{12} i\|_{s_1 + \sigma}\,.
\end{aligned}
\end{equation}  
The remainder ${\cal R}_2$ defined in \eqref{coefficienti cal L2 nuovo-R2} is an even and reversible pseudo-differential operator in $OPS^{- \infty}$. Moreover, for any $m, \alpha \geq 0$, and assuming \eqref{ansatz I delta} with $\sigma + m + \alpha \leq \mu_0$, the following estimates hold: 
\begin{equation}\label{stime cal R2 nuovo}
\begin{aligned}
& \norma {\cal R}_2 \norma_{- m, s, \alpha}^{k_0, \gamma} \lesssim_{m, s, \alpha} \e \gamma^{- 1} 
\big(1 + \| \fracchi_0 \|^{k_0, \gamma}_{s + \sigma + m + \alpha} \big)\,, \quad \forall s \geq s_0 \\
& \norma \Delta_{12} {\cal R}_2 \norma_{- m , s_1 , \alpha} \lesssim_{m, s_1, \alpha} \e \gamma^{- 1} \| \Delta_{12} i\|_{s_1 + \sigma + m + \alpha} \, . 
\end{aligned}
\end{equation}
\end{lemma}

\begin{proof}
The transformations ${\cal A}^{\pm 1} - {\rm Id}, ({\cal A}^{\pm 1} - {\rm Id})^*$ are even and reversibility preserving because $\alpha$ and $\breve \alpha$ are odd functions. 
Estimate \eqref{stime nuovo cal A inv} can be proved by using \eqref{stime alpha nuove ansatz}, \eqref{stime alpha breve nuovo dopo ansatz}, 
arguing as in the proof of Lemma \ref{stime trasformazione cal B}.

Estimate \eqref{stime a4 a5 a6 nuovo} follows by definitions 
\eqref{coefficienti cal L2 nuovo-a4}-\eqref{coefficienti cal L2 nuovo-a6}, 
by estimates 
\eqref{stime alpha nuove ansatz}, 
\eqref{stime alpha breve nuovo dopo ansatz}, 
\eqref{stime nuovo cal A inv}, 
\eqref{stime coefficienti cal L1},
\eqref{stime coefficienti cal L1 Delta},
and by applying Lemma \ref{lemma:LS norms}. 
Estimates \eqref{stime cal R2 nuovo} of the remainder ${\cal R}_2$ follow by using the same arguments we used in Lemma \ref{lem:stime coefficienti cal L1} to get estimates \eqref{stima resto cal L1}, \eqref{stima resto cal L1 Delta} for the remainder ${\cal R}_1$. 
\end{proof}

In the sequel we rename in \eqref{nuovo cal L2 prima}-\eqref{coefficienti cal L2 nuovo-R2} the space variable $ y $ by $ x $.

\section{Symmetrization of the order $1/2$}\label{sec:Symm}

The aim of this section is to conjugate the operator $\mL_2$ defined in \eqref{nuovo cal L2 prima} 
to a new operator $\mL_4$ in which the highest order derivatives appear in the off-diagonal entries with the same order and opposite coefficients (see \eqref{0104.2}-\eqref{0104.14}). In the complex variables $(u, \bar u)$ that we will introduce in Section \ref{sec:10}, this amounts to the symmetrization of the linear operator at the 
highest order, see \eqref{0104.21}-\eqref{resti-bis}.

We first conjugate $\mL_2$ by  the real, even and reversibility preserving transformation 
\begin{equation}
\label{M_2}
\mM_2 := \begin{pmatrix} 
\Lm_\h & \  0  \vspace{1mm} \\
0  & \ \Lm_\h^{-1}
\end{pmatrix}, 
\end{equation}
where 
$\Lambda_\h$ is the Fourier multiplier, acting on the periodic functions,
\begin{equation}\label{formula compatta Lambda h inverso}
\Lambda_\h := \pi_0 + |D|^{\frac14} T_\h^{\frac14}\,, 
\qquad \text{with\,\,inverse} \qquad   \Lambda_\h^{- 1} = \pi_0 + |D|^{- \frac14} T_\h^{- \frac14}\,,
\end{equation}
with $T_\h = \tanh ( {\mathtt h} |D| ) $   
and $\pi_0$  defined in \eqref{def pi0}.  
The conjugated operator is 
\begin{equation} \label{mL2 costruz}
\mL_3 
:= \mM_2^{-1} \mL_2 \mM_2 
= \om \cdot \pa_\vphi +  \begin{pmatrix}  
\Lm_\h^{-1} a_4 \Lm_\h \ & \ \Lm_\h^{-1} (- a_5 \pa_x \mH T_\h + \mR_2) \Lm_\h^{-1}  \\
\Lm_\h a_6 \Lm_\h \ & \ 0  \end{pmatrix}
=: \om \cdot \pa_\vphi + \begin{pmatrix} A_3 & B_3 \\ 
C_3 & 0 \end{pmatrix}.
\end{equation}
We develop the operators in \eqref{mL2 costruz} up to order $- 1/2$.  
First we write
\begin{align}
 A_3 & = \Lm_\h^{-1} a_4 \Lm_\h = a_4 + {\cal R}_{A_3} 
  \qquad {\rm where} \qquad {\cal R}_{A_3} := [\Lm_\h^{-1},  a_4] \Lm_\h  \in OPS^{- 1}   \label{espansione A2 simmetrizzazione} 
 \end{align}
by Lemma \ref{lemma tame norma commutatore}. 
Using that $|D|^m \pi_0 = \pi_0 |D|^m = 0$ for any $m \in \R$ and that $\pi_0^2 = \pi_0$
on the periodic functions, one has   
\begin{align}
C_3 & =  \Lm_\h a_6 \Lm_\h = a_6 \Lambda_\h^2 + [\Lambda_\h, a_6] \Lambda_\h = a_6 (\pi_0 + |D|^{\frac14} T_\h^{\frac14})^2  + [\Lambda_\h, a_6]  \Lambda_\h\nonumber\\
& = a_6 |D|^{\frac12} T_\h^{\frac12} + \pi_0 + {\cal R}_{C_3}
\qquad {\rm where } \qquad 
{\cal R}_{C_3} := (a_6 - 1) \pi_0 + [\Lambda_\h, a_6] \Lambda_\h   \,. \label{espansione C2}
\end{align}
Using that $|D| = \mH \partial_x$, \eqref{formula compatta Lambda h inverso} and 
$|D| \pi_0 = 0$ on the periodic functions, 
we write $B_3$ in \eqref{mL2 costruz} as 
\begin{align}
B_3 & = \Lm_\h^{-1} (- a_5 \pa_x \mH T_\h + \mR_2) \Lm_\h^{-1} 
= - a_5 |D| T_\h \Lambda_\h^{-2} - [\Lm_\h^{-1}, a_5] |D| T_\h \Lm_\h^{-1} 
+ \Lm_\h^{-1} \mR_2 \Lm_\h^{-1} \nonumber \\
& 
= - a_5 |D| T_\h \big(\pi_0 + |D|^{- \frac14} T_\h^{- \frac14} \big)^2 
- [ \Lm_\h^{-1}, a_5] |D| T_\h \Lm_\h^{-1} 
+ \Lm_\h^{-1} \mR_2 \Lm_\h^{-1} \nonumber \\
& = - a_5 |D|^{\frac12} T_\h^{\frac12} + {\cal R}_{B_3} \qquad 
{\rm where} \qquad 
{\cal R}_{B_3} := - [\Lm_\h^{-1}, a_5] |D| T_\h \Lm_\h^{-1} + \Lm_\h^{-1} \mR_2 \Lm_\h^{-1} \,.   \label{espansione B2 simmetrizzazione}
\end{align}

\begin{lemma}\label{stime simmetrizzazione ordine principale}
The operators $\Lambda_{\mathtt h} \in OPS^{\frac14}$, 
$\Lambda_{\mathtt h}^{-1} \in OPS^{-\frac14}$ 
and ${\cal R}_{A_3}, {\cal R}_{B_3}, {\cal R}_{C_3} \in OPS^{-\frac12} $. Furthermore, there exists $\sigma(k_0, \tau, \nu) > 0$ such that for any $\alpha > 0$, assuming \eqref{ansatz I delta} with $\mu_0 \geq \sigma + \alpha$, then for all $s \geq s_0$,  
\begin{align}\label{stima Lambda mathtt h}
& \norma \Lambda_{\mathtt h} \norma_{\frac14, s, \alpha}^{k_0, \gamma}\,, \norma \Lambda_{\mathtt h}^{- 1} \norma_{- \frac14, s, \alpha}^{k_0, \gamma} \lesssim_\alpha 1\,, \\
& \label{stima cal R 3 B C D}
\norma {\cal R} \norma_{- \frac12, s, \alpha}^{k_0, \gamma} \lesssim_{s,  \alpha} 
\e \g^{-1} \big(1 + \| \fracchi_0 \|_{s + \sigma  +  \alpha}^{k_0, \gamma} \big)\,, \quad \norma \Delta_{12}{\cal R}\norma_{- \frac12, s_1, \alpha} \lesssim_{s_1, \alpha} \e \g^{-1} \| \Delta_{12} i \|_{s_1 + \sigma  +  \alpha}
\end{align}
for all $\mR \in \{ \mR_{A_3}, \mR_{B_3}, \mR_{C_3} \}$. The operator $\mL_3$ in \eqref{mL2 costruz} is real, even and reversible. 
\end{lemma}

\begin{proof}
The lemma follows by the definitions of ${\cal R}_{A_3}$, ${\cal R}_{B_3}$, ${\cal R}_{C_3}$ in 
\eqref{espansione A2 simmetrizzazione}, \eqref{espansione B2 simmetrizzazione}, \eqref{espansione C2},  by Lemmata \ref{lemma stime Ck parametri} and \ref{lemma tame norma commutatore}, 
recalling \eqref{norma a moltiplicazione} and using 
\eqref{stime a4 a5 a6 nuovo}, \eqref{stime cal R2 nuovo}.
\end{proof}

Consider now a transformation $\mM_3$ of the form
\begin{equation} \label{0104.1}
\mM_3 := \begin{pmatrix} p & 0 \\ 
0 & 1 \end{pmatrix},
\quad 
\mM_3^{-1} = \begin{pmatrix} p^{-1} & 0 \\ 
0 & 1 \end{pmatrix},
\end{equation}
where $p(\ph,x)$ is a real-valued periodic function, with $p-1$ small (see \eqref{0104.14}).
The conjugated operator is 
\begin{equation} \label{0104.2}
\mL_4 := \mM_3^{-1} \mL_3 \mM_3 
= \om \cdot \pa_\vphi + \begin{pmatrix} 
p^{-1} (\ompaph p) + p^{-1} A_3 p \ & \ p^{-1} B_3 \\ 
C_3 p & 0 \end{pmatrix}
= \om \cdot \pa_\vphi + \begin{pmatrix} A_4 & B_4 \\ 
C_4 & 0 \end{pmatrix} 
\end{equation}
where, recalling \eqref{espansione A2 simmetrizzazione},  
\eqref{espansione B2 simmetrizzazione}, \eqref{espansione C2}, 
one has 
\begin{align}
A_4 & =  \breve a_4 + {\cal R}_{A_4}\,, \quad 
\breve a_4 := a_4 + p^{-1} (\ompaph p)\,, \quad {\cal R}_{A_4} := p^{- 1}{\cal R}_{A_3} p \label{definizione A3 simmetrizzazione}  \\
B_4 & =  - p^{- 1}a_5 |D|^{\frac12} T_\h^{\frac12}  + {\cal R}_{B_4}\,, \quad {\cal R}_{B_4} := p^{- 1}{\cal R}_{B_3} \label{definizione B3 simmetrizzazione} \\
C_4 & = a_6 p  |D|^{\frac12} T_\h^{\frac12} + \pi_0 + {\cal R}_{C_4}\,, \quad {\cal R}_{C_4} :=  a_6 [|D|^{\frac12} T_\h^{\frac12}, p] + \pi_0 (p - 1)+ {\cal R}_{C_3} p  \label{definizione C3 simmetrizzazione}
\end{align}
and therefore ${\cal R}_{A_4}, {\cal R}_{B_4}, {\cal R}_{C_4} \in OPS^{- \frac12}$. 
The coefficients of the highest order term in $B_4$ in \eqref{definizione B3 simmetrizzazione} 
and $C_4$ in \eqref{definizione C3 simmetrizzazione} are opposite if $a_6 p = p^{-1} a_5$. 
Therefore we fix the real valued function
\begin{equation}  \label{0104.14}
p := \sqrt{\frac{a_5}{a_6}}\,, \qquad 
a_6 p = p^{-1} a_5 = \sqrt{a_5 a_6} \, 
=: a_7 \,.
\end{equation} 

\begin{lemma}\label{lemma simmetrizzazione ordine principale}
There exists $\sigma := \sigma(\tau, \nu, k_0) > 0$ such that for any $\alpha > 0$, assuming \eqref{ansatz I delta} with $\mu_0 \geq \sigma + \alpha$, then for any $s \geq s_0$ the following holds. 
The transformation $\mM_3$ defined in \eqref{0104.1} 
is real, even and reversibility preserving and  
satisfies 
\begin{equation}\label{stime M3}
\norma \mM^{\pm 1}_3 - {\rm Id} \norma_{0, s, 0}^{k_0, \gamma} 
\lesssim_s \e \g^{-1} \big(1 + \| \fracchi_0 \|_{s + \sigma}^{k_0, \gamma} \big) \,.
\end{equation}
The real valued functions $\breve a_4, a_7$ defined in \eqref{definizione A3 simmetrizzazione}, \eqref{0104.14} satisfy  
\begin{equation} \label{0104.15}
\breve a_4 = {\rm odd}(\vphi) {\rm even}(x)\,, \quad 
\quad a_7 = {\rm even}(\vphi) {\rm even}(x)\, ,
\end{equation}
and, for any $ s \geq s_0 $, 
\begin{equation}\label{stime an bn cn (3)}
\| \breve a_4 \|_s^{k_0, \gamma}, \| a_7 - 1\|_s^{k_0, \gamma}  \lesssim_{s} \e \g^{-1} 
\big(1 + \| \fracchi_0 \|_{s + \sigma}^{k_0, \gamma} \big)\,.
\end{equation}
 The remainders ${\cal R}_{A_4}, {\cal R}_{B_4}, {\cal R}_{C_4} \in OPS^{ - \frac12}$ defined in \eqref{definizione A3 simmetrizzazione}-\eqref{definizione C3 simmetrizzazione} satisfy 
\begin{equation}\label{stime R A3 B3 C3 N}
\norma {\cal R} \norma_{ - \frac12, s, \alpha}^{k_0, \gamma} \lesssim_{s,  \alpha} 
\e \g^{-1} \big(1 + \| \fracchi_0 \|_{s + \sigma  +  \alpha}^{k_0, \gamma} \big)\,, \qquad {\cal R} \in \{ {\cal R}_{A_4}, {\cal R}_{B_4}, {\cal R}_{C_4}\}\,.
\end{equation}
Let $ i_1, i_2$ be given embedded tori. Then 
\begin{align}\label{stime Delta M3}
& \norma \Delta_{12} \mM^{\pm 1}_3 \norma_{0, s_1, 0} \lesssim_{s_1} \e \g^{-1} \| \Delta_{12} i\|_{s_1 + \sigma}\,, \\
\label{stime Delta an bn cn (3)}
& \|\Delta_{12} \breve a_4 \|_{s_1}, \| \Delta_{12} a_7\|_{s_1} \lesssim_{s_1} \e \g^{-1} \| \Delta_{12} i  \|_{s_1 + \sigma }\, \, , \\
\label{stime Delta R A3 B3 C3 N}
& \norma \Delta_{12}{\cal R} \norma_{ - \frac12, s_1, \alpha} \lesssim_{s_1,  \alpha} 
\e \g^{-1} \| \Delta_{12} i \|_{s_1 + \sigma  + \alpha}\,, \qquad {\cal R} \in \{ {\cal R}_{A_4}, {\cal R}_{B_4}, {\cal R}_{C_4}\}\,.
\end{align}
The operator $ \mL_4 $ in \eqref{0104.2} is real, even and reversible.
\end{lemma} 

\begin{proof}
By \eqref{parity a4 a5 a6},
the functions $a_5, a_6 $ are $ \even(\ph)\even(x)$, and therefore
$p $ is $ \even(\ph) \even(x)$. Moreover, since $a_4$ is $\odd(\ph)\even(x)$, we 
deduce \eqref{0104.15}.
Since $p$ is $\even(\ph) \even(x)$, 
the transformation $\mM_3$ is real, even and reversibility preserving. 

By definition \eqref{0104.14}, Lemma \ref{Moser norme pesate}, the interpolation estimate \eqref{p1-pr} and applying estimates \eqref{stime a4 a5 a6 nuovo} on $a_5$ and $a_6$, one gets that $p$ satisfies the estimates 
\begin{equation}\label{stime p}
\| p^{\pm 1} - 1 \|_s^{k_0, \gamma} \lesssim_s \e \g^{-1} 
\big(1 + \| \fracchi_0 \|_{s + \sigma}^{k_0, \gamma}\big)\,, \quad  \| \Delta_{12} p^{\pm 1}\|_{s_1} \lesssim_{s_1}  \e \g^{-1} \| \Delta_{12} i \|_{s_1 + \sigma}
\end{equation}
for some $\sigma = \sigma(\tau, \nu, k_0) > 0$. 
Hence estimates \eqref{stime M3}, \eqref{stime Delta M3} for $\mM_3^{\pm 1}$ follow 
by definition \eqref{0104.1}, using estimates \eqref{norma a moltiplicazione}, \eqref{stime p}.
Estimates \eqref{stime an bn cn (3)}, \eqref{stime Delta an bn cn (3)} for $\breve a_4, a_7$ follow by definitions \eqref{definizione A3 simmetrizzazione}, \eqref{0104.14} and applying estimates \eqref{stime a4 a5 a6 nuovo} on $a_4$, $a_5$ and $a_6$, estimates \eqref{stime p} on $p$, Lemma \ref{Moser norme pesate} and the interpolation estimate \eqref{p1-pr}. 
Estimates \eqref{stime R A3 B3 C3 N}, \eqref{stime Delta R A3 B3 C3 N} follow by 
definitions \eqref{definizione A3 simmetrizzazione}-\eqref{definizione C3 simmetrizzazione}, 
estimate \eqref{norma a moltiplicazione}, 
Lemmata \ref{lemma stime Ck parametri} and \ref{lemma tame norma commutatore}, 
bounds \eqref{stime a4 a5 a6 nuovo} on $a_4, a_5, a_6$, 
\eqref{stime p} on $p$, 
and Lemma \ref{stime simmetrizzazione ordine principale}. 
\end{proof}

\section{Symmetrization of the lower orders}\label{sec:10}

To symmetrize the linear operator $\mL_4 $ in \eqref{0104.2}, with $ p $ fixed in \eqref{0104.14}, 
at lower orders, 
it is convenient to introduce the complex coordinates $(u,\bar u):= \mC^{-1}(\eta,\psi)$, with $\mC$ defined in \eqref{def:mC}, namely $u = \eta + \ii \psi$, $\bar u = \eta - \ii \psi$. 
In these  complex coordinates the linear operator $\mL_4 $ becomes, 
using \eqref{operatori in coordinate complesse} and \eqref{0104.14}, 
\be \label{0104.21}
\mL_5 
:= \mC^{-1} \mL_4 \mC
= \ompaph 
+ \ii a_7 |D|^{\frac12} T_\h^{\frac12} \Sigma + a_8  {\mathbb I}_2 + \ii \Pi_0 
+ {\cal P}_5 + {\cal Q}_5\,,  
\qquad a_8 := \frac{\breve a_4}{2}\,,
\ee
where the real valued functions $ a_7 $, $\breve a_4$  
are defined in \eqref{0104.14}, \eqref{definizione A3 simmetrizzazione} and 
satisfy \eqref{0104.15},  
\begin{align} \label{def pauli matrix cal R4}
& \Sigma := \begin{pmatrix}
1 & 0 \\
0 & - 1
\end{pmatrix}  \, , \qquad 
\Pi_0 := \frac12 \begin{pmatrix}
 \pi_0 &  \pi_0 \\
- \pi_0 & -  \pi_0
\end{pmatrix} \,, 
\qquad 
{\mathbb I}_2 := 
\begin{pmatrix}
1 & 0 \\
0 & 1
\end{pmatrix} \, , 
\end{align}
$\pi_0$ is defined in \eqref{def pi0}, 
and
\be\label{resti-bis}
\begin{aligned}
& \qquad \qquad \quad 
{\cal P}_5 := \begin{pmatrix}
P_5 & 0 \\
0 & \overline P_5
\end{pmatrix}\,, 
\quad  {\cal Q}_5 := \begin{pmatrix}
0 & Q_5 \\
\overline Q_5 & 0
\end{pmatrix}\,,   \\
&  P_5  := \frac12 \big\{ {\cal R}_{A_4}  + \ii ({\cal R}_{C_4} - {\cal R}_{B_4}) \big\}\,,
\quad \quad 
Q_5  := a_8 + \frac12 \big\{ {\cal R}_{A_4}  + \ii ({\cal R}_{C_4} + {\cal R}_{B_4}) \big\}\, . 
\end{aligned}
\ee
By the estimates of Lemma \ref{lemma simmetrizzazione ordine principale} we have
\begin{align}
& \| a_7 - 1\|_s^{k_0, \gamma} \lesssim_s \e \g^{-1} \big(1 + \| \fracchi_0\|_{s + \sigma}^{k_0, \gamma}\big)\,, \quad \| \Delta_{12} a_7\|_{s_1} \lesssim_{s_1} \e \g^{-1} \| \Delta_{12} i \|_{s_1 + \sigma} \label{stime b0 (4) prima di decoupling} \\
& \| a_8\|_s^{k_0, \gamma} \lesssim_s \e \g^{-1} \big(1 + \| \fracchi_0\|_{s + \sigma}^{k_0, \gamma} \big)\,, \quad \| \Delta_{12} a_8\|_{s_1} \lesssim_{s_1} \e \g^{-1} \| \Delta_{12} i \|_{s_1 + \sigma}\,, \label{stime a0 (4) prima di decoupling} \\
& \norma {\cal P}_5 \norma_{- \frac12, s, \alpha}^{k_0, \gamma}\,,\,
\norma {\cal Q}_5 \norma_{ 0, s, \alpha}^{k_0, \gamma} 
\lesssim_{s, \alpha} \e \g^{-1} \big(1 + \| \fracchi_0 \|_{s + \sigma + \alpha}^{k_0, \gamma} \big) \label{stima cal P4 Q4 inizio decoupling} \\
& \norma \Delta_{12} {\cal P}_5 \norma_{- \frac12, s_1, \alpha}\,,\, 
\norma \Delta_{12} {\cal Q}_5 \norma_{ 0, s_1, \alpha} \lesssim_{s_1, \alpha} \e \g^{-1} \| \Delta_{12} i \|_{s_1 + \sigma + \alpha}\,. \label{stima derivate cal P4 Q4 inizio decoupling}
\end{align}
Now we define inductively a finite number of transformations to remove all the terms of orders $\geq -M$ 
from the off-diagonal operator $ {\cal Q}_5$. The constant $M$ will be fixed in \eqref{relazione mathtt b N}.

Let ${\cal L}_5^{(0)}:= {\cal L}_5$, $P_{5}^{(0)} := P_5$ and $Q_{5}^{(0)} := Q_5$. 
In the rest of the section we prove the following inductive claim: 
\begin{itemize}
\item 
{\sc Symmetrization of ${\cal L}_5^{(0)} $ in decreasing orders.} 
For $m \geq 0$, there is a real,  even and reversible 
 operator of the form 
\begin{equation}\label{cal Lm}
{\cal L}_5^{(m)} := \omega \cdot \partial_\vphi + \ii a_7 |D|^{\frac12} T_\h^{\frac12} \Sigma + a_8 {\mathbb I}_2 + \ii \Pi_0 +  {\cal P}_5^{(m)} + {\cal Q}_5^{(m)}\,,
\end{equation}
where 
\be
\begin{aligned}\label{cal Rm cal Qm decoupling}
{\cal P}_5^{(m)} = \begin{pmatrix}
P_5^{(m)} &0\\
0 & \overline P_5^{(m)}
\end{pmatrix}\,, \quad {\cal Q}_5^{(m)} = \begin{pmatrix}
0 & Q_5^{(m)}\\
\overline Q_5^{(m)} & 0
\end{pmatrix}\,, \\  
 P_5^{(m)} = {\rm Op}(p_m) \in OPS^{-\frac12}\,, \quad 
 Q_5^{(m)} = {\rm Op}(q_m)\in OPS^{- \frac{m}{2} }\, . 
\end{aligned}
\ee
For any $\alpha \in \N$, assuming \eqref{ansatz I delta} with $\mu_0 \geq \aleph_4(m, \alpha) + \sigma$, 
where the increasing constants $ \aleph_4 (m, \alpha)$  are defined inductively by 
\begin{equation}\label{definizione kn(alpha)}
\aleph_4(0, \alpha) := \alpha \,, \qquad  \aleph_4({m + 1}, \alpha) :=  \aleph_4(m, \alpha + 1) + \frac{m}{2} + 2 \alpha + 4 \, ,
\end{equation}
we have
\begin{align}
& \norma {\cal P}_5^{(m)} \norma_{- \frac12, s, \alpha}^{k_0, \gamma} \, , \ \norma {\cal Q}_5^{(m)} \norma_{- \frac{m}{2}, s, \alpha} \lesssim_{m, s, \alpha} \e \g^{-1} \big(1 + \| \fracchi_0\|_{s + \aleph_4(m, \alpha) + \sigma}^{k_0, \gamma} \big)\,, \label{stima induttiva disaccoppiamento} \\
& \norma \Delta_{12} {\cal P}_5^{(m)} \norma_{- \frac12 , s_1, \alpha}\,,\, \norma \Delta_{12} {\cal Q}_5^{(m)} \norma_{- \frac{m }{2}, s_1, \alpha} \lesssim_{m, s_1, \alpha}  \e \g^{-1} \| \Delta_{12} i\|_{s_1 + \aleph_4(m, \alpha) + \sigma} \label{stima induttiva disaccoppiamento derivate in i}\, .
\end{align}
For $ m \geq 1 $,  there exist real, even, reversibility preserving, invertible maps $ \Phi_{m-1} $ 
of the form 
\be\label{defPhi-n-m-1}
\begin{aligned}
& \Phi_{m-1} := {\mathbb I}_2 + \Psi_{m-1}\,, \qquad \Psi_{m-1} := \begin{pmatrix}
0 & \psi_{m-1}(\vphi, x, D) \\
\overline {\psi_{m-1}(\vphi, x, D)} & 0
\end{pmatrix}\,,
\end{aligned}
\ee
with $ \psi_{m-1}(\vphi, x, D) $ 
in $ OPS^{- \frac{m-1}{2} - \frac12} $, 
such that   
\be\label{def:L5m-indu}
{\cal L}_5^{(m )} = \Phi_{m-1}^{- 1} {\cal L}_5^{(m-1)} \Phi_{m-1} \, . 
\ee
\end{itemize}
{\bf Initialization.} 
The real, even and reversible operator $ {\cal L}_5^{(0)} = {\cal L}_5 $ in \eqref{0104.21} 
satisfies the assumptions
\eqref{cal Lm}-\eqref{stima induttiva disaccoppiamento derivate in i} for $m = 0$ 
by \eqref{stima cal P4 Q4 inizio decoupling}-\eqref{stima derivate cal P4 Q4 inizio decoupling}.
\\[1mm]
\noindent
{\bf Inductive step.}  We conjugate $ {\cal L}_5^{(m)} $ in \eqref{cal Lm} by  a
real  operator of the form (see \eqref{defPhi-n-m-1}) 
\be\label{defPhi-n}
\begin{aligned}
& \Phi_m := {\mathbb I}_2 + \Psi_m\,, \ \Psi_m := \begin{pmatrix}
0 & \psi_m(\vphi, x, D) \\
\overline {\psi_m(\vphi, x, D)} & 0
\end{pmatrix}\,, \ 
 \psi_m(\vphi, x, D) := {\rm Op} ( \psi_m )  \in OPS^{- \frac{m}{2} - \frac12}\, . 
\end{aligned}
\ee
We compute 
\begin{align}
{\cal L}_5^{(m)} \Phi_m & =  \Phi_m \big( \Dom   + \ii a_7 |D|^{\frac12} T_\h^{\frac12} \Sigma + a_8{\mathbb I}_2 + \ii \Pi_0 +  {\cal P}_5^{(m)}  \big) \nonumber\\
& \quad + \big[\ii a_7 |D|^{\frac12} T_\h^{\frac12} \Sigma  + a_8 {\mathbb I}_2 + \ii \Pi_0 +   
{\cal P}_5^{(m)}, \Psi_m \big] + (\Dom \Psi_m)
+ {\cal Q}_5^{(m)} + {\cal Q}_5^{(m)} \Psi_m\,. \label{coniugato-n-decoupling}
\end{align}
In the next lemma we choose $ \Psi_m $ to decrease the order of the off-diagonal operator $ {\cal Q}_5^{(m)} $. 

\begin{lemma}\label{lemma-bassan-dec}
Let 
\be\label{def:psin} 
\psi_m(\vphi, x, \xi) := \begin{cases} - \dfrac{\chi(\xi) q_m(\vphi, x, \xi) }{2 \ii a_7(\vphi, x) |\xi|^{\frac12} 
\tanh^{\frac12} (\h |\xi|)} \ \, \   
\qquad \ \rm{if} \quad |\xi| > \frac13 \, , \\
0 \qquad \qquad \qquad\qquad \, \qquad \qquad \qquad \ \rm{if}  \quad   |\xi| \leq \frac13\,,
\end{cases} \qquad \psi_m \in S^{- \frac{m}{2} - \frac12} \, , 
\ee 
where the cut-off function $\chi$ is defined in \eqref{cut off simboli 1}. 
Then the operator $ \Psi_m $ in \eqref{defPhi-n} solves  
\be\label{ourgoaln}
\ii \big[ a_7 |D|^{\frac12} T_\h^{\frac12} \Sigma, \Psi_m \big] + {\cal Q}_5^{(m)}  =   {\cal Q}_{ \psi_m} 
\ee
where
\begin{equation}\label{R T psi n definizione}
 {\cal Q}_{\psi_m} :=   \begin{pmatrix}
0 & q_{ \psi_m}(\vphi, x, D) \\
\overline{q_{\psi_m}(\vphi, x, D)} 
\end{pmatrix},\quad  q_{\psi_m} \in S^{- \frac{m}{2} - 1}  \, .
\end{equation}
Moreover, there exists $\sigma(k_0, \tau, \nu) > 0$ such that,
for any $\alpha > 0$, if \eqref{ansatz I delta} holds with 
$\mu_0 \geq  \aleph_4(m, \alpha + 1)+ \alpha + \frac{m}{2} + \sigma + 4 $, then 
\begin{equation}\label{R T psi n}
\norma q_{\psi_m}(\vphi, x, D)  \norma_{ - \frac{m}{2} - 1, s, \alpha}^{k_0, \gamma} 
\lesssim_{s, \alpha}   \e \g^{-1} \big( 1 + \| \fracchi_0 \|_{s +  \aleph_4(m, \alpha + 1) + \frac{m}{2} + \alpha  + \sigma + 4}^{k_0, \gamma} \big)\, . 
\end{equation}
The map $ \Psi_m $ is real, even, reversibility preserving and
\begin{align}\label{stima psi n}
& \norma \psi_m(\vphi, x, D) 
\norma_{ - \frac{m}{2} - \frac12 , s, \alpha}^{k_0, \gamma}  
\lesssim_{m, s, \alpha} \e \g^{-1} \big(1 + \| \fracchi_0 \|_{s + \sigma + \aleph_4(m, \alpha)}^{k_0, \gamma} \big)\,,\\
& \label{derivate i psi n}  \norma \Delta_{12}\psi_m( \vphi, x, D)  \norma_{- \frac{m}{2} - \frac12, s_1, \alpha} \lesssim_{m, s_1, \alpha} 
\e \g^{-1} \| \Delta_{12} i \|_{s_1 + \sigma + \aleph_4(m, \alpha)}\,, \\
& \label{derivate i R T psi n}
 \norma \Delta_{12} q_{\psi_m}(\vphi, x, D)  \norma_{- \frac{m}{2} - 1 , s_1, \alpha} \lesssim_{m, s_1, \alpha} \e \g^{-1} \| \Delta_{12} i \|_{s_1 + \aleph_4(m, \alpha + 1) + \frac{m}{2} + \alpha  + \sigma + 4  }\,.
\end{align}
\end{lemma}

\begin{proof}
We first note that in \eqref{def:psin} the denominator 
$ a_7 |\xi|^{\frac12} \tanh(\h |\xi|)^{\frac12} \geq c | \xi |^{\frac12}  $ with $ c >  0 $
for all $ | \xi | \geq 1 / 3 $, since $ a_7 - 1 = O(\e \gamma^{-1})$ by  \eqref{stime an bn cn (3)} and \eqref{ansatz I delta}.
Thus the symbol $ \psi_m $ is well defined and estimate  \eqref{stima psi n} follows by \eqref{def:psin},  
\eqref{lemma composizione multiplier} and \eqref{stima induttiva disaccoppiamento}, \eqref{stime an bn cn (3)},
Lemma \ref{Moser norme pesate}, \eqref{ansatz I delta}. Recalling the definition \eqref{def pauli matrix cal R4} of $\Sigma$,
the vector valued commutator $\ii [ a_7 |D|^{\frac12} T_\h^{\frac12} \Sigma, \Psi_m ] $ is 
\be\label{VVcommutator} 
\ii \big[ a_7 |D|^{\frac12} T_\h^{\frac12} \Sigma, \Psi_m \big]  
=   \begin{pmatrix}
0 & A  \\
\bar A  & 0
\end{pmatrix} \, , \qquad 
A := \ii \big(a_7 |D|^{\frac12} T_\h^{\frac12}  {\rm Op} ( \psi_m )  + {\rm Op} ( \psi_m )  a_7 |D|^{\frac12} T_\h^{\frac12} \big) \, . 
\ee
 By \eqref{VVcommutator}, in order to solve \eqref{ourgoaln}
 with a  remainder $ {\cal Q}_{\psi_m} \in OPS^{- \frac{m}{2} - 1 } $  as in \eqref{R T psi n definizione},  
we have to solve 
\be\label{sol-voluta}
\ii a_7 |D|^{\frac12} T_\h^{\frac12} {\rm Op} ( \psi_m )  + 
\ii {\rm Op} ( \psi_m ) a_7 |D|^{\frac12} T_\h^{\frac12} + 
{\rm Op} ( q_m ) = {\rm Op} ( q_{\psi_m}) \in OPS^{- \frac{m}{2} - 1} \, . 
\ee
By \eqref{expansion symbol}, applied with $N = 1$, $ A = a_7 |D|^{\frac12} T_\h^{\frac12} $, $ B = {\rm Op}(\psi_m) $, and
\eqref{definizione Dm}, we have  the expansion
\begin{align}
& a_7 |D|^{\frac12} T_\h^{\frac12}  {\rm Op}(\psi_m)+ {\rm Op}(\psi_m) a_7 |D|^{\frac12} T_\h^{\frac12}   =  {\rm Op}\big(2 a_7 |\xi|^{\frac12} \tanh^{\frac12} (\h |\xi|) \psi_m  \big) 
+ {\rm Op}({\mathfrak q}_{\psi_m}) 
\label{primo-svi1}
\end{align}
where,  using that 
$ a_7 \chi(\xi) |\xi|^{\frac12} \tanh^{\frac12} (\h \chi (\xi)  |\xi|) \in S^{\frac12} $ and 
$ \psi_m  \in S^{-\frac{m}{2} -\frac12} $,  the symbol
\be\label{eq:frakq}
\mathfrak q_{\psi_m} = r_{1, AB} + r_{1, BA}  + 
2 a_7 |\xi|^{\frac12} 
\big(  \tanh^{\frac12} (\h \chi (\xi) |\xi|) \chi (\xi)  -  \tanh^{\frac12} (\h |\xi|) \big)  
\psi_m  
\in S^{- \frac{m}{2} - 1} \, , 
\ee
recalling that $ 1- \chi (\xi)  \in S^{-\infty} $ by  \eqref{cut off simboli 1}.
The symbol  $ \psi_m  $ in \eqref{def:psin} is the solution of  
\be\label{equazione-per-decrescere}
2\ii a_7 |\xi|^{\frac12} \tanh^{\frac12} (\h |\xi|) \psi_m  + \chi(\xi) q_m   = 0 \, ,
\ee
and therefore, by  \eqref{primo-svi1}-\eqref{equazione-per-decrescere}, 
the remainder $ q_{\psi_m}  $ in \eqref{sol-voluta} is 
\begin{equation}\label{puffo 0}
q_{\psi_m}   = \ii  \mathfrak q_{ \psi_m} + (1 - \chi(\xi)) q_m \in S^{- \frac{m}{2}- 1}\,.
\end{equation}
This proves \eqref{ourgoaln}-\eqref{R T psi n definizione}. We now prove \eqref{R T psi n}. 
We first estimate \eqref{eq:frakq}. 
By \eqref{stima RN derivate xi parametri} (applied  with $N = 1$, 
$A = a_7 |D|^{\frac12} T_\h^{\frac12}$, $B = {\rm Op}(\psi_m) $, $m = 1/2$, $m' 
= - \frac{m}{2} -  \frac12 $ and also by inverting the role of $A$ and $B$), and the estimates 
\eqref{stima psi n},  \eqref{stime b0 (4) prima di decoupling}, \eqref{ansatz I delta}
we have $\norma \mathfrak q_{\psi_m}(\vphi, x, D) \norma_{- \frac{m}{2} - 1, s, \alpha}^{k_0, \gamma} 
\lesssim_{m, s, \alpha}
\e \g^{-1} \big( 1 + \| \fracchi_0 \|^{k_0, \gamma}_{s + \sigma + \aleph_4(m, \alpha + 1) + \frac{m}{2} + \alpha + 4} \big)$
and the estimate \eqref{R T psi n} for $q_{ \psi_m}(\vphi, x, D)$  follows by \eqref{puffo 0} 
 using  \eqref{stima induttiva disaccoppiamento}, recalling that 
 $1 - \chi(\xi) \in S^{- \infty}$ and by applying \eqref{lemma composizione multiplier} with $g(D) = 1 - \chi(D)$ and 
 $A = q_m(\vphi, x, D)$. 
Bounds \eqref{derivate i psi n}-\eqref{derivate i R T psi n} follow by similar arguments and by a repeated use of the triangular inequality. 

Finally, the map $ \Psi_m $ defined by \eqref{defPhi-n}, \eqref{def:psin}  is real, even and reversibility preserving 
because $ {\cal Q}_5^{(m)} $ is real, even, reversible and $a_7$ is ${\rm even}(\vphi) {\rm even}(x)$ (see \eqref{0104.15}).   
\end{proof}

For $\e \g^{-1}$ small enough, by \eqref{stima psi n} and \eqref{ansatz I delta} 
the operator $\Phi_m$ is invertible, and, by Lemma \ref{Neumann pseudo diff},
\begin{equation} \label{2017.2704.1}
\norma \Phi_m^{-1} - \mathbb{I}_2 \norma^{k_0,\g}_{0,s,\a} \lesssim_{s,\a} 
\norma \Psi_m \norma^{k_0,\g}_{0,s,\a} \lesssim_{s,\a}
\e \g^{-1} \big( 1 + \| \fracchi_0 \|_{s + \sigma + \aleph_4(m, \alpha)}^{k_0, \gamma} \big) \, .
\end{equation}
By \eqref{coniugato-n-decoupling} and \eqref{ourgoaln}, 
the conjugated operator is 
\be\label{defin:Ln+1}
{\cal L}_5^{(m + 1)} := \Phi_m^{- 1} {\cal L}_5^{(m)} \Phi_m = \Dom + 
\ii a_7 |D|^{\frac12} T_\h^{\frac12} \Sigma + a_8 {\mathbb I}_2 + \ii \Pi_0 +  {\cal P}_5^{(m)} + \breve{\cal P}_{m + 1} 
\ee
where $ \breve{\cal P}_{m + 1} := \Phi_m^{- 1} {\cal P}^*_{m + 1} $ and 
\be \label{bf R n + 1-star}
{\cal P}^*_{m + 1} := 
{\cal Q}_{ \psi_m} + \big[ \ii \Pi_0, \Psi_m \big] + \big[ a_8 {\mathbb I}_2 + {\cal P}_5^{(m)}, \Psi_m \big] + 
(\omega \cdot \partial_\vphi \Psi_m) + {\cal Q}_5^{(m)} \Psi_m \, . 
\ee
Thus \eqref{def:L5m-indu} at order $ m + 1 $ is proved. 
Note that $ \breve{\cal P}_{m + 1} $ and $\Pi_0$ are the only operators in \eqref{defin:Ln+1} containing off-diagonal terms. 

\begin{lemma}\label{stima-Rn+1-dec}
The operator $\breve{\cal P}_{m + 1} \in OPS^{- \frac{m}{2} - \frac12}$. Furthermore, for any $\alpha > 0$, assuming \eqref{ansatz I delta} with $\mu_0 \geq \sigma + \aleph_4( m + 1, \alpha)$, the following estimates hold:   
\begin{align}\label{stima cal R n+1}
& \norma \breve{\cal P}_{m + 1} \norma_{ - \frac{m}{2} - \frac12, s, \alpha}^{k_0, \gamma} 
\lesssim_{m, s, \alpha} \e \g^{-1} 
\big( 1 + \| \fracchi_0 \|_{s + \sigma + \aleph_4( m + 1, \alpha)}^{k_0, \gamma} \big) \,, \quad \forall s \geq s_0 \, , \\
& \label{stima derivata i cal R n+1}
\norma \Delta_{12} \breve{\cal P}_{m + 1}  \norma_{ - \frac{m}{2} - \frac12, s_1 , \alpha} 
\lesssim_{m, s_1, \alpha} \e \g^{-1} \| \Delta_{12} i \|_{s_1 + \sigma + \aleph_4( m + 1, \alpha)}
\end{align}
where the  constant $ \aleph_4( m + 1 , \alpha )$  is defined  in \eqref{definizione kn(alpha)}. 
\end{lemma}

\begin{proof}
Use  Lemma \ref{lemma-bassan-dec}, \eqref{cal Rm cal Qm decoupling}, \eqref{defPhi-n},  \eqref{estimate composition parameters}, \eqref{stime a0 (4) prima di decoupling}, \eqref{stima induttiva disaccoppiamento}, \eqref{stima induttiva disaccoppiamento derivate in i}, \eqref{crescente-m-neg}, \eqref{bf R n + 1-star}, \eqref{2017.2704.1}. 
 \end{proof}

The operator $ {\cal L}_5^{(m +1)} $ in \eqref{defin:Ln+1} 
has the same form \eqref{cal Lm} as $\mL_5^{(m)}$ 
with diagonal operators $ {\cal P}_5^{(m + 1)} $ 
and off-diagonal operators $ {\cal Q}_5^{(m+1)} $ like in \eqref{cal Rm cal Qm decoupling},
with 
$ {\cal P}_5^{(m + 1)} + {\cal Q}_5^{(m + 1)} = \mP_5^{(m)} + \breve \mP_{m+1} $, 
satisfying 
\eqref{stima induttiva disaccoppiamento}-\eqref{stima induttiva disaccoppiamento derivate in i}
at the step $ m +  1$ 
thanks to \eqref{stima cal R n+1}-\eqref{stima derivata i cal R n+1} 
and \eqref{stima induttiva disaccoppiamento}-\eqref{stima induttiva disaccoppiamento derivate in i}
at the step $m$.
This proves the inductive claim.
Applying it $ 2M $ times (the constant $M$ will be fixed in \eqref{relazione mathtt b N}), we derive the following lemma.

\begin{lemma}\label{Lemma finale decoupling}
For any $\alpha > 0$, assuming \eqref{ansatz I delta} with $\mu_0 \geq \aleph_5(M, \alpha) + \sigma$ where the constant $\aleph_5(M, \alpha) := \aleph_4(2 M, \alpha)$ is defined recursively by \eqref{definizione kn(alpha)}, the following holds. The  real, even, reversibility preserving, invertible map 
\begin{equation}\label{def bf Phi M}
{ \bf \Phi}_{M} := \Phi_0 \circ \ldots \circ \Phi_{2 M-1} 
\end{equation}
where $\Phi_m$, $m=0,\ldots, 2M-1$, are defined in \eqref{defPhi-n},  satisfies  
\begin{align}\label{definizione bf PhiN}
& \norma {\bf \Phi}_M^{\pm 1} - {\mathbb I}_2 \norma_{0,s,0}^{k_0, \gamma}\,,\, \norma ({\bf \Phi}_M^{\pm 1} - {\mathbb I}_2)^* \norma_{0,s,0}^{k_0, \gamma} \lesssim_{s, M } 
\e \g^{-1} \big( 1 + \|\fracchi_0 \|_{s + \sigma + \aleph_5(M, 0)}^{k_0, \gamma} \big) \, , \quad \forall s \geq s_0 \, , \\
& \label{bf PhiN derivate in i}
\norma \Delta_{12} {\bf \Phi}_M^{\pm 1}  \norma_{0, s_1, 0} \,,\,\norma \Delta_{12} ({\bf \Phi}_M^{\pm 1})^*  \norma_{0, s_1, 0}  \lesssim_{M, s_1}  
\e \g^{-1} \| \Delta_{12} i \|_{s_1 + \sigma + \aleph_5 (M, 0)} \, . 
\end{align}
The map $ { \bf \Phi}_{M} $ conjugates $ {\cal L}_5 $ to the real, even and reversible operator 
\begin{equation}\label{cal LN decoupling}
{\cal L}_6  := {\bf \Phi}_M^{- 1} {\cal L}_5 {\bf \Phi}_M =  \Dom + \ii a_7 |D|^{\frac12} T_\h^{\frac12} \Sigma + a_8 {\mathbb I}_2 +  \ii \Pi_0 + {\cal P}_6 + {\cal Q}_6 
\end{equation}
where the functions $ a_7, a_8 $ are defined in \eqref{0104.14}, \eqref{0104.21}, and 
\be\label{resti prima Egorov}
{\cal P}_6 := \begin{pmatrix}
P_6 & 0 \\
0 & \overline {P}_6
\end{pmatrix} 
\in OPS^{- \frac12}  \,, \quad 
{\cal  Q}_6 := \begin{pmatrix}
0 & Q_6\\
 \overline Q_6 & 0
 \end{pmatrix} 
\in OPS^{- M} 
\ee
given by $  \mP_6 := \mP_5^{(2M)}$, $  \mQ_6  :=  \mQ_5^{(2M)} $
 in \eqref{cal Lm}-\eqref{cal Rm cal Qm decoupling} for $m = 2M$, satisfy
\begin{align}\label{stima resti prima Egorov}
& \norma {\cal P}_6 \norma_{- \frac12, s, \alpha}^{k_0, \gamma} + \norma {\cal Q}_6 
 \norma_{ - M , s, \alpha}^{k_0, \gamma} \lesssim_{M, s,\a}  
\e \g^{-1} \big(1 + \| \fracchi_0 \|_{s + \sigma + \aleph_5 (M, \alpha)}^{k_0, \gamma} \big)\,, \quad \forall s \geq s_0\,, \\
& \norma \Delta_{12} {\cal P}_6\norma_{- \frac12, s_1, \alpha} + \norma  \Delta_{12} {\cal Q}_6
 \norma_{ - M  , s_1 , \alpha} \lesssim_{M, s_1,\a}  
\e \g^{-1} \| \Delta_{12} i \|_{s_1 + \sigma + \aleph_5 (M, \alpha)} \, .  \label{stima derivate resti prima Egorov} 
\end{align}
\end{lemma}

\begin{proof}
We use \eqref{stima induttiva disaccoppiamento}, \eqref{stima induttiva disaccoppiamento derivate in i}, \eqref{defPhi-n}, \eqref{stima psi n}, \eqref{estimate composition parameters}, \eqref{2017.2704.1} and Lemma \ref{stima pseudo diff aggiunto}.
\end{proof}

\section{Reduction of the order $1/2$} \label{sec:primo semi-FIO}

We have obtained the operator $\mL_6$ in \eqref{cal LN decoupling}, 
where ${\cal P}_6 $ is in $ OPS^{- \frac12}$ and the off-diagonal term ${\cal Q}_6 $ is in $ OPS^{-M} $. 
The goal of this section is to reduce to {\it constant coefficient} the leading term 
$\ii a_7(\vphi, x) |D|^{\frac12} T_\h^{\frac12} \Sigma $.
To this end, we study how the operator 
$
{\cal L}_6
$
transforms under the action of the flow $\Phi(\tau) := \Phi(\tau, \ph)$ 
\begin{equation}\label{flow-propagator-beta}
\begin{cases}
\partial_\tau \Phi( \tau) = \ii  A(\vphi)  \Phi(\tau) \\
\Phi(0) = {\rm Id} \,,
\end{cases}
\qquad A(\vphi) := \beta (\vphi, x) |D|^{\frac12} 
\end{equation}
  where the function $\beta(\vphi, x)$ is a real valued smooth function, which will be defined in \eqref{defbetaSFIO}.  
 Since  $ \beta (\vphi, x) $ is real valued, usual energy estimates imply that the flow $ \Phi (\t, \vphi)$ is a 
bounded operator on Sobolev spaces satisfying tame estimates, see Section \ref{AppendiceA}.

Let $\Phi:= \Phi(\ph) := \Phi(1,\ph)$. Note that $\Phi^{-1}=\overline\Phi$ (see Section \ref{AppendiceA}) and 
\begin{equation} \label{2017.0316.1}
\Phi \pi_0 = \pi_0 = \Phi^{-1} \pi_0\,. 
\end{equation}
We write  the operator  $ {\cal L}_6$ in \eqref{cal LN decoupling} as 
$$ 
{\cal L}_6 =  
 \Dom +   \ii \Pi_0  + 
\begin{pmatrix} P_6^{(0)}   & Q_6 \\ \overline Q_6 & \overline P_6^{(0)} 
\end{pmatrix} 
$$ 
where $\Pi_0$ is defined in \eqref{def pauli matrix cal R4}, $Q_6$  in \eqref{resti prima Egorov}, and
\begin{equation}\label{definizione PM primo Egorov}
P_6^{(0)} := P_6^{(0)}(\vphi, x, D) :=\ii a_7 |D|^{\frac12} T_\h^{\frac12} +a_8 +  P_6
\end{equation}
with $P_6$ defined in \eqref{resti prima Egorov}. 
Conjugating  $ {\cal L}_6$ with the real operator % flow transformation
\begin{equation}\label{0806.1}
{\bf \Phi}:= \begin{pmatrix} \Phi & 0 \\ 0 & \overline\Phi \end{pmatrix} 
% = \begin{pmatrix} \Phi & 0 \\ 0 & \Phi^{-1} \end{pmatrix}
\end{equation}
we get, since ${\bf \Phi}^{- 1} \Pi_0 {\bf \Phi} = \Pi_0 {\bf \Phi}$ by \eqref{2017.0316.1},  
\begin{equation} \label{0806.2}
\mL_7 := {\bf \Phi}^{- 1}\mL_6 {\bf \Phi} 
= \ompaph + {\bf \Phi}^{-1}\big( \ompaph {\bf \Phi} \big) 
+ \ii \Pi_0 {\bf \Phi} 
+ \begin{pmatrix} \Phi^{-1} P_6^{(0)}   \Phi & \Phi^{-1} Q_6 \ov \Phi \\ 
{\ov \Phi}^{-1} \, \overline Q_6 \Phi & {\ov \Phi}^{-1} \, \ov P_6^{(0)}   \ov \Phi  \end{pmatrix}\, . 
\end{equation}
Let us study the operator
\begin{equation}\label{coniugio flusso di una PDE iperbolica}
L_7:= \ompaph + \Phi^{-1} \big( \ompaph \Phi \big) + \Phi^{-1} P_6^{(0)} \Phi \, .
\end{equation}

\noindent
{\sc Analysis of the term $\Phi^{-1} P_6^{(0)} \Phi$.}
Recalling  \eqref{flow-propagator-beta}, the operator $P (\t,\ph) := \Phi(\t,\ph)^{-1} P_6^{(0)} \Phi(\t,\ph)$ 
satisfies the equation 
$$
\pa_\tau P(\tau, \vphi) =  - \ii \Phi(\t,\ph)^{-1} \big[ A(\vphi)  , P_6^{(0)}  \big] \Phi(\t,\ph) \,. 
$$
Iterating this formula, and using the notation $ {\rm Ad}_{A(\vphi)} P_6^{(0)} := \big[ A(\vphi), P_6^{(0)} \big]$, 
we obtain the following Lie series expansion of the conjugated operator 
\begin{align} 
\Phi(1,\ph)^{-1} P_6^{(0)} \Phi(1,\ph) & = P_6^{(0)}
- \ii [A, P_6^{(0)}] +
\sum_{n=2}^{2M} \frac{(-\ii)^n}{n!} {\rm Ad}_{A(\vphi)}^n P_6^{(0)}
\nonumber \\
& \quad \ + \frac{(-\ii)^{2M+1}}{(2M)!} \int_0^1 (1- \tau)^{2M} 
\Phi(\t,\ph)^{-1}  {\rm Ad}_{A(\vphi)}^{2M+1} P_6^{(0)} \Phi(\t,\ph) \, d \tau \, .
\label{conj-op}
\end{align}
The order $M$ of the expansion will be fixed in \eqref{relazione mathtt b N}.
We remark that \eqref{conj-op} is an expansion in operators with decreasing orders (and size)  because  
each commutator with $ A(\vphi )= \b(\vphi,x)|D|^{\frac12} $ gains $ \frac12 $ order (and it has the size of $ \beta $). 
By \eqref{flow-propagator-beta} and \eqref{definizione PM primo Egorov}, % the first term of the expansion is 
\begin{align}
- \ii \big[ A, P_6^{(0)} \big] 
= \big[ \beta |D|^{\frac12}, a_7 |D|^{\frac12} \big] + 
\big[ \beta |D|^{\frac12}, a_7 |D|^{\frac12} (T_\h^{\frac12} - {\rm Id}) \big]  - 
\ii \big[ \beta |D|^{\frac12}, a_8 + P_6 \big]\,. \label{primo levi 0}
\end{align}
Moreover, by \eqref{symbol commutator}, \eqref{Expansion Moyal bracket} one has 
\begin{align} \label{primo levi 1}
[\beta |D|^{\frac12}, a_7 |D|^{\frac12} ] 
& = {\rm Op}\Big(- \ii \{ \beta \chi(\xi) |\xi|^{\frac12}, a_7 \chi(\xi) |\xi|^{\frac12} \} 
+ \mathtt r_2(\beta \chi(\xi) |\xi|^{\frac12},  a_7 \chi(\xi) |\xi|^{\frac12}) \Big) 
\\
& = \ii \big( (\partial_x \beta) a_7 - \beta(\partial_x a_7) \big) 
{\rm Op} \Big( \frac12 \chi^2(\xi) {\rm sign}(\xi) 
+ \chi(\xi) \pa_\xi \chi(\xi) |\xi| \Big) 
+ {\rm Op}\big( \mathtt r_2(\beta \chi(\xi) |\xi|^{\frac12},  a_7 \chi(\xi) |\xi|^{\frac12})  \big) 
\notag
\end{align}
where the symbol $\mathtt r_2(\beta \chi(\xi) |\xi|^{\frac12},  a_7 \chi(\xi) |\xi|^{\frac12})\in S^{-1}$ is defined according to \eqref{def:r2ab}.
Therefore \eqref{primo levi 0}, \eqref{primo levi 1} imply the expansion
\begin{align}
- \ii \big[A, P_6^{(0)} \big]  & = -\frac12 \big((\partial_x \beta) a_7 - \beta(\partial_x a_7) \big) {\cal H} + R_{A, P_6^{(0)}} \label{primo levi 2} 
\end{align}
where the remainder
\begin{align}
R_{A, P_6^{(0)}} & :=  
\ii \big((\partial_x \beta) a_7 - \beta(\partial_x a_7) \big) 
\Op \Big( \chi(\xi) \pa_\xi \chi(\xi) |\xi|  
+ \frac{1}{2} (\chi^2(\xi) - \chi (\xi)) \mathrm{sign}(\xi) \Big)  \nonumber\\
& \quad + {\rm Op}\big( \mathtt r_2(\beta \chi(\xi) |\xi|^{\frac12},  a_7 \chi(\xi) |\xi|^{\frac12})  \big) \label{primo levi 3}
 + \big[ \beta |D|^{\frac12}, a_7 |D|^{\frac12} (T_\h^{\frac12} - {\rm Id}) \big]  - \ii 
\big[ \beta |D|^{\frac12}, a_8 + P_6 \big] 
\end{align}
is an operator of order $- \frac12 $ (because of the term $ [ \beta |D|^{\frac12}, a_8 ]$). 
\\[1mm]
{\sc Analysis of the term $\ompaph + \Phi^{-1} \{ \ompaph \Phi \} =  \Phi^{-1} \circ \ompaph \circ \Phi$.}
We argue as above, differentiating 
\begin{align*}
\pa_\tau \big\{ \Phi(\tau,\ph)^{-1} \circ \ompaph \circ \Phi(\tau,\ph) \big\} & 
= - \ii \Phi(\t,\ph)^{-1} \big[ A(\vphi)  , \ompaph  \big] \Phi(\t,\ph) \\ 
& = - \ii \Phi(\t,\ph)^{-1} \big( {\rm Ad}_{A(\vphi)} \ompaph \big)  \Phi(\t,\ph) \, .
\end{align*}
Therefore, by iteration,  we get the Lie series expansion 
\begin{align}
\Phi(1,\ph)^{-1} \circ \ompaph \circ \Phi(1,\ph) 
& = \ompaph - \ii {\rm Ad}_{A(\vphi)} \ompaph 
+ \frac{(-\ii)^2}{2} {\rm Ad}_{A(\vphi)}^2 \ompaph 
%\\ & \quad  
+ \sum_{n=3}^{2M + 1} \frac{(-\ii)^n}{n!} {\rm Ad}_{A(\vphi)}^n \ompaph 
\notag \\ 
& \quad  + 
\frac{(-\ii)^{2M+2}}{(2M+1)!} \int_0^1 
(1- \tau)^{2M+1} 
\Phi(\t, \vphi)^{-1}  \big( {\rm Ad}_{A(\vphi)}^{2M+2}  \ompaph  \big)  \Phi(\t, \vphi) \, d \tau. \label{2302.2}
\end{align}
We compute the commutator
\begin{equation}\label{2302.1}
{\rm Ad}_{A(\ph)} \ompaph
= \big[ A(\vphi),  \ompaph \big] = - ( \ompaph A(\vphi)) \stackrel{\eqref{flow-propagator-beta}} = 
- (\ompaph \beta (\vphi, x)) |D|^{1/2}
\end{equation}
and, using \eqref{symbol commutator}, \eqref{Expansion Moyal bracket},
\begin{align}
{\rm Ad}^2_{A(\vphi)}\omega \cdot \partial_\vphi 
& = \big[ (\ompaph A(\vphi)) , A(\vphi)  \big] 
= \big[ (\omega \cdot \partial_\vphi \beta)|D|^{\frac12}, \beta |D|^{\frac12} \big] \nonumber\\
& = {\rm Op} \Big( - \ii \big\{ (\omega \cdot \partial_\vphi \beta) \chi(\xi) |\xi|^{\frac12}, \beta \chi(\xi) |\xi|^{\frac12} \big\}  + \mathtt r_2((\omega \cdot \partial_\vphi \beta) \chi(\xi) |\xi|^{\frac12}, \beta \chi(\xi) |\xi|^{\frac12})\Big)\,. \nonumber
\end{align}
According to \eqref{Expansion Moyal bracket} the term with the Poisson bracket is 
\begin{align*}
- \ii \big\{ (\omega \cdot \partial_\vphi \beta) \chi(\xi) |\xi|^{\frac12}, \beta \chi(\xi) |\xi|^{\frac12} \big\} 
& = \ii \big( \beta \, \omega \cdot \partial_\vphi \beta_x - \beta_x \, \omega \cdot \partial_\vphi \beta \big) 
\Big( 
\frac{1}{2} \chi(\xi)^2 {\rm sign}(\xi) 
+ \chi(\xi) \pa_\xi \chi(\xi) |\xi| \Big)
\end{align*}
and therefore 
\begin{equation}\label{2302.3}
\frac{(- \ii)^2}{2} {\rm Ad}^2_{A(\vphi)} \omega \cdot \partial_\vphi  
= \frac{1}{4} \big( \beta \, \omega \cdot \partial_\vphi \beta_x - \beta_x \, \omega \cdot \partial_\vphi \beta \big) {\cal H} 
+ R_{A, \omega \cdot \partial_\vphi}
\end{equation}
where 
\begin{align}
R_{A, \omega \cdot \partial_\vphi} 
& := - \frac{\ii}{4} \, 
\big( \beta \, \omega \cdot \partial_\vphi \beta_x - \beta_x \, \omega \cdot \partial_\vphi \beta \big) \,
{\rm Op} \Big( (\chi(\xi)^2 -  \chi (\xi) ) {\rm sign}(\xi) 
+ 2 \chi(\xi) \pa_\xi \chi(\xi) |\xi| \Big) 
\notag \\ & \quad \ 
- \frac12 \Op \Big( \mathtt r_2 \big( (\omega \cdot \partial_\vphi \beta) \chi(\xi) |\xi|^{\frac12}, \beta \chi(\xi) |\xi|^{\frac12} \big) \Big).  
\label{primo levi 7}
\end{align}
is an operator in $ OPS^{-1} $
(the first line of \eqref{primo levi 7} reduces to the zero operator when acting on the periodic functions, because $\chi^2 - \chi$ and $\pa_\xi \chi$ vanish on $\Z$).

Finally, by \eqref{2302.2}, \eqref{2302.1} and \eqref{2302.3}, we get
\begin{align}
\Phi(1,\ph)^{-1} \circ \ompaph \circ \Phi(1,\ph) 
& = \ompaph   
+ \ii  (\ompaph \beta) (\vphi, x) |D|^{\frac12} + \frac14 \big( \beta (\omega \cdot \partial_\vphi \beta_x) - \beta_x (\omega \cdot \partial_\vphi \beta) \big){\cal H} 
+ R_{A, \omega \cdot \partial_\vphi} \nonumber \\
& \quad \, 
- \sum_{n=3}^{2M+1} \frac{(-\ii)^n}{n!} {\rm Ad}_{A(\vphi)}^{n-1} \big( \ompaph A(\vphi) \big) 
\label{recur syst} \\ 
& \quad \,  
- \frac{(-\ii)^{2M+2}}{(2M+1)!} \int_0^1 
(1- \tau)^{2M+1} 
\Phi(\t, \vphi)^{-1}  \big( {\rm Ad}_{A(\vphi)}^{2M+1}  \big( \ompaph A(\vphi) \big)  \big)  \Phi(\t, \vphi) \, d \tau \, . 
\nonumber  
\end{align}
This is an expansion in operators with decreasing orders (and size). 

In conclusion, by 
\eqref{coniugio flusso di una PDE iperbolica},  \eqref{conj-op},
\eqref{definizione PM primo Egorov}, \eqref{primo levi 2}, 
\eqref{recur syst}, 
the term of order $ | D |^{\frac12} $ in $ L_7 $ in \eqref{coniugio flusso di una PDE iperbolica} is
\be\label{eq:passaggio}
\ii \big( (\ompaph \b) + a_7 T_\h^\frac12 \big) | D |^{\frac12} \, .
\ee

\noindent
{\bf Choice of the functions $\beta(\vphi, x)$ and $\alpha(x)$.}
We choose the function $ \b (\vphi, x ) $ such that 
\be\label{def:a51}
(\ompaph \b)  (\vphi, x )  + a_7  (\vphi, x ) = \langle a_7 \rangle_\vphi(x)\,, \quad  
\langle a_7 \rangle_\vphi(x):=  \frac{1}{(2 \pi)^\nu} \int_{\T^\nu}  a_7 (\vphi, x) \, d \vphi.  
\ee
For all $\om \in \mathtt{DC}(\g,\t)$, the solution of \eqref{def:a51} is the periodic function
\be\label{defbetaSFIO}
\b (\vphi, x) := - (\ompaph)^{-1} \big( a_7 (\vphi, x) - \langle a_7 \rangle_\vphi (x) \big) \, ,
\ee
which we extend to the whole parameter space $\R^\nu \times [\h_1, \h_2]$ by setting 
$ \b_{ext} := - (\omega \cdot \partial_\vphi)^{- 1}_{ext} ( a_7 - \langle a_7 \rangle_\vphi )$ 
via the operator $ (\omega \cdot \partial_\vphi)^{- 1}_{ext} $ defined in Lemma \ref{lemma:WD}.
For simplicity we still denote by $\b$ this extension.

\begin{lemma}\label{lem:stim-beta}
The real valued function $\beta  $ defined in \eqref{defbetaSFIO} is 
$\odd(\vphi) \even(x)$. Moreover   
there exists $\sigma (k_0, \tau, \nu) > 0$ such that, if \eqref{ansatz I delta} holds with $\mu_0 \geq \sigma$, then $ \beta $
satisfies the following estimates: 
\begin{align}
\| \beta\|_s^{k_0, \gamma} & \lesssim_s \e \gamma^{-2} \big(1 + \| \fracchi_0\|_{s + \sigma}^{k_0, \gamma} \big)\,, \quad \| \omega \cdot \partial_\vphi \beta \|_s^{k_0,\gamma} \lesssim_s \e \g^{-1} \big(1 + \| \fracchi_0\|_{s + \sigma}^{k_0, \gamma}\big)  \label{stima beta primo egorov} \\
\| \Delta_{12} \beta\|_{s_1}  & \lesssim_{s_1} \e \gamma^{-2} \| \Delta_{12} i\|_{s_1 + \sigma}\,, \quad \| \omega \cdot \partial_\vphi \Delta_{12} \beta\|_{s_1} \lesssim_{s_1} \e \g^{-1} \| \Delta_{12} i \|_{s_1 + \sigma} \,. 
\label{stima derivata beta primo egorov}
\end{align}
\end{lemma} 

\begin{proof}
The function $a_7 $ is $\even(\vphi) \even(x) $ (see \eqref{0104.15}), 
and therefore, by \eqref{defbetaSFIO}, $\b$ is $\odd (\vphi)\even(x)$. 
Estimates \eqref{stima beta primo egorov}-\eqref{stima derivata beta primo egorov} 
follow by \eqref{def:a51}, \eqref{defbetaSFIO}, \eqref{stime b0 (4) prima di decoupling} 
and Lemma \ref{lemma:WD}.
\end{proof}

By \eqref{0104.14}, \eqref{coefficienti cal L2 nuovo}, \eqref{coefficienti cal L2 nuovo-a6} one has 
\[
a_7 
= \sqrt{a_5 a_6 } 
= \sqrt{{\cal A}^{- 1} (a_2) {\cal A}^{- 1}(a_3) {\cal A}^{- 1}(1 + \alpha_x)} 
= {\cal A}^{- 1}( \sqrt{a_2 a_3}) {\cal A}^{- 1}\big(  \sqrt{1 + \alpha_x} \big). 
\]
We now choose the $ 2 \pi $-periodic function $\alpha(x)$ (introduced as a free 
parameter in \eqref{diffeo solo x nuovo}) so that 
\be\label{eq:a7-constant}
\langle a_7 \rangle_\vphi(x) = \mathtt m_{\frac12}
\ee 
is independent of $ x $, for some real constant $\mathtt m_{\frac12}  $. This
is equivalent to solve the equation 
$$
\langle \sqrt{a_2 a_3} \, \rangle_\vphi (x) \, \sqrt{1 + \alpha_x (x)} = \mathtt m_{\frac12}
$$
whose solution is 
\be\label{equazione omologica per alpha}
\mathtt m_{\frac12} := \Big( \frac{1}{2 \pi} \int_{\T} \frac{dx}{ \langle \sqrt{a_2 a_3} \, \rangle_\vphi^2 (x) } \Big)^{- \frac12} \, , \qquad
 \alpha (x) := \pa_{x}^{-1} \Big( \frac{\mathtt m_{\frac12}^2}{ \langle \sqrt{a_2 a_3} \, \rangle_\vphi^2 (x) } -1 \Big) \, . 
\ee

\begin{lemma}
The real valued function $ \alpha (x) $ defined in  \eqref{equazione omologica per alpha} 
is $\odd(x)$ and \eqref{stime alpha nuove ansatz} holds. 
Moreover 
\begin{align}
& |\mathtt m_{\frac12} - 1|^{k_0, \gamma} 
\lesssim \e \g^{-1} \,, \quad 
|\Delta_{12} \mathtt m_{\frac12}| 
\lesssim \e \g^{-1} \| \Delta_{12} i\|_{s_1} \label{stime lambda 1} \, . 
\end{align}
\end{lemma}

\begin{proof}
Since $ a_2, a_3 $ are $ {\rm even}(x)$ by  \eqref{rem:eps a1 b1 c1}, 
the function $ \alpha (x) $ defined in \eqref{equazione omologica per alpha} is $\odd(x)$. 
Estimates \eqref{stime lambda 1} follow by the definition of $ \mathtt m_{\frac12} $ in 
\eqref{equazione omologica per alpha} and 
\eqref{stime coefficienti cal L1}, \eqref{stime coefficienti cal L1 Delta}, \eqref{ansatz I delta}, 
applying also Lemma \ref{Moser norme pesate} and \eqref{p1-pr}. 
Similarly $ \alpha $ satisfies \eqref{stime alpha nuove ansatz}
by \eqref{stime coefficienti cal L1}, \eqref{stime coefficienti cal L1 Delta}, 
\eqref{stime lambda 1}, Lemma \ref{Moser norme pesate} and \eqref{p1-pr}. 
\end{proof}

By \eqref{def:a51} and \eqref{eq:a7-constant} the  term in \eqref{eq:passaggio} reduces to 
\begin{equation}\label{1006.1}
\ii \big( \ompaph \b (\vphi, x) + a_7 (\vphi, x) T_\h^\frac12 \big)  | D |^{\frac12} = \ii 
 \mathtt m_{\frac12} T_\h^\frac12  | D |^{\frac12} + {\mathtt R}_\beta
\end{equation}
where ${\mathtt R}_\beta $ is the  $ OPS^{-\infty} $ operator defined by 
\be \label{def:Rb}
\mathtt R_\beta := \ii (\omega \cdot \partial_\vphi \beta) ({\rm Id} - T_\h^{\frac12} )|D|^{\frac12} \, .
\ee
Finally,  the operator $ L_7 $ in \eqref{coniugio flusso di una PDE iperbolica} is,
in view of  \eqref{conj-op}, \eqref{definizione PM primo Egorov}, 
\eqref{primo levi 2}, \eqref{recur syst}, \eqref{1006.1},
\begin{equation}\label{L7}
L_7 = \omega \cdot \partial_\vphi + \ii  \mathtt m_{\frac12}  T_\h^{\frac12} |D|^{\frac12}  
+ a_8 + a_9 {\cal H}  +  P_7 + T_7 
\end{equation}
where $a_9  $ is the real valued function
\begin{align}
a_9 := a_9(\vphi,x) := -\frac12 \big(  \beta_x \, a_7 - \beta (\partial_x a_7) \big) -
\frac14 \big(   \beta_x \,   \ompaph \beta - \beta  \,  \ompaph \beta_x \big)\,,  \label{parte moltiplicazione dopo primo egorov}
\end{align}
$ P_7 $ is the  operator in $ OPS^{-1/2 } $ given by 
\begin{align}
P_7  := R_{A, P_6^{(0)}} +  R_{A, \omega \cdot \partial_\vphi} -   
\sum_{n=3}^{2M+1} \frac{(-\ii)^n}{n!} {\rm Ad}_{A(\vphi)}^{n-1} \big( \ompaph A(\vphi) \big) +
 \sum_{n=2}^{2M} \frac{(-\ii)^n}{n!} {\rm Ad}_{A(\vphi)}^n P_6^{(0)} + P_6 + 
 \mathtt R_\beta \label{definizione RM (1) Egorov}
\end{align}
(the operators $R_{A, P_6^{(0)}}, R_{A, \omega \cdot \partial_\vphi}, P_6, \mathtt R_\beta $ 
are defined respectively in \eqref{primo levi 3}, \eqref{primo levi 7}, \eqref{resti prima Egorov}, \eqref{def:Rb}),
and 
\be\label{def:T7}
\begin{aligned}
T_7  
& :=  
- \frac{(-\ii)^{2M+2}}{(2M+1)!} \int_0^1 
(1- \tau)^{2M+1} 
\Phi(\t, \vphi)^{-1}  \big( {\rm Ad}_{A(\vphi)}^{2M+1}  \big( \ompaph A(\vphi) \big)  \big)  \Phi(\t, \vphi) \, d \tau 
\\
& \quad 
+ \frac{(-\ii)^{2M+1}}{(2M)!} \int_0^1 (1- \tau)^{2M} 
\Phi(\t,\ph)^{-1}  {\rm Ad}_{A(\vphi)}^{2M+1} P_6^{(0)} \Phi(\t,\ph) \, d \tau 
\end{aligned}
\ee
($ T_7 $ stands for ``tame remainders", namely remainders satisfying tame 
estimates together with their derivatives, see \eqref{stima tame cal T (1)}, without controlling 
their pseudo-differential structure).
In conclusion, we have  the following lemma.

\begin{lemma}
Let $ \beta (\vphi, x) $ and $ \alpha  (x) $  be the functions defined in 
\eqref{defbetaSFIO} and 
\eqref{equazione omologica per alpha}.
Then $ {\cal L}_7 :=  {\bf \Phi}^{- 1}\mL_6 {\bf \Phi} $ in \eqref{0806.2} is the 
 real, even and reversible operator
\begin{equation}\label{cal LM (1) Egorov}
{\cal L}_7 = \omega \cdot \partial_\vphi + \ii  \mathtt m_{\frac12}  T_\h^{\frac12} |D|^{\frac12} \Sigma + 
\ii\Pi_0 + (a_8 + a_9 {\cal H}) {\mathbb I}_2  +  {\cal P}_7 + {\cal T}_7 
\end{equation}
where $ \mathtt m_{\frac12} $ is the real constant defined in \eqref{equazione omologica per alpha}, 
$ a_8, a_9 $ are the real valued functions in \eqref{0104.21}, \eqref{parte moltiplicazione dopo primo egorov}, 
\be\label{parity:a8a9}
a_8 =  \odd(\ph)\even(x) \, , \qquad
a_9 = \odd(\vphi) \odd(x) \, , 
\ee
and  ${\cal P}_7$, ${\cal T}_7 $ are the real operators
\be
{\cal P}_7  := 
\begin{pmatrix}
P_7 & 0 \\
0 & \overline P_7
\end{pmatrix} \in OPS^{- \frac12} \, , \qquad
{\cal T}_7 
 :=   \ii \Pi_0 ({\bf \Phi} - {\mathbb I}_2)  + {\bf \Phi}^{- 1} {\cal Q}_6 {\bf \Phi}  +
\begin{pmatrix}
T_7 & 0 \\
0 & \overline T_7
\end{pmatrix} \, , 
\label{cal RM (1) Egorov} 
\ee
where $ P_7 $ is defined in \eqref{definizione RM (1) Egorov}  and  $ T_7 $ in \eqref{def:T7}. 
\end{lemma}

\begin{proof}
Formula \eqref{cal LM (1) Egorov} follows by 
\eqref{0806.2} and \eqref{L7}. 
By  Lemma \ref{lem:stim-beta} the real function $\beta $  
is $ {\rm odd}(\vphi) {\rm even}(x)$.
Thus, by Sections \ref{sezione operatori reversibili e even} and \ref{AppendiceA}, 
the flow map $ {\bf \Phi} $ in  \eqref{0806.1} is real, even and reversibility preserving 
and therefore the conjugated operator $ {\cal L}_7 $ is real, even and reversible.
%\eqref{definizione PM primo Egorov}, \eqref{conj-op}, \eqref{primo levi 2}, \eqref{recur syst}, \eqref{1006.1}.
Moreover the function $ a_7 $ is $ {\rm even}(\vphi) {\rm even}(x) $ by \eqref{0104.15} and $ a_9 $ 
defined in \eqref{parte moltiplicazione dopo primo egorov} is ${\rm odd}(\vphi){\rm odd}(x)$. 
\end{proof}

Note that formulas 
\eqref{parte moltiplicazione dopo primo egorov} and \eqref{cal RM (1) Egorov} (via \eqref{definizione RM (1) Egorov}, \eqref{def:T7})
define $a_9 $ and $ \mP_7, {\cal T}_7 $ on the whole parameter space $\R^\nu \times [\h_1, \h_2]$
by means of the extended function $\b$ and the corresponding flow $\Phi$.
Thus the right hand side of \eqref{cal LM (1) Egorov} defines an extended operator 
on $\R^\nu \times [\h_1, \h_2]$, which we still denote by $\mL_7 $. 

In the next lemma we provide some estimates on the operators ${\cal P}_7 $ and ${\cal T}_7 $.

\begin{lemma}\label{lemma stime primo Egorov}
There exists $\sigma (k_0, \tau, \nu ) > 0$ such that, if \eqref{ansatz I delta} holds with $\mu_0 \geq \sigma$, then 
\begin{align}
 \| a_9 \|_s^{k_0, \gamma} \lesssim_s \e \g^{-2} (1 + \| \fracchi_0\|_{s + \sigma}^{k_0, \gamma})\,, \quad \forall s \geq s_0\,, \quad 
 \quad \| \Delta_{12} a_9 \|_{s_1} \lesssim_{s_1} \e \g^{-2}  \| \Delta_{12} i\|_{s_1 + \sigma}\,. \label{stime a6 primo egorov}
\end{align}
For any $s \geq s_0$ there exists $\delta(s) > 0$ small enough such that if  $\e \gamma^{- 2} \leq \delta(s)$, then  
\begin{align} 
\label{flusso PseudoPDE fracchi}
\| (\Phi^{\pm 1} - {\rm Id}) h \|_s^{k_0, \gamma}, \| (\Phi^* - {\rm Id}) h \|_s^{k_0, \gamma} 
& \lesssim_s \e \gamma^{-2}
\big(\| h \|_{s + \sigma}^{k_0, \gamma} + \| \fracchi_0\|_{s + \sigma}^{k_0, \gamma} \| h \|_{s_0 + \sigma}^{k_0, \gamma} \big)\,,
\\
\| \Delta_{12} \Phi^{\pm 1} h \|_{s_1} 
& \lesssim_{s_1} \e \g^{-2} \| \Delta_{12} i \|_{s_1 + \s} \| h \|_{s_1 + \frac12}\,.
\label{QUELLA DOPO LA (12.37)}
\end{align}
The pseudo-differential operator ${\cal P}_7$ defined in \eqref{cal RM (1) Egorov} %-\eqref{definizione RM (1) Egorov}
is in $OPS^{- \frac12}$. Moreover for any $M, \alpha > 0$, there exists a constant $\aleph_6(M, \alpha) > 0$ such that assuming \eqref{ansatz I delta} with $\mu_0 \geq \aleph_6(M, \alpha) + \sigma$, the following estimates hold: 
\begin{align}
& \norma {\cal P}_7 \norma_{- \frac12, s, \alpha}^{k_0, \gamma} \lesssim_{M, s, \alpha} \e \g^{-2}  
\big(1 + \| \fracchi_0\|_{s + \aleph_6(M, \alpha) + \sigma}^{k_0, \gamma} \big)\,, \label{stime cal RM (1)} \\
& \norma \Delta_{12} {\cal P}_7 \norma_{- \frac12, s_1, \alpha} \lesssim_{M, s_1, \alpha} \e \g^{-2}  \| \Delta_{12} 
i \|_{s_1 +   \aleph_6(M, \alpha) + \sigma} \,.\label{stime derivate cal RM (1)}
\end{align}
Let $S > s_0$, $\b_0 \in \N$, and $M > \frac12 (\beta_0 + k_0)$. 
There exists a constant $\aleph_6'(M, \beta_0) > 0$ such that, assuming \eqref{ansatz I delta} with $\mu_0 \geq \aleph_6'(M, \beta_0) + \sigma$, for any $m_1, m_2 \geq 0$, 
with $m_1 + m_2  \leq M - \frac12 (\beta_0 + k_0) $, 
for any $\beta \in \N^\nu$, $|\beta| \leq \beta_0$, 
the operators $ \langle D \rangle^{m_1}\partial_\vphi^\beta {\cal T}_7 \langle D \rangle^{m_2}$, $\langle D \rangle^{m_1}\partial_\vphi^\beta \Delta_{12} {\cal T}_7 \langle D \rangle^{m_2}$ are ${\cal D}^{k_0}$-tame with tame constants satisfying 
\begin{align}
 \mathfrak M_{\langle D \rangle^{m_1}\partial_\vphi^\beta {\cal T}_7 \langle D \rangle^{m_2}}(s) 
 \lesssim_{M, S} \e \g^{-2} \big(1 + \| \fracchi_0 \|_{s + \aleph_6'(M, \beta_0) + \sigma}\big)\,, \qquad \forall s_0 \leq s \leq S \label{stima tame cal T (1)} \\
\| \langle D \rangle^{m_1} \Delta_{12}\partial_\vphi^\beta {\cal T}_7 \langle D \rangle^{m_2} \|_{{\cal L}(H^{s_1})} \lesssim_{M, S} \e \g^{-2}  \| \Delta_{12} i \|_{s_1 + \aleph_6'(M, \beta_0) + \sigma}\,.  \label{stima tame derivata cal T (1)}
\end{align}
\end{lemma}

\begin{proof}
% {\sc Proof of \eqref{stime a6 primo egorov}} 
Estimates \eqref{stime a6 primo egorov} for $a_9 $ defined in \eqref{parte moltiplicazione dopo primo egorov} 
follow by \eqref{stime b0 (4) prima di decoupling}, \eqref{stima beta primo egorov}, \eqref{stima derivata beta primo egorov}, \eqref{p1-pr} and \eqref{ansatz I delta}. % and Lemma \ref{lemma:WD}. 

\noindent
{\sc Proof of \eqref{flusso PseudoPDE fracchi}-\eqref{QUELLA DOPO LA (12.37)}.} 
It follows by applying Proposition \ref{proposition 2.40 unificata}, Lemma \ref{lemma 2.42 unificato}, 
estimates \eqref{stima beta primo egorov}-\eqref{stima derivata beta primo egorov} 
and using formula $\partial_\lambda^k \big( (\Phi^{\pm 1} - {\rm Id})h \big) = \sum_{k_1 + k_2 = k} C(k_1, k_2) \partial_\lambda^{k_1} (\Phi^{\pm 1} - {\rm Id}) \partial_\lambda^{k_2} h$, for any $k \in \N^{\nu + 1}$, $|k| \leq k_0$. 

\noindent
{\sc Proof of \eqref{stime cal RM (1)}-\eqref{stime derivate cal RM (1)}.}
First we prove \eqref{stime cal RM (1)}, estimating each term in the definition
\eqref{definizione RM (1) Egorov} of $ P_7$.  
The operator $ A = \beta (\vphi, x)  | D|^{\frac12} $ in \eqref{flow-propagator-beta} satisfies,
by \eqref{lemma composizione multiplier} and \eqref{stima beta primo egorov},  
\begin{equation}\label{stima pseudo A nel lemma}
\norma A \norma_{\frac12, s, \alpha}^{k_0, \gamma} \lesssim_{s, \alpha} \| \beta\|_{s }^{k_0, \gamma}
\lesssim_{s, \alpha}  \e \gamma^{-2} \big(1 + \| \fracchi_0\|_{s + \sigma}^{k_0, \gamma} \big) \,.
\end{equation}
The operator $ P_6^{(0)} $ in \eqref{definizione PM primo Egorov} satisfies, 
by \eqref{stime b0 (4) prima di decoupling}, \eqref{stime a0 (4) prima di decoupling}, 
\eqref{lemma composizione multiplier}, \eqref{stima resti prima Egorov},
\begin{equation}\label{stima pseudo P5 nel lemma}
\norma P_6^{(0)} \norma_{\frac12, s, \alpha}^{k_0, \gamma} 
{\lesssim_{M, s, \alpha}} \, 1 + \| \fracchi_0\|_{s + \aleph_5(M, \alpha) + \sigma}^{k_0, \gamma}\,.
\end{equation}
The estimate of the term $- \sum_{n=3}^{2M+1} \frac{(-\ii)^n}{n!} {\rm Ad}_{A(\vphi)}^{n-1} \big( \ompaph A(\vphi) \big) +
 \sum_{n=2}^{2M} \frac{(-\ii)^n}{n!} {\rm Ad}_{A(\vphi)}^n P_6^{(0)}$ in \eqref{definizione RM (1) Egorov} then follows by \eqref{stima pseudo A nel lemma}, \eqref{stima pseudo P5 nel lemma} and by applying Lemma \ref{lemma stime Ck parametri} and the estimate \eqref{stima Ad pseudo diff astratta}. 
The term $\mathtt R_\beta \in OPS^{- \infty}$ % in \eqref{definizione RM (1) Egorov} 
defined in \eqref{def:Rb} can be estimated by \eqref{lemma composizione multiplier} 
(applied with $A := \omega \cdot \partial_\vphi \beta$, 
$g(D) := (T_\h^{\frac12} - {\rm Id}) |D|^{\frac12} \in OPS^{- \infty}$) 
and using \eqref{stima beta primo egorov}, \eqref{tangente iperbolica espansione}. 
The estimate of the terms $R_{A, P_6^{(0)}}, R_{A, \omega \cdot \partial_\vphi}$ in \eqref{definizione RM (1) Egorov} follows by their definition given in \eqref{primo levi 3}, \eqref{primo levi 7} and by estimates 
\eqref{stime b0 (4) prima di decoupling}, 
\eqref{stime a0 (4) prima di decoupling}, 
\eqref{stima resti prima Egorov}, 
\eqref{stima beta primo egorov}, 
\eqref{p1-pr}, 
\eqref{lemma composizione multiplier},
and Lemmata \ref{lemma stime Ck parametri}, \ref{lemma tame norma commutatore}. 
Since $P_6$ satisfies \eqref{stima resti prima Egorov}, 
estimate \eqref{stime cal RM (1)} is proved. 
Estimate \eqref{stime derivate cal RM (1)} can be proved by similar arguments. 

\noindent
{\sc Proof of \eqref{stima tame cal T (1)}, \eqref{stima tame derivata cal T (1)}}. 
We estimate the term ${\bf \Phi}^{- 1} {\cal Q}_6 {\bf \Phi}$ in \eqref{cal RM (1) Egorov}. 
For any $k \in \N^{\nu+1}, \beta \in \N^\nu$, $|k| \leq k_0$, $|\beta| \leq \beta_0$, $\lambda = (\omega, \h)$, one has 
\begin{align} \label{labello bello}
& \partial_\lambda^k \partial_\vphi^\beta ({\bf \Phi}^{- 1} {\cal Q}_6 {\bf \Phi})  = \sum_{\begin{subarray}{c}
\beta_1 + \beta_2 + \beta_3 = \beta \\
k_1 + k_2 + k_3 = k
\end{subarray}}C(\beta_1, \beta_2 , \beta_3, k_1, k_2, k_3) (\partial_{\lambda}^{k_1} \partial_\vphi^{\beta_1} {\bf \Phi}^{- 1}) (\partial_\lambda^{k_2} \partial_\vphi^{\beta_2} {\cal Q}_6) (\partial_\lambda^{k_3}\partial_\vphi^{\beta_3} {\bf \Phi}) \,.
\end{align}
For any $m_1, m_2 \geq 0$ satisfying $m_1 + m_2 \leq M - \frac12 (\beta_0 + k_0)$, 
we have to provide an estimate for the operator 
\begin{equation}\label{stima coniugazione egorov cal Q5}
\langle D \rangle^{ m_1 }(\partial_{\lambda}^{k_1} \partial_\vphi^{\beta_1} {\bf \Phi}^{- 1}) (\partial_\lambda^{k_2} \partial_\vphi^{\beta_2} {\cal Q}_6) (\partial_\lambda^{k_3}\partial_\vphi^{\beta_3} {\bf \Phi}) \langle D \rangle^{m_2}\,. 
\end{equation}
We write 
\begin{align}
\eqref{stima coniugazione egorov cal Q5} & =  \Big( \langle D \rangle^{m_1} \partial_{\lambda}^{k_1} \partial_\vphi^{\beta_1} {\bf \Phi}^{- 1} \langle D \rangle^{- \frac{|\beta_1| + |k_1|}{2} - m_1}  \Big) \label{stima coniugazione egorov cal Q5a} \\
& \qquad \circ \Big( \langle D \rangle^{\frac{|\beta_1| + |k_1|}{2} + m_1} \partial_\lambda^{k_2} \partial_\vphi^{\beta_2} 
{\cal Q}_6 \langle D \rangle^{\frac{|\beta_3| + |k_3|}{2} + m_2}  \Big) \label{stima coniugazione egorov cal Q5b} \\
& \qquad \circ \Big( \langle D\rangle^{- m_2 - \frac{|\beta_3| + |k_3|}{2}} \partial_\lambda^{k_3}\partial_\vphi^{\beta_3} {\bf \Phi} \langle D \rangle^{ m_2}  \Big) \,.\label{stima coniugazione egorov cal Q5c}
\end{align}
The terms \eqref{stima coniugazione egorov cal Q5a}-\eqref{stima coniugazione egorov cal Q5c} can be estimated separately. 
To estimate the terms \eqref{stima coniugazione egorov cal Q5a} and \eqref{stima coniugazione egorov cal Q5c}, we apply 
\eqref{copenaghen B omega} of Proposition \ref{proposition 2.40 unificata},
\eqref{stima derivate flusso generalissima derivate i}
of Lemma \ref{lemma 2.42 unificato}, 
and \eqref{stima beta primo egorov}-\eqref{stima derivata beta primo egorov}. 
The pseudo-differential operator in \eqref{stima coniugazione egorov cal Q5b} 
is estimated in $\norma \ \norma_{0,s,0}$ norm by using 
\eqref{Norm Fourier multiplier}, \eqref{estimate composition parameters}, \eqref{lemma composizione multiplier}, 
bounds \eqref{stima resti prima Egorov}, \eqref{stima derivate resti prima Egorov}
on ${\cal Q}_6 $,
and the fact that $\frac{|\beta_1| + |k_1|}{2} + m_1 + \frac{|\beta_3| + |k_3|}{2} + m_2 - M \leq 0$.
Then its action on Sobolev functions is deduced by Lemma \ref{lemma: action Sobolev}.
As a consequence, each operator in \eqref{stima coniugazione egorov cal Q5},
and hence the whole operator \eqref{labello bello}, 
satisfies \eqref{stima tame cal T (1)}.

The estimates of the terms in \eqref{def:T7} can be done arguing similarly, using also 
the estimates \eqref{stima Ad pseudo diff astratta}, \eqref{stima pseudo A nel lemma}-\eqref{stima pseudo P5 nel lemma}. 
The term $\langle D \rangle^{m_1} \partial_\vphi^\beta\Pi_0 ({\bf \Phi} - \mathbb{I}_2) \langle D \rangle^{m_2}$ can be estimated by applying Lemma \ref{lemma coniugazione proiettore pi 0} 
(with $A = \mathbb{I}_2 $, $B = {\bf \Phi } $) and \eqref{flusso PseudoPDE fracchi}, 
\eqref{stima beta primo egorov}, \eqref{stima derivata beta primo egorov}. 
\end{proof}

\section{Reduction of the lower orders} \label{sezione descent method}

In this section we complete the reduction of the operator $\mL_7$ in \eqref{cal LM (1) Egorov} to constant coefficients,  
up to a regularizing remainder of order $|D|^{-M}$. 
We write 
\be\label{forma-cal-L7}
{\cal L}_7 = \begin{pmatrix}
L_7 & 0 \\
0 & \overline L_7
\end{pmatrix} + \ii \Pi_0 + {\cal T}_7\, ,
\ee
where 
\begin{equation}\label{primo operatore descent method}
L_7 := \omega \cdot \partial_\vphi + \ii \mathtt m_{\frac12} T_\h^{\frac12} |D|^{\frac12} 
+ a_8 + a_9 {\cal H} + P_7\,,
\end{equation}
 the real valued functions $a_8, a_9$ are introduced in 
 \eqref{0104.21}, \eqref{parte moltiplicazione dopo primo egorov}, satisfy \eqref{parity:a8a9}, 
and the operator $P_7 \in OPS^{- \frac12}$  in \eqref{definizione RM (1) Egorov} is  even and reversible. 
We first conjugate the operator  $ L_7$.

\subsection{Reduction of the order 0} \label{subsec:luglio.1} 

In this subsection we reduce to constant coefficients 
the term $a_8 + a_9 {\cal H}$ of order zero of $L_7$ in \eqref{primo operatore descent method}. 
We begin with removing the dependence of $ a_8 + a_9 {\cal H} $ on $\ph$.
It turns out that, 
since  $a_8, a_9$ are odd functions in  $\ph$ by  \eqref{parity:a8a9}, thus with zero average,
this step removes completely the terms of order zero. 
Consider the transformation 
\begin{equation}\label{definizione M0 descent method}
W_0 := {\rm Id} + f_0(\vphi, x) + g_0(\vphi, x) {\cal H}\,, 
\end{equation}
where $f_0, g_0$ are real valued functions to be determined.
Since $\mH^2 = - {\rm Id} + \pi_0$ on the periodic functions 
where $\pi_0$ is defined in \eqref{def pi0}, one has  
\begin{align}
L_7 W_0 & = W_0 \big( \omega \cdot \partial_\vphi + \ii \mathtt m_{\frac12} T_\h^{\frac12} |D|^{\frac12} \big) 
+ ( \omega \cdot \partial_\vphi f_0 + a_8 +  a_8 f_0 - a_9 g_0 ) 
\nonumber\\ & \quad 
+ ( \omega \cdot \partial_\vphi g_0 +  a_9 + a_8 g_0 + a_9 f_0 ) {\cal H} 
+ \breve P_7 
\label{florida 2}
\end{align}
where $ \breve P_7\in OPS^{-\frac12}$ is the operator
\begin{align}
\breve P_7 :=  
a_9 [{\cal H}, f_0] + a_9 [{\cal H}, g_0] {\cal H}+ [\ii \mathtt m_{\frac12} T_\h^{\frac12} |D|^{\frac12}, W_0] + P_7 W_0 
+ a_9 g_0 \pi_0 \,.\label{florida 3}
\end{align}
In order to eliminate the zero order terms in \eqref{florida 2} 
we choose the functions $f_0,g_0$ such that 
\begin{equation}\label{equazione omologica primo step descent method}
\begin{cases}
\omega \cdot \partial_\vphi f_0 + a_8 +  a_8 f_0 - a_9 g_0 = 0 \\
\omega \cdot \partial_\vphi g_0 +  a_9 + a_8 g_0 + a_9 f_0 = 0\,.
\end{cases}
\end{equation}
Writing $ z_0 = 1 + f_0 + \ii g_0 $, 
the real system \eqref{equazione omologica primo step descent method} 
is equivalent to the complex scalar equation 
\begin{equation}\label{equazione complessa descent method grado 0}
\omega \cdot \partial_\vphi z_0 + (a_8 + \ii a_9) z_0 = 0\,. 
\end{equation}
Since $a_8, a_9$ are odd functions in $\ph$, 
we choose, for all $\om \in \mathtt{DC}(\g,\t) $, the periodic function
\begin{equation}\label{scelta p0 descent}
z_0 := \exp(p_0), \quad 
p_0 := - (\ompaph)^{-1} (a_8 + \ii a_9),
\end{equation}
which solves \eqref{equazione complessa descent method grado 0}.
Thus the real functions
\begin{equation}
\begin{aligned} \label{def:f0g0}
f_0 & := \Re(z_0) - 1 
= \exp(-(\ompaph)^{-1} a_8) \cos( (\ompaph)^{-1} a_9 ) - 1,
\\ 
g_0 & := \Im(z_0) 
= - \exp(-(\ompaph)^{-1} a_8) \sin( (\ompaph)^{-1} a_9 )
\end{aligned}
\end{equation}
solve \eqref{equazione omologica primo step descent method}, 
and, for $\om \in \DC(\g,\t)$, equation \eqref{florida 2} reduces to 
\be\label{L7-intermedia}
L_7 W_0 = W_0 (\omega \cdot \partial_\vphi + \ii \mathtt m_{\frac12} T_\h^{\frac12} |D|^{\frac12})
+ \breve P_7\,, \quad \breve P_7\in OPS^{-\frac12} \, .
\ee
We extend the function $p_0$ in \eqref{scelta p0 descent}
to the whole parameter space $\R^\nu \times [\h_1, \h_2]$ by using % the extended operator 
$ (\omega \cdot \partial_\vphi)_{ext}^{- 1} $ introduced in Lemma \ref{lemma:WD}. 
Thus the functions $ z_0, f_0, g_0 $ in \eqref{scelta p0 descent}, \eqref{def:f0g0}
are defined on $\R^\nu \times [\h_1, \h_2]$ as well. 

\begin{lemma}
The real valued functions $f_0, g_0 $ in \eqref{def:f0g0} satisfy 
\be\label{parityp0q0}
f_0 = \even(\vphi) \even(x) \, , \quad g_0 = \even(\vphi) \odd (x) \, . 
\ee
Moreover, there exists $\sigma(k_0, \tau, \nu) > 0$ such that, if \eqref{ansatz I delta} holds with $\mu_0 \geq \sigma$, then 
\begin{equation}\label{stime f0 g0}
\| f_0 \|_s^{k_0, \gamma}\,,\, \| g_0 \|_s^{k_0, \gamma} \lesssim_s \e \gamma^{- 3} 
\big(1 + \| \fracchi_0\|_{s + \sigma}^{k_0, \gamma}\big)\,, \quad \| \Delta_{12} f_0\|_{s_1}, \| \Delta_{12} g_0 \|_{s_1} \lesssim_{s_1} \e \gamma^{- 3} \| \Delta_{12} i \|_{s_1 + \sigma}\,.
\end{equation}
 The operator $W_0$ defined in \eqref{definizione M0 descent method} is 
 even, reversibility preserving, invertible and for any $\alpha > 0$, assuming \eqref{ansatz I delta} with $\mu_0 \geq \alpha + \sigma$, the following estimates hold:  
\begin{equation}\label{stime M0 descent method}
\norma W_0^{\pm 1} - {\rm Id} \norma_{0, s, \alpha}^{k_0, \gamma} \lesssim_{s, \alpha} \e \gamma^{- 3} 
\big(1 + \| \fracchi_0\|_{s + \alpha + \sigma}^{k_0, \gamma} \big)\,, \quad \norma \Delta_{12} W_0^{\pm 1} \norma_{0, s_1, \alpha} \lesssim_{s_1, \alpha} \e \gamma^{- 3} \| \Delta_{12} i \|_{s_1 + \alpha +  \sigma}\,.
\end{equation}
\end{lemma}

\begin{proof}
The parities in \eqref{parityp0q0} follow by \eqref{def:f0g0} and \eqref{parity:a8a9}. 
Therefore $ W_0 $ in \eqref{definizione M0 descent method} is even and reversibility preserving. 
Estimates \eqref{stime f0 g0} follow by \eqref{def:f0g0}, 
\eqref{stime a0 (4) prima di decoupling}, 
\eqref{stime a6 primo egorov},
\eqref{p1-pr}, \eqref{2802.2}, 
\eqref{0811.10}. %\eqref{Diophantine-1}.
The operator $W_0$ defined in \eqref{definizione M0 descent method} is invertible
by Lemma \ref{Neumann pseudo diff},  \eqref{stime f0 g0}, \eqref{ansatz I delta}, for $ \e \gamma^{-3} $ small enough. 
Estimates \eqref{stime M0 descent method} then follow by \eqref{stime f0 g0}, 
using \eqref{norma a moltiplicazione}, \eqref{lemma composizione multiplier} and Lemma \ref{Neumann pseudo diff}. 
\end{proof}

For $\om \in \DC(\g,\t)$, by \eqref{L7-intermedia} we obtain the even and reversible operator
\begin{equation}\label{definizione P7 (1)}
L_7^{(1)} := W_0^{- 1} L_7 W_0 
= \omega \cdot \partial_\vphi + \ii \mathtt m_{\frac12} T_\h^{\frac12} |D|^{\frac12} + P_7^{(1)}\,,
\qquad 
P_7^{(1)}  := W_0^{-1} \breve P_7 \, , 
\end{equation}
where $\breve P_7 $ is the operator in $ OPS^{-\frac12} $  defined in \eqref{florida 3}.

Since the functions $ f_0, g_0 $ are defined on $\R^\nu \times [\h_1, \h_2]$, 
the operator $\breve P_7$  in  \eqref{florida 3} is defined on $\R^\nu \times [\h_1, \h_2]$, and   
% the operator 
$ \omega \cdot \partial_\vphi + \ii \mathtt m_{\frac12} T_\h^{\frac12} |D|^{\frac12}  + P_7^{(1)}$
% $L_7^{(1)} := \omega \cdot \partial_\vphi + \ii \mathtt m_{\frac12} T_\h^{\frac12} |D|^{\frac12}  + P_7^{(1)}$ 
in \eqref{definizione P7 (1)} is an extension of $L_7^{(1)}$ to $\R^\nu \times [\h_1, \h_2]$, 
still denoted $L_7^{(1)}$. 

\begin{lemma}
For any $M, \alpha > 0$, there exists a constant $\aleph_7^{(1)}(M, \alpha) > 0$ such that if \eqref{ansatz I delta} holds with $\mu_0 \geq \aleph_7^{(1)}(M, \alpha)$, the remainder $P_7^{(1)} \in OPS^{- \frac12}$, 
defined in \eqref{definizione P7 (1)}, satisfies 
\begin{equation}\label{stime P7 (1)}
\begin{aligned}
& \norma P_7^{(1)} \norma_{- \frac12, s, \alpha}^{k_0, \gamma} 
\lesssim_{M, s, \alpha} \e \gamma^{- 3} \big( 1 + \| \fracchi_0\|_{s + \aleph_7^{(1)}(M, \alpha)}^{k_0, \gamma} \big) \, , \\
& \norma \Delta_{12} P_7^{(1)} \norma_{- \frac12, s_1, \alpha} 
\lesssim_{M, s_1, \alpha} \e \gamma^{- 3} \| \Delta_{12} i\|_{s_1 + \aleph_7^{(1)}(M, \alpha)}\,. 
\end{aligned}
\end{equation}
\end{lemma}

\begin{proof}
Estimates \eqref{stime P7 (1)} follow by the definition of $P_7^{(1)}$ given in \eqref{definizione P7 (1)}, by estimates \eqref{stime f0 g0}, \eqref{stime M0 descent method}, \eqref{stime lambda 1}, \eqref{stime a6 primo egorov}, 
\eqref{stime cal RM (1)}, \eqref{stime derivate cal RM (1)}, 
by applying \eqref{norma a moltiplicazione}, 
\eqref{estimate composition parameters}, \eqref{lemma composizione multiplier}, \eqref{stima commutator parte astratta} 
and using also Lemma \ref{lem: commutator aH}. 
The fact that $P_7^{(1)}$ has size $\e \gamma^{- 3}$ 
is due to the term $[\ii \mathtt m_{\frac12} T_\h^{\frac12} |D|^{\frac12}, W_0] 
= [\ii \mathtt m_{\frac12} T_\h^{\frac12} |D|^{\frac12}, W_0 - {\rm Id}]$,
because $\mathtt m_{\frac12} = 1 + O(\e \g^{-1})$ and $W_0 - {\rm Id} = O(\e \gamma^{- 3})$. 
\end{proof}

We underline that the operator $L_7^{(1)}$ 
in \eqref{definizione P7 (1)} does not contain  terms of order zero.

\subsection{Reduction at  negative orders} \label{subsec:luglio.2}

In this subsection we define inductively a finite number of transformations to the aim of 
reducing to constant coefficients all the symbols of orders $ > -M $ 
of the operator $ L_7^{(1)} $ in \eqref{definizione P7 (1)}.
The constant $M$ will be fixed in \eqref{relazione mathtt b N}. 
In the rest of the section we prove the following inductive claim: 
\begin{itemize}
\item 
{\bf Diagonalization of $ L_7^{(1)} $ in decreasing orders.}  
For any $m \in \{1, \ldots, 2M \}$, 
we have an even and  reversible operator of the form 
\begin{equation}\label{L7 (m)}
L_7^{(m)} := \omega \cdot \partial_\vphi +  \Lambda_m(D) + P_7^{(m)}\,,
\quad P_7^{(m)} \in OPS^{- \frac{m}{2}} \, , 
\end{equation}
where 
\begin{equation}\label{definizione Lambda m (D)}
\Lambda_m(D) :=  \ii \mathtt m_{\frac12} T_\h^{\frac12} |D|^{\frac12} +r_m(D)\,, \qquad r_m(D) \in OPS^{- \frac12} \,.
\end{equation}
The operator $ r_m(D) $ is an even and reversible Fourier multiplier, independent of $(\vphi, x)$. 
Also the operator $ P_7^{(m)} $ is even and reversible. 
%therefore $r_m(\xi) \in \ii \R$ for all $\xi \in \R$ and $r_m(\xi) = r_m(- \xi)$ 
%(see Section \ref{sezione operatori reversibili e even}). 

For any $M, \alpha > 0$, there exists a constant $\aleph_7^{(m)}(M, \alpha) > 0$ (depending also on $\tau, k_0, \nu$) such that,  if \eqref{ansatz I delta} holds with $\mu_0 \geq \aleph_7^{(m)}(M, \alpha)$, then the following estimates hold: 
\begin{align}
 & \norma r_m(D)\norma_{- \frac12, s, \alpha}^{k_0, \gamma} \lesssim_{M, \alpha} \e \gamma^{- (m + 1)}\,, \quad \norma \Delta_{12} r_m(D) \norma_{- \frac12, s_1, \alpha} \lesssim_{M, \alpha} \e \gamma^{- (m + 1)} \| \Delta_{12} i \|_{s_1 + \aleph_7^{(m)}(M, \alpha)}\,, 
\label{2007.4} \\ 
 & \norma  P_7^{(m)}  \norma_{- \frac{m}{2}, s, \alpha}^{k_0, \gamma} \lesssim_{M, s, \alpha} \e \gamma^{- (m + 2)} 
 \big(1 + \|\fracchi_0 \|_{s + \aleph_7^{(m)}(M, \alpha) }^{k_0, \gamma}\big)\,, 
\label{2007.5} \\ 
& \norma \Delta_{12} P_7^{(m)} \norma_{- \frac{m}{2}, s_1, \alpha} \lesssim_{M, s_1, \alpha} \e \gamma^{- (m + 2)} \| \Delta_{12} i \|_{s_1 + \aleph_7^{(m)}(M, \alpha) }\,. 
\label{2007.50} 
\end{align}
Note that by \eqref{definizione Lambda m (D)}, 
using \eqref{stime lambda 1}, \eqref{2007.4} 
and \eqref{Norm Fourier multiplier} (applied for $g(D) = T_\h^{\frac12} |D|^{\frac12}$) 
one gets
\begin{align}
& \norma \Lambda_m(D)\norma_{ \frac12, s, \alpha}^{k_0, \gamma} 
\lesssim_{M,\alpha} 1 \,, \quad 
\norma \Delta_{12} \Lambda_m(D) \norma_{\frac12, s_1, \alpha} 
\lesssim_{M,\alpha} \e \gamma^{- (m + 1)} \| \Delta_{12} i \|_{s_1 + \aleph_7^{(m)}(M, \alpha)}\,. \label{stima induttiva Lambda m (D)}
\end{align}
For $ m \geq 2 $ there exist real, even, reversibility preserving, invertible maps $ W_{m-1}^{(0)} $, $ W_{m-1}^{(1)} $ of the
form
\begin{equation}\label{ansatz-W0Wm}
\begin{aligned}
& W_{m-1}^{(0)} := {\rm Id} + w_{m-1}^{(0)}(\ph,x, D) 
\qquad \text{with} \qquad  w_{m-1}^{(0)}(\vphi, x, \xi) \in S^{- \frac{m-1}{2}} \, ,\\
& W_{m-1}^{(1)} := {\rm Id} + w_{m-1}^{(1)}(x, D) \qquad \quad 
{\rm with} \qquad w_{m-1}^{(1)}(x, \xi) \in S^{- \frac{m-1}{2} + \frac12}
\end{aligned}
\ee
such that, for all $ \omega \in {\mathtt {DC}}(\gamma, \tau) $, 
\be\label{coniugazione-descent}
L_7^{(m)} =  (W_{m-1}^{(1)})^{-1} (W_{m-1}^{(0)})^{-1}  L_7^{(m-1)} W_{m-1}^{(0)} W_{m-1}^{(1)} \, . 
\ee
\end{itemize}

\noindent
{\bf Initialization. } 
For $m=1$, 
the even and reversible operator $ L_7^{(1)}  $ in \eqref{definizione P7 (1)} has the form 
 \eqref{L7 (m)}-\eqref{definizione Lambda m (D)} with % $r_1(D)=0$, thus
\begin{equation}\label{r1(D) = 0}
r_1(D) = 0, \quad 
\Lambda_1(D) = \ii \mathtt m_{\frac12} T_\h^{\frac12} |D|^{\frac12}\,. 
\end{equation}
Since $ \Lambda_1(D) $ is even and reversible, by difference,  
the operator $  P_7^{(1)} $ is even and reversible as well. 
 At $m=1$, estimate \eqref{2007.4} is trivial and \eqref{2007.5}-\eqref{2007.50} are \eqref{stime P7 (1)}.

\medskip

\noindent
{\bf Inductive step. } 
In the next two subsections, we prove the above inductive claim, see \eqref{definizione L7 m+1}-\eqref{nuovo resto m+1 descent} and Lemma \ref{lemma detto 13.6}.
We perform this reduction in  two steps: 
\begin{enumerate}
\item
First we look for a transformation $ W_m^{(0)}$ to remove the dependence on $\ph $ of the terms 
of order $-m/2$ of the operator 
$ L_7^{(m)} $ in \eqref{L7 (m)},  
see  \eqref{equazione omologica descent method m0}. 
The resulting conjugated operator is $ L_7^{(m,1)} $ in \eqref{Wm0 L7m Wm0}.
\item Then we look for a transformation $ W_m^{(1)}$ to remove the dependence on $ x $ 
of the terms  of order $-m/2$ of the operator   $ L_7^{(m,1)} $ in \eqref{Wm0 L7m Wm0}, 
see \eqref{equazione omologica descent method m1} and \eqref{2007.8}.
\end{enumerate}

\subsubsection{Elimination of the dependence on $\ph$}\label{subsub1}

In this subsection we eliminate the dependence on $\ph$ from the terms of order $-m/2$ in $P_7^{(m)}$
in \eqref{L7 (m)}.  
We conjugate the operator $ L_7^{(m)}  $ in \eqref{L7 (m)} by a transformation of the form (see \eqref{ansatz-W0Wm})
\begin{equation}\label{definizione Wm (0)}
W_m^{(0)} := {\rm Id} + w_m^{(0)}(\ph,x, D)\,, 
\quad \text{with} \quad  w_m^{(0)}(\vphi, x, \xi) \in S^{- \frac{m}{2}} \, ,
\end{equation}
which we shall fix in \eqref{definizione wm (0)}. We compute
\begin{align}
L_7^{(m)} W_m^{(0)} & = W_m^{(0)} \big( \omega \cdot \partial_\vphi + \Lambda_m(D) \big) 
+ (\omega \cdot \partial_\vphi w_m^{(0)})(\vphi, x, D) + P_7^{(m)} 
\nonumber\\ & \quad \,
+ \big[ \Lambda_m (D), w_m^{(0)}(\vphi, x, D) \big]  + P_7^{(m)} w_m^{(0)}(\vphi, x, D) \, .
\label{L7m Wm0}
\end{align}
Since $\Lambda_m(D) \in OPS^{\frac12}$ and the operators
 $P_7^{(m)}$, $w_m^{(0)}(\vphi, x, D) $ are in $ OPS^{- \frac{m}{2}}$, with $m \geq 1$, 
we have that the commutator $ [\Lambda_m (D), w_m^{(0)}(\vphi, x, D)]$ is in $OPS^{- \frac{m}{2} - \frac12}$ and $P_7^{(m)} w_m^{(0)}(\vphi, x, D)$ is in $OPS^{-m} \subseteq OPS^{- \frac{m}{2} - \frac12}$. 
Thus the term of order $- m/2$ in \eqref{L7m Wm0} is  
$ (\omega \cdot \partial_\vphi w_m^{(0)})(\vphi, x, D) + P_7^{(m)} $.

Let $p_7^{(m)}(\vphi, x, \xi) \in S^{- \frac{m}{2}}$ be the symbol of $P_7^{(m)}$. 
We look for $ w_m^{(0)}(\vphi, x, \xi) $ such that 
\begin{equation} \label{equazione omologica descent method m0}
\omega \cdot \partial_\vphi w_m^{(0)}(\vphi, x, \xi) + p_7^{(m)}(\vphi, x, \xi) = \langle p_7^{(m)} \rangle_\vphi(x, \xi) 
\end{equation}
where 
\begin{equation} \label{2007.6} 
\langle p_7^{(m)} \rangle_\vphi(x, \xi) := \frac{1}{(2 \pi)^\nu} \int_{\T^\nu} p_7^{(m)}(\vphi, x, \xi)\, d \vphi\, . 
\end{equation}
For all $\om \in \mathtt{DC}(\g,\t)$, 
we choose the solution of \eqref{equazione omologica descent method m0} 
given by the periodic function
\begin{equation}\label{definizione wm (0)}
w_m^{(0)}(\vphi, x, \xi) := (\omega \cdot \partial_\vphi)^{- 1}
\Big( \langle p_7^{(m)} \rangle_\vphi(x, \xi) -  p_7^{(m)}(\vphi, x, \xi)\Big) \,.
\end{equation}
We extend the symbol $w_m^{(0)}$ in \eqref{definizione wm (0)} 
to the whole parameter space $\R^\nu \times [\h_1, \h_2]$ 
by using the extended operator $ (\omega \cdot \partial_\vphi)_{ext}^{- 1} $ 
introduced in Lemma \ref{lemma:WD}. 
As a consequence, the operator $W_m^{(0)}$ in \eqref{definizione Wm (0)}
is extended accordingly. 
We still denote by $w_m^{(0)}, W_m^{(0)}$ these extensions.

\begin{lemma} \label{lemma stima induttiva descent passo m-1}
The operator $W_m^{(0)}$ defined in \eqref{definizione Wm (0)}, \eqref{definizione wm (0)}
is even and reversibility preserving. 
For any $\alpha, M > 0$ there exists a constant $\aleph_7^{(m,1)}(M, \alpha) > 0$ 
(depending also on $k_0, \tau, \nu$), 
larger than the constant $\aleph_7^{(m)}(M, \alpha)$ 
appearing in \eqref{2007.4}-\eqref{stima induttiva Lambda m (D)}
such that, if \eqref{ansatz I delta} holds with $\mu_0 \geq \aleph_7^{(m, 1)}(M, \alpha)$, then for any $s \geq s_0$
\begin{align}
& \norma {\rm Op}(w_m^{(0)}) \norma_{- \frac{m}{2}, s, \alpha}^{k_0, \gamma} 
\lesssim_{M, s, \alpha} \e \gamma^{-(m + 3)} 
\big(1 + \| \fracchi_0 \|_{s + \aleph_7^{(m, 1)}(M, \alpha)}^{k_0, \gamma} \big) 
\label{stima wm0} \\
&  \norma \Delta_{12} {\rm Op}(w_m^{(0)}) \norma_{- \frac{m}{2}, s_1, \alpha} 
\lesssim_{M, s_1, \alpha} \e \gamma^{-(m + 3)} 
\| \Delta_{12} i \|_{s_1 + \aleph_7^{(m, 1)}(M, \alpha)}  \, . 
\label{stima derivata wm0} 
\end{align}
As a consequence, the transformation $ W_m^{(0)} $ defined in \eqref{definizione Wm (0)},
\eqref{definizione wm (0)} is invertible and 
\begin{align}
& \norma (W_m^{(0)})^{\pm 1} - {\rm Id} \norma_{0, s, \a}^{k_0, \gamma} 
\lesssim_{M, s, \a} \e \gamma^{-(m + 3)} 
\big(1 + \| \fracchi_0 \|_{s + \aleph_7^{(m, 1)}(M, \a)}^{k_0, \gamma} \big) 
\label{stima Wm0 - Id} \\
& \norma \Delta_{12} (W_m^{(0)})^{\pm 1} \norma_{0, s_1, \a} 
\lesssim_{M, s_1, \a} \e \gamma^{-(m + 3)}  \| \Delta_{12} i \|_{s_1 + \aleph_7^{(m, 1)}(M,\a)} \,. \label{stima derivata Wm0} 
\end{align}
\end{lemma}

\begin{proof}
We begin with proving \eqref{stima wm0}. 
By \eqref{norm1 parameter}-\eqref{norm1} one has 
\[
\norma \mathrm{Op}(w_m^{(0)}) \norma_{-\frac{m}{2}, s,\, \a}^{k_0, \g} 
\lesssim_{k_0,\nu} \, \max_{\b \in [0,\a]} \sup_{\xi \in \R} \, 
\langle \xi \rangle^{\frac{m}{2} +\b} 
\big\| \pa_\xi^\b w_m^{(0)}(\cdot, \cdot, \cdot, \xi) \big\|_s^{k_0,\g}.
\]
By \eqref{definizione wm (0)} 
and \eqref{2802.2}, for each $\xi \in \R$ one has 
\[
\| \pa_\xi^\b w_m^{(0)}(\cdot, \cdot, \cdot, \xi) \|_s^{k_0,\g} 
\lesssim_{k_0,\nu} \, \g^{-1} \big\| \pa_\xi^\b \big( 
\langle p_7^{(m)} \rangle_\vphi(\cdot, \xi) -  p_7^{(m)}(\cdot, \cdot, \xi) \big) \big\|_{s+\mu}^{k_0,\g}
\]
where $\mu$ is defined in \eqref{Diophantine-1} with $k+1 = k_0$. 
Hence 
$ \norma \mathrm{Op}(w_m^{(0)}) \norma_{-\frac{m}{2}, s,\, \a}^{k_0, \g} 
\lesssim_{k_0,\nu} \, \g^{-1} \norma P_7^{(m)} \norma_{-\frac{m}{2}, s + \mu, \a}^{k_0,\g} $ 
and \eqref{stima wm0} follows by \eqref{2007.5}. 
The other bounds are proved similarly, 
using the explicit formula \eqref{definizione wm (0)}, 
estimates \eqref{2007.5}-\eqref{2007.50} 
and \eqref{2802.2}, 
\eqref{estimate composition parameters}, 
and Lemma \ref{Neumann pseudo diff}. 
\end{proof}

By \eqref{L7m Wm0} and \eqref{equazione omologica descent method m0} we get 
%\begin{align}
%L_7^{(m)} W_m^{(0)} & = W_m^{(0)} \big( \omega \cdot \partial_\vphi + \Lambda_m(D) \big) 
%+ \langle p_7^{(m)} \rangle_\vphi(x, D) + \big[ \Lambda_m (D), w_m^{(0)}(\vphi, x, D) \big]  
%+ P_7^{(m)} w_m^{(0)}(\vphi, x, D) 
%\nonumber \\
%& = W_m^{(0)} \big( \omega \cdot \partial_\vphi + \Lambda_m(D) + \langle p_7^{(m)} \rangle_\vphi(x, D) \big) 
%- w_m^{(0)}(\vphi, x, D) \langle p_7^{(m)} \rangle_\vphi(x, D)  
%\nonumber \\ & \quad \, 
%+ \big[ \Lambda_m (D), w_m^{(0)}(\vphi, x, D) \big]  + P_7^{(m)} w_m^{(0)}(\vphi, x, D) \nonumber 
%\end{align}
%and therefore we obtain 
the even and reversible operator
\begin{equation}\label{Wm0 L7m Wm0}
L_7^{(m, 1)} := (W_m^{(0)})^{- 1} L_7^{(m)} W_m^{(0)}  
= \omega \cdot \partial_\vphi + \Lambda_m(D) + \langle p_7^{(m)} \rangle_\vphi(x, D) + P_7^{(m, 1)} 
\end{equation}
where 
\begin{equation}\label{P7 m1}
P_7^{(m, 1)} := (W_m^{(0)})^{- 1} \Big( \big[\Lambda_m (D), w_m^{(0)}(\vphi, x, D) \big]  + P_7^{(m)} w_m^{(0)}(\vphi, x, D) -  w_m^{(0)}(\vphi, x, D) \langle p_7^{(m)} \rangle_\vphi(x, D) \Big) 
\end{equation}
is in $ OPS^{- \frac{m}{2} - \frac12}$, 
as we prove in Lemma  \ref{lemma stima induttiva descent passo m0} below.   
Thus the term of order $-\frac{m}{2}$ in \eqref{Wm0 L7m Wm0}  
is $\langle p_7^{(m)} \rangle_\vphi(x, D)$, which does not depend on $\ph$ any more. 

\begin{lemma}\label{lemma stima induttiva descent passo m0}
The operators $ \langle p_7^{(m)} \rangle_\vphi(x, D) $ and $ P_7^{(m, 1)} $ are even and reversible. 
The operator $ P_7^{(m, 1)}$ in \eqref{P7 m1} is in $ OPS^{- \frac{m}{2} - \frac12}$.
For any $\alpha, M > 0$ there exists a constant $\aleph_7^{(m, 2)}(M, \alpha) > 0$ 
(depending also on $k_0, \tau, \nu$), 
larger than the constant $\aleph_7^{(m,1)}(M, \alpha)$ 
appearing in Lemma \ref{lemma stima induttiva descent passo m-1},
such that, if \eqref{ansatz I delta} holds with $\mu_0 \geq \aleph_7^{(m,2)}(M, \alpha)$, 
then for any $s \geq s_0$
\begin{align}
 & \norma P_7^{(m, 1)} \norma_{- \frac{m}{2} - \frac12, s, \alpha}^{k_0, \gamma} \lesssim_{M, s, \alpha} \e \gamma^{- (m + 3)} (1 + \| \fracchi_0\|_{s + \aleph_7^{(m,2)}(M, \alpha)}^{k_0, \gamma})\,, \label{stima P7m1} \\
 & \norma \Delta_{12} P_7^{(m, 1)} \norma_{- \frac{m}{2} - \frac12, s_1, \alpha} \lesssim_{M, s_1, \alpha} \e \gamma^{- (m + 3)} \| \Delta_{12} i\|_{s_1 + \aleph_7^{(m,2)}(M, \alpha)}\,. 
 \label{stima derivata P7m1}
\end{align}
\end{lemma}

\begin{proof}
Since $ P_7^{(m)} (x, D)  $ is even and reversible by the inductive claim, its 
$ \vphi $-average $ \langle p_7^{(m)} \rangle_\vphi(x, D) $ defined in  \eqref{2007.6} is  even and reversible as well. 
Since $ \Lambda_m (D) $ is reversible and $ W_m^{(0)} $ is reversibility preserving we obtain that 
$ P_7^{(m, 1)}$  in \eqref{P7 m1}
is even and reversible. 

Let us prove that $P_7^{(m, 1)} $ is in $ OPS^{- \frac{m}{2} - \frac12}$. 
Since $\Lambda_m(D) \in OPS^{\frac12}$ and the operators
 $P_7^{(m)}$, $w_m^{(0)}(\vphi, x, D) $ are in $ OPS^{- \frac{m}{2}}$, with $m \geq 1$, 
we have that $ [\Lambda_m (D), w_m^{(0)}(\vphi, x, D)]$ is in $OPS^{- \frac{m}{2} - \frac12}$ and $P_7^{(m)} w_m^{(0)}(\vphi, x, D)$ is in $OPS^{-m} \subseteq OPS^{- \frac{m}{2} - \frac12}$. 
Moreover also $w_m^{(0)}(\vphi, x, D) \langle p_7^{(m)} \rangle_\vphi(x, D)   \in OPS^{- m} \subseteq OPS^{- \frac{m}{2} - \frac{1}{2}}$, since $m \geq 1$. 
Since $ (W_m^{(0)})^{- 1} $ is in $ OPS^0 $, 
the remainder $P_7^{(m, 1)}$ is in $ OPS^{- \frac{m}{2} - \frac12}$. 
Bounds \eqref{stima P7m1}-\eqref{stima derivata P7m1} follow by the explicit expression in  
\eqref{P7 m1}, Lemma \ref{lemma stima induttiva descent passo m-1}, 
estimates \eqref{2007.4}-\eqref{stima induttiva Lambda m (D)}, 
and \eqref{stima astratta simbolo mediato}, 
\eqref{estimate composition parameters}, \eqref{stima commutator parte astratta}.
\end{proof}

\subsubsection{Elimination of the dependence on $x$}

In this subsection we eliminate the dependence on $x$ 
from $\langle p_7^{(m)} \rangle_\vphi(x, D)$, 
which is the only term of order $-m/2$ in \eqref{Wm0 L7m Wm0}.
To this aim we conjugate $ L_7^{(m, 1)} $ in \eqref{Wm0 L7m Wm0} by 
a transformation of the form 
\begin{equation}\label{definizione Wm1}
W_m^{(1)} := {\rm Id} + w_m^{(1)}(x, D), \quad 
\text{where} \quad w_m^{(1)}(x, \xi) \in S^{- \frac{m}{2} + \frac12}
\end{equation}
is a $\vphi$-independent symbol, which we shall fix in \eqref{2007.7} (for $ m = 1 $) 
and \eqref{2007.9} (for $ m \geq 2 $). 
We denote the space average of the function $ \langle p_7^{(m)} \rangle_\vphi(x, \xi) $ 
defined in  \eqref{2007.6} by
\begin{equation}
\label{correzione media passo m descent}
\langle p_7^{(m)} \rangle_{\ph,x} (\xi) 
:= \frac{1}{2 \pi} \int_\T \langle p_7^{(m)} \rangle_\vphi(x, \xi)\, d x 
= \frac{1}{(2 \pi)^{\nu + 1}} \int_{\T^{\nu + 1}} p_7^{(m)}(\vphi, x, \xi)\,d \vphi\, d x \,.
\end{equation}
By  \eqref{Wm0 L7m Wm0}, we compute 
\begin{align}
L_7^{(m, 1)} W_m^{(1)} & = W_m^{(1)} \Big( \omega \cdot \partial_\vphi + \Lambda_m(D) + \ppp \Big) + 
\big[ \Lambda_m(D), w_m^{(1)}(x, D) \big] + \langle p_7^{(m)} \rangle_\vphi(x, D) - \ppp (D) \nonumber\\
& \qquad + \langle p_7^{(m)} \rangle_\vphi(x, D) w_m^{(1)}(x, D) - w_m^{(1)}(x, D) \ppp(D)  +  P_7^{(m, 1)} W_m^{(1)} \,. \label{L7m1 Wm1}
\end{align}
By formulas \eqref{expansion symbol}, \eqref{rNTaylor} (with $N = 1$) and \eqref{symbol commutator}, \eqref{Expansion Moyal bracket}, 
\begin{align}
\langle p_7^{(m)} \rangle_\vphi(x, D) w_m^{(1)}(x, D) & = {\rm Op}\Big(\langle p_7^{(m)} \rangle_\vphi(x, \xi) w_m^{(1)}(x, \xi) \Big) + r_{\langle p_7^{(m)} \rangle_\vphi , w_m^{(1)} }( x, D)\,, \label{georgia 0} \\
w_m^{(1)}(x, D) \ppp(D)  & = {\rm Op}\Big( w_m^{(1)}(x, \xi) \ppp(\xi) \Big) + r_{w_m^{(1)}, \ppp}( x, D)\,, \label{georgia 1} \\
 \big[\Lambda_m(D), w_m^{(1)}(x, D) \big]  & = {\rm Op}\Big( - \ii \partial_\xi \Lambda_m(\xi) \partial_x w_m^{(1)}(x, \xi) \Big) + \mathtt r_2(\Lambda_m, w_m^{(1)})(x, D)\, \label{georgia 2}
\end{align}
where $r_{\langle p_7^{(m)} \rangle_\ph , w_m^{(1)} }$, $r_{w_m^{(1)}, \ppp} \in$ $S^{- m - \frac12} \subset S^{- \frac{m}{2} - \frac12}$, $ \mathtt r_2(\Lambda_m, w_m^{(1)})(x, D) \in$ $S^{- \frac{m}{2} - 1} \subset S^{- \frac{m}{2} - \frac12}$. Let  $\chi_0 \in {\cal C}^\infty(\R, \R)$ be a cut-off function satisfying 
\begin{equation}\label{altro cut off chi 0}
\begin{aligned}
\chi_0(\xi) = \chi_0(- \xi) \ \ \forall \xi \in \R\,, \quad 
\chi_0(\xi) = 0 \ \  \forall |\xi| \leq \frac45 \, , \quad 
\chi_0(\xi) = 1 \ \  \forall |\xi| \geq \frac78\,.
\end{aligned}
\end{equation}
By \eqref{L7m1 Wm1}-\eqref{georgia 2}, one has  
\begin{align}
L_7^{(m, 1)} W_m^{(1)} 
& = W_m^{(1)} \big( \omega \cdot \partial_\vphi + \Lambda_m(D) + \ppp(D) \big)  
\nonumber\\
& \quad + {\rm Op} \Big( - \ii \partial_\xi \Lambda_m(\xi) \partial_x w_m^{(1)}(x, \xi) 
+ \chi_0(\xi) \big( \langle p_7^{(m)} \rangle_\vphi(x, \xi) 
- \ppp(\xi) \big)
\label{gricia 0} \\ 
& \quad \qquad \quad \ 
+ \chi_0(\xi)\big( \langle p_7^{(m)} \rangle_\vphi(x, \xi) - \ppp(\xi) \big) 
w_m^{(1)}(x, \xi) \Big) 
\label{gricia 1}\\
& \quad + {\rm Op} \Big(  
\big( 1 - \chi_0(\xi) \big) \big( \langle p_7^{(m)} \rangle_\vphi(x, \xi) 
- \ppp(\xi) \big) \big(1 + w_m^{(1)}(x, \xi) \big) \Big) 
\nonumber\\
& \quad + \mathtt r_2(\Lambda_m, w_m^{(1)})(x, D)  + r_{\langle p_7^{(m)} \rangle_\vphi , w_m^{(1)} }( x, D) - r_{w_m^{(1)}, \ppp}( x, D) +  P_7^{(m, 1)} W_m^{(1)} \,. 
\label{georgia 4}
\end{align}
The terms containing $1 - \chi_0(\xi)$ are in $S^{- \infty}$, 
by definition \eqref{altro cut off chi 0}. The term in \eqref{gricia 0} is of order $-\frac{m}{2}$ and the term in \eqref{gricia 1} is of order $- m + \frac12$, which equals $- \frac{m}{2}$ for $m = 1$, 
and is strictly less than $-\frac{m}{2}$ for $m \geq 2$.
Hence we split the two cases $m=1$ and $m \geq 2$. 

\smallskip

\textbf{First case:} $m=1$. \ 
We look for $w_m^{(1)}(x, \xi) = w_1^{(1)}(x, \xi)$ such that 
\begin{equation}\label{equazione omologica descent method m1}
- \ii \partial_\xi \Lambda_1(\xi) \partial_x w_1^{(1)}(x, \xi) 
+ \chi_0(\xi) \Big(\langle p_7^{(1)} \rangle_\vphi(x, \xi) - \pppuno(\xi) \Big)
% + \chi_0(\xi)\big(\langle p_7^{(1)} \rangle_\vphi(x, \xi) - \pppuno(\xi) \big) 
( 1 + w_1^{(1)}(x, \xi)) = 0\,. 
\end{equation}
By \eqref{r1(D) = 0} and recalling \eqref{definizione Dm}, \eqref{cut off simboli 1}, 
for $|\xi| \geq 4/5$ one has $\Lambda_1(\xi) = \ii \mathtt m_{\frac12} \tanh^{\frac12}(\h |\xi|) |\xi|^{\frac12} $. 
Since, by \eqref{stime lambda 1}, $|\mathtt m_{\frac12}| \geq 1/2$ for $\e \gamma^{- 1}$ small enough, we have 
\begin{equation}\label{lower bound partial xi Lambda1}
\inf_{|\xi| \geq \frac45} |\xi|^{\frac12}|\partial_\xi \Lambda_1(\xi)| \geq \delta > 0\,, 
\end{equation} 
where $\d$ depends only on $\h_1$.
% After these preliminary observations, 
%We solve the equation \eqref{equazione omologica descent method m1} by variation of the constants. 
Using that $\langle p_7^{(1)} \rangle_\vphi - \pppuno$ has zero average in $x$, 
we choose the solution of \eqref{equazione omologica descent method m1} 
given by the periodic function
\begin{equation} \label{2007.7}
w_1^{(1)}(x, \xi) := \exp \big( g_1(x, \xi) \big) - 1, 
\quad \,  
g_1(x, \xi) := \begin{cases}
\dfrac{\chi_0(\xi) \partial_x^{-1} 
\big( \langle p_7^{(1)} \rangle_\vphi(x, \xi) - \pppuno(\xi) \big)}{\ii \partial_\xi \Lambda_1(\xi)}
& \text{if} \ |\xi| \geq \frac45 \\
0 
& \text{if} \ |\xi| \leq \frac45\,. 
\end{cases}
\end{equation}
Note that, by the definition of the cut-off function $\chi_0$ given in \eqref{altro cut off chi 0}, recalling \eqref{r1(D) = 0}, \eqref{lower bound partial xi Lambda1} 
and applying estimates \eqref{Norm Fourier multiplier}, \eqref{stime lambda 1}, 
the Fourier multiplier $\frac{\chi_0(\xi)}{\partial_\xi \Lambda_1(\xi)}$ is a symbol in $S^{\frac12}$ and satisfies 
\begin{equation}\label{stima chi 0 Lambda1}
\Bignorma
{\rm Op} \Big( \frac{\chi_0(\xi)}{\partial_\xi \Lambda_1(\xi)}\Big)
\Bignorma_{\frac12, s , \alpha}^{k_0, \gamma} 
\lesssim_{\alpha} 1\,, \quad 
\Bignorma \Delta_{12} {\rm Op}  \Big(  \frac{\chi_0(\xi)}{\partial_\xi \Lambda_1(\xi)}\Big) \Bignorma_{\frac12, s_1 , \alpha} 
\lesssim_{\alpha} \e \gamma^{- 1}  \| \Delta_{12} i\|_{s_1}\,.
\end{equation} 
Therefore the function $g_1(x, \xi)$ is a well-defined symbol in $S^{0}$. 

\smallskip

\textbf{Second case:} $m \geq 2$. \ 
We look for $w_m^{(1)}(x, \xi)$ such that 
\begin{equation} \label{2007.8}
- \ii \partial_\xi \Lambda_m(\xi) \partial_x w_m^{(1)}(x, \xi) 
+ \chi_0(\xi)\big( \langle p_7^{(m)} \rangle_\vphi(x, \xi) - \ppp(\xi) \big)= 0 \, .
\end{equation}
Recalling \eqref{definizione Lambda m (D)}-\eqref{2007.4} and \eqref{lower bound partial xi Lambda1}, one has that 
\begin{align}
\inf_{|\xi| \geq \frac45} |\xi|^{\frac12} |\partial_\xi \Lambda_m(\xi)| & \geq \inf_{|\xi| \geq \frac45} |\xi|^{\frac12} |\partial_\xi \Lambda_1(\xi)| - \sup_{\xi \in \R} |\xi|^{\frac12} |\partial_\xi r_m(\xi)| \geq \delta - \norma r_m(D) \norma_{- \frac12, 0, 1}  \nonumber\\
& \geq \delta - C \e \gamma^{- (m + 1)} \geq \delta/2 \label{lower bound partial xi Lambda m}
\end{align}
for $\e \gamma^{- (m + 1)}$ small enough. 
Since $\langle p_7^{(m)} \rangle_\vphi(x, \xi) - \ppp(\xi)$ has zero average in $x $, 
we choose the solution of \eqref{2007.8} given by the periodic function  
\begin{equation} \label{2007.9}
w_m^{(1)}(x, \xi) 
:= \begin{cases}
\dfrac{\chi_0(\xi)\pa_x^{-1} \big( \langle p_7^{(m)} \rangle_\vphi(x, \xi) - \ppp(\xi) \big) }{\ii \partial_\xi \Lambda_m(\xi) } & \, \text{if} \  |\xi| \geq \frac45 \\
0 & \, \text{if} \  |\xi| \leq \frac45\,. 
\end{cases}
\end{equation}
By the definition of the cut-off function $\chi_0$ in \eqref{altro cut off chi 0}, 
recalling \eqref{r1(D) = 0}, \eqref{definizione Lambda m (D)}, 
\eqref{lower bound partial xi Lambda m}, 
and applying estimates \eqref{Norm Fourier multiplier}, \eqref{stime lambda 1}, \eqref{2007.4}, 
the Fourier multiplier $\frac{\chi_0(\xi)}{\partial_\xi \Lambda_m(\xi)}$ 
is a symbol in $S^{\frac12}$ and satisfies 
\begin{equation}\label{stima chi 0 Lambdam}
\Bignorma {\rm Op} \Big( \frac{\chi_0(\xi)}{\partial_\xi \Lambda_m(\xi)}\Big)
\Bignorma_{\frac12, s , \alpha}^{k_0, \gamma} 
\lesssim_{M,\alpha} 1\,, \quad 
\Bignorma \Delta_{12} {\rm Op}  \Big(  \frac{\chi_0(\xi)}{\partial_\xi \Lambda_m(\xi)}\Big) \Bignorma_{\frac12, s_1 , \alpha} 
\lesssim_{M,\alpha} \e \gamma^{- (m + 1)}  \| \Delta_{12} i\|_{s_1 + \aleph_7^{(m)}(M, \alpha)}\,.
\end{equation} 
By \eqref{lower bound partial xi Lambda m}, 
the function $w_m^{(1)}(x, \xi)$ is a well-defined symbol in $S^{- \frac{m}{2} + \frac12}$.

\medskip

In both cases $m=1$ and $m \geq 2$, we have eliminated the terms of order $-\frac{m}{2}$ 
from the right hand side of \eqref{georgia 4}. 

\begin{lemma} \label{lemma stima induttiva descent passo m1}
The operators $W_m^{(1)}$ defined in \eqref{definizione Wm1}, 
\eqref{2007.7} for $ m = 1 $, and 
\eqref{2007.9} for $ m \geq 2 $, 
are even and reversibility preserving.
For any $M, \alpha > 0$ there exists a constant 
$\aleph_7^{(m,3)}(M, \alpha) > 0$ (depending also on $k_0, \tau, \nu$), 
larger than the constant $\aleph_7^{(m,2)}(M, \alpha)$ 
appearing in Lemma \ref{lemma stima induttiva descent passo m0}, 
such that, if \eqref{ansatz I delta} holds with $\mu_0 \geq \aleph_7^{(m,3)}(M, \alpha)$, 
then for any $s \geq s_0$
\begin{align}
& \norma {\rm Op}(w_m^{(1)}) \norma_{- \frac{m}{2} + \frac12, s, \alpha}^{k_0, \gamma} 
\lesssim_{M, s, \alpha} \e \gamma^{-(m + 3)} 
\big(1 + \| \fracchi_0 \|_{s + \aleph_7^{(m,3)}(M, \alpha)}^{k_0, \gamma} \big) 
\label{stima wm1} \\
& \norma \Delta_{12} {\rm Op}(w_m^{(1)}) \norma_{- \frac{m}{2} + \frac12, s_1, \alpha} 
\lesssim_{M, s_1, \alpha} \e \gamma^{-(m + 3)} 
\| \Delta_{12} i \|_{s_1 + \aleph_7^{(m,3)}(M, \alpha)}\, .  
\label{stima derivata wm1} 
\end{align}
As a consequence, the transformation $ W_m^{(1)} $ is invertible and 
\begin{align}
& \norma (W_m^{(1)})^{\pm 1} - {\rm Id} \norma_{0, s, \a}^{k_0, \gamma} 
\lesssim_{M, s,\a} 
\e \gamma^{-(m + 3)} \big(1 + \| \fracchi_0 \|_{s + \aleph_7^{(m,3)}(M,\a)}^{k_0, \gamma}\big) \label{stima Wm1 - Id} \\
& \norma \Delta_{12} (W_m^{(1)})^{\pm 1} \norma_{0, s_1, \a} 
\lesssim_{M, s_1, \a} \e \gamma^{-(m + 3)} \| \Delta_{12} i \|_{s_1 + \aleph_7^{(m,3)}(M,\a)} \,. \label{stima derivata Wm1} 
\end{align}
\end{lemma}

\begin{proof}
The lemma follows by the explicit expressions in 
\eqref{definizione Wm1}, 
\eqref{2007.7}, \eqref{2007.9}, \eqref{correzione media passo m descent}, 
by estimates \eqref{Norm Fourier multiplier}, 
\eqref{stima astratta simbolo mediato}, 
\eqref{lemma composizione multiplier}, 
Lemmata \ref{lemma stime Ck parametri}, 
\ref{lemma tame norma commutatore}, 
\ref{Neumann pseudo diff} 
and estimates \eqref{2007.5}, \eqref{2007.50}, 
\eqref{stima chi 0 Lambda1}, \eqref{stima chi 0 Lambdam}. 
\end{proof}

In conclusion, by \eqref{georgia 4}, \eqref{equazione omologica descent method m1} and \eqref{2007.8},
we obtain the even and reversible operator
\begin{equation}\label{definizione L7 m+1}
L_7^{(m + 1)} := (W_m^{(1)})^{- 1} L_7^{(m, 1)} W_m^{(1)} = \omega \cdot \partial_\vphi + \Lambda_{m + 1}(D) 
+ P_7^{(m + 1)}
\end{equation}
where 
\be
\begin{aligned} \label{nuova parte diagonale descent method} 
\Lambda_{m + 1}(D) 
& := \Lambda_m(D) + \ppp(D)  = \ii \mathtt m_{\frac12} T_\h^{\frac12} |D|^{\frac12} 
+ r_{m + 1}(D)\,, \\ 
r_{m + 1}(D) & := r_{m }(D) + \ppp(D)\,,
\end{aligned}
\ee
and 
\begin{align}
P_7^{(m + 1)} & := 
(W_m^{(1)})^{-1} \Big\{ 
\mathtt r_2(\Lambda_m, w_m^{(1)})(x, D)  + r_{\langle p_7^{(m)} \rangle_\vphi , w_m^{(1)} }( x, D) - r_{w_m^{(1)}, \ppp}( x, D) +  P_7^{(m, 1)} W_m^{(1)} 
\notag \\ & \qquad \qquad \qquad 
+ \chi_{(m \geq 2)}  \Op \Big(  
\chi_0(\xi)\big( \langle p_7^{(m)} \rangle_\vphi(x, \xi) - \ppp(\xi) \big)
w_m^{(1)}(x,\xi) \Big)  \nonumber\\
& \quad \quad \qquad \qquad  + {\rm Op}\Big( (1 - \chi_0(\xi)) \big( \langle p_7^{(m)} \rangle_\vphi(x, \xi) - \ppp(\xi) \big) \big(1 + w_m^{(1)}(x, \xi) \big)  \Big)\Big\}
\label{nuovo resto m+1 descent}
\end{align}
with $\chi_{(m \geq 2)}$ equal to $1$ if $m \geq 2$, and zero otherwise.

\begin{lemma} \label{lemma detto 13.6}
The operators $\Lambda_{m + 1}(D)$, $r_{m + 1}(D)$, $P_7^{(m + 1)}$ are even and reversible. 
For any $M, \alpha > 0$ there exists a constant $\aleph_7^{(m + 1)}(M, \alpha) > 0$ 
(depending also on $k_0, \tau, \nu$), 
larger than the constant $\aleph_7^{(m,3)}(M, \alpha)$ appearing in Lemma 
\ref{lemma stima induttiva descent passo m1}, 
such that, if \eqref{ansatz I delta} holds with $\mu_0 \geq \aleph_7^{(m + 1)}(M, \alpha)$, 
then for any $s \geq s_0$
\begin{align}
& \norma r_{m + 1}(D) \norma_{- \frac12, s, \alpha}^{k_0, \gamma} \lesssim_{M, \alpha} \e \gamma^{- (m+2)}\,, \quad \norma \Delta_{12} r_{m + 1}(D) \norma_{- \frac12, s_1, \alpha} \lesssim_{M, \alpha} \e \gamma^{- (m+2)} \| \Delta_{12} i \|_{s_1 + \aleph_7^{(m + 1)}(M, \alpha)} 
\label{2007.1} \\ 
 & \norma P_7^{(m + 1)} \norma_{- \frac{m}{2} - \frac12, s, \alpha}^{k_0, \gamma} \lesssim_{M, s, \alpha} \e \gamma^{- (m + 3)} \big(1 + \| \fracchi_0\|_{s + \aleph_7^{(m +1)}(M, \alpha)}^{k_0, \gamma} \big)\,, \label{2007.2} \\ 
 & \norma \Delta_{12} P_7^{(m + 1)} \norma_{- \frac{m}{2} - \frac12, s_1, \alpha} \lesssim_{M, s_1, \alpha} \e \gamma^{- (m + 3)} \| \Delta_{12} i \|_{s_1 + \aleph_7^{(m + 1)}(M, \alpha)}\,. \label{2007.3} 
\end{align}
\end{lemma}

\begin{proof}
Since the operator $ \langle p_7^{(m)} \rangle_\vphi(x, D)  $ is even and 
reversible by Lemma \ref{lemma stima induttiva descent passo m0}, % we get that 
the average $\ppp(D)$ defined in \eqref{correzione media passo m descent} 
is even and reversible as well (we use Remark \ref{rem:change}). 
Since $r_m(D)$, $\Lambda_{m }(D)$ are even and reversible by the inductive claim, 
then also $r_{m + 1}(D)$, $\Lambda_{m + 1}(D)$  defined in
 \eqref{nuova parte diagonale descent method} are even and reversible.
  
Estimates \eqref{2007.1}-\eqref{2007.3} for 
$r_{m + 1}(D)$ and $ P_7^{(m + 1)} $ defined respectively 
in \eqref{nuova parte diagonale descent method} and \eqref{nuovo resto m+1 descent}
follow by the explicit expressions of $ \langle p_7^{(m)} \rangle_{\ph,x} (\xi) $ 
in \eqref{correzione media passo m descent} 
and $ w_m^{(1)} $ in \eqref{2007.7} and \eqref{2007.9} (for $ m = 1 $ and $ m \geq 2 $ respectively), 
by applying \eqref{stima astratta simbolo mediato}, \eqref{Norm Fourier multiplier}, 
 \eqref{stima Wm1 - Id}-\eqref{stima derivata Wm1}, \eqref{stima P7m1}-\eqref{stima derivata P7m1}, 
\eqref{lemma composizione multiplier}, 
Lemmata \ref{lemma stime Ck parametri}, 
\ref{lemma tame norma commutatore}, 
% \ref{Neumann pseudo diff} 
and the inductive estimates 
\eqref{2007.4}-\eqref{stima induttiva Lambda m (D)}. 
\end{proof}

Thus, the proof of the inductive claims \eqref{2007.4}-\eqref{coniugazione-descent} is complete.

\subsubsection{Conclusion of the reduction of $ L_7^{(1)} $} 

Composing all the previous transformations,
we obtain  the even and reversibility preserving map
\begin{equation}\label{mappa finale descent method}
W := W_0 \circ W_1^{(0)} \circ W_1^{(1)} \circ \ldots \circ W_{2M-1}^{(0)} \circ W_{2M-1}^{(1)}\,,
\end{equation}
where $W_0$ is defined in \eqref{definizione M0 descent method} 
and for $m = 1, \ldots, 2M-1$, 
$ W_m^{(0)}, W_m^{(1)}$ are defined in \eqref{definizione Wm (0)}, \eqref{definizione Wm1}. 
 The order $M$ will be fixed in \eqref{relazione mathtt b N}.
By \eqref{L7 (m)}, \eqref{definizione Lambda m (D)}, \eqref{coniugazione-descent} at $ m = 2 M  $, 
 the operator $L_7$ in \eqref{primo operatore descent method} is conjugated,  for all $\om \in \mathtt{DC}(\g,\t)$, 
to the even and reversible operator
\begin{equation}\label{definizione L8}
L_8 := L_7^{(2M)} = W^{- 1} L_7 W 
= \omega \cdot \partial_\vphi + \Lambda_{2M}(D) + P_{2M}
\end{equation}
where $ P_{2M} := P_7^{(2M)}  \in OPS^{-M} $ and 
\begin{equation}\label{lambdone M}
\Lambda_{2M} (D) = \ii \mathtt m_{\frac12} T_\h^{\frac12} |D|^{\frac12} + r_{2M}(D) \, , 
\quad   r_{2M}(D) \in OPS^{- \frac12} \, . 
\end{equation}

\begin{lemma}\label{lemma conclusivo descent method}
Assume \eqref{ansatz I delta} with $\mu_0 \geq \aleph_7^{(2M)}(M, 0)$. Then, for any $s \geq s_0$, the following estimates hold: 
\begin{align}
& \norma r_{2 	M}(D)\norma_{- \frac12, s, 0}^{k_0, \gamma} \lesssim_{M} \e \gamma^{- (2M + 1)}\,, \quad \norma \Delta_{12} r_{2 M}(D) \norma_{- \frac12, s_1, 0} \lesssim_{M} \e \gamma^{- (2 M + 1)} \| \Delta_{12} i \|_{s_1 + \aleph_7^{(2M)}(M, 0)}\,, 
\label{2007.10} \\
& \norma  P_{2M}  \norma_{- M, s, 0}^{k_0, \gamma} \lesssim_{M, s} \e \gamma^{- 2(M+1)} 
\big(1 + \|\fracchi_0 \|^{k_0, \gamma}_{s + \aleph_7^{(2M)}(M, 0) } \big)\,, 
\label{2007.11} \\  
& \norma \Delta_{12} P_{2M} \norma_{- M , s_1, 0} \lesssim_{M, s_1} \e \gamma^{- 2(M+1)} \| \Delta_{12} i \|_{s_1 + \aleph_7^{(2M)}(M, 0) }\,, 
\label{2007.12} \\
& \norma W^{\pm 1} - {\rm Id}\norma^{k_0, \gamma}_{0, s, 0} 
\lesssim_{M, s} \e \gamma^{- 2(M+1)} \big(1 + \| \fracchi_0 \|_{s + \aleph_7^{(2M)}(M, 0)}^{k_0, \gamma} \big)\,, \label{stima composta tutte le mappe descent} \\
& \norma \Delta_{12} W^{\pm 1} \norma_{0, s_1, 0} 
\lesssim_{M, s_1} \e \gamma^{- 2(M+1)} \| \Delta_{12} i \|_{s_1 + \aleph_7^{(2M)}(M, 0)}\,. 
\label{stima derivata composta tutte le mappe descent} 
\end{align}
\end{lemma}

\begin{proof}
Estimates \eqref{2007.10}, \eqref{2007.11}, \eqref{2007.12} follow by 
\eqref{2007.4}, \eqref{2007.5}, \eqref{2007.50} applied for $m = 2M $. 
Estimates \eqref{stima composta tutte le mappe descent}-\eqref{stima derivata composta tutte le mappe descent} for the map $ W $ defined in  \eqref{mappa finale descent method}, and its inverse $ W^{-1} $, 
follow by 
\eqref{stime M0 descent method}, \eqref{stima Wm0 - Id}, \eqref{stima derivata Wm0}, 
\eqref{stima Wm1 - Id}, \eqref{stima derivata Wm1}, 
applying the composition estimate \eqref{estimate composition parameters} (with $m=m'=\a=0$).
\end{proof}

%\begin{remark}\label{alpha descent}
%In order to estimate the remainder $ \norma P_{2M} \norma_{- M, s, 0}^{k_0, \gamma}$ 
%we need to estimate $\norma P_7^{(m)} \norma_{- \frac{m}{2}, s, \alpha}^{k_0, \gamma}$ 
%for $\alpha > 0$ at each step. 
%The reason is that, at each step of the iteration, $P_7^{(m +1)}$ involves two $\xi$-derivatives of $w_m^{(1)}$ (and hence of $p_m(x, \xi)$, see the term $\mathtt r_2(\Lambda_m, w_m^{(1)})$ in \eqref{georgia 4}). 
%For the estimate of the last remainder $P_7^{(2M)}$ it is enough to choose $\alpha = 0$. 
%\end{remark}

Since $ \Lambda_{2M}(D) $ is even and reversible, we have that 
\be\label{Lambda2M}
\Lambda_{2M}(\xi), r_{2M}(\xi) \in \ii \R  \qquad {\rm and} \qquad 
\Lambda_{2M}(\xi) = \Lambda_{2M}(- \xi) \, , \ r_{2M}(\xi) = r_{2M}(- \xi) \, .
\ee 
In conclusion, we write the even and reversible operator $ L_8 $ in \eqref{definizione L8} as
\be\label{def:newL8}
L_8 = \omega \cdot \partial_\vphi + \ii D_8  + P_{2M} 
\ee
where $ D_8 $ is the diagonal operator 
\begin{align}\label{introduzione primo operatore diagonale riducibilita}
 D_8 := - \ii \Lambda_{2M} (D) := {\rm diag}_{j \in \Z} (\mu_j) \, ,  
& \qquad \mu_j := \mathtt m_{\frac12} |j|^{\frac12} \tanh (\h |j|)^{\frac12} + r_j \,, 
\quad r_j := - \ii \, r_{2M}(j) \, , \\
& \qquad  \mu_j, r_j \in \R\,, \qquad \mu_{j} = \mu_{-j} \, , \ r_{j} = r_{-j}\,, \quad \forall j \in \Z\,,
\label{code - 1/2 inizio riducibilita}
\end{align}
with $ r_j \in \R $ satisfying, by 
\eqref{2007.10}, 
\begin{align}
& \label{stima code - 1/2 inizio riducibilita}
\sup_{j \in \Z}|j|^{\frac12} |r_j|^{k_0, \gamma} \lesssim_M \e \gamma^{- (2M + 1)}\,, \quad \sup_{j \in \Z} |j|^{\frac12} |\Delta_{12} r_j| \lesssim_M \e \gamma^{- (2M + 1)} \| \Delta_{12} i \|_{s_1 + \aleph_7^{(2M)}(M, 0)}\, 
\end{align}
and $ P_{2M} \in OPS^{-M}$ satisfies \eqref{2007.11}-\eqref{2007.12}. 

From now on, we do not need to expand further the operators in 
decreasing orders and  we will only estimate the tame constants of the operators acting on 
periodic functions 
(see Definitions \ref{def:Ck0} and \ref{def:op-tame}).

\begin{remark} \label{rem:25 maggio 2017}
In view of Lemma \ref{lemma: action Sobolev}, the tame constants of $ P_{2M} $
can be deduced by estimates 
%it is sufficient to estimate the remainder $ P_{2M}$ 
\eqref{2007.11}-\eqref{2007.12} of the pseudo-differential norm 
$\norma P_{2M} \norma_{-M,s,\a}$ with $\a = 0$. 
% instead of their pseudo-differential norms (Definition \ref{def:pseudo-norm}).
%In the next sections we will only estimate the tame constants of the operators 
%(see Definitions \ref{def:Ck0} and \ref{def:op-tame})
%instead of their pseudo-differential norms (Definition \ref{def:pseudo-norm}).
% (thus differentiate symbols with respect to $\xi$) 
%and we can restrict the frequencies to $\xi = j \in \Z$,
%see \eqref{introduzione primo operatore diagonale riducibilita}-\eqref{code - 1/2 inizio riducibilita}.
% On the other hand, 
The iterative reduction in decreasing orders % up to regularizing operators 
performed in the previous sections 
cannot be set in $\norma \ \norma_{-M,s,0}$ norms, 
because each step of the procedure requires some derivatives of symbols 
with respect to $\xi$ (in the remainder of commutators, in the Poisson brackets of symbols, 
and also in \eqref{2007.9}), 
and $\a$ keeps track of the regularity of symbols with respect to $\xi$.
\end{remark}

%\begin{remark} \label{rem:25 maggio 2017}
%In the next sections % From now on, pseudo-differential calculus is no longer needed, 
%we will only estimate the tame constants of the operators 
%(see Definitions \ref{def:Ck0} and \ref{def:op-tame})
%% Lemma \ref{lemma operatore e funzioni dipendenti da parametro}
%instead of their pseudo-differential norms (Definition \ref{def:pseudo-norm}).
%From now on, we do not need to expand further the operators in % asymptotic  
%decreasing orders
%% (thus differentiate symbols with respect to $\xi$) 
%and we can restrict the frequencies to $\xi = j \in \Z$,
%see \eqref{introduzione primo operatore diagonale riducibilita}-\eqref{code - 1/2 inizio riducibilita}.
%For this reason, it is sufficient to estimate the remainder $P_{2M}$ 
%(see \eqref{2007.11}-\eqref{2007.12}) in the pseudo-differential norm 
%$\norma \ \norma_{-M,s,\a}$ with $\a = 0$, in view of Lemma \ref{lemma: action Sobolev}.
%On the other hand, the iterative procedure we have performed in this section 
%cannot be set in $\norma \ \norma_{-M,s,0}$ norms, 
%because each step of the procedure involves some derivatives of symbols 
%with respect to $\xi$ (in the remainder of commutators, in the Poisson brackets of symbols, 
%and also in \eqref{2007.9}), 
%and $\a$ keeps track of the regularity of symbols with respect to $\xi$.
%\end{remark}

\subsection{Conjugation of ${\cal L}_7$} \label{subsec:luglio.3}

In the previous subsections \ref{subsec:luglio.1}-\ref{subsec:luglio.2} 
we have conjugated the operator $L_7$ defined in \eqref{primo operatore descent method} 
to $L_8$ in \eqref{definizione L8}, whose symbol is constant in $ (\vphi, x) $, up to 
smoothing remainders of order $-M$.
Now we conjugate the whole operator $\mL_7$ in \eqref{forma-cal-L7} 
by the real, even and reversibility preserving map
\be\label{definizione cal WM}
{\cal W} := \begin{pmatrix}
W & 0 \\
0 & \overline W
\end{pmatrix}
\ee
where $ W $ is defined in \eqref{mappa finale descent method}. By 
\eqref{definizione L8}, \eqref{def:newL8} we obtain,  for all $\om \in \mathtt{DC}(\g,\t)$, the real, even and
reversible operator
\begin{equation}\label{cal L8}
{\cal L}_8 := {\cal W}^{- 1} {\cal L}_7 {\cal W} 
= \omega \cdot \partial_\vphi + \ii {\cal D}_8 + \ii \Pi_0 +  {\cal T}_8
\end{equation}
where $ {\cal D}_8 $ is the diagonal operator 
\begin{equation}\label{def:cal-L8}
{\cal D}_8 := \begin{pmatrix}
D_8 & 0 \\
0 & - D_8
\end{pmatrix}\,, 
\end{equation}
with $ D_8 $ defined in \eqref{introduzione primo operatore diagonale riducibilita}, and the remainder 
${\cal T}_8$ is 
\begin{equation}\label{definizione cal T8}
{\cal T}_8 := \ii {\cal W}^{-1}  \Pi_0 {\cal W} - \ii \Pi_0 
+ {\cal W}^{- 1} {\cal T}_7 {\cal W} + {\cal P}_{2M}\, , 
\quad 
\mP_{2M} := \begin{pmatrix} 
P_{2M} & 0 \\ 
0 & \overline{P_{2M}} \end{pmatrix}
\end{equation}
with  $P_{2M}$ defined in \eqref{definizione L8}.
Note that ${\cal T}_8$ 
is defined on the whole 
parameter space $\R^\nu \times [\h_1, \h_2]$. 
Therefore the operator in the right hand side in \eqref{cal L8} is defined on $\R^\nu \times [\h_1, \h_2]$ as well. 
This defines the extended operator $\mL_8$ on $\R^\nu \times [\h_1, \h_2]$.  
 
\begin{lemma}\label{lemma cal L8}
For any $M > 0$, there exists a constant $\aleph_8(M) > 0$ (depending also on $\tau, \nu, k_0$) such that, if \eqref{ansatz I delta} holds with $\mu_0 \geq \aleph_8(M)$, then for any $s \geq s_0$
\begin{align}
& \norma {\cal W}^{\pm 1} - {\rm Id}\norma^{k_0, \gamma}_{0, s, 0}, \, 
\norma {\cal W}^* - {\rm Id}\norma^{k_0, \gamma}_{0, s, 0} 
\lesssim_{M, s} \e \gamma^{- 2(M+1)} \big(1 + \| \fracchi_0 \|_{s + \aleph_8(M)}^{k_0, \gamma}\big)\,, \label{stima composta tutte le mappe descent0} \\
& \norma \Delta_{12} {\cal W}^{\pm 1} \norma_{0, s_1, 0}, \, 
\norma \Delta_{12} {\cal W}^* \norma_{0, s_1, 0} 
\lesssim_{M, s_1} \e \gamma^{- 2(M+1)} \| \Delta_{12} i \|_{s_1 + \aleph_8(M)}\,. 
\label{stima derivata composta tutte le mappe descent0} 
\end{align}
Let $S > s_0$, $\b_0 \in \N$, and $M > \frac12 (\beta_0 + k_0)$. 
There exists a constant $ \aleph_8'(M, \beta_0) > 0$ such that, assuming \eqref{ansatz I delta} with $\mu_0 \geq \aleph_8'(M, \beta_0)$, for any $m_1, m_2 \geq 0$, 
with $m_1 + m_2 \leq M - \frac12 (\beta_0 + k_0) $, 
for any $\b \in \N^\nu$, $|\b| \leq \b_0$, 
the operators $ \langle D \rangle^{m_1} (\partial_\vphi^\beta {\cal T}_8) \langle D \rangle^{m_2}$, $\langle D \rangle^{m_1} \Delta_{12} (\partial_\vphi^\beta {\cal T}_8) \langle D \rangle^{m_2}$ are ${\cal D}^{k_0}$-tame with tame constants satisfying 
\begin{align}
 \mathfrak M_{\langle D \rangle^{m_1} (\pa_\vphi^\beta {\cal T}_8) \langle D \rangle^{m_2}}(s) \lesssim_{M, S} \e \gamma^{- 2(M+1)} \big(1 + \| \fracchi_0 \|_{s +\aleph_8'(M, \beta_0)} \big)\,, \qquad \forall s_0 \leq s \leq S 
\label{stima tame cal T 8} \\
\| \langle D \rangle^{m_1} \Delta_{12} (\pa_\vphi^\beta {\cal T}_8) \langle D \rangle^{m_2} \|_{{\cal L}(H^{s_1})} \lesssim_{M, S} \e \gamma^{- 2(M+1)}  \| \Delta_{12} i \|_{s_1 +\aleph_8'(M, \beta_0)}\,.  \label{stima tame derivata cal T 8}
\end{align}
\end{lemma}

\begin{proof}
Estimates \eqref{stima composta tutte le mappe descent0}, 
\eqref{stima derivata composta tutte le mappe descent0} 
follow by definition \eqref{definizione cal WM}, by estimates \eqref{stima composta tutte le mappe descent}, \eqref{stima derivata composta tutte le mappe descent} 
and using also Lemma \ref{stima pseudo diff aggiunto} to estimate the adjoint operator. 
Let us prove \eqref{stima tame cal T 8} 
(the proof of \eqref{stima tame derivata cal T 8} follows by similar arguments). 
First we analyze the term ${\cal W}^{- 1} {\cal T}_7 {\cal W}$. 
Let $m_1, m_2 \geq 0$, with $m_1 + m_2 \leq M - \frac12 (\beta_0 + k_0) $ and $\beta \in \N^\nu$ with $|\beta| \leq \beta_0$. Arguing as in the proof of Lemma \ref{lemma stime primo Egorov}, we have to analyze, for any $\beta_1, \beta_2, \beta_3 \in \N^\nu$ with $\beta_1 + \beta_2 + \beta_3 = \beta$, the operator 
$( \partial_\vphi^{\beta_1}{\cal W}^{- 1}) (  \partial_{\vphi}^{\beta_2} {\cal T}_7)(  \partial_\vphi^{\beta_3} {\cal W})$. We write 
\begin{align}
& \langle D \rangle^{m_1} ( \partial_\vphi^{\beta_1}{\cal W}^{- 1}) (  \partial_{\vphi}^{\beta_2} {\cal T}_7)(  \partial_\vphi^{\beta_3} {\cal W})  \langle D \rangle^{m_2} \nonumber\\
& = \Big( \langle D \rangle^{m_1} \partial_\vphi^{\beta_1} {\cal W} \langle D \rangle^{- m_1} \Big)  \circ \Big(\langle D \rangle^{m_1}  \partial_{\vphi}^{\beta_2} {\cal T}_7 \langle D \rangle^{m_2} \Big) \circ \Big( \langle D \rangle^{- m_2}  \partial_\vphi^{\beta_3} {\cal W} \langle D \rangle^{m_2} \Big)\,.  \label{alaska 3}
\end{align}
For any $m \geq 0$, $\beta \in \N^\nu$, $|\beta| \leq \beta_0$, 
by 
\eqref{interpolazione parametri operatore funzioni}, 
\eqref{Norm Fourier multiplier}, \eqref{lemma composizione multiplier}, 
\eqref{estimate composition parameters}, 
one has 
$$
\mathfrak M_{\langle D \rangle^m (\pa_\vphi^\b {\cal W}^{\pm 1}) \langle D \rangle^{-m}}(s) 
\lesssim_s \norma \langle D \rangle^{m} (\pa_\vphi^\b {\cal W}^{\pm 1}) \langle D \rangle^{-m} \norma_{0, s, 0}^{k_0, \gamma}  
\lesssim_s \norma \partial_\vphi^{\beta} {\cal W}^{\pm 1} \norma_{0, s + m, 0}^{k_0, \gamma} \lesssim_s \norma {\cal W}^{\pm 1} \norma_{0, s + \beta_0 + m, 0}^{k_0, \gamma}
$$
and $\norma {\cal W}^{\pm 1} \norma_{0, s + \beta_0 + m, 0}^{k_0, \gamma}$ 
can be estimated by using \eqref{stima composta tutte le mappe descent0}. 
The estimate of \eqref{alaska 3} then follows by \eqref{stima tame cal T (1)} and 
Lemma \ref{composizione operatori tame AB}. 
The tame estimate of $ \langle D \rangle^{m_1}  \pa_\vphi^\beta {\cal P}_{2M} \langle D \rangle^{m_2}$ follows by \eqref{interpolazione parametri operatore funzioni}, \eqref{2007.11}, \eqref{2007.12}.  
The tame estimate of the term $\ii \langle D \rangle^{m_1}  \pa_\vphi^\beta  \big( {\cal W}^{- 1}  \Pi_0 {\cal W} -  \Pi_0 \big) \langle D \rangle^{m_2}$ follows by Lemma \ref{lemma coniugazione proiettore pi 0} (applied with 
$ A = {\cal W}^{- 1}  $ and $ B = {\cal W} $) and
\eqref{interpolazione parametri operatore funzioni}, 
\eqref{stima composta tutte le mappe descent0}, 
\eqref{stima derivata composta tutte le mappe descent0}. 
\end{proof}

\section{Conclusion: reduction of $ {\cal L}_\om $ up to smoothing operators} 
\label{coniugio cal L omega}

By Sections \ref{linearizzato siti normali}-\ref{sezione descent method},
for all $\lm = (\om, \h) \in \mathtt{DC}(\g,\t) \times [\h_1, \h_2]$ 
the real, even and reversible operator ${\cal L}$ in \eqref{linearized vero} is conjugated 
to the real, even and reversible operator ${\cal L}_8$ defined in \eqref{cal L8}, namely
\begin{align}\label{bf RN (3) bot}
&  {\cal P}^{- 1} {\cal L} {\cal P} = {\cal L}_8 = \omega \cdot \partial_\vphi + \ii {\cal D}_8 + \ii \Pi_0 +  {\cal T}_8  \,,
\end{align}
where $ {\cal P}$   is the real, even and reversibility preserving map 
\begin{equation}\label{semiconiugio cal L8}
 {\cal P} :=  {\cal Z} {\cal B} {\cal A} {\cal M}_2 {\cal M}_3 \mC {\bf \Phi}_M {\bf \Phi}  {\cal W}\,.
\end{equation}
Moreover, as already noticed below \eqref{definizione cal T8}, the operator 
$\mL_8$ is defined on the whole parameter space
$\R^\nu \times [\h_1, \h_2]$.

Now we deduce a similar conjugation result for the projected linearized operator $\mL_\om$ defined
in \eqref{Lomega def}, which  acts on the normal subspace $ H_{\Splus}^\bot $,  
whose relation with $\mL$ is stated in \eqref{representation Lom}. 
The operator $\mL_\om$ is even and reversible as stated in Lemma \ref{thm:Lin+FBR}.

Let ${\mathbb S} := \Splus \cup (-\Splus)$ and ${\mathbb S}_0 := {\mathbb S} \cup \{ 0 \}$. 
We denote by $\Pi_{{\mathbb S}_0}$ the corresponding $L^2$-orthogonal projection 
and $\Pi_{{\mathbb S}_0}^\bot : = {\rm Id} - \Pi_{{\mathbb S}_0}$. 
We also denote 
$H_{\mathbb S_0}^\bot := \Pi_{\mathbb S_0}^\bot L^2(\T) $ and $H^s_\bot : = H^s(\T^{\nu + 1}) \cap H_{\mathbb S_0}^\bot$.

\begin{lemma}{\bf (Restriction of the conjugation map to $H_{\mathbb S_0}^\bot$)}\label{Lemma:trasformazioni finali W}
Let $M > 0$. There exists a constant $\sigma_M > 0$ (depending also on $k_0, \tau, \nu$) such that,
assuming \eqref{ansatz I delta} with $\mu_0 \geq \sigma_M$, the following holds: for any $s > s_0$ there exists a constant 
$\delta(s) > 0$ such that, if $\e \gamma^{- 2(M+1)} \leq \delta(s)$, then the operator
\begin{equation}\label{Phi 1 Phi 2 proiettate}
{\cal P}_\bot := \Pi_{{\mathbb S}_0}^\bot {\cal P} \Pi_{{\mathbb S}_0}^\bot 
\end{equation}
is invertible and for each family of functions $h : = h(\lambda) \in H^{s + \sigma_M }_\bot \times H^{s + \sigma_M }_\bot$ it satisfies 
\begin{align} \label{stima Phi 1 Phi 2 proiettate}
\| {\cal P}_\bot^{\pm 1} h \|_s^{k_0, \gamma} 
& \lesssim_{M, s} \| h \|_{s + \sigma_M }^{k_0, \gamma} 
+ \| \fracchi_0 \|_{s + \sigma_M }^{k_0, \gamma} \| h \|_{s_0 + \sigma_M }^{k_0, \gamma}\,, 
\\
\| (\Delta_{12} \mP_\bot^{\pm 1}) h \|_{s_1} 
& \lesssim_{M,s_1} \e \g^{-2(M+1)} \| \Delta_{12} i \|_{s_1 + \sigma_M} \| h \|_{s_1 + 1}\,.
\label{Delta 12 mP}
\end{align}
The operator $  {\cal P}_\bot $ is real, even and reversibility preserving. 
The operators $\mP, \mP^{-1}$ also satisfy \eqref{stima Phi 1 Phi 2 proiettate}, \eqref{Delta 12 mP}.
\end{lemma}

\begin{proof}
Applying \eqref{interpolazione parametri operatore funzioni (2)} and \eqref{est Z-Id}, 
\eqref{est BAPQ (1)}, \eqref{stime nuovo cal A inv},  \eqref{stima Lambda mathtt h}, 
\eqref{stime M3}, \eqref{def:mC}, \eqref{definizione bf PhiN}, \eqref{flusso PseudoPDE fracchi}, 
\eqref{stima composta tutte le mappe descent0} we get
$$
\| A h \|_s^{k_0, \gamma} \lesssim_s \| h \|_{s + \mu_M}^{k_0 , \gamma} + \| \fracchi_0 \|_{s + \mu_M}^{k_0, \gamma} \| h \|_{s_0 + \mu_M}^{k_0, \gamma}\,, \quad A \in \{ {\cal Z}^{\pm 1}, {\cal B}^{\pm 1}, {\cal A}^{\pm1}, 
{\cal M}_2^{\pm 1}, {\cal M}_3^{\pm 1}, \mC^{\pm 1}, {\bf \Phi}_M^{\pm 1}, {\bf \Phi}^{\pm 1} ,{\cal W}^{\pm 1} \}\,,
$$
for some $\mu_M > 0 $. Then by the definition \eqref{semiconiugio cal L8} of ${\cal P}$, 
by composition, one gets that $ \| {\cal P}^{\pm 1} h \|_s^{k_0, \gamma} 
\lesssim_{M, s} \| h \|_{s +  \sigma_M}^{k_0, \gamma} 
+ \| \fracchi_0 \|_{s + \sigma_M }^{k_0, \gamma} \| h \|_{s_0 +  \sigma_M}^{k_0, \gamma} $ 
for some constant $\sigma_M > 0$ larger than $\mu_M > 0$, thus $ {\cal P}^{\pm 1} $ 
satisfy \eqref{stima Phi 1 Phi 2 proiettate}.  
In order to prove that ${\cal P}_\bot$ is invertible, it is sufficient to prove that 
$\Pi_{{\mathbb S}_0} {\cal P} \Pi_{{\mathbb S}_0}$ is invertible, and argue as 
in the proof of Lemma 9.4 in \cite{Alaz-Bal}, or Section 8.1 of \cite{BBM-auto}. 
This follows by a perturbative argument, for $\e \gamma^{- 2(M+1)}$ small, 
using that $ \Pi_{{\mathbb S}_0} $ is a finite dimensional projector. 
The proof of \eqref{Delta 12 mP} follows similarly by using
\eqref{derivate in i cal Z},
\eqref{Delta 12 cal B}, 
\eqref{stime nuovo cal A inv},
% $\Delta_{12} \Lambda_\h^{\pm 1} = 0$ 
\eqref{stime Delta M3},
\eqref{bf PhiN derivate in i},
\eqref{QUELLA DOPO LA (12.37)},
\eqref{stima derivata composta tutte le mappe descent0}.
\end{proof}

Finally, for all $\lm = (\om, \h) \in \mathtt{DC}(\g,\t) \times [\h_1, \h_2]$, 
the operator ${\cal L}_\omega $ defined in \eqref{Lomega def} is conjugated to 
\begin{equation}\label{definizione cal L bot}
 {\cal L}_\bot  
:= {\cal P}_\bot^{- 1} {\cal L}_\omega {\cal P}_\bot = \Pi_{{\mathbb S}_0}^\bot {\cal L}_8 
\Pi_{{\mathbb S}_0}^\bot  + R_M  
\end{equation}
where 
\begin{align}\label{R8}
 R_M & := 
{\cal P}_\bot^{- 1} \Pi_{{\mathbb S}_0}^\bot  
\big( {\cal P} \Pi_{{\mathbb S}_0} {\cal L}_8 \Pi_{{\mathbb S}_0}^\bot  - 
{\cal L} \Pi_{{\mathbb S}_0}{\cal P} \Pi_{{\mathbb S}_0}^\bot  +  
\e  R {\cal P}_\bot \big)  \\
& = 
{\cal P}_\bot^{- 1} \Pi_{{\mathbb S}_0}^\bot  
% ({\cal P} -  {\mathbb I}_2) 
\mP \Pi_{{\mathbb S}_0}   {\cal T}_8 \Pi_{{\mathbb S}_0}^\bot   
+ {\cal P}_\bot^{- 1} \Pi_{{\mathbb S}_0}^\bot   
J \partial_u  \nabla_u H (T_\delta(\vphi)) \Pi_{{\mathbb S}_0} 
% ({\cal P} - \mathbb I_2) 
\mP \Pi_{{\mathbb S}_0}^\bot 
+ \e {\cal P}_\bot^{- 1} \Pi_{{\mathbb S}_0}^\bot R {\cal P}_\bot 
\label{R8-linea2}
\end{align}
is a finite dimensional operator. 
To prove \eqref{definizione cal L bot}-\eqref{R8} we first use \eqref{representation Lom} and 
\eqref{Phi 1 Phi 2 proiettate} to get 
$\mL_\om \mP_\bot = \Pi_{{\mathbb S}_0}^\bot (\mL + \e R) \Pi_{{\mathbb S}_0}^\bot 
\mP \Pi_{{\mathbb S}_0}^\bot$, 
then we use \eqref{bf RN (3) bot} to get $\Pi_{{\mathbb S}_0}^\bot \mL \mP \Pi_{{\mathbb S}_0}^\bot$ $= \Pi_{{\mathbb S}_0}^\bot \mP \mL_8 \Pi_{{\mathbb S}_0}^\bot$, 
and we also use the decomposition $\mathbb{I}_2 = \Pi_{{\mathbb S}_0} + \Pi_{{\mathbb S}_0}^\bot$. To get % the second line in 
\eqref{R8-linea2}, 
we use \eqref{bf RN (3) bot}, \eqref{representation Lom}, 
and we note that 
$ \Pi_{{\mathbb S}_0} \, \om \cdot \pa_\vphi  \, \Pi_{{\mathbb S}_0}^\bot = 0 $, 
$ \Pi_{{\mathbb S}_0}^\bot \, \om \cdot \pa_\vphi  \, \Pi_{{\mathbb S}_0} = 0 $, 
and $\Pi_{{\mathbb S}_0} \ii {\cal D}_8   \Pi_{{\mathbb S}_0}^\bot = 0 $, 
by \eqref{def:cal-L8} and 
\eqref{introduzione primo operatore diagonale riducibilita}.

\begin{lemma}\label{lemma forma buona resto}
The operator $R_M$ in \eqref{R8} has the finite dimensional form \eqref{forma buona resto}. 
Moreover, let $S > s_0$ and  $M >  \frac12(\beta_0 + k_0)$. 
For any $\beta \in \N^\nu$, $|\beta| \leq \beta_0$, there exists a constant $ \aleph_9(M, \beta_0) > 0$ (depending also on $k_0, \tau, \nu$) such that, if \eqref{ansatz I delta} holds with $\mu_0 \geq \aleph_9(M, \beta_0)$, then for any $m_1, m_2 \geq 0$, 
with $m_1 + m_2 \leq M - \frac12(\beta_0 + k_0)$, one has that the operators $ \langle D \rangle^{m_1}\partial_\vphi^\beta R_M \langle D \rangle^{m_2}$, $\langle D \rangle^{m_1 }\partial_\vphi^\beta \Delta_{12} R_M \langle D \rangle^{m_2}$ are ${\cal D}^{k_0}$-tame with tame constants 
\begin{align}
\mathfrak M_{\langle D \rangle^{m_1}\partial_\vphi^\beta R_M \langle D \rangle^{m_2}}(s) 
& \lesssim_{M, S} \e \gamma^{- 2 (M+1)} \big(1 + \| \fracchi_0 \|^{k_0, \gamma}_{s +\aleph_9(M, \beta_0)}\big)\,, \qquad \forall s_0 \leq s \leq S 
\label{RM} 
\\
\| \langle D \rangle^{m_1} \Delta_{12}\partial_\vphi^\beta R_M \langle D \rangle^{m_2} \|_{{\cal L}(H^{s_1})} 
& \lesssim_{M, S} \e \gamma^{- 2 (M+1)} \| \Delta_{12} i \|_{s_1 + \aleph_9(M,\b_0)}\,.  
\label{derivata RM}
\end{align}
\end{lemma}

\begin{proof} 
To prove that the operator $R_M$ has the finite dimensional form \eqref{forma buona resto}, 
notice that in the first two terms in \eqref{R8-linea2} there is the finite dimensional projector 
$ \Pi_{{\mathbb S}_0} $, that the operator $ R $ in the third term in \eqref{R8-linea2}  
already has the finite dimensional form \eqref{forma buona resto}, 
and use the property that $\mP_\bot ( a(\ph) h) = a(\ph) \mP_\bot h$ 
for all $h = h(\ph,x)$ and all $a(\ph)$ independent of $ x $, see also 
the proof of Lemma \ref{lemma coniugazione proiettore pi 0} 
(and Lemma 6.30 in \cite{BertiMontalto} and Lemma 8.3 in \cite{BBM-auto}).
To estimate $R_M$, use 
\eqref{stima Phi 1 Phi 2 proiettate}, 
\eqref{Delta 12 mP} 
for $\mP$,
\eqref{stima tame cal T 8}, 
\eqref{stima tame derivata cal T 8} 
for $\mathcal{T}_8$,
\eqref{representation Lom}, 
\eqref{linearized vero}, 
\eqref{stima V B a c},
\eqref{stima derivate i primo step},
\eqref{estimate DN} 
for $J \partial_u  \nabla_u H (T_\delta(\vphi))$,
\eqref{forma buona resto}, 
\eqref{stime gj chij} 
for $R$.
The term $ \Pi_{{\mathbb S}_0}^\bot   
J \partial_u  \nabla_u 
H (T_\delta(\vphi))  \Pi_{{\mathbb S}_0} $ is small
because $ \Pi_{{\mathbb S}_0}^\bot \begin{psmallmatrix} 0 & - D \tanh ( {\mathtt h} D) \\ 
1 & 0 \end{psmallmatrix} \Pi_{{\mathbb S}_0} $ is zero. 
\end{proof}

By \eqref{definizione cal L bot} and \eqref{cal L8} we get
\begin{equation}\label{forma finale operatore pre riducibilita:0}
{\cal L}_\bot =  \omega \cdot \partial_\vphi \mathbb I_\bot+ \ii {\cal D}_\bot + {\cal R}_\bot
\end{equation}
where  $\mathbb I_\bot$ denotes the identity map of 
$H_{\mathbb S_0}^\bot  $ (acting on scalar functions $u$, as well as on pairs $(u,\bar u)$ in a diagonal manner),
\begin{equation} \label{forma finale parte diagonale pre riducibilita}
{\cal D}_\bot := \begin{pmatrix} D_\bot & 0 \\
0  & - D_\bot \end{pmatrix}	, 
\qquad 
D_\bot := \Pi_{\mathbb S_0}^\bot D_8 \Pi_{\mathbb S_0}^\bot \, , 
\end{equation}
and  ${\cal R}_\bot$ is the operator
\begin{equation}\label{definizione cal R bot pre riducibilita}
{\cal R}_\bot := \Pi_{\mathbb S_0}^\bot  {\cal T}_8 \Pi_{\mathbb S_0}^\bot  + R_M\, ,
\quad 
\mR_{\bot} = \begin{pmatrix} \mR_{\bot,1} & \mR_{\bot,2} \\ 
\overline{\mR_{\bot,2}} & \overline{\mR_{\bot,1}} \end{pmatrix} \, . 
\end{equation}
The operator $ {\cal R}_\bot $ in 
\eqref{definizione cal R bot pre riducibilita} is defined 
for all $\lm = (\om, \h) \in \R^\nu \times [\h_1, \h_2]$, 
because ${\cal T}_8$ in \eqref{definizione cal T8} 
and the operator in the right hand side of \eqref{R8-linea2} 
% (which defines $R_M$) 
are defined on the whole parameter space. As a consequence, the right hand side of \eqref{forma finale operatore pre riducibilita:0}
extends the definition of $\mL_\bot$ to $ \R^\nu \times [\h_1, \h_2]$.
We still denote the extended operator by ${\cal L}_\perp$.

In conclusion, we have obtained the following proposition. 

\begin{proposition}\label{prop: sintesi linearized}
{\bf (Reduction of $\mL_\om$ up to smoothing remainders)}
For all $ \lm = (\om, \h) \in \mathtt{DC}(\g,\t) \times [\h_1, \h_2] $,
the operator $ {\cal L}_\om $ in \eqref{representation Lom} is conjugated 
by the map $ {\cal P}_\bot $ defined in \eqref{Phi 1 Phi 2 proiettate} 
 to the real, even and reversible operator 
${\cal L}_\perp $ in \eqref{definizione cal L bot}. %  defined in \eqref{forma finale operatore pre riducibilita:0}.
For all $\lm \in \R^\nu \times [\h_1, \h_2]$, the extended operator ${\cal L}_\bot$ 
defined by the right hand side of \eqref{forma finale operatore pre riducibilita:0}
has the form 
\begin{equation}\label{forma finale operatore pre riducibilita}
{\cal L}_\bot =  \omega \cdot \partial_\vphi \mathbb I_\bot+ \ii {\cal D}_\bot + {\cal R}_\bot
\end{equation}
where $ {\cal D}_\bot $ is the diagonal operator 
\begin{equation} \label{forma finale parte diagonale pre riducibilita 1}
{\cal D}_\bot := \begin{pmatrix} D_\bot & 0 \\
0  & - D_\bot \end{pmatrix}	, 
\qquad 
D_\bot  
= {\rm diag}_{j \in {\mathbb S}_0^c} \, \mu_j\,, \qquad \mu_{-j} = \mu_j \, ,
\end{equation} 
with eigenvalues $ \mu_j $, 
defined in \eqref{introduzione primo operatore diagonale riducibilita}, given by 
\begin{equation} \label{2802.4}
\mu_j = \mathtt m_{\frac12} |j|^{\frac12} \tanh^{\frac12}(\h |j|) + r_j \in \R \, , \quad \  r_{-j} = r_j  \, , 
\end{equation}
where $\mathtt m_{\frac12}, r_j \in \R$ satisfy 
\eqref{stime lambda 1}, \eqref{stima code - 1/2 inizio riducibilita}. 
The operator ${\cal R}_\bot$ defined in  \eqref{definizione cal R bot pre riducibilita} 
is real, even and reversible. 

Let $S > s_0$, $\b_0 \in \N$, and $M > \frac12 (\beta_0 + k_0)$. 
There exists a constant $ \aleph(M, \beta_0) > 0$ (depending also on $k_0, \tau, \nu$) such that, 
assuming \eqref{ansatz I delta} with $\mu_0 \geq \aleph(M, \beta_0)$, 
for any $m_1, m_2 \geq 0$, with $m_1 + m_2 \leq M - \frac12 (\beta_0 + k_0) $,
for any $\beta \in \N^\nu$, $|\beta| \leq \beta_0$, 
the operators $ \langle D \rangle^{m_1}\partial_\vphi^\beta {\cal R}_\bot \langle D \rangle^{m_2}$, $\langle D \rangle^{m_1 }\partial_\vphi^\beta \Delta_{12} {\cal R}_\bot \langle D \rangle^{m_2}$ are ${\cal D}^{k_0}$-tame with tame constants satisfying 
\begin{align}
 \mathfrak M_{\langle D \rangle^{m_1}\partial_\vphi^\beta {\cal R}_\bot \langle D \rangle^{m_2}}(s) \lesssim_{M, S} \e \gamma^{- 2 (M+1)} \big(1 + \| \fracchi_0 \|^{k_0, \gamma}_{s +\aleph(M, \beta_0)}\big)\,, \qquad \forall s_0 \leq s \leq S \label{stima tame cal R bot} \\
\| \langle D \rangle^{m_1} \Delta_{12}\partial_\vphi^\beta {\cal R}_\bot \langle D \rangle^{m_2} \|_{{\cal L}(H^{s_1})} \lesssim_{M, S} \e \gamma^{- 2 (M+1)}  \| \Delta_{12} i \|_{s_1 + \aleph(M, \beta_0)}\,.  \label{stima tame derivata cal R bot}
\end{align}
\end{proposition}

\begin{proof}
Estimates \eqref{stima tame cal R bot}-\eqref{stima tame derivata cal R bot} for the term $\Pi_{\mathbb S_0}^\bot  {\cal T}_8 \Pi_{\mathbb S_0}^\bot $ in \eqref{definizione cal R bot pre riducibilita} follow directly by  \eqref{stima tame cal T 8}, \eqref{stima tame derivata cal T 8}. 
Estimates \eqref{stima tame cal R bot}-\eqref{stima tame derivata cal R bot} for $R_M$ are \eqref{RM}-\eqref{derivata RM}.
\end{proof}

\section{Almost-diagonalization and invertibility of $ {\cal L}_\om $}\label{sec: reducibility}

In Proposition \ref{prop: sintesi linearized} we obtained the operator $\mL_\bot = \mL_\bot(\ph)$ in \eqref{forma finale operatore pre riducibilita}
which is diagonal up to the smoothing operator $ \mR_\bot $.
In this section we implement a diagonalization KAM iterative scheme 
to reduce the size of the non-diagonal term $ \mR_\bot $. 
 
We first replace the operator $ \mL_\bot $ in 
\eqref{forma finale operatore pre riducibilita}
%\eqref{forma finale operatore pre riducibilita:0}
with  the operator $\mL_\bot^\sym $ defined in \eqref{2701.1} below, 
which coincides with $\mL_\bot $  on the subspace of functions even in $  x $, see Lemma \ref{lemma:R bot sym}.
%This trick enables us to exploit the fact that we are working with even operators on the subspace of functions $\even(x)$, and to use the fact that on this subspace the eigenvalues are simple in order to impose the second order Melnikov conditions.
%This is a trick to exploit the fact that we are working with even operators on the subspace of functions $\even(x)$, while using their matrix representation on the exponential basis $(e^{ijx})_{j\in\Z}$.
This trick enables to reduce an even operator using its matrix representation in the exponential basis $(e^{\ii jx})_{j\in\Z}$ and exploiting the fact that on the subspace of functions $\even(x)$ its eigenvalues are simple.
We define the linear operator $\mL_\bot^{sym} $, acting on 
$H_{\mathbb S_0}^\bot $, as 
\begin{equation} \label{2701.1}
{\cal L}_\bot^{sym} := \omega \cdot \partial_\vphi \mathbb I_\bot 
+ \ii {\cal D}_\bot + {\cal R}_\bot^{sym},
\quad 
\mR_{\bot}^{sym} := \begin{pmatrix} \mR_{\bot,1}^{sym} & \mR_{\bot,2}^{sym} \\ 
\overline{\mR_{\bot,2}^{sym}} & \overline{\mR_{\bot,1}^{sym}} \end{pmatrix},
\end{equation}
where $\mR_{\bot,i}^{sym}$, $i=1,2$, are defined by 
their matrix entries
\begin{equation} \label{2601.1}
(\mR_{\bot,i}^{sym})_j^{j'}(\ell) := 
\begin{cases} (\mR_{\bot,i})_j^{j'}(\ell) + (\mR_{\bot,i})_j^{-j'}(\ell) 
\ \  & \text{if} \ jj' > 0, \\
0 & \text{if} \ jj' < 0, 
\end{cases}  
\qquad j, j' \in \mathbb S_0^c \, ,    \qquad \ i = 1,2, 
\end{equation}
and 
$\mR_{\bot,i} $, $ i = 1, 2 $ 
are introduced in \eqref{definizione cal R bot pre riducibilita}. 
Note that, in particular, $(\mR_{\bot, i}^{sym})_j^{j'} = 0$, $i = 1, 2$ on the anti-diagonal $j' = -j$. 
Using definition \eqref{2601.1}, one has the following lemma.

\begin{lemma} \label{lemma:R bot sym} The operator $\mR_{\bot}^{sym}$ coincides with $\mR_\bot$ 
on the subspace of functions $\even(x)$ in $H_{\mathbb S_0}^\bot \times H_{\mathbb S_0}^\bot$, namely
\begin{equation} \label{2601.3}
\mR_\bot h = \mR_\bot^{sym} h \, ,  \quad \forall h \in H_{\mathbb S_0}^\bot \times H_{\mathbb S_0}^\bot \, , 
\quad h = h(\ph,x) = \even(x) \, .  
\end{equation}
$\mR_{\bot}^{sym}$ is real, even and reversible, 
and it satisfies the same bounds \eqref{stima tame cal R bot}, \eqref{stima tame derivata cal R bot} as $\mR_\bot$.
\end{lemma}

As a starting point of the recursive scheme, 
we consider the real, even, reversible linear operator 
$\mL_\bot^{sym}$ in \eqref{2701.1}, acting on $ H_{{\mathbb S}_0}^\bot $,
defined for all % $ (\om, \h) \in \R^{ |\Splus|} \times [\h_1, \h_2] $, 
$ (\om, \h) \in \R^\nu \times [\h_1, \h_2] $, 
which we rename 
\be\label{defL0-red}
{\cal L}_0 := \mL_\bot^{sym}
:= \omega \cdot \partial_\vphi {\mathbb I}_\bot + \ii {\cal D}_0 + {\cal R}_0 \,,
\qquad 
\mD_0 := \mD_\bot, 
\quad  
\mR_0 := \mR_\bot^{sym} \, , 
\ee
with 
\be\label{op-diago0}
{\cal D}_0:= \begin{pmatrix}
D_0 & 0 \\
0 & - {D}_0
\end{pmatrix}\,, \qquad 
D_0 := {\rm diag}_{j \in {\mathbb S}_0^c} \, \mu_j^{0} \,, \qquad 
\mu_j^{0} := \mathtt m_{\frac12} |j|^{\frac12} \tanh^{\frac12} (\h |j|) + r_j\,, 
\ee
where 
$ \mathtt m_{\frac12} := \mathtt m_{\frac12} (\om, \h) \in \R $ satisfies \eqref{stime lambda 1}, 
$r_j := r_j(\omega, \h) \in \R$, $r_j = r_{- j}$ 
satisfy \eqref{stima code - 1/2 inizio riducibilita},  
and  
\begin{align}   \label{defRQ0}
{\cal R}_0 := \begin{pmatrix}
R_1^{(0)} & R_2^{(0)} \\
\overline R_2^{(0)} & \overline R_1^{(0)}
\end{pmatrix}\,, \quad R_i^{(0)} : H_{{\mathbb S}_0}^\bot \to H_{{\mathbb S}_0}^\bot\,, 
\qquad i = 1, 2 \, . 
\end{align}
% is real, even and reversible.

\noindent
\textbf{Notation.} 
In  this section we use the following  notation:
given an operator $R$, we denote by $\pa_{\ph_i}^{s} \langle D \rangle^\fm R \langle D \rangle^\fm$ 
the operator $\langle D \rangle^\fm \circ \big( \pa_{\ph_i}^{s} R(\ph) \big) \circ \langle D \rangle^\fm$.
Similarly $\langle \pa_{\ph,x} \rangle^{\mathtt{b}} \langle D \rangle^\fm R \langle D \rangle^\fm$ 
denotes  $\langle D \rangle^\fm \circ \big( \langle \pa_{\ph,x} \rangle^{\mathtt{b}} R \big) \circ \langle D \rangle^\fm$ where $\langle \pa_{\ph,x} \rangle^{\mathtt{b}} $ 
is introduced in Definition \ref{def:maj}. 

\smallskip

The operator ${\cal R}_0$ in \eqref{defRQ0} satisfies the tame estimates of Lemma \ref{lem:tame iniziale} below.
Define the constants 
\be
\begin{aligned}\label{alpha beta}
& {\mathtt b} := [{\mathtt a}] + 2 \in \N \,,
\quad {\mathtt a} := \max \{ 3 \tau_1, \chi (\tau + 1)(4 \mathtt d + 1) + 1 \} \, , 
\quad \chi := 3 / 2 \, , \\
&  \quad \tau_1 := \tau(k_0 + 1) + k_0 + \fm\,, \quad 
\fm := \perd (k_0 + 1) + \frac{k_0}{2}\,, 
\end{aligned}
\ee
where  $\perd > \frac34 k_0^* $, by  \eqref{relazione tau k0}. The condition $\mathtt a \geq \chi (\tau + 1)(4 \mathtt d + 1) + 1$ in \eqref{alpha beta} 
will be used in Section \ref{sec:NM} in order to verify inequality \eqref{cond-su-p}. 
Proposition \ref{prop: sintesi linearized} implies that
$ {\cal R}_0$ satisfies the tame estimates of Lemma \ref{lem:tame iniziale}
by fixing the constant $ M  $ large enough 
(which means that one has to perform a sufficiently large number of regularizing steps in Sections \ref{sec:10} and \ref{sezione descent method}), namely
\begin{equation}\label{relazione mathtt b N}
M := \Big[ 2 \fm + 2 \mathtt b
+ 1 + \frac{\mathtt b + s_0 + k_0}{2} \Big] + 1 \in \N \,  
\end{equation}
where $[  \, \cdot \, ] $ denotes the integer part, and 
$\fm$ and  $\mathtt b$ are defined in \eqref{alpha beta}. We also set   
 \begin{equation}\label{definizione bf c (beta)}
 \mu (\mathtt b) :=  \aleph(M, s_0 + \mathtt b)\,,
\end{equation}
where the constant $\aleph(M, s_0 + \mathtt b)$ is given in Proposition \ref{prop: sintesi linearized}. 

\begin{lemma} {\bf (Tame estimates of ${\cal R}_0 := {\cal R}_\bot^\sym$)} \label{lem:tame iniziale}
Assume \eqref{ansatz I delta} with $ \mu_0 \geq \mu (\mathtt b)$.  
Then ${\cal R}_0$ in \eqref{defL0-red}
satisfies the following property:
the operators 
\begin{align} \label{1309.3}
& \langle D \rangle^\fm {\cal R}_0 \langle D \rangle^{\fm + 1}, \qquad  \partial_{\vphi_i}^{s_0} \langle D \rangle^\fm {\cal R}_0 \langle D \rangle^{\fm  + 1}\,, \qquad 
\forall i =1, \ldots, \nu  \, , \\
& \label{1309.2}
\langle D \rangle^{\fm + \mathtt b} {\cal R}_0 \langle D \rangle^{\fm + \mathtt b +  1}, \quad \partial_{\vphi_i}^{s_0 +  {\mathtt b}} \langle D \rangle^{\fm + \mathtt b} {\cal R}_0 \langle D \rangle^{\fm + \mathtt b +  1}\, , \
\end{align}
where $\fm, {\mathtt b }$ are defined in \eqref{alpha beta},
are $ {\cal D}^{k_0} $-tame with tame constants
\begin{align} \label{tame cal R0 cal Q0}
& {\mathbb M}_0 (s) := \max_{i =1, \ldots, \nu } \big\{ 
{\mathfrak M}_{\langle D \rangle^\fm {\cal R}_0 \langle D \rangle^{\fm + 1} }(s),
{\mathfrak M}_{\partial_{\vphi_i}^{s_0} \langle D \rangle^\fm {\cal R}_0 \langle D \rangle^{\fm + 1} }(s) \big\} \\  
& \label{tame norma alta cal R0 cal Q0}
{\mathbb M}_0(s, {\mathtt b})  := \max_{i =1, \ldots,  \nu} \big\{ {\mathfrak M}_{ \langle D \rangle^{\fm + \mathtt b} {\cal R}_0 \langle D \rangle^{\fm + \mathtt b +  1} }(s), 
{\mathfrak M}_{ \partial_{\vphi_i}^{s_0 +  {\mathtt b}} \langle D \rangle^{\fm + \mathtt b} {\cal R}_0 \langle D \rangle^{\fm + \mathtt b +  1} }(s)  \big\} 
\end{align}
satisfying, for all  $ s_0 \leq s \leq S $, 
% the tame estimates  \eqref{tame cal R0 cal Q0}-\eqref{tame norma alta cal R0 cal Q0} with 
\begin{equation}\label{costanti resto iniziale}
{\mathfrak M}_0(s, {\mathtt b}) := {\rm max}\{ {\mathbb M}_0 (s), {\mathbb M}_0(s, {\mathtt b}) \}
\lesssim_S \e \gamma^{- 2(M+1)} \big( 1 + \| \fracchi_0 \|_{s + \mu(\mathtt b)}^{k_0, \gamma} \big)\,.
\end{equation} 
%where $ {\mathbb M}_0 (s) $, $ {\mathbb M}_0(s, {\mathtt b}) $ are defined in 
%\eqref{tame cal R0 cal Q0}, \eqref{tame norma alta cal R0 cal Q0}. 
In particular we have
\be\label{def:costanti iniziali tame}
{\mathfrak M}_0 (s_0, {\mathtt b}) 
\leq C(S) \e \gamma^{- 2(M+1)}  \, . 
\ee  
Moreover, for all $ i = 1, \ldots,  \nu $, $ \b \in \N $, $ \b \leq s_0 + {\mathtt b} $, we have
%the operators $\partial_{\vphi_i}^\beta  \langle D \rangle^\fm\Delta_{12} {\cal R}_0 \langle D \rangle^{\fm + 1}$, $\partial_{\vphi_i}^\beta \langle D \rangle^{\fm + \mathtt b}\Delta_{12} {\cal R}_0 \langle D \rangle^{\fm + \mathtt b + 1}$ satisfy the bounds 
\begin{align}
\!\!\!  \| \partial_{\vphi_i}^\beta  \langle D \rangle^\fm\Delta_{12} {\cal R}_0 \langle D \rangle^{\fm + 1} \|_{{\cal L}(H^{s_0})},   \| \partial_{\vphi_i}^\beta  \langle D \rangle^{\fm + \mathtt b}\Delta_{12} {\cal R}_0 \langle D \rangle^{\fm + \mathtt b + 1} \|_{{\cal L}(H^{s_0})}  \lesssim_{ S} \e \gamma^{- 2 (M+1)}  \| \Delta_{12} i \|_{s_0 + \mu(\mathtt b)}\,.  \label{costanti variazioni i resto iniziale}
\end{align}
\end{lemma} 

\begin{proof}
Estimate \eqref{costanti resto iniziale} follows by Lemma \ref{lemma:R bot sym}, 
by \eqref{stima tame cal R bot} with $m_1 = \fm$, $m_2 = \fm + 1$ for $\mathbb M_0 (s)$, with $m_1 = \fm + \mathtt b$, $m_2 = \fm + \mathtt b + 1$ for $\mathbb M_0(s, \mathtt b)$, and by definitions \eqref{alpha beta}, \eqref{relazione mathtt b N}, \eqref{definizione bf c (beta)}. 
Estimates \eqref{costanti variazioni i resto iniziale} follow similarly, 
applying \eqref{stima tame derivata cal R bot} with the same choices of $m_1, m_2$ and with $s_1 = s_0$.  
\end{proof}

We perform the almost-reducibility of $ {\cal L}_0 $ along the scale 
\be\label{def Nn}
N_{- 1} := 1\,,\quad 
N_\vnu := N_0^{\chi^\vnu} \quad \forall \vnu \geq 0\,, \quad 
\chi = 3/2\, ,
\ee
requiring inductively at each step the second order Melnikov non-resonance conditions in
\eqref{Omega nu + 1 gamma}. Note that the non-diagonal remainder ${\cal R}_\vnu$ in \eqref{cal L nu} is small according to the first inequality in \eqref{stima cal R nu}.

\begin{theorem} \label{iterazione riducibilita}
{\bf (Almost-reducibility of $\mL_0$: KAM iteration)} 
% Let $\mR_0 = \mR_\bot^\sym$, $\mL_0 = \mL_\bot^\sym$ in \eqref{2701.1}-\eqref{2601.1}.
There exists $ \tau_2 := \tau_2 (\t, \nu ) > \t_1 + \mathtt a$ 
(where $\t_1, \mathtt a$ are defined in \eqref{alpha beta})
such that, for all $ S > s_0 $, 
there are $ N_0 := N_0 (S, {\mathtt b}) \in \N$, $\d_0 := \d_0 (S, {\mathtt b})\in (0,1)$ 
such that, if 
\begin{equation}\label{KAM smallness condition1}
\e \g^{-2(M+1)} \leq \d_0, \quad
N_0^{\tau_2}  {\mathfrak M}_0(s_0, {\mathtt b}) \gamma^{- 1} \leq 1
\end{equation}
(see \eqref{def:costanti iniziali tame}), then, for all $ n \in \N $, 
$\vnu = 0, 1 , \ldots, n$: 
\begin{itemize}
\item[${\bf(S1)_{\vnu}}$] 
There exists a real, even and reversible operator 
\begin{equation}\label{cal L nu}
{\cal L}_\vnu:= \Dom {\mathbb I}_\bot + \ii {\cal D}_\vnu + {\cal R}_\vnu\,, \quad 
{\cal D}_\vnu := \begin{pmatrix}
D_\vnu & 0 \\
0 & - D_\vnu 
\end{pmatrix}\,,\quad D_\vnu := {\rm diag}_{j \in {\mathbb S}_0^c} \mu_j^\vnu\,,
\end{equation}
%--- which acts on the space of functions even in $ x $ --- 
defined for all $ (\om, \h) $ in $\R^{\nu} \times [\h_1,\h_2]$
where $ \mu_j^\vnu $ 
are $ k_0 $ times differentiable functions of the form 
\begin{equation}\label{mu j nu}
\mu_j^\vnu(\omega, \h) := \mu_j^0(\omega, \h) + r_j^\vnu(\omega, \h)
\in \R
\end{equation}
where $\mu_j^0$ are defined in \eqref{op-diago0}, 
satisfying 
\begin{equation}\label{stima rj nu}
\mu_j^\vnu = \mu_{- j}^\vnu\,, \quad i.e. \ r_j^\vnu = r_{- j}^\vnu\,, \quad |r_j^\vnu|^{k_0, \gamma} \leq C(S, \mathtt b) \e \gamma^{- 2(M+1)} |j|^{- 2 \fm}\,,\ \  \forall j \in {\mathbb S}_0^c
\end{equation}
and, for $\vnu \geq 1$,
\be\label{vicinanza autovalori estesi}
|\mu_j^\vnu - \mu_j^{\vnu - 1}|^{k_0, \gamma} 
\leq C |j|^{- 2\fm} 
{\mathfrak M}_{\langle D \rangle^\fm {\cal R}_{\vnu-1} \langle D \rangle^\fm }^\sharp (s_0) 
\leq C( S, \mathtt b) \e \gamma^{- 2 (M+1)} |j|^{- 2\fm}  N_{\vnu - 2}^{- {\mathtt a}}\,. 
\ee
The remainder
\begin{equation}\label{forma cal R nu}
{\cal R}_\vnu := \begin{pmatrix}
R_1^{(\vnu)} & R_2^{(\vnu)} \\
\overline R_2^{(\vnu)} & \overline R_1^{(\vnu)}
\end{pmatrix}
\end{equation}
satisfies
\begin{equation} \label{2701.3}
( R_1^{(\vnu)} )_j^{j'}(\ell) = ( R_2^{(\vnu)} )_j^{j'}(\ell) = 0 
\quad \forall (\ell,j, j'), \ jj'< 0,
\end{equation}
and it is $ {\cal D}^{k_0} $-modulo-tame: 
more precisely, the operators $ \langle D \rangle^\fm {\cal R}_\vnu \langle D \rangle^\fm$ 
and $ \langle \partial_{\vphi, x} \rangle^{\mathtt b} \langle D \rangle^\fm {\cal R}_\vnu \langle D \rangle^\fm $ are $ {\cal D}^{k_0} $-modulo-tame 
and there exists a constant $C_* := C_*(s_0, \mathtt b) > 0$ such that, for any $s \in [s_0, S] $, 
\begin{equation}\label{stima cal R nu}
{\mathfrak M}_{\langle D \rangle^\fm {\cal R}_\vnu \langle D \rangle^\fm}^\sharp (s) \leq 
\frac{C_*{\mathfrak M}_0 (s, {\mathtt b}) }{N_{\vnu - 1}^{ {\mathtt a}}}\,,\quad 
 {\mathfrak M}_{\langle \partial_{\vphi, x} \rangle^{\mathtt b} \langle D \rangle^\fm {\cal R}_\vnu \langle D \rangle^\fm}^\sharp ( s) \leq  C_*{\mathfrak M}_0 (s, {\mathtt b}) N_{\vnu - 1}\,.
\end{equation}
Define the sets $ \tLm_\vnu^\gamma $ by
 $ \tLm_0^\gamma := \mathtt{DC}(2 \gamma,\tau) \times [\h_1, \h_2]$, and, for all $\vnu \geq 1$, 
\begin{align}
\tLm_\vnu^\gamma & :=  \tLm_\vnu^\gamma (i) :=  
\Big\{ \lambda = (\omega, \h) \in \tLm_{\vnu - 1}^\gamma  \ : \ 
\notag \\
&  |\omega \cdot \ell  + \mu_j^{\vnu - 1} - \mu_{j'}^{\vnu - 1}| \geq  
\gamma j^{- \perd} j'^{- \perd} \langle \ell \rangle^{-\tau}  
\ \  \forall |\ell  |, |j - j'| \leq N_{\vnu - 1}, 
\ \ j, j' \in \N^+ \setminus \Splus, 
\ \ (\ell, j, j') \neq (0, j, j),  
 \nonumber\\
&  |\omega \cdot \ell  + \mu_j^{\vnu - 1} + \mu_{j'}^{\vnu - 1}| \geq  
\gamma (\sqrt{j} + \sqrt{j'})\langle \ell \rangle^{-\tau}
\ \ \forall |\ell|, |j - j'| \leq N_{\vnu - 1}\,, 
\ \ j, j' \in \N^+ \setminus \Splus  \Big\} \, . 
\label{Omega nu + 1 gamma}
\end{align}
For $ \vnu \geq 1 $, there exists a real,  even and reversibility preserving map, 
defined for all $ (\om, \h) $ in $\R^{\nu} \times [\h_1,\h_2]$, of the form 
\begin{equation}\label{Psi nu - 1}
 {\Phi}_{\vnu - 1} := {\mathbb I}_\bot + { \Psi}_{\vnu - 1}\,, \quad 
{\Psi}_{\vnu - 1} := \begin{pmatrix}
\Psi_{\vnu - 1, 1} & \Psi_{\vnu - 1, 2} \\
\overline \Psi_{\vnu - 1, 2} & \overline \Psi_{\vnu - 1, 1}
\end{pmatrix}\,
\end{equation}
such that for all $\lm = (\om, \h) \in \tLm_\vnu^\gamma $ 
the following conjugation formula holds:
\be\label{coniugionu+1}
{\cal L}_\vnu = { \Phi}_{\vnu - 1}^{-1} {\cal L}_{\vnu - 1} {\Phi}_{\vnu - 1}\, .
\ee
The operators $ \langle D \rangle^{\pm \fm} \Psi_{\vnu - 1} \langle D \rangle^{\mp \fm} $ 
and $ \langle \partial_{\vphi,x} \rangle^{\mathtt b} \langle D \rangle^{\pm \fm}\Psi_{\vnu - 1} \langle D \rangle^{\mp \fm} $ 
are $ {\cal D}^{k_0} $-modulo-tame on $\R^{\nu} \times [\h_1,\h_2]$ with modulo-tame constants satisfying, for all $s \in [s_0, S] $,
($ \tau_1, \mathtt a $ are defined in \eqref{alpha beta})
\begin{align}\label{tame Psi nu - 1}
&  {\mathfrak M}_{ \langle D \rangle^{ \pm \fm}\Psi_{\vnu - 1} \langle D \rangle^{ \mp \fm}}^\sharp \! (s) \leq 
C(s_0, \mathtt b) \g^{-1} 
N_{\vnu - 1}^{\tau_1} N_{\vnu - 2}^{- \mathtt a} {\mathfrak M}_0 (s, {\mathtt b}) \, , 
\\
& {\mathfrak M}_{\langle \partial_{\vphi,x} \rangle^{\mathtt b}  \langle D \rangle^{ \pm \fm}\Psi_{\vnu - 1} \langle D \rangle^{ \mp \fm} }^\sharp \! (s) \leq 
C(s_0, \mathtt b) \g^{-1} 
N_{\vnu - 1}^{\tau_1} N_{\vnu - 2}  {\mathfrak M}_0 (s, {\mathtt b}) \, , \label{tame Psi nu - 1 vphi x b}\\
& \label{tame Psi nu - 1 senza Dm}
 {\mathfrak M}_{\Psi_{\vnu - 1} }^\sharp \! (s) \leq 
 C(s_0, \mathtt b) \g^{-1} 
N_{\vnu - 1}^{\tau_1} N_{\vnu - 2}^{- \mathtt a} {\mathfrak M}_0 (s, {\mathtt b}) \, .
\end{align}

\item[${\bf(S2)_{\vnu}}$] Let $ i_1(\omega, \h) $, $ i_2(\omega, \h) $ be such that 
${\cal R}_0(i_1)$,  ${\cal R}_0(i_2 )$ satisfy \eqref{def:costanti iniziali tame}. 
Then % for all $\vnu = 0, \ldots n$, 
for all $(\omega, \h) \in \tLm_\vnu^{\gamma_1}(i_1) \cap \tLm_\vnu^{\gamma_2}(i_2)$
%\begin{color}{red} (OPPURE su tutto $\R^{\nu} \times [\h_1,\h_2]$ ??) \end{color}
with $\gamma_1, \gamma_2 \in [\gamma/2, 2 \gamma]$, the following estimates hold
\begin{align}\label{stima R nu i1 i2}
& \! \! \| | \langle D \rangle^\fm\Delta_{12} {\cal R}_\vnu \langle D \rangle^\fm| \|_{{\cal L}(H^{s_0})}  
\lesssim_{S,\mathtt b}  \e \gamma^{- 2(M+1)} N_{\vnu - 1}^{- \mathtt a}
 \| i_1 - i_2\|_{s_0 +  \mu(\mathtt b)}, 
\\ 
& \label{stima R nu i1 i2 norma alta}
\! \! \|  |\langle \partial_{\vphi, x}  \rangle^{\mathtt b} \langle D \rangle^\fm \Delta_{12} {\cal R}_\vnu \langle D \rangle^\fm | \|_{{\cal L}(H^{s_0})}
\lesssim_{S, \mathtt b} \e \gamma^{- 2(M+1)} N_{\vnu - 1} \| i_1 - i_2\|_{ s_0 +  \mu(\mathtt b)}\,.
\end{align}
Moreover for $\vnu \geq 1$, for all $j \in {\mathbb S}^c_0$, 
\begin{align}\label{r nu - 1 r nu i1 i2}
& \big|\Delta_{12}(r_j^\vnu - r_j^{\vnu - 1}) \big| \lesssim_{S, \mathtt b} \e \gamma^{- 2(M+1)} |j|^{- 2\fm}   N_{\vnu - 2}^{- \mathtt a}  \| i_1 - i_2  \|_{ s_0  + \mu(\mathtt b)}\,, \\
& \ |\Delta_{12} r_j^{\vnu}| \lesssim_{S, \mathtt b} \e \gamma^{- 2(M+1)} |j|^{- 2 \fm} \| i_1 - i_2  \|_{ s_0  + \mu(\mathtt b)}\,. \label{r nu i1 - r nu i2}
\end{align}

\item[${\bf(S3)_{\vnu}}$]  Let $i_1$, $i_2$ be like in ${\bf(S2)_{\vnu}}$ 
and $0 < \rho \leq \gamma/2$. 
Then
\begin{equation}\label{inclusione cantor riducibilita S4 nu}
C(S) N_{\vnu-1}^{  (\tau + 1)(4 \perd + 1) } \gamma^{-4\perd } \|i_2 - i_1 \|_{s_0 + \mu(\mathtt b)} 
\leq \rho \quad \Longrightarrow \quad 
\tLm_\vnu^\gamma(i_1) \subseteq \tLm_\vnu^{\gamma - \rho}(i_2) \, . 
\end{equation}
\end{itemize}
\end{theorem}

We make some comments: 

\begin{enumerate}
\item 
Note that in \eqref{r nu - 1 r nu i1 i2}-\eqref{r nu i1 - r nu i2}  
we do not need norms $ | \ |^{k_0, \gamma}$.
This is the reason why we did not estimate the derivatives with respect to $(\om, \h)$ 
of the operators $ \Delta_{12} {\cal R} $ in the previous sections.

\item
Since the second Melnikov conditions 
$ |\omega \cdot \ell  + \mu_j^{\vnu - 1} - \mu_{j'}^{\vnu - 1}| \geq  
\gamma |j|^{- \perd} |j'|^{- \perd } \langle \ell \rangle^{-\tau} $
lose regularity both in $\ph$ and in $x$, 
for the convergence of the reducibility scheme 
we use the smoothing operators $\Pi_N$, defined in \eqref{proiettore-oper}, 
which regularize in both $\ph$ and $x$. 
As a consequence, the natural smallness condition to impose at the zero step of the recursion is \eqref{stima cal R nu} at $\vnu = 0$ that we verify in the step ${\bf(S1)}_{0}$ thanks to Lemma \ref{lem: Initialization astratto} and \eqref{costanti resto iniziale}. 
% to verify such a smallness condition it is sufficient to control the tame constants of the operators \eqref{1309.2}.

\item
An important point of Theorem \ref{iterazione riducibilita} is to require 
bound \eqref{KAM smallness condition1} for $ {\mathfrak M}_0 (s_0, {\mathtt b}) $ 
only in low norm, which is verified in Lemma \ref{lem:tame iniziale}. 
On the other hand, Theorem \ref{iterazione riducibilita} 
provides the smallness \eqref{stima cal R nu}  
of the tame constants $ {\mathfrak M}_{\langle D \rangle^\fm {\cal R}_\vnu \langle D \rangle^\fm }^\sharp (s) $ and proves that 
$ {\mathfrak M}_{\langle \partial_{\vphi, x} \rangle^{\mathtt b} \langle D \rangle^\fm {\cal R}_\vnu \langle D \rangle^\fm }^\sharp  (s, {\mathtt b})  $, $ \vnu \geq 0 $, do not diverge too much. 
\end{enumerate}

Theorem \ref{iterazione riducibilita} implies that the invertible operator
\be\label{defUn}
{\cal U}_n := { \Phi}_0 \circ \ldots \circ {\Phi}_{n-1} \,, \quad n \geq 1,
\ee
has almost-diagonalized  $ {\cal L}_0 $, i.e. \eqref{cal L infinito} below holds. 
As a corollary, we deduce the following theorem.

\begin{theorem}\label{Teorema di riducibilita} {\bf (Almost-reducibility of $\mL_0$)}
Assume \eqref{ansatz I delta} with $ \mu_0 \geq   \mu (\mathtt b)$. 
Let $\mR_0 = \mR_\bot^\sym$, $\mL_0 = \mL_\bot^\sym$ in \eqref{2701.1}-\eqref{2601.1}.
For all $S > s_0$ there exists $N_0 := N_0(S, \mathtt b) > 0$, $ \d_0 := \delta_0(S)  >  0 $ 
such that, if the smallness condition 
\begin{equation}\label{ansatz riducibilita}
N_0^{\tau_2} \e \gamma^{- (2M + 3)}  \leq \d_0  
\end{equation}
holds, where the constant $ \tau_2 := \tau_2 (\tau, \nu )  $ is defined in Theorem \ref{iterazione riducibilita} and $M$ is defined in \eqref{relazione mathtt b N}, 
then, for all $ n \in \N$, 
for all $ \lambda = (\omega, \h) \in \R^{\nu} \times [\h_1,\h_2]$, 
the operator $ {\cal  U}_n$ in \eqref{defUn} and its inverse $ {\cal U}_n^{- 1} $ 
are real, even, reversibility preserving, and $ {\cal D}^{k_0}$-modulo-tame,
%The operators $ {\cal U}_{n}^{\pm 1} -  {\mathbb I}_\bot$ are  
with 
\begin{equation}\label{stima Phi infinito}
{\mathfrak M}^\sharp_{{\cal U}_{n}^{\pm 1} -  {\mathbb I}_\bot} (s) 
\lesssim_S \e \g^{-(2M + 3)} N_0^{\tau_1} 
\big(1 +  \| \fracchi_0 \|_{s  + \mu (\mathtt b)}^{k_0, \gamma} \big) 
\quad \forall s_0 \leq s \leq S \, , 
\end{equation} 
where $\tau_1$ is defined in \eqref{alpha beta}.  

The operator $ {\cal L}_n = \Dom {\mathbb I}_\bot + \ii {\cal D}_n + {\cal R}_n$ 
defined in \eqref{cal L nu} (with $\vnu = n$)
is real, even and reversible. 
The operator $ \langle D \rangle^\fm {\cal R}_n \langle D \rangle^\fm$ is $ {\cal D}^{k_0} $-modulo-tame, with 
\be\label{stima resto operatore quasi diagonalizzato}
{\mathfrak M}^\sharp_{\langle D \rangle^\fm {\cal R}_n \langle D \rangle^\fm}(s)  
\lesssim_S  \e \gamma^{- 2(M+1)} N_{n - 1}^{- {\mathtt a}}
\big(1 +  \| \fracchi_0 \|_{s   + \mu (\mathtt b)}^{k_0, \gamma}\big) 
\quad \forall s_0 \leq s \leq S\,.
\ee
Moreover, for all $\lm = (\om, \h)$ in the set 
\be\label{Cantor set}
{ \tLm}_{n}^{ \gamma} 
= \bigcap_{\vnu = 0}^{n} \tLm_\vnu^\gamma
\ee
defined in \eqref{Omega nu + 1 gamma},
the following conjugation formula holds: 
\begin{equation}\label{cal L infinito}
{\cal  L}_n = {\cal  U}_n^{- 1} {\cal L}_0 {\cal  U}_n \,.
\end{equation}
\end{theorem}

\begin{proof}
%[Proof of Theorem \ref{Teorema di riducibilita}, assuming Theorem \ref{iterazione riducibilita}]
Assumption \eqref{KAM smallness condition1} of Theorem \ref{iterazione riducibilita} 
holds by \eqref{costanti resto iniziale}, \eqref{ansatz I delta} 
with $ \mu_0 \geq \mu (\mathtt b)$,  
and \eqref{ansatz riducibilita}.
Estimate \eqref{stima resto operatore quasi diagonalizzato} follows by 
\eqref{stima cal R nu} (for $\vnu = n$) and \eqref{costanti resto iniziale}.
It remains to prove \eqref{stima Phi infinito}.  
The estimates of $\mathfrak M^\sharp_{\Phi_\vnu^{\pm 1} - \mathbb I_\bot}(s)$, $\vnu = 0, \ldots, n-1$, are obtained by using \eqref{tame Psi nu - 1 senza Dm}, \eqref{KAM smallness condition1}
and Lemma \ref{serie di neumann per maggiorantia}.
Then the estimate of ${\cal U}_n^{\pm 1} - {\mathbb I}_\bot$ follows
as in the proof of Theorem 7.5 in \cite{BertiMontalto}, 
using Lemma \ref{interpolazione moduli parametri}.
\end{proof}

\subsection{Proof of Theorem \ref{iterazione riducibilita}}

{\bf Initialization.}

\noindent
{\sc Proof of ${\bf(S1)}_{0}$}. The real, even and 
reversible operator $ {\cal L}_0 $ defined in \eqref{defL0-red}-\eqref{defRQ0} has the form \eqref{cal L nu}-\eqref{mu j nu} for $ \vnu = 0 $ with $r_j^0(\omega, \h) = 0 $, 
and \eqref{stima rj nu} holds trivially. 
Moreover \eqref{2701.3} is satisfied for $ \vnu = 0 $ 
by the definition of $ \mR_0 := \mR_\bot^{sym} $ in \eqref{2601.1}. 
The estimate \eqref{stima cal R nu} for $ \vnu = 0 $ follows by applying Lemma \ref{lem: Initialization astratto} to $A \in  \{ R_1^{(0)}, R_2^{(0)} \}$ and by recalling  definition of $\mathfrak{M}_0(s,\ttb)$ in \eqref{costanti resto iniziale}.
%\begin{lemma}\label{lem: Initialization} 
%$ {\mathfrak M}_{\langle D \rangle^\fm {\cal R}_0 \langle D \rangle^\fm}^\sharp (s) $,  $ {\mathfrak M}_{\langle \partial_{\vphi, x} \rangle^{\mathtt b}\langle D \rangle^\fm {\cal R}_0 \langle D \rangle^\fm}^\sharp ( s) \lesssim_{s_0, \mathtt b}  {\mathfrak M}_0 (s, {\mathtt b})  $. 
%\end{lemma}
%\begin{proof}
%Apply Lemma \ref{lem: Initialization astratto} to $A \in  \{ R_1^{(0)}, R_2^{(0)} \}$ and recall the definition of $\mathfrak{M}_0(s,\ttb)$ 
%in \eqref{costanti resto iniziale}. 
%\end{proof}
\\[1mm]
{\sc Proof of ${\bf(S2)}_0$}. The proof of \eqref{stima R nu i1 i2}, \eqref{stima R nu i1 i2 norma alta} for $\vnu = 0$ follows similarly using Lemma \ref{lem: Initialization astratto} and \eqref{costanti variazioni i resto iniziale}.   
\\[1mm]
{\sc Proof of ${\bf(S3)}_{0}$}. It is trivial because, by definition, 
$\mathtt{\Lambda}_0^\g = \mathtt{DC}(2\g,\t) \times [\h_1, \h_2]
\subseteq \mathtt{DC}(2\g-2\rho,\t) \times [\h_1, \h_2] 
= \mathtt{\Lambda}_0^{\g-\rho}$.

\subsubsection{Reducibility step}

In this section we describe the inductive step and show how to define ${\cal L}_{\vnu + 1}$ 
(and ${ \Psi}_\vnu$, ${ \Phi}_\vnu$, etc). 
To simplify the notation we drop the index $\vnu$ and write $+$ instead of $\vnu + 1$, so that we write
${\cal  L} := {\cal  L}_{\vnu } $, $ {\cal D} := {\cal D}_\vnu $, $D := D_\vnu$, $ \mu_j = \mu_j^\vnu $, $ {\cal R} := {\cal R}_\vnu $, $R_1 := R_1^{(\vnu)}, R_2 := R_2^{(\vnu)}$, 
and $\mL_+ := \mL_{\vnu+1}$, $\mD_+ := \mD_{\vnu+1}$, and so on. 

We conjugate the operator $ {\cal L} $ in \eqref{cal L nu} by a  transformation of the form (see \eqref{Psi nu - 1})
\be\label{defPsi}
{\Phi} := {\mathbb I}_\bot + {\Psi}\,, \qquad {\Psi} := \begin{pmatrix}
\Psi_{1} & \Psi_2 \\
\overline \Psi_2 & \overline \Psi_1
\end{pmatrix}.
\ee
We have 
\be\label{Ltra1}
{\cal L} {\Phi} = {\Phi} ( \Dom {\mathbb I}_\bot + \ii {\cal D} ) 
+ ( \om \cdot \pa_\ph \Psi + \ii [{\cal D}, {\Psi}] + \Pi_{N} {\cal R} ) 
+ \Pi_{N}^\bot {\cal R}  + {\cal R} {\Psi} 
\ee
where the projector $ \Pi_N $ is defined in \eqref{proiettore-oper}, 
$ \Pi_N^\bot := {\mathbb I}_2 - \Pi_N$, 
and $\om \cdot \pa_\ph \Psi$
is the commutator $[\Dom , \Psi]$.
We want to solve the homological equation 
\begin{equation}\label{equazione omologica}
\Dom{ \Psi} + \ii [{\cal D}, { \Psi}] + \Pi_{N} {\cal R}  = [{\cal R}] 
\ee
where
\be\label{def[R]}
[{\cal R}] := \begin{pmatrix}
[{R_1}] & 0 \\
0 & [\overline{R}_1] 
\end{pmatrix}\, , \quad  [{R_1}] := \diag_{j \in {\mathbb S}_0^c} (R_1)_j^j(0)\,.
\ee
By \eqref{cal L nu}, \eqref{forma cal R nu}, \eqref{defPsi}, 
equation \eqref{equazione omologica} is equivalent to the two scalar  homological equations 
\be\label{homo1r}
\Dom \Psi_1 + \ii [{ D}, \Psi_1] + \Pi_{N} { R_1} = [{ R_1}]  \, , \qquad 
\Dom \Psi_2 + \ii ( { D} \Psi_2 + \Psi_2 { D} ) + \Pi_{N} {R_2} = 0 \,
\ee
(note that $[R_1] = [ \Pi_N R_1]$).
We choose the solution of \eqref{homo1r} given by 
\begin{align}\label{shomo1}
& (\Psi_1)_j^{j'}(\ell) := \begin{cases}
- \dfrac{(R_1)_j^{j'}(\ell)}{\ii (\omega \cdot \ell  + \mu_j - \mu_{j'})} 
\quad & 
\forall (\ell ,  j,  j') \neq (0, j,  \pm j) \, , \, | \ell  |, |j - j'| \leq N,  
\\
\, 0 & \text{otherwise;}
\end{cases} 
\\
& \label{shomo2}
(\Psi_2)_j^{j'}(\ell) := \begin{cases}
- \dfrac{(R_2)_j^{j'}(\ell)}{\ii (\omega \cdot \ell  +\mu_j + \mu_{j'})} 
\quad & 
\forall (\ell  , j,  j') \in \Z^{ \nu} \times {\mathbb S}_0^c \times {\mathbb S}_0^c \,, \,
| \ell |, |j - j'| \leq N,  
\\
\, 0 & \text{otherwise.}
\end{cases}
\end{align}
Note that, since $ \mu_j = \mu_{- j} $ for all $ j \in {\mathbb S}_0^c $ (see \eqref{stima rj nu}),
the denominators in \eqref{shomo1}, \eqref{shomo2} are different from zero 
for $ (\om, \h) \in \tLm_{\vnu+1}^\gamma $ (see \eqref{Omega nu + 1 gamma}
with  $ \vnu \rightsquigarrow \vnu+1 $) and the maps $ \Psi_1 $, $ \Psi_2 $ 
are well defined on $ \tLm_{\vnu+1}^\gamma $.
Also note that the term $[R_1]$ in \eqref{def[R]}  
(which is the term we are not able to remove by conjugation with $\Psi_1$
in \eqref{homo1r}) 
contains only the diagonal entries $j' = j$ and not the anti-diagonal ones $j' = -j$, 
because $\mR$ is zero on $j' = -j$ by \eqref{2701.3}.
Thus, by construction,
\begin{equation} \label{2701.4}
(\Psi_1)_j^{j'}(\ell) = (\Psi_2)_j^{j'}(\ell) = 0 \quad \forall (\ell, j, j'), \ jj'< 0 \, .
\end{equation}

\begin{lemma} {\bf (Homological equations)}\label{Homological equations tame}
The operators
$ \Psi_1 $, $ \Psi_2 $ defined in \eqref{shomo1}, \eqref{shomo2} (which, for all  $\lm \in \tLm_{\vnu+1}^\gamma $,  solve the homological equations 
 \eqref{homo1r}) admit an extension to the whole parameter space $\R^{\nu} \times [\h_1,\h_2]$. Such extended operators are $ {\cal D}^{k_0} $-modulo-tame with modulo-tame constants satisfying
\begin{align}
& \label{stima tame Psi} 
{\mathfrak M}_{\langle D \rangle^{\pm \fm}\Psi \langle D \rangle^{\mp \fm}}^\sharp (s) \lesssim_{k_0} N^{\tau_1} \g^{-1} {\mathfrak M}^\sharp_{\langle D \rangle^\fm {\cal R} \langle D \rangle^\fm} (s), 
\\
& 
{\mathfrak M}_{\langle \partial_{\vphi, x} \rangle^{\mathtt b} \langle D \rangle^{\pm \fm}\Psi \langle D \rangle^{\mp \fm}}^\sharp (s)
 \lesssim_{k_0} N^{\tau_1} \g^{-1} {\mathfrak M}^\sharp_{\langle \partial_{\vphi, x}  \rangle^{\mathtt b} \langle D \rangle^\fm {\cal R} \langle D \rangle^\fm} (s) 
\label{stima tame Psi 2001.3}
\\
& 
\label{stima tame Psi senza Dm} {\mathfrak M}_{\Psi }^\sharp (s) \lesssim_{k_0} N^{\tau_1} \g^{-1} {\mathfrak M}^\sharp_{ {\cal R} } (s)
\end{align}
where $\tau_1, {\mathtt b}, \fm$ are defined in \eqref{alpha beta}.

Given $i_1 $, $i_2 $, let $ \Delta_{12} \Psi := \Psi (i_2) - \Psi (i_1) $. 
If $\g_1, \g_2 \in [\gamma/2, 2 \gamma]$, 
then, for all $(\omega, \h) \in \tLm_{\vnu + 1}^{\gamma_1}(i_1) \cap \tLm_{\vnu + 1}^{\gamma_2}(i_2)$,
\begin{align}
% \!\!\! \!  
\| \, |\langle D \rangle^{\pm \fm}\Delta_{12} \Psi \langle D \rangle^{\mp \fm} | \, \|_{{\cal L}(H^{s_0})} & \lesssim 
N^{2 \tau + 2 \perd + \frac12} \gamma^{- 1} \big( \| \, | \langle D \rangle^\fm {\cal R}(i_2) \langle D \rangle^\fm| \, \|_{{\cal L}(H^{s_0})} \| i_1 - i_2 \|_{ s_0 + \mu(\mathtt b)} 
\nonumber \\
& \quad \, + \| \, | \langle D \rangle^\fm \Delta_{1 2} {\cal R} \langle D \rangle^\fm | \, \|_{{\cal L}(H^{s_0})} \big) \, , \label{stime delta 12 Psi bassa}
\\  
%\!\!\! \! 
\| \, |\langle \partial_{\vphi, x} \rangle^{\mathtt b} \langle D \rangle^{\pm \fm}\Delta_{12} \Psi \langle D \rangle^{\mp \fm} | \, \|_{{\cal L}(H^{s_0})}
& \lesssim 
N^{2 \tau + 2 \perd + \frac12} \gamma^{- 1} \big( \| \, |\langle \partial_{\vphi, x} \rangle^{\mathtt b}\langle D \rangle^\fm {\cal R}(i_2) \langle D \rangle^\fm | \, \|_{{\cal L}(H^{s_0})} \| i_1 - i_2 \|_{s_0 + \mu(\mathtt b)} 
\nonumber \\
& \quad \, + \| \, | \langle \partial_{\vphi, x} \rangle^{\mathtt b} \langle D \rangle^\fm \Delta_{12} {\cal R} \langle D \rangle^\fm | \, \|_{{\cal L}(H^{s_0})}\big)\,. 
\label{stime delta 12 Psi alta}
\end{align}
Moreover  ${\Psi}$ is real, even and reversibility preserving. 
\end{lemma}

\begin{proof}
For all $\lm \in \tLm_{\vnu+1}^\gamma$,  $(\ell, j, j') \neq (0, j, \pm j)$, $ j, j' \in {\mathbb S}_0^c $
 $|\ell|,|j-j'| \leq N$, we have the small divisor estimate 
 $$
 |\om \cdot \ell + \mu_j - \mu_{j'}| =  
 |\om \cdot \ell + \mu_{|j|} - \mu_{|j'|}| \geq \gamma |j|^{- \perd} |j'|^{- \perd} \langle \ell \rangle^{-\tau}
 $$ 
by \eqref{Omega nu + 1 gamma}, because $ ||j| - |j'|| \leq | j - j' | \leq N $.
As in Lemma \ref{lemma:cut-off sd}, we extend the restriction to $F = \tLm_{\vnu+1}^\gamma$
of the function $ (\om \cdot \ell + \mu_j - \mu_{j'})^{-1}$ 
to the whole parameter space $\R^{\nu} \times [\h_1,\h_2] $ 
by setting 
$$
g_{\ell,j,j'} (\lm) := \frac{\chi \big( f(\lambda) \rho^{-1} \big) }{f (\lambda)  } \, , \quad
f(\lambda) := \om \cdot \ell + \mu_j - \mu_{j'} \, ,
\quad \rho := \gamma \langle \ell \rangle^{-\tau} |j|^{- \perd} |j'|^{- \perd}  \,,
$$
where $\chi$ is the cut-off function in \eqref{cut off simboli 1}.
We now estimate the corresponding constant $M$ in \eqref{0103.1}.
For $n \geq 1$, $x > 0$, the $n$-th derivative of the function 
$\tanh^{\frac12} (x)$ is $P_n(\tanh(x)) \tanh^{\frac12 - n}(x) (1 - \tanh^2(x))$, 
where $P_n$ is a polynomial of degree $\leq 2n-2$. 
Hence $| \pa_\h^n \{ \tanh^{\frac12} (\h |j|) \} | \leq C$ 
for all $n = 0, \ldots, k_0$, for all $\h \in [\h_1, \h_2]$, 
for all $j \in \Z$,
for some $C = C(k_0, \h_1)$ independent of $n,\h,j$.
By \eqref{mu j nu}, \eqref{stima rj nu}, 
\eqref{op-diago0}, 
\eqref{stime lambda 1}, 
\eqref{stima code - 1/2 inizio riducibilita}
(and recalling that $\mu_j$ here denotes $\mu_j^\vnu$), 
since $\e \g^{-2(M+1)} \leq \g$, we deduce that 
\begin{equation} \label{2017.palma.1}
\g^{|\a|} | \pa_\lm^\a \mu_j | \lesssim \g |j|^{\frac12} \quad 
\forall \a \in \N^{\nu + 1}, \ 1 \leq |\a| \leq k_0 \, .
\end{equation}
Since $\g^{|\a|} | \pa_\lm^\a (\om \cdot \ell) | \leq \g |\ell|$ for all $|\a| \geq 1$, 
we conclude that 
\begin{equation} \label{2001.1}
\g^{|\a|} | \pa_\lm^\a ( \om \cdot \ell + \mu_j - \mu_{j'} ) |
\lesssim \g ( |\ell| + |j|^{\frac12} + |j'|^{\frac12} )
\lesssim \g \langle \ell \rangle |j|^{\frac12} |j'|^{\frac12} \, , 
\quad \forall \ 1 \leq |\a| \leq k_0 \, .
\end{equation}
Thus \eqref{0103.1} holds with  $M = C \g \langle \ell \rangle |j|^{\frac12} |j'|^{\frac12}$ (which is $ \geq \rho $) and
 \eqref{0103.2} implies that  
\begin{equation}\label{17.0103.1}
| g_{\ell,j,j'} |^{k_0,\g}
\lesssim \gamma^{- 1} \langle \ell \rangle^{\tau(k_0 + 1) + k_0} |j|^\fm |j'|^\fm  \quad {\rm with} \quad 
\fm = (k_0 + 1) \perd + \frac{k_0}{2} 
\end{equation}
defined in \eqref{alpha beta}.
Formula \eqref{shomo1} with $(\om \cdot \ell + \mu_j - \mu_{j'})^{-1}$ 
replaced by $g_{\ell,j,j'} (\lm)$ defines the extended operator $\Psi_1$
to $\R^{\nu} \times [\h_1,\h_2]$. 
Analogously, we construct an extension of the function $(\om \cdot \ell + \mu_j + \mu_{j'})^{-1}$ 
to the whole $\R^{\nu} \times [\h_1,\h_2] $, 
and we obtain an extension of the operator $\Psi_2$ in \eqref{shomo2}. 

\smallskip

\noindent
{\sc Proof of \eqref{stima tame Psi}, \eqref{stima tame Psi 2001.3}, \eqref{stima tame Psi senza Dm}}. 
We prove \eqref{stima tame Psi 2001.3} for $ \Psi_1 $,
then the estimate for $\Psi_2$ follows in the same way, 
as well as \eqref{stima tame Psi}, \eqref{stima tame Psi senza Dm}. 
Furthermore, we analyze $\langle D \rangle^\fm \partial_\lambda^k\Psi_1 \langle D \rangle^{- \fm}$, since $\langle D \rangle^{- \fm} \partial_\lambda^k\Psi_1 \langle D \rangle^\fm$ can be treated in the same way.  
Differentiating $(\Psi_1)_j^{j'}(\ell) = g_{\ell, j, j'} (R_1)_j^{j'}(\ell)$, 
one has that, for any $|k| \leq k_0$, 
\begin{align}
|\partial_\lambda^k (\Psi_1)_j^{j'}(\ell) |
& \lesssim \sum_{k_1 + k_2 = k} |\partial_\lambda^{k_1} g_{\ell, j, j'}| |\partial_\lambda^{k_2} (R_1)_j^{j'}(\ell)|
\lesssim \sum_{k_1 + k_2 = k} \gamma^{- |k_1|}|g_{\ell, j, j'}|^{k_0, \gamma} |\partial_\lambda^{k_2} (R_1)_j^{j'}(\ell)| \nonumber\\
& \stackrel{\eqref{17.0103.1}}{\lesssim}  
\langle \ell \rangle^{\tau(k_0 + 1) + k_0} |j|^\fm |j'|^\fm \gamma^{- 1 - |k|}  \sum_{ |k_2| \leq |k|} \gamma^{|k_2|} |\partial_\lambda^{k_2} (R_1)_j^{j'}(\ell)| \, . \label{stima-Psi-R}
\end{align}
For $|j - j'| \leq N$, $j, j' \neq 0$, one has 
\begin{equation} \label{nuova stima eq omologica 1}
|j|^{2\fm} \lesssim |j|^\fm (|j'|^\fm + |j-j'|^\fm) 
\lesssim  |j|^\fm (|j'|^\fm + N^\fm) 
\lesssim |j|^\fm |j'|^\fm N^\fm.
\end{equation}
Hence, by \eqref{stima-Psi-R} and \eqref{nuova stima eq omologica 1}, 
for all $|k| \leq k_0$, $j, j' \in \mathbb S_0^c$, $\ell \in \Z^{\nu}$, $|\ell| \leq N$, $|j - j'| \leq N$, one has 
\begin{align} 
|j|^\fm |\partial_\lambda^k (\Psi_1)_j^{j'}(\ell) | |j'|^{- \fm} 
 & \lesssim N^{\tau_1}   \gamma^{- 1 - |k|} 
  \sum_{  |k_2| \leq |k|} \gamma^{|k_2|} |j|^\fm |\partial_\lambda^{k_2} (R_1)_j^{j'}(\ell)| |j'|^\fm \label{nuova stima eq omologica 2}
\end{align}
where $\tau_1 = \tau(k_0 + 1) + k_0 + \fm $ is defined in \eqref{alpha beta}.
Therefore, for all $0 \leq |k| \leq k_0 $,  we get 
\begin{align}
& \| \, |\langle \partial_{\vphi, x} \rangle^{\mathtt b} \langle D \rangle^\fm \partial_\lambda^k \Psi_1 \langle D \rangle^{- \fm}| \, h \|_s^2  
\notag \\ 
& \leq \sum_{\ell, j} \langle \ell, j \rangle^{2 s} \Big(\sum_{ |\ell' - \ell|, |j' - j| \leq N} 
\langle \ell - \ell', j - j' \rangle^{\mathtt b} \langle j \rangle^\fm 
| \partial_\lambda^k (\Psi_1)_j^{j'}(\ell - \ell') | \langle j' \rangle^{- \fm} |h_{\ell', j'}| \Big)^2 \nonumber\\
&  \stackrel{\eqref{nuova stima eq omologica 2}} {{\lesssim}_{k_0}} 
N^{2 \tau_1} \gamma^{- 2(1 + |k|)} \sum_{|k_2| \leq |k|}  \gamma^{2 |k_2|} \sum_{\ell, j} \langle \ell, j \rangle^{2 s} 
\Big(\sum_{\ell', j'} | \langle \ell - \ell', j - j' \rangle^{\mathtt b} \langle j \rangle^\fm 
\partial_\lambda^{k_2}(R_1)_j^{j'}(\ell - \ell') \langle j' \rangle^\fm | |h_{\ell', j'}| \Big)^2 \nonumber\\
& \lesssim_{k_0} N^{2 \tau_1} \gamma^{- 2(1 + |k|)} \sum_{|k_2| \leq |k|} \gamma^{2 |k_2|} 
\big\| \, | \langle \partial_{\vphi, x} \rangle^{\mathtt b} \langle D \rangle^\fm \partial_\lambda^{k_2}(R_1) \langle D \rangle^\fm | \, [ \, \norma h \norma \, ] \, \big\|_s^2 
\nonumber\\
& \stackrel{ \eqref{CK0-tame}, \eqref{Soboequals}} {{\lesssim }_{k_0}} 
 N^{2 \tau_1} \gamma^{- 2(1 + |k|)} 
 \Big( {\mathfrak M}^\sharp_{\langle \partial_{\vphi, x} \rangle^{\mathtt b} 
 \langle D \rangle^\fm R_1 \langle D \rangle^\fm } (s) \|  h \|_{s_0} + 
{\mathfrak M}^\sharp_{\langle \partial_{\vphi, x} \rangle^{\mathtt b} \langle D \rangle^\fm R_1 \langle D \rangle^\fm } (s_0) 
\|  h \|_{s}  \Big)^2 \label{caserta}
\end{align}
and, recalling Definition \ref{def:op-tame}, inequality \eqref{stima tame Psi 2001.3} follows. The proof of \eqref{stime delta 12 Psi bassa}-\eqref{stime delta 12 Psi alta} follow similarly. 
\end{proof}

If $ \Psi $, with $ \Psi_1, \Psi_2 $ defined in 
\eqref{shomo1}-\eqref{shomo2},
 satisfies the smallness condition  
\begin{equation}\label{piccolezza neumann-in concreto}
4 C(\mathtt b) C(k_0)  {\mathfrak M}_{\Psi}^\sharp (s_0)  \leq 1/ 2 \, , 
\end{equation}
then,  by Lemma \ref{serie di neumann per maggiorantia},  $ {\Phi} $ is invertible,  and  
\eqref{Ltra1}, \eqref{equazione omologica} imply that, for all $\lm \in \tLm_{\vnu+1}^\gamma $, 
\begin{equation}\label{coniugio L piu}
{\cal L}_+  = {\Phi}^{- 1} {\cal L} { \Phi} = \Dom {\mathbb I}_\bot + \ii {\cal D}_+ + {\cal R}_+ 
\end{equation}
which  proves \eqref{coniugionu+1} and \eqref{cal L nu} at the step $ \vnu + 1 $,  with 
\be\label{new-diag-new-rem} 
\ii {\cal D}_+ := \ii {\cal  D} + [{\cal R}] \, , \quad 
{\cal R}_+: = 
{\Phi}^{- 1} \big( \Pi_{N}^\bot {\cal R} + {\cal R} {\Psi} - { \Psi} [{\cal R}] \big) \, . 
\ee
We note that ${\cal R}_+$ satisfies 
\begin{equation} \label{2017.1505.5}
\mR_+ = \begin{pmatrix} (R_+)_1 & (R_+)_2 \\ 
(\overline R_+)_2 & (\overline R_+)_1 \end{pmatrix}, 
\quad 
[(R_+)_1]_j^{j'}(\ell) = [(R_+)_2]_j^{j'}(\ell) = 0 
\quad 
\forall (\ell, j, j'), \ j j' < 0 \, ,
\end{equation}
similarly as $\mR_\vnu$ in \eqref{2701.3},
because 
the property of having zero matrix entries for $j j' < 0$ is preserved by matrix product, 
and $\mR, \Psi, [\mR]$ satisfy such a property (see \eqref{2701.3}, \eqref{2701.4}, \eqref{def[R]}), 
and therefore, by Neumann series, also $\Phi^{-1}$ does.

The right hand sides of \eqref{coniugio L piu}-\eqref{new-diag-new-rem} define an extension of ${\cal L}_+$ to the whole parameter space $\R^{\nu} \times [\h_1,\h_2]$, 
since  $\mR$ and $\Psi$ are defined on $\R^{\nu} \times [\h_1,\h_2]$.

The new operator $ {\cal L}_+  $ in \eqref{coniugio L piu} has the same form as $ {\cal L} $ in \eqref{cal L nu}, with 
the non-diagonal remainder
$ {\cal R}_+$ defined in \eqref{new-diag-new-rem}
which is the sum of a quadratic function of 
$ { \Psi} $, $ {\cal R}$ and a term $ \Pi_{N}^\bot {\cal R}$ supported on high frequencies. 
The new normal form $ {\cal D}_+  $ in \eqref{new-diag-new-rem} is diagonal: 

\begin{lemma}\label{nuovadiagonale} 
{\bf (New diagonal part).}  
For all $(\om,\h) \in \R^{\nu} \times [\h_1,\h_2]$ we have 
\be\label{new-NF}
\ii {\cal D}_+ = \ii {\cal D} + [{\cal R}] = \ii \begin{pmatrix}
 {D}_+ & 0 \\
0 & - { D}_+
\end{pmatrix}\,,\quad { D}_+ := {\rm diag}_{j \in {\mathbb S}^c_0} \mu_j^+\,,\quad \mu_j^+ := 
\mu_j + {\mathtt r}_j\in \R\, , 
\ee
with $ {\mathtt r}_j = {\mathtt r}_{- j} $, $ \mu_j^+ = \mu_{- j}^+ $ 
for all $j \in {\mathbb S}_0^c $, and, on $\R^{\nu} \times [\h_1,\h_2]$,
\be\label{rj+-stima}
|{\mathtt r}_j|^{k_0, \gamma} = |\mu_j^+ - \mu_j|^{k_0, \gamma} 
\lesssim |j|^{- 2 \fm}{\mathfrak M}^\sharp_{\langle D \rangle^\fm {\cal R} \langle D \rangle^\fm} (s_0).
\ee
Moreover, given tori $ i_1(\omega, \h) $, $ i_2(\omega, \h)$, the difference 
\be\label{diff:r1r2}
| {\mathtt r}_j (i_1) -{\mathtt r}_j (i_2)   |  
\lesssim |j|^{- 2\fm } \| |  \langle D \rangle^\fm\Delta_{1 2} {\cal R} \langle D \rangle^\fm | \|_{{\cal L}(H^{s_0})}\,.
\ee
\end{lemma}

\begin{proof}
Identity \eqref{new-NF} follows  
by \eqref{cal L nu} and \eqref{def[R]}
 with $ {\mathtt r}_j := -\ii (R_1)_j^j (0) $. 
Since  $ R_1$ satisfies \eqref{2701.3} and it is even, 
we deduce, by \eqref{even operators Fourier}, that 
$ {\mathtt r}_{-j} = {\mathtt r}_j  $. Since $ {\cal R} $ is reversible, 
\eqref{2601.4} implies that 
$ {\mathtt r}_j := -\ii (R_1)_j^j (0) $ satisfies
$ {\mathtt r}_j = \overline{{\mathtt r}_{-j}} $. Therefore 
$ {\mathtt r}_j = \overline{{\mathtt r}_{-j}} =\overline{{\mathtt r}_{j}}  $ and  each $ {\mathtt r}_j  \in \R $. 
% Since $  \overline{{\mathtt r}_{j}} = \overline{{\mathtt r}_{-j}} $
%In addition, since  $ R_1  = A + \ii B $ is reversible we have $ R_1(-\vphi) = - \overline{R}_1(\vphi) $, and
%so the maps $ \vphi \mapsto A_j^{j'} (\vphi) $ are odd and 
%the average $ A_j^j(0) = \int_{\T^\nu} A_j^j(\vphi)\, d \vphi = 0 $.
%Hence $ (R_1)_j^j(0) = \ii B_j^j(0) \in \ii \R $ and  each $ {\mathtt r}_j  \in \R $. 

Recalling
Definition \ref{def:op-tame},  we have
$ \| |\partial_\lambda^k (\langle D \rangle^\fm R_1 \langle D \rangle^\fm)| h \|_{s_0} \leq 2 \gamma^{- |k|} {\mathfrak M}^\sharp_{\langle D \rangle^\fm R_1 \langle D \rangle^\fm} (s_0) \| h \|_{s_0} $, 
for all  $\lambda = (\omega, \h)$, $0 \leq |k| \leq k_0$, 
and therefore (see \eqref{tame-coeff})
$$  
|\partial_\lambda^k (R_1)_j^j(0)| \lesssim |j|^{- 2\fm} \gamma^{- |k|}  {\mathfrak M}^\sharp_{\langle D \rangle^\fm R_1 \langle D \rangle^\fm} (s_0)  \lesssim |j|^{- 2\fm} \gamma^{- |k|}  {\mathfrak M}^\sharp_{\langle D \rangle^\fm {\cal R} \langle D \rangle^\fm} (s_0) 
$$
which implies \eqref{rj+-stima}. 
Estimate \eqref{diff:r1r2} 
follows  by $|\Delta_{1 2} (R_1)_j^j(0)| \lesssim |j|^{- 2\fm} \| | \langle D \rangle^\fm \Delta_{1 2} {\cal R} \langle D \rangle^\fm | \|_{{\cal L}(H^{s_0})}$. 
\end{proof}

\subsubsection{Reducibility iteration}

Let $ n \geq 0 $ and suppose that $({\bf S1})_{\vnu} $-$({\bf S3})_{\vnu} $ are true
for all $\vnu = 0,\ldots,n$.
We prove $({\bf S1})_{n + 1}$-$({\bf S3})_{n + 1}$. For simplicity of notation % (as in other parts of the paper) 
we omit to write the dependence on $k_0$ which is considered as a fixed constant. 

\smallskip

\noindent
{\sc Proof of $({\bf S1})_{n + 1}$}. 
By \eqref{stima tame Psi}-\eqref{stima tame Psi senza Dm}, \eqref{stima cal R nu}, and 
using that $\mathfrak M^\sharp_{{\cal R}_n}(s) \lesssim \mathfrak M^\sharp_{\langle D \rangle^\fm {\cal R}_n  \langle D \rangle^\fm}(s)$,
the operator $\Psi_n$ 
defined in Lemma \ref{Homological equations tame}
satisfies estimates \eqref{tame Psi nu - 1}-\eqref{tame Psi nu - 1 senza Dm} with $\vnu = n + 1 $.
In particular at $ s = s_0 $ we have
\begin{equation}\label{costante tame Psi nu i (s0)}
{\mathfrak M}^\sharp_{\langle D \rangle^{\pm \fm}\Psi_n \langle D \rangle^{\mp \fm}}(s_0 ) \,, \  {\mathfrak M}^\sharp_{\Psi_n}(s_0 )  \leq
C(s_0, \mathtt b) N_n^{\tau_1} N_{n - 1}^{- \mathtt a} \gamma^{- 1} {\mathfrak M}_0(s_0, {\mathtt b})\,.
\end{equation}
Therefore, by  \eqref{costante tame Psi nu i (s0)}, \eqref{alpha beta}, \eqref{KAM smallness condition1}, choosing $\t_2>\t_1$,
the smallness condition \eqref{piccolezza neumann-in concreto}
holds for $ N_0 := N_0 (S, {\mathtt b}) $ large enough (for any $n \geq 0$),
and the map $\Phi_n = {\mathbb I}_\bot + {\Psi}_n $ is invertible, with inverse
\begin{equation}\label{inverso Phi nu}
\Phi_n^{-1} = {\mathbb I}_\bot + \check{\Psi}_n\,, \qquad  
\check{\Psi}_n :=  \begin{pmatrix}
\check{\Psi}_{n, 1} & \check{\Psi}_{n, 2} \\
\overline{\check{\Psi}}_{n, 2} & \overline{\check{\Psi}}_{n, 1} \end{pmatrix} . 
\end{equation}
Moreover also the smallness condition
\eqref{piccolezza neumann tame} (of Corollary \ref{serie di neumann per maggioranti})
with $ A = \Psi_n $, holds, and Lemma \ref{serie di neumann per maggiorantia}, Corollary \ref{serie di neumann per maggioranti} and Lemma \ref{Homological equations tame} imply that 
the maps $\check \Psi_n$, 
$ \langle D \rangle^{\pm \fm}{\check \Psi}_n \langle D \rangle^{\mp \fm}$ 
and $\langle \partial_{\vphi, x} \rangle^{\mathtt b} \langle D \rangle^{\pm \fm}{\check \Psi}_n \langle D \rangle^{\mp \fm}$ 
are $ {\cal D}^{k_0} $-modulo-tame with modulo-tame constants satisfying
%satisfying the same bounds \eqref{tame Psi nu - 1}-\eqref{tame Psi nu - 1 senza Dm}$|_{\nu = n+1}$ 
%as $\Psi_n$, i.e.
\begin{align}
{\mathfrak M}_{{\check \Psi}_n }^\sharp (s), \ 
{\mathfrak M}_{ \langle D \rangle^{\pm \fm}{\check \Psi}_n \langle D \rangle^{\mp \fm}}^\sharp (s)  & \lesssim_{ s_0, \mathtt b} 
N_n^{\tau_1}  \g^{-1}  {\mathfrak M}^\sharp_{ \langle D \rangle^\fm {\cal R}_n \langle D \rangle^\fm} (s)   \label{stima tame Psi tilde} \\
& \stackrel{\eqref{stima cal R nu}_{|n}}{\lesssim_{ s_0, \mathtt b}} 
N_n^{\tau_1} N_{n - 1}^{- \mathtt a} \g^{-1}  {\mathfrak M}_0(s, \mathtt b) \, ,   \label{stima tame Psi tilde A} 
\end{align}
and 
\begin{align}
{\mathfrak M}_{\langle \partial_{\vphi, x} \rangle^{\mathtt b} \langle D \rangle^{\pm \fm}{\check \Psi}_n \langle D \rangle^{\mp \fm}}^\sharp (s)  
& \lesssim_{s_0, \mathtt b}  N_n^{\tau_1}  \g^{-1} {\mathfrak M}^\sharp_{\langle \partial_{\vphi, x}  \rangle^{\mathtt b} \langle D \rangle^\fm {\cal R}_n \langle D \rangle^\fm} (s)  \nonumber \\
& \quad +   N_n^{2 \tau_1} \g^{-2} {\mathfrak M}^\sharp_{\langle \partial_{\vphi, x}  \rangle^{\mathtt b} \langle D \rangle^\fm {\cal R}_n \langle D \rangle^\fm} (s_0) {\mathfrak M}^\sharp_{ \langle D \rangle^\fm {\cal R}_n \langle D \rangle^\fm} (s) \label{1502.1} \\
& \stackrel{\eqref{stima cal R nu}_{|n}, \eqref{alpha beta}, \eqref{KAM smallness condition1}}{\lesssim_{s_0, \mathtt b}}  N_n^{\tau_1} N_{n-1} \g^{-1} {\mathfrak M}_0(s, \mathtt b) \, . \label{1502.1 A}
\end{align}
Conjugating $ {\cal L}_n $ by  $ \Phi_n $, we obtain, 
by \eqref{coniugio L piu}-\eqref{new-diag-new-rem}, 
for all $\lm \in \mathtt{\Lambda}_{n+1}^\g$, 
\begin{equation}\label{coniugio L nu+1}
{\cal L}_{n+1} = \Phi_n^{-1} {\cal L}_n \Phi_n 
= \Dom {\mathbb I}_\bot + \ii {\cal D}_{n+1} + {\cal R}_{n+1} \, ,
\end{equation}
namely \eqref{coniugionu+1} at $\vnu = n+1$, where 
\begin{equation}\label{cal R nu + 1 tame}
\ii \mD_{n+1} := \ii \mD_n + [\mR_n] \, ,  \qquad 
{\cal R}_{n + 1}: = 
\Phi^{-1}_n \big( \Pi_{N_n}^\bot {\cal R}_n + {\cal R}_n \Psi_n - \Psi_n [{\cal R}_n] \big) \,. 
\end{equation}
The operator ${\cal L}_{n+1}$ is real, even and reversible 
because $ \Phi_n $ is real, even and reversibility preserving 
(Lemma \ref{Homological equations tame}) and ${\cal L}_n$ is real, even and reversible.
Note that the operators $\mD_{n+1}, \mR_{n+1}$ are defined 
on $\R^{\nu} \times [\h_1,\h_2]$, 
and the identity \eqref{coniugio L nu+1} holds on $\mathtt{\Lambda}_{n+1}^\g$.

By Lemma \ref{nuovadiagonale} the operator $ {\cal D}_{n+1} $ is diagonal and, 
by \eqref{def:costanti iniziali tame}, \eqref{stima cal R nu}, 
\eqref{costanti resto iniziale}, 
its eigenvalues $ \mu_j^{n+1} : \R^{\nu} \times [\h_1,\h_2] \to \R $ 
satisfy 
$$ 
|{\mathtt r}_j^n|^{k_0, \gamma} = | \mu_j^{n + 1} - \mu_j^n|^{k_0, \gamma} 
\lesssim |j|^{- 2\fm} {\mathfrak M}_{\langle D \rangle^\fm {\cal R}_n \langle D \rangle^\fm}^\sharp (s_0) 
\leq C(S, {\mathtt b}) \e \gamma^{- 2(M+1)} |j|^{- 2\fm} N_{n-1}^{-{\mathtt a}} \, ,
$$ 
which is \eqref{vicinanza autovalori estesi} with $\vnu = n+1$.
Thus also \eqref{stima rj nu} at $ \vnu = n+1 $ holds, by a telescoping sum.
In addition, by \eqref{2017.1505.5} the operator $ {\cal R}_{n+1} $ satisfies \eqref{2701.3} 
with $\vnu = n+1$.
In order to prove that \eqref{stima cal R nu} holds with $ \vnu = n + 1 $, 
we first provide the following  inductive estimates 
on the new remainder $ {\cal R}_{n + 1} $. 

\begin{lemma}\label{estimate in low norm} 
The operators $ \langle D \rangle^\fm {\cal R}_{n+1} \langle D \rangle^\fm$ 
and $ \langle \partial_{\vphi, x} \rangle^{ {\mathtt b}} \langle D \rangle^\fm {\cal R}_{n + 1} \langle D \rangle^\fm$ 
are $ {\cal D}^{k_0} $-modulo-tame, with 
\begin{align}
& {\mathfrak M}_{ \langle D \rangle^\fm {\cal R}_{n + 1} \langle D \rangle^\fm}^\sharp (s ) 
\lesssim_{s_0, \mathtt b}  N_n^{- {\mathtt b}} {\mathfrak M}_{\langle \partial_{\vphi, x} \rangle^{\mathtt b} \langle D \rangle^\fm {\cal R}_n \langle D \rangle^\fm}^\sharp (s) 
+ \frac{N_n^{\tau_1}}{ \gamma } {\mathfrak M}_{\langle D \rangle^\fm {\cal R}_n \langle D \rangle^\fm}^\sharp (s) {\mathfrak M}_{\langle D \rangle^\fm {\cal R}_n \langle D \rangle^\fm}^\sharp (s_0)\,, \label{schema quadratico tame} \\
& {\mathfrak M}_{\langle \partial_{\vphi, x} \rangle^{ {\mathtt b}} \langle D \rangle^\fm 
{\cal R}_{n + 1} \langle D \rangle^\fm}^\sharp (s) 
\lesssim_{s_0, \mathtt b}
 {\mathfrak M}_{\langle \partial_{\vphi, x} \rangle^{ {\mathtt b}} \langle D \rangle^\fm {\cal R}_n \langle D \rangle^\fm}^\sharp (s)  \nonumber\\
 & \qquad \qquad \qquad \qquad \qquad \qquad  + N_n^{\tau_1}\gamma^{- 1} {\mathfrak M}_{\langle \partial_{\vphi, x} \rangle^{ {\mathtt b}} \langle D \rangle^\fm {\cal R}_n \langle D \rangle^\fm}^\sharp (s_0) 
{\mathfrak M}_{ \langle D \rangle^\fm {\cal R}_n \langle D \rangle^\fm}^{\sharp}(s)  \,. 
\label{M+Ms}
\end{align}
\end{lemma}

\begin{proof}
By \eqref{cal R nu + 1 tame} and \eqref{inverso Phi nu}, we write 
\begin{align}
\langle D \rangle^\fm {\cal R}_{n + 1} \langle D \rangle^\fm & =  \langle D \rangle^\fm \Pi_{N_n}^\bot {\cal R}_n \langle D \rangle^\fm + (\langle D \rangle^\fm \check{ \Psi}_n \langle D \rangle^{- \fm}) ( \langle D \rangle^\fm \Pi_{N_n}^\bot {\cal R}_n \langle D \rangle^\fm)   \nonumber\\
& \qquad + \Big(\mathbb I_\bot + \langle D \rangle^\fm \check \Psi_n \langle D \rangle^{- \fm} \Big) \Big( (\langle D \rangle^\fm {\cal R}_n \langle D \rangle^\fm)( \langle D \rangle^{- \fm} \Psi_n \langle D \rangle^\fm ) \Big) \nonumber\\
& \qquad -  \Big(\mathbb I_\bot + \langle D \rangle^\fm \check \Psi_n \langle D \rangle^{- \fm} \Big) \Big( (\langle D \rangle^\fm \Psi_n \langle D \rangle^{- \fm}) (\langle D \rangle^\fm [{\cal R}_n] \langle D \rangle^\fm) \Big)\,. \label{cal R nu + 1 tame Dm} 
\end{align}
The proof of \eqref{schema quadratico tame} follows by estimating separately all the terms in \eqref{cal R nu + 1 tame Dm}, applying  Lemmata \ref{lemma:smoothing-tame},
\ref{interpolazione moduli parametri},
and \eqref{stima tame Psi}, \eqref{stima tame Psi tilde}, \eqref{stima cal R nu}$_{| n}$, \eqref{alpha beta}, \eqref{KAM smallness condition1}.
The proof of \eqref{M+Ms} follows by formula \eqref{cal R nu + 1 tame Dm}, 
Lemmata \ref{interpolazione moduli parametri}, \ref{lemma:smoothing-tame} and 
estimates \eqref{stima tame Psi}, \eqref{stima tame Psi 2001.3}, \eqref{stima tame Psi tilde}, \eqref{stima cal R nu}$_{|n}$, \eqref{alpha beta}, \eqref{KAM smallness condition1}.
\end{proof}

In the next lemma we prove that \eqref{stima cal R nu} holds at $\vnu = n+1$, 
concluding the proof of $({\bf S1})_{n + 1} $.

\begin{lemma}\label{stima M nu + 1 K nu + 1}
For $N_0 = N_0(S, \mathtt b) > 0$ large enough we have 
\begin{align*} 
& {\mathfrak M}_{\langle D \rangle^\fm {\cal R}_{n + 1} \langle D \rangle^\fm}^\sharp (s) \leq C_*(s_0, \mathtt b) N_n^{- \mathtt a} {\mathfrak M}_0(s, \mathtt b) \\
&  
 {\mathfrak M}_{\langle \partial_{\vphi, x} \rangle^{\mathtt b}\langle D \rangle^\fm {\cal R}_{n + 1} \langle D \rangle^\fm}^\sharp (s) \leq  C_*(s_0, \mathtt b)N_n {\mathfrak M}_0(s, \mathtt b)  \, .
 \end{align*} 
\end{lemma}

\begin{proof}
By \eqref{schema quadratico tame} and \eqref{stima cal R nu} we get 
$$
\begin{aligned}
{\mathfrak M}_{\langle D \rangle^\fm {\cal R}_{n + 1} \langle D \rangle^\fm}^\sharp (s)  
& \lesssim_{s_0, \mathtt b} N_{n}^{- \mathtt b} N_{n - 1} {\mathfrak M}_0(s, \mathtt b) + N_n^{\tau_1} \gamma^{- 1} 
{\mathfrak M}_0(s, \mathtt b) {\mathfrak M}_0(s_0, \mathtt b) N_{n - 1}^{- 2 \mathtt a} \\
&  \leq C_*(s_0, \mathtt b) N_n^{- \mathtt a} {\mathfrak M}_0(s, \mathtt b)
 \end{aligned}
$$
by \eqref{alpha beta}, \eqref{KAM smallness condition1},
taking $N_0(S, \mathtt b) > 0$ large enough 
and $\tau_2 > \tau_1 + \mathtt{a}$.
Then by \eqref{M+Ms}, \eqref{stima cal R nu}  we get that 
$$
\begin{aligned}
{\mathfrak M}_{\langle \partial_{\vphi, x} \rangle^{\mathtt b}\langle D \rangle^\fm {\cal R}_{n + 1} \langle D \rangle^\fm}^\sharp (s) &
 \lesssim_{s_0, \mathtt b} N_{n - 1} {\mathfrak M}_0(s, \mathtt b) + 
N_n^{\tau_1} N_{n - 1}^{1 - \mathtt a} \gamma^{- 1} {\mathfrak M}_0(s, \mathtt b) {\mathfrak M}_0(s_0, \mathtt b) \\
& \leq C_*(s_0, \mathtt b) N_n {\mathfrak M}_0(s, \mathtt b) \nonumber
\end{aligned}
$$
by \eqref{alpha beta}, \eqref{KAM smallness condition1} and taking $N_0(S, \mathtt b)> 0$ large enough.  
\end{proof}

\noindent {\sc Proof of $({\bf S2})_{n + 1}$.} The proof of the estimates \eqref{stima R nu i1 i2}, \eqref{stima R nu i1 i2 norma alta} 
for  $ \vnu = n + 1$ for the term $\Delta_{12}{\cal R}_{n + 1}$ (where ${\cal R}_{n + 1}$ is defined in \eqref{cal R nu + 1 tame}) follow as above. The proof of \eqref{r nu - 1 r nu i1 i2} for $\vnu = n + 1$ follows 
estimating $ \Delta_{12}(r_j^{n+1} - r_j^{n}) = \Delta_{12} {\mathtt r}_j^{n} $ by 
\eqref{diff:r1r2} of Lemma \ref{nuovadiagonale} and by \eqref{stima R nu i1 i2} for $\vnu = n$. Estimate \eqref{r nu i1 - r nu i2} for $\vnu = n + 1$ follows by a telescoping argument 
using \eqref{r nu - 1 r nu i1 i2} and \eqref{stima R nu i1 i2}.

\medskip

\noindent {\sc Proof of $({\bf S3})_{n + 1}$.}
First we note that the non-resonance conditions imposed in \eqref{Omega nu + 1 gamma} are actually finitely many.
We prove the following 

\begin{itemize}
\item{\sc Claim:} Let $ \omega \in  \mathtt{DC}( 2 \gamma, \tau) $ and $ \e \gamma^{- 2(M+1)} \leq 1 $. 
Then there exists $C_0 > 0$ such that, 
for any $\vnu = 0, \ldots, n$, for all  
$|\ell|, |j - j'| \leq N_\vnu$, $ j, j' \in \N^+ \setminus \Splus$, if 
\begin{equation}\label{condizione minimo j j'}
{\rm min}\{j, j' \} \geq  C_0 N_\vnu^{2 (\tau + 1)} \gamma^{-2},
\end{equation}
then 
$ |\omega \cdot \ell + \mu_j^\vnu - \mu_{j'}^{\vnu}| \geq \gamma \langle \ell \rangle^{-\tau} $.
\end{itemize} 

\noindent
{\sc Proof of the claim.} By \eqref{mu j nu}, \eqref{stima rj nu} and recalling also 
\eqref{stima code - 1/2 inizio riducibilita}, one has 
\begin{equation}\label{deep purple 0}
\mu_j^\vnu =  \mathtt m_{\frac12} j^{\frac12} \tanh^{\frac12}(\h j) + \mathfrak{r}_j^\vnu, \quad 
\mathfrak{r}_j^\vnu := r_j + r_j^\vnu\,, \quad 
\sup_{j \in \mathbb S^c} j^{\frac12} |\mathfrak{r}_j^\vnu |^{k_0, \gamma} 
\lesssim_S \e \gamma^{- 2 (M+1)}\,. 
\end{equation}
For all $j, j' \in \N \setminus \{ 0 \}$, one has
\begin{equation} \label{1309.4}
| \sqrt{j \tanh(\h j)} - \sqrt{ j' \tanh(\h j')} | 
\leq \frac{C(\h)}{\min\{ \sqrt{j}, \sqrt{j'} \} } \, |j - j'|.
\end{equation}
Then, using \eqref{1309.4} and that $\omega \in \mathtt{DC}( 2 \gamma, \tau)$, we have, 
for $|j-j'| \leq N_\vnu$, $|\ell| \leq N_\vnu$,  
\begin{align}
|\omega \cdot \ell + \mu_j^\vnu - \mu_{j'}^\vnu| 
& \geq  |\omega \cdot \ell| 
- |\mathtt m_{\frac12}| \frac{C(\h)}{\min\{ \sqrt{j}, \sqrt{j'} \} } \, |j - j'|
- |\mathfrak{r}_j^\vnu| - |\mathfrak{r}_{j'}^\vnu| 
\nonumber \\
& \stackrel{\eqref{stime lambda 1}, \eqref{deep purple 0}}{\geq } 
\frac{2 \gamma}{\langle \ell \rangle^{\tau}} 
- \frac{2 C(\h) N_\vnu }{\min\{ \sqrt{j}, \sqrt{j'} \} } \, 
- \frac{C(S) \e \gamma^{- 2(M+1)}}{{\rm min}\{ \sqrt{j}, \sqrt{j'} \}} 
\stackrel{\eqref{condizione minimo j j'}}{\geq} 
\frac{\gamma}{\langle \ell \rangle^\tau}\,, \nonumber
\end{align}
where the last inequality holds for $C_0$ large enough.
This proves the claim.

Now we prove ${\bf(S3)}_{n+1}$, namely that 
\begin{equation} \label{1309.5}
C(S) N_n^{(\tau + 1) (4 \perd + 1)} \gamma^{-4\perd} \|i_2 - i_1 \|_{s_0 + \mu(\mathtt b)} 
\leq \rho \quad \Longrightarrow \quad
\mathtt \Lambda_{n+1}^{\g}(i_1) \subseteq \mathtt \Lambda_{n+1}^{\g- \rho}(i_2) \, . 
\end{equation}
Let $ \lambda \in \mathtt \Lambda_{n+1}^{\g}(i_1) $.
Definition \eqref{Omega nu + 1 gamma} and \eqref{inclusione cantor riducibilita S4 nu}
with $\vnu = n$ (i.e.\  ${\bf(S3)}_n$) imply that 
$ \mathtt \Lambda_{n+1}^{\g}(i_1) \subseteq \mathtt \Lambda_{n}^{\g}(i_1) \subseteq \mathtt \Lambda_{n}^{\g- \rho}(i_2) $. 
Moreover $ \l \in \mathtt \Lambda_n^{\gamma - \rho}(i_2) \subseteq \mathtt \Lambda_n^{\gamma/2}(i_2)  $ because $\rho \leq \g/2$. 
Thus $\mathtt \Lambda_{n + 1}^\gamma(i_1) \subseteq \mathtt \Lambda_n^{\gamma - \rho}(i_2) \subseteq \mathtt \Lambda_n^{\gamma/2}(i_2) $. 
Hence $\mathtt \Lambda_{n+1}^{\g}(i_1) \subseteq \mathtt \Lambda_n^{\g}(i_1) \cap \mathtt \Lambda_n^{\g/2}(i_2)$, and estimate \eqref{r nu i1 - r nu i2} 
on $|\Delta_{12} r_j^n| = |r_{j}^n(\l, i_2(\l)) - r_{j}^n(\l, i_1(\l))|$ holds 
for any $\lambda \in \mathtt \Lambda_{n+1}^{\g}(i_1)$. 
By the previous claim, since $\omega \in \mathtt{DC}(2 \gamma, \tau)$, 
for all $|\ell|, |j - j'| \leq N_n$ satisfying \eqref{condizione minimo j j'} with $\vnu = n$ 
we have 
$$
|\omega \cdot \ell + \mu_j^n (\lambda, i_2(\lambda)) - \mu_{j'}^n (\lambda, i_2(\lambda)) | 
\geq \frac{\g}{\langle \ell \rangle^\tau}
\geq \frac{\g}{\langle \ell \rangle^\tau j^\perd j'^\perd}
\geq \frac{\gamma - \rho}{\langle \ell \rangle^\tau j^\perd j'^\perd}\,.
$$
It remains to prove that the second Melnikov conditions in \eqref{Omega nu + 1 gamma} with $\vnu = n+1$ also hold for $j,j'$ violating \eqref{condizione minimo j j'}$|_{\vnu = n}$, 
namely that
\begin{align}
|\omega \cdot \ell + \mu_j^n(\lambda, i_2 (\lambda)) - \mu_{j'}^n(\lambda, i_2(\lambda))| 
\geq \frac{\gamma - \rho}{\langle \ell \rangle^\tau j^\perd j'^\perd} \, , 
\ \forall |\ell|, |j - j'| \leq N_n\,, \  
\min \{ j, j' \} \leq C_0 N_n^{2 (\tau + 1)} \gamma^{-2}\,. 
\label{j j' max min j - j'}
\end{align}
The conditions on $j, j'$ in \eqref{j j' max min j - j'} imply that 
\begin{equation}\label{j j' max min j - j'0}
\max \{ j, j'\} = \min \{ j, j' \} + |j - j'| 
\leq C_0 N_n^{2 (\tau + 1)} \gamma^{-2} + N_n 
\leq 2 C_0 N_n^{2 (\tau + 1)} \gamma^{-2}\,.
\end{equation}
Now by \eqref{mu j nu}, \eqref{stima rj nu}, \eqref{1309.4}, recalling 
\eqref{stime lambda 1}, \eqref{stima code - 1/2 inizio riducibilita}, 
\eqref{r nu i1 - r nu i2} and the bound $\e \g^{-2(M+1)} \leq 1$, 
we get 
\begin{align}
 |(\mu_j^n - \mu_{j'}^n)(\l, i_2(\l)) - (\mu_j^n - \mu_{j'}^n)(\l, i_1(\l))| 
& \leq |(\mu_j^{0} - \mu_{j'}^{0})(\l, i_2(\l)) 
- (\mu_j^{0} - \mu_{j'}^{0})(\l, i_1(\l))| \nonumber \\
& \quad + |r_{j}^n(\l, i_2(\l)) - r_{j}^n(\l, i_1(\l))| 
 + |r_{j'}^n(\l, i_2(\l)) - r_{j'}^n(\l, i_1(\l))| 
\nonumber \\
&  \leq 
\frac{C(S) N_n}{{\rm min}\{ \sqrt{j}, \sqrt{j'} \}} \Vert i_2 - i_1 \Vert_{ s_0 + \mu(\mathtt b)} \,.  
\label{legno3} %% ci sarebbe anche una piccolezza \e \g^{-2M} che non stiamo usando, qui.
\end{align}
Since $\lm \in \mathtt{\Lambda}_{n+1}^{\g}(i_1)$, by \eqref{legno3} we have, 
for all $ | \ell | \leq N_n $, $ | j - j' | \leq N_n $,  
\begin{align}
| \omega \cdot \ell + \mu_j^n (i_2) - \mu_{j'}^n (i_2) | 
& \geq | \omega \cdot \ell + \mu_j^n (i_1) - \mu_{j'}^n (i_1) | 
- |(\mu_j^n - \mu_{j'}^n)(i_2) - (\mu_j^n - \mu_{j'}^n)(i_1) | 
\nonumber\\
& \geq 
\frac{\gamma}{\langle \ell \rangle^{\tau} j^{\perd} j'^{ \perd}} 
- \frac{C(S) N_n}{{\rm min}\{ \sqrt{j}, \sqrt{j'} \}}\Vert i_2 - i_1 \Vert_{ s_0 + \mu(\mathtt b)} \nonumber\\
& \geq
\frac{\gamma}{\langle \ell \rangle^{\tau} j^{\perd} j'^{ \perd}} 
- C(S) N_n \Vert i_2 - i_1 \Vert_{ s_0 + \mu(\mathtt b)} 
\geq \frac{\gamma - \rho}{\langle \ell \rangle^{\tau} j^{\perd} j'^{ \perd}} 
\nonumber
\end{align}
provided 
$ C(S) N_n \langle \ell \rangle^\tau j^\perd j'^\perd \|i_2 - i_1 \|_{s_0 + \mu(\mathtt b)} 
\leq \rho $. 
Using that $|\ell| \leq N_n$ and \eqref{j j' max min j - j'0}, 
the above inequality is implied by the inequality assumed in \eqref{1309.5}.
The proof for the second Melnikov conditions for $\om \cdot \ell + \mu_j^n + \mu_{j'}^n$ 
can be carried out similarly (in fact, it is simpler). 
This completes the proof of \eqref{inclusione cantor riducibilita S4 nu} with $\vnu = n + 1$.  
\qed

\subsection{Almost-invertibility  of  $ {\cal L}_\om $}\label{quasi invertibilita}

By \eqref{definizione cal L bot}, $\mL_\om = \mP_\bot \mL_\bot \mP_\bot^{-1}$, 
where ${\cal P}_\bot$ is defined in \eqref{semiconiugio cal L8}, \eqref{Phi 1 Phi 2 proiettate}. 
By \eqref{cal L infinito}, for any $\lambda \in \mathtt \Lambda^\gamma_n$, we have that %  in Theorem \ref{Teorema di riducibilita}, 
$\mL_0 = \mU_n \mL_n \mU_n^{-1}$, where $\mU_n$ is defined in \eqref{defUn},
$\mL_0 = \mL_\bot^\sym$, and $\mL_\bot^\sym = \mL_\bot$ 
on the subspace of functions even in $x$ (see \eqref{2601.3}).
Thus 
\begin{equation}\label{final semi conjugation}
{\cal L}_\omega =  {\cal V}_n  {\cal  L}_n  {\cal V}_n^{- 1},  
\quad {\cal V}_n := {\cal P}_\bot {\cal U}_n .
\end{equation}
By Lemmata \ref{lemma operatore e funzioni dipendenti da parametro}, \ref{A versus |A|}, 
by estimate \eqref{stima Phi infinito}, 
using the smallness condition \eqref{ansatz riducibilita} 
and $\t_2 > \t_1$ (see Theorem \ref{iterazione riducibilita}), 
the operators ${\cal U}_n^{\pm 1}$ satisfy, for all $s_0 \leq s \leq S$, $\| {\cal U}_n^{\pm 1} h\|_s^{k_0, \gamma} \lesssim_S \| h \|_s^{k_0, \gamma} + \| \fracchi_0\|_{s + \mu(\mathtt b)}^{k_0, \gamma} \| h \|_{s_0 }^{k_0, \gamma}$. Therefore, by definition \eqref{final semi conjugation} and recalling \eqref{stima Phi 1 Phi 2 proiettate}, \eqref{relazione mathtt b N}, \eqref{definizione bf c (beta)}, the operators ${\cal V}_n^{\pm 1}$ satisfy, for all $s_0 \leq s \leq S$,
\begin{equation}\label{stime W1 W2}
\| {\cal V}_n^{\pm 1} h \|_s^{k_0, \gamma} \lesssim_S \| h \|_{s + \sigma}^{k_0, \gamma}  + 
\| \fracchi_0 \|_{s + \mu(\mathtt b)}^{k_0, \gamma} \| h \|_{s_0 + \sigma}^{k_0, \gamma} \,,
\end{equation}
for some $\sigma = \sigma(k_0, \tau, \nu) > 0$. 

In order to verify the inversion assumption \eqref{inversion assumption}-\eqref{tame inverse} 
% that is required to construct an approximate inverse 
%(and thus to define the next approximate solution of the Nash-Moser nonlinear iteration), 
we decompose the operator $ {\cal L}_n $ in \eqref{cal L infinito} as 
\begin{equation}\label{decomposizione bf Ln}
{\cal L}_n = {\mathfrak L}_n^{<}  + {\cal R}_n + {\cal R}_n^\bot   
\end{equation}
where 
\be\label{Rn-bot}
{\mathfrak L}_n^{<} := \Pi_{K_n} \big( \Dom {\mathbb I}_\bot + \ii {\cal D}_n \big) \Pi_{K_n} + \Pi_{K_n}^\bot \, , \quad 
{\cal R}_n^\bot := 
\Pi_{K_n}^\bot \big( \Dom {\mathbb I}_\bot + \ii {\cal D}_n \big) \Pi_{K_n}^\bot 
- \Pi_{K_n}^\bot \, , 
\end{equation}
the diagonal operator $ {\cal D}_n $ is defined in \eqref{cal L nu} (with $ \vnu = n $), and $K_n := K_0^{\chi^n}$ is the scale of the nonlinear Nash-Moser iterative scheme.  

\begin{lemma} {\bf (First order Melnikov non-resonance conditions)}\label{lem:first-Mel}
For all $ \lambda = (\om, \h)  $ in 
\begin{equation}\label{prime di melnikov}
 {\mathtt \Lambda}_{n + 1}^{\gamma, I}  :=  {\mathtt \Lambda}_{n + 1}^{\gamma, I} ( i ) := 
\big\{ \lambda \in  \R^\nu \times [\mathtt h_1, \mathtt h_2] : |\omega \cdot \ell  
+  \mu_j^n| \geq  2\gamma j^{\frac12} \langle \ell  \rangle^{- \tau} \,,
\quad \forall | \ell  | \leq K_n\,,\ j \in \N^+ \setminus \Splus \big\},
\end{equation}
the operator $ {\mathfrak L}_n^< $ in \eqref{Rn-bot} 
is invertible and there is an extension of the inverse operator (that we denote in the same way) to the whole $\R^\nu \times [\mathtt h_1, \mathtt h_2]$ satisfying the estimate 
\begin{equation}\label{stima tilde cal Dn}
\| ({\mathfrak L}_n^<)^{- 1} g \|_s^{k_0, \gamma} 
\lesssim_{k_0} \gamma^{- 1} \| g \|_{s + \mu}^{k_0, \gamma}\,, 
\end{equation}
where $\mu = k_0 + \t (k_0 +1)$ is the constant in \eqref{Diophantine-1} with $ k_0 = k + 1 $.
\end{lemma}

\begin{proof}
By \eqref{2017.palma.1}, similarly as in \eqref{2001.1} one has 
$\g^{|\a|} | \pa_\lm^\a( \om \cdot \ell + \mu_j^n ) | 
\lesssim \g \langle \ell \rangle |j|^{\frac12}$ for all $1 \leq |\a| \leq k_0$.
Hence Lemma \ref{lemma:cut-off sd} can be applied to $f(\lm) = \om \cdot \ell + \mu_j^n(\lm)$
with $M = C \g \langle \ell \rangle |j|^{\frac12}$ and 
$\rho = 2\gamma j^{\frac12} \langle \ell  \rangle^{- \tau}$.
Thus, following the proof of Lemma \ref{lemma:WD} with 
$\om \cdot \ell + \mu_j^n(\lm)$ instead of $\om \cdot \ell$, 
we obtain \eqref{stima tilde cal Dn}. 
\end{proof}

Standard smoothing properties imply that 
the operator ${\cal R}_n^\bot $ defined in \eqref{Rn-bot} satisfies, for all $ b  > 0$,   
\begin{equation}\label{stima tilde cal Rn}
\| {\cal R}_n^\bot h \|_{s_0}^{k_0, \gamma} \lesssim K_n^{- b} \| h \|_{s_0 + b + 1}^{k_0, \gamma}\,,\quad \| {\cal R}_n^\bot h\|_s^{k_0, \gamma} \lesssim \| h \|_{s + 1}^{k_0, \gamma} \, .
\end{equation}
By \eqref{final semi conjugation}, \eqref{decomposizione bf Ln}, 
Theorem \ref{Teorema di riducibilita}, Proposition \ref{prop: sintesi linearized}, 
and estimates \eqref{stima tilde cal Dn}, \eqref{stima tilde cal Rn}, \eqref{stime W1 W2}, 
we deduce the following theorem.

\begin{theorem}\label{inversione parziale cal L omega}
{\bf (Almost-invertibility of $ {\cal L}_\om $)}
Assume \eqref{ansatz 0}. Let $ {\mathtt a}, {\mathtt b} $ as in \eqref{alpha beta} and $M$ as in \eqref{relazione mathtt b N}. 
Let $S > s_0$, and assume the smallness condition \eqref{ansatz riducibilita}.
Then for all 
\begin{equation}\label{Melnikov-invert}
(\omega, \h) \in  {\bf \Lambda}_{n + 1}^{\g}  := {\bf \Lambda}_{n + 1}^{\g} (i) 
:= \tLm_{n + 1}^\gamma  \cap  {\mathtt \Lambda}_{n + 1}^{\gamma, I}
\end{equation}
(see \eqref{Cantor set}, \eqref{prime di melnikov}) 
the operator $ {\cal L}_\omega$ defined in \eqref{Lomega def} (see also \eqref{representation Lom}) 
can be decomposed as (cf. \eqref{inversion assumption})
\begin{align}\label{splitting cal L omega}
&  {\cal L}_\omega  = {\cal L}_\omega^{<} + {\cal R}_\omega +  {\cal R}_\omega^\bot \,, \quad  {\cal L}_\omega^{<} := {\cal V}_n {\mathfrak L}_n^{<} {\cal V}_n^{- 1} \,,\quad {\cal R}_\omega := {\cal V}_n {\cal R}_n {\cal V}_n^{- 1}\,,\quad {\cal R}_\omega^\bot := {\cal V}_n {\cal R}_n^\bot {\cal V}_n^{- 1}   
\end{align}
where $ {\cal L}_\omega^{<} $ is invertible and there is an extension of the inverse operator (that we denote in the same way) to the whole $\R^\nu \times [\mathtt h_1, \mathtt h_2]$ satisfying,  for some $\sigma := \sigma(k_0, \tau, \nu) > 0$ and for all $ s_0 \leq s \leq S $, estimates \eqref{stima R omega corsivo}-\eqref{tame inverse}, with $ \mu (\mathtt b) $  defined in \eqref{definizione bf c (beta)}. 
Notice that these latter estimates hold on the whole $\R^\nu \times [\mathtt h_1, \mathtt h_2] $. 
 \end{theorem}
 
This result
%theorem provides the decomposition \eqref{inversion assumption} 
%with estimates \eqref{stima R omega corsivo}-\eqref{tame inverse}. 
% As a consequence, it 
allows to deduce Theorem \ref{thm:stima inverso approssimato}, which is
the key step for a Nash-Moser iterative scheme.

\section{Proof of Theorem \ref{main theorem}}
\label{sec:NM}

%In this section we prove Theorem \ref{main theorem}. 
%It will be a consequence of Theorem \ref{iterazione-non-lineare} below
%where we construct iteratively a sequence of better and better approximate solutions 
%of the equation $ {\cal F} ( i, \alpha) = 0$, 
%with $\mF(i,\a)$ defined in \eqref{operatorF}. 
We consider the finite-dimensional subspaces 
$$
E_n := \Big\{ \fracchi (\vphi ) = (\Theta , I , z) (\vphi) , \ \  
\Theta = \Pi_n \Theta, \ I = \Pi_n I, \ z = \Pi_n z \Big\}
$$
where $ \Pi_n  $ is the projector
\be\label{truncation NM}
\Pi_n := \Pi_{K_n} : z(\ph,x) = \sum_{\ell  \in \Z^\nu,  j \in {\mathbb S}_0^c} z_{\ell, j} e^{\ii (\ell  \cdot \ph + jx)} 
\ \mapsto \Pi_n z(\ph,x) 
:= \sum_{|(\ell ,j)| \leq K_n} 
z_{\ell,  j} e^{\ii (\ell  \cdot \ph + jx)}  
\ee
with $ K_n = K_0^{\chi^n} $ (see \eqref{definizione Kn}) and 
we denote with the same symbol 
$ \Pi_n p(\ph) :=  \sum_{|\ell | \leq K_n}  p_\ell e^{\ii \ell  \cdot \ph} $. 
We define $ \Pi_n^\bot := {\rm Id} - \Pi_n $.  
The projectors $ \Pi_n $, $ \Pi_n^\bot$ satisfy the smoothing properties \eqref{p2-proi}, \eqref{p3-proi} for the weighted Whitney-Sobolev norm $\| \cdot \|_s^{k_0, \gamma}$ 
defined in \eqref{def norm Lip Stein uniform}.

In view of the Nash-Moser Theorem \ref{iterazione-non-lineare} we introduce the following constants: 
\begin{align}
& {\mathtt a}_1 :=  {\rm max}\{6  \sigma_1 + 13, \chi p (\tau + 1) (4\perd + 1) + \chi(\mu(\mathtt b) + 2 \sigma_1) +1 \},  
\qquad 
\mathtt a_2 := \chi^{- 1} \mathtt a_1  - \mu(\mathtt b) - 2 \sigma_1 , 
\label{costanti nash moser}
\\
& \mu_1 :=  3( \mu({\mathtt b}) + 2\sigma_1 ) + 1, 
\qquad {\mathtt b}_1 := {\mathtt a}_1 + \mu({\mathtt b}) +  3 \sigma_1 + 3 + \chi^{-1} \mu_1 ,
\qquad \chi = 3/ 2,
\label{costanti nash moser 1} 
\\
& \sigma_1 := \max \{ \bar \sigma\,, s_0 + 2 k_0 + 5 \},
\qquad 
S := s_0 + \mathtt b_1 
\label{costanti nash moser 2} 
\end{align}
where $\bar \sigma := \bar \sigma(\tau, \nu, k_0) > 0$ is defined in Theorem \ref{thm:stima inverso approssimato}, 
$ s_0 + 2 k_0 + 5 $ is the largest loss of regularity in the estimates of the Hamiltonian vector field $X_P$ in Lemma \ref{lemma quantitativo forma normale}, 
$\mu(\mathtt b)$ is defined in \eqref{definizione bf c (beta)},  
$\mathtt b$ is the constant $ {\mathtt b} := [{\mathtt a}] + 2 \in \N $ 
where $ {\mathtt a} $ is defined in \eqref{alpha beta}. The constants $\mathtt b_1$, $\mu_1$ appear in $({\cal P}3)_n$ of Theorem \ref{iterazione-non-lineare} below: $\mathtt b_1$ gives the maximal Sobolev regularity $S = s_0 + \mathtt b_1$ which has to be controlled along the Nash Moser iteration and $\mu_1$ gives the rate of divergence of the high norms $\| \tilde W_n \|_{s_0 + \mathtt b_1}^{k_0, \gamma}$. The constant $\mathtt a_1$ appears in \eqref{P2n} and gives the rate of convergence of ${\cal F}(\tilde U_n)$ in low norm. 

The exponent $ p $ in \eqref{NnKn} which links the scale $(N_n)_{n \geq 0}$ of the reducibility scheme (Theorem \ref{Teorema di riducibilita}) and the scale $(K_n)_{n \geq 0}$ of the Nash-Moser iteration ($N_n = K_n^p$ ) is required to satisfy   
\be\label{cond-su-p}
p {\mathtt a} > (\chi - 1 ) {\mathtt a}_1 + \chi \sigma_1 = \frac12 {\mathtt a}_1 + \frac32 \sigma_1 \, . 
\ee
By \eqref{alpha beta}, $ {\mathtt a} \geq \chi (\tau + 1)(4 \mathtt d + 1) + 1$. 
Hence, by the definition of $ {\mathtt a}_1 $ in \eqref{costanti nash moser}, there exists 
$ p := p(\tau, \nu, k_0) $ such that  \eqref{cond-su-p} holds. For example we fix $p := 3 (\mu(\mathtt b) + 3 \sigma_1 + 1)/{\mathtt a}$.

Given  
$  W = ( \fracchi, \beta ) $ where
$   \fracchi = \fracchi (\lambda) $ 
is the periodic component of a torus as in \eqref{componente periodica}, and $ \b = \b (\l) \in \R^\nu $
we denote $ \|  W \|_{s}^{k_0, \gamma} := \max\{ \|  \fracchi \|_{s}^{k_0, \gamma} ,  |  \beta |^{k_0, \gamma} \} $, 
where  $ \|  \fracchi \|_{s}^{k_0, \gamma}  $ is defined in \eqref{def:norma-cp}.

\begin{theorem}\label{iterazione-non-lineare} 
{\bf (Nash-Moser)} 
There exist $ \d_0$, $ C_* > 0 $, such that, if
\begin{equation}\label{nash moser smallness condition}  
K_0^{\tau_3} \e \g^{-2M - 3} < \d_0, 
\quad \tau_3 := \max \{ p \tau_2, 2 \sigma_1 + {\mathtt a}_1 + 4 \},  
\quad K_0 := \gamma^{- 1}, 
\quad \gamma:= \e^a, 
\quad 0 < a < \frac{1}{\tau_3 + 2M + 3}\,,
\end{equation}
where the constant $M$ is defined in \eqref{relazione mathtt b N} and $ \tau_2 := \tau_2(\tau, \nu)$ is  defined in Theorem \ref{iterazione riducibilita}, 
then, for all $ n \geq 0 $: 

\begin{itemize}
\item[$({\cal P}1)_{n}$] 
there exists a $k_0$ times differentiable function $\tilde W_n : \R^\nu  \times [\h_1, \h_2] 
\to E_{n -1} \times \R^\nu $, $ \lambda = (\om, \h) \mapsto \tilde W_n (\lambda) 
:=  (\tilde \fracchi_n, \tilde \alpha_n - \om ) $, for  $ n \geq 1$,  
and $\tilde W_0 := 0 $, satisfying 
\begin{equation}\label{ansatz induttivi nell'iterazione}
\| \tilde W_n \|_{s_0 + \mu({\mathtt b}) + \sigma_1}^{k_0, \gamma} \leq C_*   \e  \gamma^{-1}\,. 
\end{equation}
Let $\tilde U_n := U_0 + \tilde W_n$ where $ U_0 := (\vphi,0,0, \om )$.
The difference $\tilde H_n := \tilde U_{n} - \tilde U_{n-1}$, $ n \geq 1 $,  satisfies
\begin{equation}  \label{Hn}
\|\tilde H_1 \|_{s_0 + \mu({\mathtt b}) + \sigma_1}^{k_0, \gamma} \leq	 
C_* \e \gamma^{- 1} \,,\quad \| \tilde H_{n} \|_{ s_0 + \mu({\mathtt b}) + \sigma_1}^{k_0, \gamma} \leq C_* \e \gamma^{-1} K_{n - 1}^{- \mathtt a_2} \,,\quad \forall n \geq 2. 
\end{equation}

\item[$({\cal P}2)_{n}$]   
Setting $ {\tilde \imath}_n := (\vphi, 0, 0) + \tilde \fracchi_n $, we define 
\be\label{def:cal-Gn}
{\cal G}_{0} := \tOm \times [\h_1, \h_2]\,, \quad 
{\cal G}_{n+1} :=  {\cal G}_n \cap {\bf \Lambda}_{n + 1 }^{ \gamma}({\tilde \imath}_n)
\,, \quad n \geq 0 \, , 
\ee
where $  {\bf \Lambda}_{n + 1}^{ \gamma}({\tilde \imath}_n) $ is defined in \eqref{Melnikov-invert}. 
Then, for all $\lambda \in {\cal G}_n$,  setting $ K_{-1} := 1 $, we have 
\be\label{P2n}
\| {\cal F}(\tilde U_n) \|_{ s_{0}}^{k_0, \gamma}  \leq C_* \e K_{n - 1}^{- {\mathtt a}_1} \, .
\ee
\item[$({\cal P}3)_{n}$] \emph{(High norms).} 
$ \| \tilde W_n \|_{ s_{0}+ {\mathtt b}_1}^{k_0, \gamma} 
\leq C_* \e \gamma^{-1}  K_{n - 1}^{\mu_1}$ for all $\lambda \in {\cal G}_n$.
\end{itemize}
\end{theorem}

\begin{proof}
The  proof is the same as Theorem 8.2 in \cite{BertiMontalto}. 
It is based on an iterative Nash-Moser scheme and
uses  the almost-approximate inverse at each approximate quasi-periodic solution provided by
Theorem \ref{thm:stima inverso approssimato}. 
\end{proof}
 
We now complete the proof of Theorem \ref{main theorem}. 
Let $ \gamma = \e^a  $ with $ a \in (0, a_0) $ and $ a_0 := 1 / (2M + 3+ \tau_3 ) $ where $\tau_3$ is defined in \eqref{nash moser smallness condition}.
Then the smallness condition given by the first inequality in \eqref{nash moser smallness condition} holds for $ 0 < \e < \e_0 $ small enough and Theorem \ref{iterazione-non-lineare} applies.   
By \eqref{Hn} the  sequence of functions 
$$ 
{\tilde W}_n = {\tilde U}_n - (\vphi,0,0, \om) := (\tilde \fracchi_n, \tilde \a_n - \omega) =
\big( \tilde \imath _n - (\vphi,0,0), \tilde \a_n - \omega \big) 
$$ 
is a Cauchy sequence in $ \|  \ \|_{s_0}^{k_0, \gamma} $ and then it converges to a function
$
W_\infty := \lim_{n \to + \infty} {\tilde W}_n$. We define
$$
U_\infty := (i_\infty, \a_\infty) = (\vphi,0,0, \om) +  W_\infty \, , \quad W_\infty : \R^\nu \times [\h_1, \h_2] \to H^{s_0}_\vphi  \times H^{s_0}_\vphi \times H^{s_0}_{\vphi, x}
\times \R^\nu\,. 
$$
By \eqref{ansatz induttivi nell'iterazione} and \eqref{Hn} we also deduce that
\begin{equation}\label{U infty - U n}
\|  U_\infty -  U_0 \|_{s_0 + \mu(\mathtt b) + \sigma_1}^{k_0, \gamma} \leq C_* \e \gamma^{- 1} \,, \quad \| U_\infty - {\tilde U}_n \|_{s_0  + \mu({\mathtt b}) + \sigma_1}^{k_0, \gamma} \leq C \e \gamma^{-1} K_{n }^{- \mathtt a_2}\,, \ \  n \geq 1 \, .
\end{equation}
Moreover by Theorem \ref{iterazione-non-lineare}-$({\cal P}2)_n$, we deduce that 
$ {\cal F}(\lambda, U_\infty(\lambda)) = 0 $ for all $ \lambda  $ belonging to 
\be\label{defGinfty}
\bigcap_{n \geq 0} {\cal G}_n = 
\mG_0 \cap \bigcap_{n \geq 1} 
 {\bf \Lambda}_{n}^{\gamma}(\tilde \imath_{n-1}) 
\stackrel{\eqref{Melnikov-invert}}{=} 
\mG_0 \cap \Big[ \bigcap_{n \geq 1}  \tLm_{n}^{\gamma}(\tilde \imath_{n-1}) \Big] \cap 
 \Big[ \bigcap_{n \geq 1}   \mathtt \Lambda_{n}^{\gamma, I}(\tilde \imath_{n-1}) \Big]\, \,,
\ee
where $\mG_0 =  \mathtt \Omega \times [\h_1, \h_2] $ is defined in \eqref{def:cal-Gn}.  
By the first inequality in \eqref{U infty - U n} we deduce \eqref{mappa aep} and \eqref{stima toro finale}.  

It remains to prove that 
the Cantor set $ {\cal C}_\infty^{\gamma} $ in \eqref{Cantor set infinito riccardo} 
is contained in $\bigcap_{n \geq 0} \mG_n$.
We first consider the set 
\begin{equation}\label{cantor finale 1 riccardo}
{\cal  G}_\infty 
:= \mG_0 \cap \Big[ \bigcap_{n \geq 1} \tLm_n^{2 \gamma}( i_\infty) \Big] 
\cap \Big[ \bigcap_{n \geq 1} \mathtt \Lambda_n^{2 \gamma, I}(i_\infty)  \Big]\,.
\end{equation}

\begin{lemma} \label{lemma inclusione cantor riccardo 1}
$ {\cal  G}_\infty  \subseteq  \bigcap_{n \geq 0 } {\cal G}_n $,
where $\mG_n$ is defined in \eqref{def:cal-Gn}.
\end{lemma}

\begin{proof} See Lemma 8.6 of \cite{BertiMontalto}. \end{proof}

Then we define the ``final eigenvalues" 
\begin{equation}\label{autovalori finali riccardo}
\mu_j^\infty := \mu_j^0(i_\infty) + r_j^\infty \,, \quad j \in \N^+ \setminus \Splus \, , 
\end{equation}
where $\mu_j^0(i_\infty) $ are defined in \eqref{op-diago0} 
(with $\mathtt m_{\frac12}, r_j$ depending on $i_\infty$) and 
\begin{equation}\label{resti autovalori finali riccardo}
r_j^\infty := \lim_{n \to + \infty} r_j^n(i_\infty)\,, \quad 
j \in \N^+ \setminus \Splus\,,
\end{equation}
with $r_j^n$ given in Theorem \ref{iterazione riducibilita}-$({\bf S1})_n $.
Note that the sequence $(r_j^n(i_\infty))_{n \in \N}$ is a Cauchy sequence 
in $ | \ |^{k_0, \gamma}$ by \eqref{vicinanza autovalori estesi}. 
As a consequence its limit function $ r_j^\infty (\om, \h) $ is well defined, 
it is $ k_0 $ times differentiable and satisfies 
\be\label{distanza-rnrinfty}
| r_j^\infty - r_j^n(i_\infty)|^{k_0, \gamma} \leq 
C \e \gamma^{- 2(M+1)} |j|^{- 2 \fm} N_{n - 1}^{- {\mathtt a}} \, ,  \ n \geq 0 \, .
\ee 
In particular, since $ r_j^0 (i_\infty) = 0 $, we get  
$ | r_j^\infty |^{k_0, \gamma} \leq C \e \gamma^{- 2(M+1)} |j|^{- 2\fm} $ 
(here $C := C(S, k_0)$, with $S$ fixed in \eqref{costanti nash moser 2}).  
The latter estimate, \eqref{autovalori finali riccardo}, \eqref{op-diago0} 
and \eqref{stima code - 1/2 inizio riducibilita} 
imply \eqref{autovalori infiniti}-\eqref{stime autovalori infiniti} with $ \rin_j^\infty := r_j + r_j^\infty $
and $ \mathtt m_{\frac12}^\infty := \mathtt m_{\frac12} (i_\infty) $.
%By \eqref{autovalori finali riccardo} and \eqref{op-diago0}, 
%the corrections $\rin_j^\infty$ in \eqref{autovalori infiniti}-\eqref{stime autovalori infiniti}
%are given by the sum $r_j + r_j^\infty$.

\begin{lemma} \label{lemma inclusione cantor riccardo 2}
The {\it final Cantor set} $ {\cal C}_\infty^{\gamma} $ in \eqref{Cantor set infinito riccardo} satisfies $ {\cal C}_\infty^\gamma \subseteq {\cal  G}_\infty $, 
where $\mG_\infty$ is defined in \eqref{cantor finale 1 riccardo}. 
\end{lemma}

\begin{proof}
By \eqref{cantor finale 1 riccardo}, we have to prove that 
$ {\cal C}_\infty^\gamma \subseteq \tLm_n^{2 \gamma}(i_\infty) $, $ \forall n \in \N $. 
We argue by induction. 
For $n = 0$ the inclusion is trivial, 
since $ \tLm_0^{2 \gamma}(i_\infty) = \tOm \times [\h_1, \h_2] = \mG_0$. 
Now assume that ${\cal C}_\infty^\gamma \subseteq \tLm_n^{2 \gamma}(i_\infty)$ for some $n \geq 0$. 
For all $ \lambda \in {\cal C}_\infty^\gamma \subseteq  \tLm_n^{2 \gamma}(i_\infty)$, 
by \eqref{mu j nu}, \eqref{autovalori finali riccardo}, \eqref{distanza-rnrinfty},  we get 
\[
|(\mu_j^n- \mu_{j'}^n)(i_\infty) - (\mu_j^\infty - \mu_{j'}^\infty) | 
\leq C \e \gamma^{-2(M+1)} N_{n-1}^{- \mathtt a} \big( j^{- 2 \fm} + j'^{- 2 \fm} \big) 
\]
Therefore, for any $|\ell|, |j - j'| \leq N_n$ with $(\ell, j, j') \neq (0, j, j)$ 
(recall \eqref{Cantor set infinito riccardo}) we have
\begin{align*}
|\omega \cdot \ell + \mu_j^n(i_\infty) - \mu_{j'}^n(i_\infty)| 
& \geq |\omega \cdot \ell + \mu_j^\infty - \mu_{j'}^\infty| 
- C \e \gamma^{- 2(M+1)} N_{n-1}^{- \mathtt a} \big( j^{- 2 \fm} + j'^{- 2 \fm} \big) 
\\
& \geq 4 \gamma  \langle \ell \rangle^{- \tau} j^{-\perd } j'^{- \perd} 
- C \e \gamma^{- 2(M+1)} N_{n-1}^{- \mathtt a} \big( j^{- 2 \fm} + j'^{- 2 \fm} \big) 
\\
& \geq 2 \gamma \langle \ell \rangle^{- \tau} j^{-\perd } j'^{- \perd} 
\end{align*}
provided
$ C \e \gamma^{- 2M - 3} N_{n - 1}^{- \mathtt a} N_n^\tau 
\big( j^{- 2 \fm} + j'^{- 2 \fm} \big) j^\perd j'^\perd \leq 1 $. 
Since $\fm > \perd$ (see \eqref{alpha beta}), 
one has $(j+N_n)^\perd j^{\perd - 2\fm} \lesssim_\perd N_n^\perd$ for all $j \geq 1$. 
Hence, using $|j - j'| \leq N_n$, 
$$
\big( j^{- 2 \fm} + j'^{- 2 \fm} \big) j^\perd j'^\perd
= \frac{j'^\perd}{j^{2\fm - \perd}} \, + \frac{j^\perd}{j'^{2\fm - \perd}} 
\leq \frac{(j+N_n)^\perd}{j^{2\fm - \perd}} \, + \frac{(j'+N_n)^\perd}{j'^{2\fm - \perd}} 
\lesssim_\perd N_n^\perd.
$$
Therefore, for some $C_1 > 0$, one has, for any $ n \geq 0 $,  
$$
C \e \gamma^{- 2M - 3} N_{n - 1}^{- \mathtt a} N_n^\tau 
\big( j^{- 2 \fm} + j'^{- 2 \fm} \big) j^\perd j'^\perd
\leq C_1 \e \gamma^{- 2M - 3} N_{n - 1}^{- \mathtt a} N_n^{\tau + \perd} 
\leq 1
$$
for $\e$ small enough, 
by \eqref{alpha beta}, \eqref{nash moser smallness condition} and because
$ \tau_3 >  p (\tau + \perd) $
(that follows since  $\t_2 > \t_1 + \mathtt a$ where
$\t_2$ has been fixed in Theorem \ref{iterazione riducibilita}).
In conclusion $ {\cal C}_\infty^\gamma \subseteq \tLm_{n + 1}^{2 \gamma}(i_\infty)$
(for  the  
second Melnikov conditions with the $ + $ sign in \eqref{Omega nu + 1 gamma} we apply the
same argument). 
Similarly we prove that $ {\cal C}_\infty^{ \gamma} \subseteq  \mathtt \Lambda_{n}^{2 \gamma, I}(i_\infty)$ for all $n \in \N $. 
\end{proof}

Lemmata \ref{lemma inclusione cantor riccardo 1}, \ref{lemma inclusione cantor riccardo 2} imply $ {\cal C}_\infty^\gamma \subseteq \bigcap_{n\geq 0} {\cal G}_n $, where $\mG_n$ is defined in \eqref{def:cal-Gn}. This concludes the proof of Theorem \ref{main theorem}.

\appendix

\section{Dirichlet-Neumann operator}\label{subDN}

%We collect some fundamental properties of the Dirichlet-Neumann operator $G (\eta) $, 
%defined in \eqref{DN}, which are used in the paper.

Let $ \eta \in {\cal C}^\infty (\T) $. 
It is well-known (see e.g.\ \cite{LannesLivre}, \cite{AlM}, \cite{IP-Mem-2009}) that 
the Dirichlet-Neumann operator is a {\it pseudo-differential} operator of the form 
\be\label{sviluppo Geta}
G(\eta) = G(0)   + \mR_{G}(\eta), 
\qquad \text{where} \quad 
G(0) = |D| \tanh(\mathtt h |D|)
\ee
is the Dirichlet-Neumann operator at the flat surface $ \eta (x ) = 0 $ 
and the  remainder $ \mR_{G}(\eta) $ is in $ OPS^{-\infty} $  and it is $ O(\eta) $-small. 
Note that the profile $ \eta (x) := \eta (\omega, \mathtt h, \vphi, x )$, 
as well as the velocity potential at the free surface
$ \psi (x) := \psi (\omega, \mathtt h, \vphi, x )$, 
may depend on the angles $ \vphi \in \T^\nu $ and the parameters
$ \lambda := (\omega, {\mathtt h}) \in \R^{\nu} \times [{\mathtt h}_1, {\mathtt h}_2 ] $. 
% In Proposition \ref{lemma dirichlet Neumann} we prove formula \eqref{sviluppo Geta} and we provide 
%the quantitative estimate \eqref{estimate DN}.
For simplicity of notation we sometimes omit to write the dependence with respect 
to $ \vphi $ and $ \lambda $. 

In the sequel we use the following notation. Let $X$ and $Y$ be Banach spaces and $B \subset X$ be a bounded open set. We denote by ${\cal C}^1_b(B, Y)$ the space of the ${\cal C}^1$ functions $B \to Y$ bounded and with bounded derivatives.

\begin{proposition}\label{lemma dirichlet Neumann} {\bf (Dirichlet-Neumann)}
Assume that $\partial_\lambda^k \eta(\lambda, \cdot, \cdot)$ is ${\cal C}^\infty$ for all $|k| \leq k_0$. There exists $\delta (s_0, k_0 ) > 0$ such that, if 
\begin{equation}\label{condizione di piccolezza lemma dirichlet neumann}
\| \eta\|_{2 s_0 + 2 k_0 + 1}^{k_0, \gamma} \leq \delta(s_0, k_0) \, , 
\end{equation} 
then the Dirichlet-Neumann operator $ G(\eta) $ may be written as in \eqref{sviluppo Geta} 
where $ {\cal R}_G (\eta ) $ is an integral operator with  $ {\cal C}^\infty $ 
kernel  $ K_G $ 
(see \eqref{integral operator}) 
which satisfies, for all $m, s, \alpha \in  \N$, the estimate 
\be\label{estimate DN}
\norma {\cal R}_G (\eta) 
\norma_{-m, s, \a}^{k_0, \gamma} \leq C(s,m,\a, k_0) \| K_G \|_{{\cal C}^{s+m + \alpha}}^{k_0,\gamma} \leq
 C(s, m, \a, k_0 ) \| \eta \|_{s + 2 s_0 + 2 k_0 + m + \alpha + 3 }^{k_0, \gamma} \, .
\ee
Let $s_1 \geq 2 s_0 + 1$. There exists $\delta(s_1) > 0$  such that  
the map $ \{ \| \eta \|_{s_1 + 6} < \delta(s_1)  \}  \to H^{s_1}(\T^\nu \times \T \times \T) $, $\eta \mapsto K_G ( \eta ) $,  is ${\cal C}^1_b$.
\end{proposition}

The rest of this section is devoted to the proof of Proposition \ref{lemma dirichlet Neumann}.

In order to analyze the Dirichlet-Neumann  operator $ G(\eta ) $
 it is convenient to  transform the boundary value problem   
\eqref{BoundaryPr} (with $ h = {\mathtt h}$) defined in the closure of the free domain $ {\cal D}_\eta = \{(x,y):  - {\mathtt h } < y<\eta(x)\}$
into an elliptic problem in a flat lower   strip
\be\label{fixed-strip}
 \big\{(X,Y) :  - \mathtt h - c  \leq Y \leq  0 \big\} \, , 
\ee
via a  conformal diffeomorphism  (close to the identity for $ \eta $ small) of the form 
\be\label{conf-diffeo}
x = U(X,Y) = X + p(X,Y), \quad y = V(X,Y) = Y + q(X,Y) \, . 
\ee

\begin{remark}
If  \eqref{conf-diffeo} is a conformal map then the system obtained transforming 
\eqref{BoundaryPr}
is simply \eqref{BVP-new} (the Laplace operator and the Neumann boundary conditions 
are transformed into themselves).
\end{remark}

We require that 
$q(X,Y)$ and $p(X,Y)$ are $2\pi$-periodic in $X$, 
%\be\label{periodicity1}
% X \mapsto q(X,Y)\, , \ p(X, Y)  \quad {\rm are } \ 2 \pi{-\rm periodic } \, , 
%\ee
so that \eqref{conf-diffeo}
defines a diffeomorphism between the cylinder $ \T \times [- \mathtt h -c,0] $ 
and $ {\cal D}_\eta $.
The bottom  $ \{ Y = - \mathtt h - c \} $ is transformed in the bottom $ \{ y = - \mathtt h \} $ if
\be\label{bottom-down}
V(X, - \mathtt h - c ) = - \mathtt h \qquad \Leftrightarrow \qquad 
q(X, - \mathtt h -c) = c  \, , \quad \forall X \in \R \, , 
\ee
and the boundary $ \{ Y = 0 \} $ is transformed in the free surface  $ \{ y = \eta (x) \} $ if 
\be\label{bottom-up}
V(X, 0 ) = \eta ( U(X,0)) \qquad \Leftrightarrow \qquad 
q(X,0 ) = \eta ( X + p(X, 0)) \, . 
\ee
The diffeomorphism \eqref{conf-diffeo} is conformal if and only if the map 
$ U(X,Y) + \ii V(X,Y) $ is analytic, 
which amounts to the Cauchy-Riemann equations
$U_X = V_Y$, $U_Y = - V_X$, 
namely $p_X = q_Y$, $p_Y = - q_X$. 
%\be\label{Cauchy-Riemann}
%U_X = V_Y  \, , \ U_Y = - V_X \, , \qquad i.e. \quad p_X = q_Y  \, , \  p_Y = - q_X \, . 
%\ee
The functions $ (U, V) $, i.e. $ (p, q )$, are harmonic conjugate. 
Moreover, \eqref{bottom-down} and the Cauchy-Riemann equations % \eqref{Cauchy-Riemann} 
imply that  
\be\label{BCqY}
U_Y (X, - \mathtt h - c ) = p_Y (X, - \mathtt h -c ) = 0 \, . 
\ee 
%%%The most general function $p$ which is harmonic, namely $\Delta p = 0$, 
%%%and satisfies \eqref{periodicity1} and \eqref{BCqY} is  
%%%$$
%%%p(X,Y) = \b_0 + \sum_{k \neq 0}  \b_k \cosh (|k| (Y+ \mathtt h +c))  e^{\ii k X} 
%%%$$
%%%where $ \b_0 \in \R$, $ \b_k \in \C$, $ k \in \Z \setminus \{ 0 \}$, are fixed by specifying the value of $ p $ at the boundary $ \{ Y = 0 \} $, namely 
%%%\be\label{val:pX0}
%%%p(X,0) = {\mathtt p}(X) =  {\mathtt p}_0 + \sum_{k \neq 0}  {\mathtt p}_k e^{\ii k X} \, . 
%%%\ee
%%%As a consequence, the solution $p(X, Y)$ of % the problem 
%%%\be\label{eq:p}
%%%\Delta p =0  \, , % \quad  p(X + 2 \pi, Y) = p (X, Y) \, , \,  
%%%\quad p(X,0) ={\mathtt p} (X) , 
%%%\quad  p_Y (X, -  \mathtt h  -c ) = 0  \, , 
%%%\quad  2\pi\text{-periodic in}\,X, 
%%%\ee
%%%is  
%%%\be\label{sol:p}
%%%p(X,Y) =  \sum_{k \in \Z } 
%%%{\mathtt p}_k \frac{\cosh (|k| (Y+ \mathtt h +c))}{ \cosh (|k| ( \mathtt h  +c))} \, e^{\ii k X} \,. 
%%%\ee
Given any periodic function 
\be\label{val:pX0}
{\mathtt p}(X) =  {\mathtt p}_0 + \sum_{k \neq 0}  {\mathtt p}_k e^{\ii k X} \,,
\ee
the unique function $p(X, Y)$ that is $2\pi$-periodic in $X$ and 
solves $\Delta p =0$, $p(X,0) ={\mathtt p} (X)$, 
$p_Y (X, -  \mathtt h  -c ) = 0$ 
%%%\be\label{eq:p}
%%%\Delta p =0  \, , 
%%%\quad p(X,0) ={\mathtt p} (X) , 
%%%\quad  p_Y (X, -  \mathtt h  -c ) = 0  \, , 
%%%\quad  2\pi\text{-periodic in}\,X
%%%\ee
is  
\be\label{sol:p}
p(X,Y) =  \sum_{k \in \Z } 
{\mathtt p}_k \frac{\cosh (|k| (Y+ \mathtt h +c))}{ \cosh (|k| ( \mathtt h  +c))} \, e^{\ii k X}.
\ee
The unique function $q(X,Y)$ that is $2\pi$-periodic in $X$ 
and solves $\Delta q = 0$, \eqref{bottom-down} and $p_X = q_Y$, $p_Y = - q_X$ is 
%%%The most general function $q$ which is harmonic, namely $ \Delta q = 0 $,
%%%and satisfies \eqref{periodicity1} and  is  
%%%\be\label{prima-solq}
%%%q(X,Y) = \a_0 + \frac{\a_0-c}{\mathtt h  + c} Y 
%%%+ \sum_{k \neq 0}  \a_k \sinh (|k| (Y+ \mathtt h +c))  e^{\ii k X} \, ,
%%%\ee
%%%where $ \a_0 \in \R$, $ \a_k \in \C $, $ k \in \Z \setminus \{ 0 \} $. 
%%%By \eqref{Cauchy-Riemann},  \eqref{sol:p}, \eqref{prima-solq} we get
%%%$$
%%%\a_0 = c  \, , \quad \a_k = \ii  {\mathtt p}_k  \frac{\sign (k)}{ \cosh (|k|( \mathtt h  + c))} \,,
%%%$$
%%%so that $ q $ is uniquely determined as 
\be\label{sol:q}
q(X,Y) =  c + \sum_{k \neq 0}   \ii  {\mathtt p}_k  \frac{\sign (k)}{ \cosh (|k|( \mathtt h  + c))}  \sinh (|k| (Y+ \mathtt h +c))  e^{\ii k X} \, .
\ee
We still have to impose \eqref{bottom-up}. 
By \eqref{sol:q} we have 
\be\label{qX0}
q(X,0) = c + \sum_{k \neq 0}   \ii \, \sign (k) \tanh ( |k|( \mathtt h +c) )  
{\mathtt p}_k e^{\ii k X} 
= c - {\cal H} \tanh ((\mathtt h +c)|D|) {\mathtt p} (X) 
\ee
where $ {\mathtt p} (X) $ is defined in \eqref{val:pX0} and $ {\cal H }$ is the Hilbert transform  
defined as the Fourier multiplier in \eqref{Hilbert-transf}.
By \eqref{qX0}, since $p(X,0) = {\mathtt p}(X)$, 
condition \eqref{bottom-up} amounts to solve 
\be\label{eq:bordo-alto}
c - {\cal H} \tanh ((\mathtt h +c)|D|) {\mathtt p} (X)  = \eta (X+ {\mathtt p} (X) ) \, . 
\ee

\begin{remark}
If we had required $ c = 0 $ (fixing the strip of the straight domain \eqref{fixed-strip}), 
equation \eqref{eq:bordo-alto} would, in general, have no solution. 
For example, if $ \eta (x) = \eta_0 \neq 0 $, then 
$  - {\cal H} \tanh (\mathtt h |D|) {\mathtt p} (X)  = \eta_0 $ has no solutions because the left hand side has zero average while the 
right hand side has average $ \eta_0 \neq 0 $. 
\end{remark}

Since the range of $ {\cal H} $ are the functions with zero average, 
equation \eqref{eq:bordo-alto} is equivalent to 
\begin{equation} \label{0307.1}
c =  \langle \eta (X+ {\mathtt p} (X) ) \rangle \, , \quad 
- {\cal H} \tanh ((\mathtt h +c)|D|) {\mathtt p} (X)  = \pi_0^\bot \eta (X+ {\mathtt p} (X) ) 
\end{equation}
where $ \langle f \rangle = f_0 = \pi_0 f$ is the average in $X$ of any function $ f $, 
$ \pi_0 $ is defined in \eqref{def pi0}, 
and $ \pi_0^\bot := {\rm Id} - \pi_0$.
We look for a solution $(c(\vphi), \mathtt p(\vphi, X))$, where $\mathtt p$ has zero average in $X$, of the system
\be\label{sist-eq}
c =  \langle \eta (X+ {\mathtt p} (X) ) \rangle \, , \quad 
{\mathtt p} (X)  = \frac{ {\cal H}}{\tanh ((\mathtt h +c)|D|)}  [ \eta (X+ {\mathtt p} (X) )] \, .
\ee
Since ${\cal H}^2 = - \pi_0^\bot$, 
if $\mathtt p$ solves the second equation in \eqref{sist-eq}, 
then $\mathtt p$ also solves the second equation in \eqref{0307.1}.

\begin{lemma}\label{stima costante media DN}
Let $ \eta(\lm, \vphi, x) $ 
satisfy $\partial_\lambda^k \eta(\lambda, \cdot, \cdot) \in {\cal C}^\infty(\T^{\nu + 1})$ 
for all $|k| \leq k_0$. 
There exists $\delta(s_0, k_0) > 0$ such that, if 
$ \| \eta\|_{2 s_0 + k_0 + 2}^{k_0, \gamma} \leq \delta(s_0,k_0) $, then 
there exists a unique ${\cal C}^\infty$ solution $(c(\eta), \mathtt p(\eta))$ 
of system \eqref{sist-eq} satisfying
\begin{equation}\label{stima mathtt p senza c}
\| \mathtt p\|_s^{k_0, \gamma}, \| c \|_s^{k_0, \gamma} 
\lesssim_{s,k_0} \| \eta\|_{s + k_0}^{k_0,  \gamma} \, , 
\quad \forall s \geq s_0\,. 
\end{equation}
Moreover, let $s_1 \geq 2 s_0 + 1$. There exists $\delta(s_1) > 0$  such that 
the map $ \{ \| \eta \|_{s_1 + 2} < \delta(s_1)  \}  \to H^{s_1}_\vphi \times H^{s_1} $, 
$\eta \mapsto (c(\eta), \mathtt p(\eta)) $  is  ${\cal C}^1_b$. 
\end{lemma}

\begin{proof}
We look for a fixed point of the map  
\begin{equation}\label{Phi mathtt p forma compatta}
\Phi( \mathtt p ) := {\cal H} \mathtt f\big( (\mathtt h + c) |D| \big) [ \eta (\cdot+ {\mathtt p} (\cdot) )] \, , 
\qquad \text{where} \quad 
\mathtt f( \xi ) := \frac{1}{ \tanh( \xi )} ,  \quad \xi  \neq 0\,,
\end{equation}
and $c := \langle \eta (X+ {\mathtt p} (X) ) \rangle$. 
We are going to prove that $ \Phi $ is a contraction 
in a ball ${\cal B}_{2 s_0 + 1}( r ) :=  \{ \| \mathtt p \|_{2 s_0 + 1}^{k_0, \gamma} \leq r $, 
$ \langle \mathtt p \rangle = 0 \}$ with radius $ r $ small enough. 
We begin by proving some preliminary estimates. 

The operator ${\cal H} \mathtt f\big( (\mathtt h + c) |D| \big)$ is the Fourier multiplier, acting on the periodic functions,  
with symbol 
\[
- \ii \, \sign(\xi) \chi(\xi) \mathtt f \big( (\mathtt h + c(\lm, \ph)) |\xi| \big) 
=: g( \mathtt h + c(\lm, \ph), \xi), \quad 
\text{where} \ \ 
g(y, \xi) := - \ii \, \sign(\xi) \chi(\xi) \mathtt f(y |\xi|) \quad \forall y > 0,
\]
where the cut-off $\chi(\xi)$ is defined in \eqref{cut off simboli 1}. For all $n\in\N$, there is a constant $C_n(\h_1)>0$ such that $| \pa_y^n g(y, \xi) | \leq C_n(\h_1)$ for all $y\geq\h_1 / 2$, $\xi\in\R$.
%For all $n \in \N$, one has 
%\begin{equation} \label{2806.1 17}
%\pa_y^n g(y, \xi) = - \ii \, \sign(\xi) \chi(\xi) \mathtt f^{(n)}(y |\xi|) |\xi|^n.
%\end{equation}
%For all $x \in \R \setminus \{0\} $, denoting, in short, $T := \tanh(x)$, 
%one has $\mathtt f'(x) = - T^{-2} (1 - T^2)$, 
%$\mathtt f''(x) = 2 T^{-3} (1 - T^2)$, 
%and, by induction, $\mathtt f^{(n)}(x) = P_n(T^2) T^{-n-1} (1 - T^2)$ for all $n \geq 2$,
%where $ P_n $ is a polynomial of degree $ n - 2 $. 
%Since $1 - \tanh^2(x)$ vanishes exponentially as $ x \to + \infty $, 
%for every $\rho > 0$, $n \in \N$, there exists 
%a constant $ C(n,\rho) >  0 $ such that 
%\begin{equation} \label{2806.2 17}
%|\mathtt f^{(n)}(x)| x^n \leq C(n,\rho) \, , \quad \forall x \geq \rho \, .
%\end{equation}
%Since $\chi(\xi) = 0$ for $|\xi| \leq 1/3$, by \eqref{2806.1 17} and \eqref{2806.2 17}
%(with $\rho = \h_1 / 6 $) 
%we deduce that for every $n \in \N$ there exists a constant $C_n(\h_1) >  0 $ such that 
%\begin{equation}\label{stima partial y n g}
%| \pa_y^n g(y, \xi) | 
%\leq C_n(\h_1) \, , \quad 
%\forall y \geq \h_1 / 2 \, , \ \forall \xi \in \R \, . 
%\end{equation}
We consider a smooth extension $\tilde g(y, \xi)$ of $g(y, \xi)$, defined for any $(y, \xi) \in \R \times \R$,  satisfying the same bound as $g$. Now $|c(\lm,\ph)| \leq \| \eta \|_{L^\infty} \leq C \| \eta \|_{s_0}$, 
and therefore $ \h + c(\lm, \ph) \geq \h_1 / 2$ for all $\lm, \ph$ 
if $\| \eta \|_{s_0}$ is sufficiently small.  
Then, by Lemma \ref{Moser norme pesate}, the composition $\tilde g(\mathtt h + c(\lm, \ph), \xi)$ 
satisfies 
\[
\| \tilde g(\mathtt h + c, \xi) \|_s^{k_0,\g} \lesssim_{s,k_0,\h_1,\h_2} 1 + \| c \|_s^{k_0,\g}
\]
uniformly in $\xi \in \R$ (the dependence on $ \h_1, \h_2$ is omitted in the sequel).
As a consequence, we have the following estimates for pseudo-differential norms 
(recall Definition \ref{def:pseudo-norm})  
of the Fourier multiplier in \eqref{Phi mathtt p forma compatta}: for all $ s \geq s_0 $, 
\begin{align}
\label{stima tanh c nel lemma}
& \norma {\cal H} \mathtt f\big( (\mathtt h + c) |D| \big) \norma_{0, s, 0}^{k_0, \gamma}
\,, \ 
\norma {\cal H} |D| \mathtt f' \big( (\mathtt h + c) |D| \big) \norma_{0, s, 0}^{k_0, \gamma} \,
\lesssim_{s, k_0}  1 + \| c \|_{s}^{k_0, \gamma} \, . 
\end{align}
Estimate \eqref{pr-comp1} with $k+1 = k_0$ implies that, 
for $\| \mathtt p \|_{2 s_0 + 1}^{k_0, \gamma} \leq \d(s_0, k_0)$,
the function $ c \equiv c ( \eta,\mathtt p) = \langle \eta (X+ {\mathtt p} (X) ) \rangle $ 
satisfies, for all $  s \geq s_0 $,
 \begin{equation}\label{stima c in termini p eta nel lemma}
\| c \|_s^{k_0, \gamma} \lesssim_{s, k_0} \| \eta\|_{s + k_0}^{k_0, \gamma} + \| \mathtt p\|_s^{k_0, \gamma} \| \eta\|_{s_0 + k_0 + 1}^{k_0, \gamma} \,. 
\end{equation}
Therefore by \eqref{stima tanh c nel lemma}, \eqref{stima c in termini p eta nel lemma} we get, for all $s \geq s_0 $, 
\be\label{stima tanh c nel lemma-def}
\norma {\cal H} \mathtt f\big( (\mathtt h + c) |D| \big) \norma_{0, s, 0}^{k_0, \gamma} , 
\ \norma {\cal H} |D| \mathtt f'\big( (\mathtt h + c) |D| \big) \norma_{0, s, 0}^{k_0, \gamma} 
\lesssim_{s, k_0}  1 + \| \eta\|_{s + k_0}^{k_0, \gamma} + \| \mathtt p\|_s^{k_0, \gamma} 
\| \eta\|_{s_0 + k_0 + 1}^{k_0, \gamma} \,  . 
\ee
Now we prove that $ \Phi $ is a contraction in the ball 
${\cal B}_{2 s_0 + 1}( r ) :=  \{ \| \mathtt p \|_{2 s_0 + 1}^{k_0, \gamma} \leq r$, 
$\langle \mathtt p \rangle = 0 \} $.

\smallskip

\noindent
 {\sc Step 1: Contraction in low norm.}
For any $ \| \mathtt p \|_{2 s_0 + 1}^{k_0, \gamma} \leq r \leq  \delta (s_0, k_0) $,  
by \eqref{interpolazione parametri operatore funzioni (2)}, 
\eqref{stima tanh c nel lemma-def}, \eqref{pr-comp1}, 
and using the bound $ \| \eta\|_{s_0 + k_0 + 1}^{k_0, \gamma} \leq 1 $, 
we have, $ \forall s \geq s_0 $, 
\begin{equation} \label{2906.1}
\| \Phi( \mathtt p ) \|_s^{k_0,\gamma} 
\lesssim_{s, k_0} \| \eta \|_{s+k_0}^{k_0,\g} 
+ \| \eta \|_{s_0 + k_0 + 1}^{k_0,\g} \| \mathtt p \|_s^{k_0,\g}  \, .
\end{equation}
%In particular, \eqref{2906.1} at $s = 2s_0 + 1$ gives
%\begin{equation} \label{2906.2}
%\| \Phi( \mathtt p ) \|_{2s_0 + 1}^{k_0,\gamma} 
%\leq C(s_0,k_0) \big( \| \eta \|_{2s_0 + k_0 + 1}^{k_0,\g} 
%+ \| \eta \|_{s_0 + k_0 + 1}^{k_0,\g} \| \mathtt p \|_{2s_0 + 1}^{k_0,\g} \big).
%\end{equation}
We fix $r := 2 C(s_0,k_0) \| \eta \|_{2s_0 + k_0 + 1}^{k_0,\gamma}$ 
and we assume that $ r \leq 1$.
Then, using \eqref{2906.1} with $s = 2 s_0 + 1$, one deduces that $\Phi $ maps the ball $ {\cal B}_{2 s_0 + 1}( r ) $ into itself.
To prove that $\Phi$ is a contraction in this ball, we estimate its differential 
at any $\mathtt p \in \mB_{2s_0+1}(r)$ in the direction $\tilde{\mathtt p}$,
which is 
\begin{equation} \label{2906.3}
\Phi'(\mathtt p)[\tilde{\mathtt p}] = \mA (\mathtt m \, \tilde{\mathtt p}) \, , 
\end{equation}
where the operator $\mA$ and the function $\mathtt m$ are 
\begin{equation} \label{3006.5}
\mA (h) := \langle h \rangle \mH \mathtt f'((\h + c)|D|) |D| [ \eta(X + \mathtt p(X)) ]
+ \mH \mathtt f((\h + c)|D|) [h],  \quad
\mathtt m := \eta_x(X + \mathtt p(X)) \, .
\end{equation}
To obtain \eqref{2906.3}-\eqref{3006.5}, note that 
$\pa_{\mathtt p}c[\tilde{\mathtt p}] = \langle \mathtt m \tilde{\mathtt p} \rangle$.
By \eqref{pr-comp1}, for all $ s \geq s_0 $, 
\begin{equation} \label{2906.4}
\| \mathtt m \|_s^{k_0,\g} 
\lesssim_{s, k_0} 
\| \eta \|_{s + k_0 + 1}^{k_0,\g} 
+ \| \mathtt p \|_s^{k_0,\g} \| \eta \|_{s_0 + k_0 + 2}^{k_0,\g} \, .
\end{equation}
By  
\eqref{interpolazione parametri operatore funzioni (2)}, % \eqref{estimate composition parameters}, 
\eqref{stima tanh c nel lemma-def},  \eqref{pr-comp1}, 
using the bounds $\| \eta \|_{s_0 + k_0 + 1}^{k_0,\g} \leq 1$ 
and $\| \mathtt p \|_{s_0}^{k_0,\g} \leq 1$, 
we get, for all $ s \geq s_0 $, 
\begin{equation} \label{3006.6}
\norma \mA  \norma_{0,s,0}^{k_0,\g} 
\lesssim_{s, k_0} 	1 + 
\| \eta \|_{s + k_0}^{k_0,\g} 
+ \| \mathtt p \|_s^{k_0,\g} \| \eta \|_{s_0 + k_0 + 1}^{k_0,\g}  \, .
\end{equation}
By  \eqref{2906.3},  % \eqref{interpolazione parametri operatore funzioni (2)}, 
 \eqref{estimate composition parameters},  \eqref{2906.4}, \eqref{3006.6} 
we deduce that, for all $  s \geq s_0 $, 
\begin{equation} \label{3006.1}
\norma \Phi'(\mathtt p) \norma_{0,s,0}^{k_0,\g}
\lesssim_{s, k_0} 
\| \eta \|_{s + k_0 + 1}^{k_0,\g} 
+ \| \mathtt p \|_s^{k_0,\g} \| \eta \|_{s_0 + k_0 + 2}^{k_0,\g}  \, .
\end{equation}
In particular, by \eqref{3006.1} at $s = 2s_0+1$,
and \eqref{interpolazione parametri operatore funzioni (2)}, we get
\begin{equation} \label{3006.2}
\| \Phi'(\mathtt p)[\tilde{\mathtt p}] \, \|_{2s_0+1}^{k_0,\g} 
\leq C(s_0,k_0) \| \eta \|_{2s_0 + k_0 + 2}^{k_0,\g} \| \tilde{\mathtt p} \|_{2s_0+1}^{k_0,\g}
\leq \frac12 \| \tilde{\mathtt p} \|_{2s_0+1}^{k_0,\g} 
\end{equation}
provided $C(s_0,k_0) \| \eta \|_{2s_0 + k_0 + 2}^{k_0,\g} \leq 1/2$.
Thus $\Phi$ is a contraction in the ball ${\cal B}_{2 s_0 + 1}( r )$ and, 
by the contraction mapping theorem, there exists a unique fixed point 
$\mathtt p = \Phi (\mathtt p)$ in  ${\cal B}_{2 s_0 + 1}( r )$.
Moreover, by \eqref{2906.1}, using that ${\mathtt p} = \Phi({\mathtt p})$ there is $C(s_0,k_0)>  0 $ such that if $C(s_0,k_0) \| \eta \|_{s_0 + k_0 + 1}^{k_0,\g} \leq 1/2$
for all $s \in [s_0, 2 s_0 + 1] $, one has $\| \mathtt p \|_{s}^{k_0,\gamma} \lesssim_{s,k_0} \| \eta \|_{s+k_0}^{k_0,\g}$. Using also \eqref{stima c in termini p eta nel lemma} one deduces $\| c \|_{s}^{k_0,\gamma} \lesssim_{s,k_0} \| \eta \|_{s+k_0}^{k_0,\g}$ 
for all $s \in [s_0, 2s_0+1]$. Thus we have proved \eqref{stima mathtt p senza c} for all $s \in [s_0, 2s_0+1]$.
%\[
%\| \mathtt p \|_{s}^{k_0,\gamma} 
%= \| \Phi( \mathtt p ) \|_{s}^{k_0,\gamma} 
%\leq C(s_0,k_0) \| \eta \|_{s+k_0}^{k_0,\g} 
%+ C(s_0,k_0) \| \eta \|_{s_0 + k_0 + 1}^{k_0,\g} \| \mathtt p \|_{s}^{k_0,\g}
%\]
%and, for $C(s_0,k_0) \| \eta \|_{s_0 + k_0 + 1}^{k_0,\g} \leq 1/2$, 
%we deduce  
%the estimate  $\| \mathtt p \|_{s}^{k_0,\gamma} \lesssim_{s,k_0} \| \eta \|_{s+k_0}^{k_0,\g}$ 
%for all $s \in [s_0, 2s_0+1]$. 
%By \eqref{stima c in termini p eta nel lemma}, 
%using that $\| \eta \|_{s_0+k_0+1}^{k_0,\g} \leq 1$, we obtain 
%$\| c \|_{s}^{k_0,\gamma} \lesssim_{s,k_0} \| \eta \|_{s+k_0}^{k_0,\g}$ 
%for all $s \in [s_0, 2s_0+1]$.
%Thus we have proved \eqref{stima mathtt p senza c} for all $s \in [s_0, 2s_0+1]$.

\smallskip

\noindent
{\sc Step 2: regularity.}
Now we prove that $ \mathtt p $ is $ {\cal C}^\infty $ in $ (\vphi, x) $
and we estimate the norm $\| \mathtt p \|_s^{k_0, \gamma}$ as in \eqref{stima mathtt p senza c}
arguing by induction on $s$. 
Assume that, for a given $ s \geq 2 s_0 + 1$, we have already proved that  
\begin{equation}\label{induction hyphothesis mathtt p}
\| \mathtt p\|_s^{k_0, \gamma} \,, \ \| c\|_s^{k_0, \gamma} \,
 \lesssim_{s, k_0} \| \eta\|_{s + k_0}^{k_0, \gamma}\,. 
\end{equation}
We want to prove that \eqref{induction hyphothesis mathtt p} holds for $ s + 1 $. 
We have to estimate 
$\| \mathtt p \|_{s + 1}^{k_0, \gamma} \simeq 
\max \{ \| \mathtt p\|_s^{k_0, \gamma}$, $\| \partial_X \mathtt p\|_s^{k_0, \gamma}$, 
$\| \partial_{\vphi_i} \mathtt p \|_s^{k_0, \gamma}$, $i = 1, \ldots, \nu \}$.
Using the definition \eqref{Phi mathtt p forma compatta} of $\Phi$, 
we derive explicit formulas for the derivatives $\partial_X \mathtt p, \partial_{\vphi_i} \mathtt p$ in terms of $ {\mathtt p}, \eta, \pa_x \eta, \pa_{\vphi_i} \eta $. 
Differentiating the identity $\mathtt p = \Phi(\mathtt p) $ with respect to $ X $ we get  
\begin{align}
\mathtt p_X & = {\cal H} \mathtt f \big( (\mathtt h + c) |D| \big) 
[ \eta_x (X + {\mathtt p} (X)) (1 + \mathtt p_X) ]  
= \Phi'(\mathtt p)[\mathtt p_X] + \mA (\mathtt m)
\label{prima formula mathtt p}
\end{align}
where the operator $\Phi'(\mathtt p)$ is given by \eqref{2906.3}
and $\mA, \mathtt m$ are defined in \eqref{3006.5}
(note that $\langle \eta_x(X + \mathtt p(X)) (1 + \mathtt p_X(X)) \rangle = 0$).
By \eqref{3006.1} at  $ s = s_0$, 
for $ \| \eta \|_{s_0 + k_0 + 2}^{k_0,\g} \leq \d (s_0, k_0)$  small enough,  
condition \eqref{hyp:lemma-2.14}  for $ A = - \Phi' ( {\mathtt p})$ (with $ \a = 0 $) holds. Therefore
the operator $\Id - \Phi'(\mathtt p)$ is invertible and, by \eqref{con:lemma-2.14} (with $ \a = 0$), 
\eqref{induction hyphothesis mathtt p} 
and \eqref{interpolazione parametri operatore funzioni (2)}, 
its inverse satisfies, for all $ s \geq s_0 $, 
\begin{equation} \label{3006.3}
\| (\Id - \Phi'(\mathtt p))^{-1} h \|_s^{k_0,\g} 
\lesssim_{s,k_0} \| h \|_s^{k_0,\g} 
+ \| \eta \|_{s + k_0 + 1}^{k_0,\g} \| h \|_{s_0}^{k_0,\g}  \, .
\end{equation}
By \eqref{prima formula mathtt p}, we deduce that 
$\mathtt p_X = (\Id - \Phi'(\mathtt p))^{-1} \mA (\mathtt m )$.	
By \eqref{interpolazione parametri operatore funzioni (2)}, \eqref{2906.4}-\eqref{3006.6} % \eqref{3006.1} 
and \eqref{induction hyphothesis mathtt p}, we get 
$\| \mA (\mathtt m) \|_s^{k_0,\g} \lesssim_s \| \eta \|_{s + k_0 + 1}^{k_0,\g}$. 
Hence, by \eqref{3006.3}, using $\| \eta \|_{s_0 + k_0 + 2}^{k_0,\g} \leq 1$, we get 
\begin{equation}\label{stima norma s partial x mathtt p}
\| \mathtt p_X \|_s^{k_0, \gamma} \lesssim_{s,k_0} \| \eta \|_{s + k_0 + 1}^{k_0, \gamma}\,.
\end{equation}
We similar arguments
we get
$ \| \partial_{\vphi_i} \mathtt p\|_{s}^{k_0, \gamma} \lesssim_{s,k_0} 
\| \eta\|_{s + k_0 + 1}^{k_0, \gamma} $, $ i = 1, \ldots, \nu $, and 
using  \eqref{induction hyphothesis mathtt p}, 
\eqref{stima norma s partial x mathtt p}, 
we deduce \eqref{induction hyphothesis mathtt p} at $ s + 1 $ for $\mathtt p$. 
By \eqref{stima c in termini p eta nel lemma}, the same estimate holds for $c$, 
and the induction step is proved.
This completes the proof of \eqref{stima mathtt p senza c}.

The fact that the map 
$ \{ \| \eta \|_{s_1 + 2} < \delta(s_1) \} \to H^{s_1}_\vphi \times H^{s_1} $ defined by 
$ \eta \mapsto (c(\eta), {\mathtt p} (\eta)) $
is ${\cal C}^1_b $ follows by the implicit function theorem. 
\end{proof}

Notice that \eqref{condizione di piccolezza lemma dirichlet neumann} implies
the smallness condition of Lemma \ref{stima costante media DN}. 
%We have proved the following:
%\begin{lemma} {\bf (Conformal diffeomorphism)}
%Assume \eqref{condizione di piccolezza lemma dirichlet neumann}.
%Then the  transformation  
%\be
%\begin{aligned}\label{defUV1}
%U( X, Y ) & := X  + \sum_{k \neq 0}  {\mathtt p}_k \frac{\cosh (|k| (Y+\mathtt h +c))}{ \cosh (|k| (\mathtt h  +c))} e^{\ii k X} \\
%V( X, Y )  & := Y + c + \sum_{k \neq 0}   \ii  {\mathtt p}_k  \frac{\sign (k)}{ \cosh (|k|(\mathtt h  + c))}  \sinh (|k| (Y+\mathtt h +c))  e^{\ii k X} 
%\end{aligned}
%\ee
%where  $ c  $ and $ {\mathtt p} $ are  the solutions of \eqref{sist-eq} provided by Lemma \ref{stima costante media DN}, 
% is a conformal diffeomorphism 
%between the cylinder $ \T \times [- {\mathtt h} -c,0] $ and  $ {\cal D}_\eta $.
%The conditions \eqref{bottom-down},  \eqref{bottom-up} hold:
%the bottom  $ \{ Y = - \mathtt h  - c \} $ is transformed into the bottom $ \{ y = - \mathtt h  \} $
%and the boundary $ \{ Y = 0 \} $ is transformed into the free surface  $ \{ y = \eta (x) \} $.  
%\end{lemma}
Now we transform \eqref{BoundaryPr} via the conformal diffeomorphism
\begin{equation*}
\begin{aligned}%\label{defUV1}
U( X, Y ) & := X  + \sum_{k \neq 0}  {\mathtt p}_k \frac{\cosh (|k| (Y+\mathtt h +c))}{ \cosh (|k| (\mathtt h  +c))} e^{\ii k X} \\
V( X, Y )  & := Y + c + \sum_{k \neq 0}   \ii  {\mathtt p}_k  \frac{\sign (k)}{ \cosh (|k|(\mathtt h  + c))}  \sinh (|k| (Y+\mathtt h +c))  e^{\ii k X} 
\end{aligned}
\end{equation*}
where  $ c  $ and $ {\mathtt p} $ are  the solutions of \eqref{sist-eq} provided by Lemma \ref{stima costante media DN}.  Denote 
$ (P u) (X) := u ( X + {\mathtt p} (X)) $. 
The velocity potential  $  \phi( X,Y) := \Phi(U( X,Y),V( X,Y))  $ 
satisfies, using the Cauchy-Riemann equations $U_X = V_Y$, $U_Y = - V_X$ (or equivalently $p_X = q_Y$, $p_Y = - q_X$) and  \eqref{bottom-down}-\eqref{BCqY}, 
\be\label{BVP-new}
\D \phi = 0   \ \text{in} \ \{- \mathtt h - c < Y<0\}\, ,  \quad  
\phi( X,0) = (P \psi)(X) \, , \quad 
 \phi_Y (X, - \mathtt h - c) = 0  \, .
\ee
We calculate explicitly the solution $\phi$ of \eqref{BVP-new}, which is (see \eqref{sol:p}) 
\[
\phi(X,Y) = 
 \sum_{k \in \Z } \widehat{(P \psi)}_k \frac{ \cosh(|k| (Y+\h+c)) }{\cosh(|k|(\h+c))}\, e^{\ii k X} \, ,
\]
where $ \widehat{(P \psi)}_k $ denotes the $ k$-th Fourier coefficient of the periodic function $ P \psi $.
Therefore the Dirichlet-Neumann operator 
in the domain $ \{ - \mathtt h - c \leq Y \leq 0  \} $ at the flat surface $ Y = 0 $  
is given by
\be\label{dirichlet-neumann-diritto}
\phi_Y ( X,0) 
= \sum_{k \neq 0} \widehat{(P \psi)}_k \tanh(|k|(\h+c)) |k| e^{\ii k X}
= |D | \tanh ((\mathtt h +c) |D|) (P \psi)(X) \ . 
\ee

\begin{lemma}\label{G=mH}
$ G(\eta) = \partial_x P^{-1} \mH \, \tanh ((\mathtt h +c) |D|) \, P  $. 
\end{lemma}

\begin{proof}
The proof is the same as the one of Lemma 2.40 in \cite{BertiMontalto}. The only difference is that formula \eqref{dirichlet-neumann-diritto} in the case of infinite depth is given by $\phi_Y ( X,0) = |D| (P \psi)(X)$.
\end{proof}

\noindent
{\sc Proof of Proposition \ref{lemma dirichlet Neumann} concluded.}
By Lemma \ref{G=mH}  we write the Dirichlet-Neumann operator as 
$$
G(\eta) = \pa_x P^{-1} \mH \tanh ((\mathtt h +c)|D|) P  =  |D| \tanh (\mathtt h |D|) + {\cal R}_G (\eta) \, , 
\quad  {\cal R}_G (\eta)  := {\cal R}_G^{(1)}(\eta) + {\cal R}_G^{(2)}(\eta) \, , 
$$
where, using the decomposition \eqref{tangente iperbolica espansione},
\begin{align} \notag
{\cal R}_G^{(1)}(\eta) & :=  \partial_x \big( P^{-1} \mH \tanh ((\mathtt h +c)|D|) P - \mH \tanh ((\mathtt h +c)|D|)  \big) \\
& \ = \partial_x (P^{- 1 } {\cal H} P - {\cal H}) + 
\partial_x (P^{- 1} \mH {\rm Op}(r_{\mathtt h  + c}) P - {\cal H} {\rm Op}(r_{\mathtt h  + c}))\,.
 \label{rest1}
\end{align}
The second term ${\cal R}_G^{(2)}(\eta)$ is
\begin{align}
{\cal R}_G^{(2)}(\eta) & :=  \pa_x \mH \big( \tanh ((\mathtt h +c)|D|) - \tanh (\mathtt h |D|) \big) 
= \partial_x {\cal H}{\rm Op}(r_{\mathtt h  + c} - r_{\mathtt h}) = 
c\, \partial_x {\cal H} {\rm Op}(\breve r_{\mathtt h , c}) \in OPS^{- \infty}    \, ,\label{rest2}
\end{align}
where
\begin{equation*}%\label{volpe 0}
r_{\mathtt h + c}(\xi) - r_{\mathtt h}(\xi) = 
\breve r_{\mathtt h , c}(\xi)\, c \,, \qquad 
\breve r_{\mathtt h , c}(\xi) := 
2 |\xi| \chi(\xi) \int_0^1 \frac{2 \exp\{ 2(\h + t c) |\xi| \chi(\xi) \} }{(1 + 
\exp\{ 2(\h + t c) |\xi| \chi(\xi) \} )^2 } \, d t \, \in S^{- \infty} \, . 
\end{equation*}
Estimate \eqref{estimate DN} directly follows estimating \eqref{rest1} and \eqref{rest2} by Lemmata \ref{coniugio Hilbert}, \ref{lemma nucleo Fourier multiplier OPS - infty}, and using Lemma \ref{stima costante media DN}.
The differentiablility of the map $ \{ \| \eta \|_{s_1 + 6} < \delta(s_1)  \}  \to H^{s_1}(\T^\nu \times \T \times \T) $, $\eta \mapsto K_G ( \eta ) $ follows by the differentiability of the map $ \{ \| \eta \|_{s_1 + 2} < \delta(s_1)  \}  \to H^{s_1}_\vphi \times H^{s_1} $, 
$\eta \mapsto (c(\eta), \mathtt p(\eta)) $ proved in Lemma \ref{stima costante media DN}.  

\section{Whitney differentiable functions} \label{sec:U}

 The following definition is 
the one in Section 2.3, Chapter VI of \cite{Stein}, for Banach-valued functions.

\begin{definition}{\bf (Whitney differentiable functions)}
\label{def:Lip F}
Let $F$ be a closed subset of $\R^n$, $n \geq 1$. 
Let $Y$ be a Banach space.
Let $k \geq 0$ be an integer, and $k < \rho \leq k+1$. 
We say that a function $f : F \to Y$ belongs to $\Lip(\rho,F,Y)$ if there exist 
functions $f^{(j)} : F \to Y$, $j \in \N^n$, $0 \leq |j| \leq k$,
with $f^{(0)} = f$, and a constant $M > 0$ such that if $R_j(x,y)$ is defined by 
\begin{equation} \label{16 Stein}
f^{(j)}(x) = \sum_{\ell \in \N^n : |j+\ell| \leq k} \frac{1}{\ell!} \, f^{(j+\ell)}(y) \, (x-y)^\ell
+ R_j(x,y), \quad x,y \in F, 
\end{equation} 
then 
\begin{equation} \label{17 Stein}
\| f^{(j)}(x) \|_Y \leq M, \quad 
\| R_j(x,y) \|_Y \leq M |x-y|^{\rho - |j|} \, , \quad 
\forall x,y \in F, \ |j| \leq k \, .
\end{equation} 
An element of $\Lip(\rho,F,Y) $ 
is  in fact the collection $\{ f^{(j)} : |j| \leq k \}$. 
The norm of $ f \in \Lip(\rho,F,Y)$ is defined as the smallest $M$ 
for which the inequality \eqref{17 Stein} holds, namely
\begin{equation} \label{def norm Lip Stein}
\| f \|_{\Lip(\rho,F,Y)} := \inf \{ M > 0 : \text{\eqref{17 Stein} holds} \} \, .
\end{equation}
If $ F = \R^n $ by  $ \Lip(\rho, \R^n,Y) $  we shall mean the linear space of the functions $ f = f^{(0)} $ for which
there exist $ f^{(j)} = \pa_{x}^j f $, $ |j| \leq k $,  satisfying  \eqref{17 Stein}.
\end{definition}

Notice that, if $ F = \R^n $,  
the $ f^{(j)} $, $ | j | \geq 1 $,
 are uniquely determined by $ f^{(0)} $ (which is not the case  for a general $ F $ with for example isolated points).

In the case $F = \R^n$, $\rho = k+1$ and $ Y $ is  a Hilbert space, 
the space $\Lip(k+1,\R^n,Y)$ is isomorphic to the Sobolev space 
$W^{k+1,\infty}(\R^n,Y)$, with equivalent norms
\begin{equation} \label{equiv Sob Lip}
C_1 \| f \|_{W^{k+1,\infty}(\R^n,Y)}
\leq \| f \|_{\Lip(k+1,\R^n,Y)} 
\leq C_2 \| f \|_{W^{k+1,\infty}(\R^n,Y)}
\end{equation}
where $C_1, C_2$ depend only on $k,n$.
For $Y = \C$ this isomorphism is classical, see e.g. \cite{Stein}, 
and it is based on the Rademacher theorem concerning the a.e.\ differentiability of Lipschitz functions, and the fundamental theorem of calculus for the Lebesgue integral. 
Such a property may fail 
for a Banach valued function, 
but it holds for a Hilbert space, see Chapter 5 of \cite{BenyLind}
(more in general it holds if $ Y $ is reflexive or it satisfies the Radon-Nykodim property).

The following key result provides an extension of a Whitney differentiable function $f$ defined on a closed subset $F$ of $\R^n$ to the whole domain $\R^n$, with equivalent norm.

\begin{theorem}{\bf (Whitney extension Theorem)}  \label{thm:WET}
Let $F$ be a closed subset of $\R^n$, $n \geq 1$,  
$Y$ a Banach space, 
$k \geq 0$ an integer, and $k < \rho \leq k+1$. 
There exists a linear continuous extension operator $\mE_k : \Lip(\rho,F,Y) \to \Lip(\rho, \R^n,Y)$ 
which gives an extension $\mE_k f \in \Lip(\rho, \R^n,Y)$ to any $f \in \Lip(\rho,F,Y)$. 
The norm of $\mE_k$ has a bound independent of $F$, 
\begin{equation} \label{bound Ek Stein}
\| \mE_k f \|_{\Lip(\rho,\R^n,Y)} \leq C \| f \|_{\Lip(\rho,F,Y)} \, , \quad \forall f \in \Lip(\rho,F,Y) \, ,
\end{equation}
where $C$ depends only on $n,k $ (and not on $ F , Y $). 
\end{theorem}

\begin{proof}
This is Theorem 4 in Section 2.3, Chapter VI of \cite{Stein}. The proof in \cite{Stein}
is written for real-valued functions $f : F \to \R$, but it also holds for functions $f : F \to Y$ for any (real or complex) Banach space $Y$, \emph{with no change}. 
The extension operator $\mE_k$ is defined in formula (18) in Section 2.3, Chapter VI of \cite{Stein},
and it is linear by construction.
\end{proof}

Clearly, since $\mE_k f$ is an extension of $f$, one has 
\begin{equation} \label{basso Stein}
\| f \|_{\Lip(\rho,F,Y)} \leq \| \mE_k f \|_{\Lip(\rho,\R^n,Y)} 
\leq C \| f \|_{\Lip(\rho,F,Y)} \, .
\end{equation}
In order to extend a function defined on a closed set $F \subset \R^n$ 
with values in scales of Banach spaces (like $H^s(\T^{\nu+1})$),
we observe that the extension provided by Theorem \ref{thm:WET} 
does not depend on the index of the space (namely $s$).

\begin{lemma}  \label{lemma:2702.1}
Let $F$ be a closed subset of $\R^n$, $n \geq 1$, 
let $k \geq 0$ be an integer, and $k < \rho \leq k+1$.
Let $Y \subseteq Z$ be two Banach spaces. 
Then $\Lip(\rho,F,Y) \subseteq \Lip(\rho,F,Z)$. 
The two extension operators 
$\mE_k^{(Z)} : \Lip(\rho,F,Z) \to \Lip(\rho, \R^n,Z)$ 
and $\mE_k^{(Y)} : \Lip(\rho,F,Y) \to \Lip(\rho, \R^n,Y)$
provided by Theorem \ref{thm:WET} satisfy
\[
\mE_k^{(Z)} f = \mE_k^{(Y)} f \quad \forall f \in \Lip(\rho,F,Y) \, .
\]
As a consequence, we simply denote $\mE_k$ the extension operator.
\end{lemma}

\begin{proof}
The lemma follows directly by the construction of the extension operator $\mE_k$ in 
formula (18) in Section 2.3, Chapter VI of \cite{Stein}, which
% The explicit construction 
relies on a nontrivial decomposition in cubes of the domain $\R^n$ only.
\end{proof}

Thanks to the equivalence \eqref{basso Stein}, Lemma \ref{lemma:2702.1}, and 
\eqref{equiv Sob Lip} which holds for functions valued in $ H^s $, 
classical interpolation and tame estimates for products, projections, and composition of Sobolev 
functions can be easily extended to Whitney differentiable functions.

The difference between the Whitney-Sobolev norm introduced in Definition \ref{def:Lip F uniform}
and the norm in Definition \ref{def:Lip F} (for $ \rho = k+1 $, $ n = \nu + 1 $, 
and target space $ Y = H^s ( \T^{\nu+1},\C) $)
is the weight $ \g \in (0,1]$. 
Observe that the introduction of this weight simply amounts to the following rescaling $\mR_\g$: 
given $u = (u^{(j)})_{|j| \leq k}$, we define $\mR_\g u = U = (U^{(j)})_{|j| \leq k}$ as
\be \label{resca}
\lm = \g \mu, \qquad 
\g^{|j|} u^{(j)}(\lm) 
= \g^{|j|} u^{(j)}(\g \mu) 
=: U^{(j)}(\mu) = U^{(j)}(\g^{-1} \lm), \qquad 
U := \mR_\g u \, .
\ee
Thus $u \in \Lip(k+1,F,s,\g)$ if and only if $U \in \Lip(k+1, \g^{-1} F ,s,1)$, with 
\be\label{rescaling:gamma-1}
\| u \|_{s,F}^{k+1,\g} = \| U  \|_{s,\g^{-1} F}^{k+1,1} \,.
\ee
Under the rescaling $\mR_\g$, \eqref{equiv Sob Lip} gives the equivalence of the two norms 
\be \label{0203.1}
\| f \|_{W^{k+1,\infty, \gamma}(\R^{\nu+1}, H^s)} := 
\sum_{|\a| \leq k+1} \g^{|\a|} \| \pa_\lm^\a f \|_{L^\infty(\R^{\nu+1}, H^s)} 
\sim_{\nu,k} \| f \|_{s, \R^{\nu+1}}^\kug \,.
\ee
Moreover, given $u \in \Lip(k+1,F,s,\g)$, its extension 
\begin{equation} \label{Wg}
\tilde u := \mR_\g^{-1} \mE_k \mR_\g u \in \Lip(k+1, \R^{\nu+1}, s, \g) 
\quad \text{satisfies} \quad 
\| u \|^{k+1,\g}_{s,F} 
\sim_{\nu,k} \| \tilde u \|_{s, \R^{\nu+1}}^\kug \,.
\end{equation}

\noindent
\textbf{Proof of Lemma \ref{lemma:smoothing}}.
Inequalities \eqref{p2-proi}-\eqref{p3-proi} follow by 
\begin{align*}
 (\Pi_N u)^{(j)}(\lm) = \Pi_N [u^{(j)}(\lm)], \quad 
 R_j^{(\Pi_N u)}(\lm, \lm_0) = \Pi_N [ R_j^{(u)}(\lm, \lm_0) ], 
\end{align*}
for all $0 \leq |j|  \leq k$, $\lm, \lm_0 \in F$, and the usual smoothing estimates
$ \| \Pi_N f \|_{s}  \leq N^\alpha \| f \|_{s-\alpha} $ and 
$ \| \Pi_N^\bot f \|_{s}  
\leq N^{-\alpha} \| f \|_{s + \alpha} $ for Sobolev functions.
\qed

\smallskip

\noindent
\textbf{Proof of Lemma \ref{lemma:interpolation}}.
Inequality \eqref{2202.3} follows from the classical interpolation inequality 
$\| u \|_s \leq \| u \|_{s_0}^\theta \| u \|_{s_1}^{1-\theta}$, 
$s = \theta s_0 + (1-\theta) s_1$ for Sobolev functions,
and from the Definition \ref{def:Lip F uniform} of Whitney-Sobolev norms, 
since 
\begin{align*}
\g^{|j|} \| u^{(j)}(\lm) \|_s 
& \leq (\g^{|j|} \| u^{(j)}(\lm) \|_{s_0})^\theta
(\g^{|j|} \| u^{(j)}(\lm) \|_{s_1})^{1-\theta}
\leq (\| u \|_{s_0,F}^{k+1,\g})^\theta (\| u \|_{s_1,F}^{k+1,\g})^{1-\theta},
\\
\g^{k+1} \| R_j(\lm,\lm_0) \|_s 
& \leq (\g^{k+1} \| R_j(\lm,\lm_0) \|_{s_0})^\theta
(\g^{k+1} \| R_j(\lm,\lm_0) \|_{s_1})^{1-\theta}
\leq (\| u \|_{s_0,F}^{k+1,\g})^\theta (\| u \|_{s_1,F}^{k+1,\g})^{1-\theta}
|\lm - \lm_0|^{k+1-|j|}.
\end{align*}
Inequality \eqref{2202.2} follows from \eqref{2202.3} by using the asymmetric Young inequality
(like in Lemma 2.2 in \cite{BertiMontalto}).
\qed

\smallskip

\noindent
\textbf{Proof of Lemma \ref{lemma:LS norms}}. 
By \eqref{0203.1}-\eqref{Wg}, the lemma follows from the corresponding inequalities 
for functions in $W^{k+1,\infty,\g}(\R^{\nu+1}, H^s)$, which are proved, for instance, 
in \cite{BertiMontalto} (formula (2.72), Lemma 2.30).
\qed

\smallskip

For any $\rho > 0$, we define the  $ {\cal C}^\infty $ function $ h_\rho : \R \to \R $,  
\be\label{def-gl}
h_\rho (y) := \frac{\chi_\rho(y)}{y} = \frac{\chi(y \rho^{-1})}{y}  \, , 
\quad \forall y \in \R \setminus \{ 0 \}, 
\quad h_\rho(0) := 0  \, , 
\ee
where $\chi$ is the cut-off function introduced in \eqref{cut off simboli 1}, 
and $\chi_\rho(y) := \chi (y / \rho)$.
Notice that the function $ h_\rho $ 
is of class $ {\cal C}^\infty $ because  $ h_\rho (y) = 0 $ for $ |y| \leq \rho / 3 $.
Moreover 
by the properties of $ \chi $ in  \eqref{cut off simboli 1}
we have 
\begin{equation} \label{fatti d'acca rho}
h_\rho (y) = \frac{1}{ y} \, , \ \forall  |y| \geq \frac{2 \rho}{ 3} \, , \qquad 
| h_\rho(y) | \leq \frac{3}{\rho} \, , \ \forall y \in \R \, . 
\end{equation}
To prove Lemma \ref{lemma:WD}, we use the following preliminary lemma. 

\begin{lemma} 
\label{lemma:cut-off sd}
Let $f : \R^{\nu+1} \to \R$ and  $\rho > 0$. 
Then the function 
\begin{equation} \label{2017.1805.1}
g(\lambda) := h_\rho (f(\lambda)), \quad \forall \lambda \in \R^{\nu+1} \, , 
\end{equation}
where $h_\rho$ is defined in \eqref{def-gl},
coincides with $ 1/ f (\lambda) $ 
on the set $ F := \{ \lm \in \R^{\nu+1} : |f(\lm)| \geq \rho \}$. 

If the function $f $ is in $ W^{k+1,\infty}(\R^{\nu+1},\R)$, with estimates
\begin{equation} \label{0103.1}
\g^{|\a|} | \pa_\lm^\a f(\lm) | \leq M \, ,  \quad  \forall \a \in \N^{\nu+1} \, , \ \ 
1 \leq |\a| \leq k+1 \, , 
\end{equation} 
for some $M \geq \rho$, 
then the function $g $ is in $ W^{k+1,\infty}(\R^{\nu+1},\R)$ and 
\begin{equation} \label{0103.2}
\g^{|\a|} | \pa_\lm^\a g(\lm) | 
\leq C_k \frac{M^{k+1}}{\rho^{k+2}}\, , 
\quad \forall \a \in \N^{\nu+1}, \ \ 
0 \leq |\a| \leq k+1.
\end{equation} 
\end{lemma}

\begin{proof}
Formula \eqref{0103.2} for $\a = 0$ holds by \eqref{fatti d'acca rho}. For $|\alpha| \geq 1$, we use the Fa\`a di Bruno formula and \eqref{0103.1}. 
\end{proof}

\noindent 
{\bf Proof of Lemma \ref{lemma:WD}}.
%Let $u \in \Lip(k+1, F, s+\mu, \g)$, 
% where $u$ is the collection $u = (u^{(j)})_{|j| \leq k}$
% (recall Definition \ref{def:Lip F uniform}). 
%and let $\tilde u % = (\tilde u^{(j)})_{|j| \leq k} 
%\in \Lip(k+1, \R^{\nu+1}, s+\mu, \g)$
%be its extension defined in \eqref{Wg}.   
%Note that $\tilde u^{(j)} = \pa_x^j \tilde u$ (see the first observation after Definition \ref{def:Lip F uniform}).
The function $(\ompaph)^{-1}_{ext} u$ defined in \eqref{def ompaph-1 ext} is
\[
\big( (\ompaph)^{-1}_{ext} u \big) (\lm,\ph,x) 
= - \ii \sum_{(\ell,j) \in \Z^{\nu+1}} 
g_\ell(\lm) u_{\ell, j}(\lm) \, e^{\ii (\ell \cdot \vphi + j x )} \, ,
\]
where $g_\ell(\lm) = h_\rho( \om \cdot \ell )$ in \eqref{2017.1805.1}
with $\rho = \g \langle \ell \rangle^{-\t}$ and $f(\lm) = \om \cdot \ell$.
The function $f(\lm)$ satisfies \eqref{0103.1} 
with $M = \g |\ell|$. 
Hence $g_\ell(\lm)$ satisfies \eqref{0103.2}, namely
\begin{equation} \label{0103.6}
\g^{|\a|} | \pa_\lm^\a g_\ell(\lm) | \leq C_k \g^{-1} \langle \ell \rangle^\mu 
\quad \forall \a \in \N^{\nu+1}, \ 0 \leq |\a| \leq k+1,
\end{equation}
where $\mu = k + 1 + (k+2) \t$ is defined in \eqref{Diophantine-1}. By the product rule and using \eqref{0103.6}, we deduce 
$ \g^{|\a|} \| \pa_\lm^{\a}((\ompaph)^{-1}_{ext} u) (\lm) \|_s  
\leq C_k \g^{-1} \| u \|_{s+\mu, \R^{\nu+1}}^{k+1,\g} $
and therefore \eqref{2802.2}.
The proof is concluded by observing that the restriction of $(\ompaph)^{-1}_{ext} u$ 
to $F$ gives $(\ompaph)^{-1} u$ as defined in \eqref{def:ompaph}, 
and \eqref{Diophantine-1} follows by \eqref{Wg}.
\qed

\medskip

\noindent 
{\bf Proof of Lemma \ref{Moser norme pesate}}. 
Given $u\in\Lip(k+1,F,s,\g)$, we consider its extension 
$ \tilde u\in\Lip(k+1,\R^{\nu+1},s,\g)$ provided by \eqref{Wg}. 
Then we observe that the composition $\mathtt f(\tilde u)$ is an extension of $\mathtt f(u)$, and therefore one has the inequality 
$\| \mathtt f(u) \|_{s,F}^{k+1,\g} 
\leq \| \mathtt f(\tilde u) \|_{s,\R^{\nu+1}}^{k+1,\g} 
\sim \| \mathtt f(\tilde u) \|_{W^{k+1, \infty, \g}(\R^{\nu+1},H^s)}$ by \eqref{0203.1}. 
Then \eqref{0811.10} follows by the Moser composition estimates 
for $\| \ \|_{s, \R^{\nu+1}}^{k+1,\g}$ 
(see for instance Lemma 2.31 in \cite{BertiMontalto}), 
together with the equivalence of the norms in \eqref{0203.1}-\eqref{Wg}. 
\qed

\section{A Nash-Moser-H\"ormander implicit function theorem}
\label{sec:NMH}

%In this section we state the Nash-Moser-H\"ormander theorem of \cite{Baldi-Haus}, 
%which we apply in Section \ref{sec: change-transport equation} as a black box 
%to prove Theorem \ref{thm:nonlinear transport}.

Let $(E_a)_{a \geq 0}$ be a decreasing family of Banach spaces with continuous injections  
$E_b \hookrightarrow E_a$, 
\begin{equation} \label{S0}
\| u \|_{E_a} \leq \| u \|_{E_b} \quad \text{for} \  a \leq b.	
\end{equation}
Set $E_\infty = \cap_{a\geq 0} E_a$ with the weakest topology making the 
injections $E_\infty \hookrightarrow E_a$ continuous. 
Assume that there exist linear smoothing operators $S_j : E_0 \to E_\infty$ for $j = 0,1,\ldots$, 
satisfying the following inequalities, 
with constants $C$ bounded when $a$ and $b$ are bounded, 
and independent of $j$,
\begin{alignat}{2}
\label{S1} 
\| S_j u \|_{E_a} 
& \leq C \| u \|_{E_a} 
&& \text{for all} \ a;
\\
\label{S2} 
\| S_j u \|_{E_b} 
& \leq C 2^{j(b-a)} \| S_j u \|_{E_a} 
&& \text{if} \ a<b; 
\\
\label{S3} 
\| u - S_j u \|_{E_b} 
& \leq C 2^{-j(a-b)} \| u - S_j u \|_{E_a} 
&& \text{if} \ a>b; 
\\ 
\label{S4} 
\| (S_{j+1} - S_j) u \|_{E_b} 
& \leq C 2^{j(b-a)} \| (S_{j+1} - S_j) u \|_{E_a} 
\quad && \text{for all $a,b$.}
\end{alignat}
Set 
\begin{equation}  \label{new.24}
R_0 u := S_1 u, \qquad 
R_j u := (S_{j+1} - S_j) u, \quad j \geq 1.
\end{equation}
We also assume that 
\begin{equation} \label{2705.4}
\| u \|_{E_a}^2 \leq C \sum_{j=0}^\infty \| R_j u \|_{E_a}^2	\quad \forall a \geq 0,
\end{equation}
with $C$ bounded for $a$ bounded 
(a sort of ``orthogonality property'' of the smoothing operators).

Suppose that we have another family $F_a$ of decreasing Banach spaces with smoothing operators having the same properties as above. We use the same notation also for the smoothing operators. 

\begin{theorem}[\cite{Baldi-Haus}] \label{thm:NMH}
{\bf (Existence)} Let $a_1, a_2, \a, \b, a_0, \mu$ be real numbers with 
\begin{equation} \label{ineq 2016}
0 \leq a_0 \leq \mu \leq a_1, \qquad 
a_1 + \frac{\b}{2} \, < \a < a_1 + \b , \qquad 
2\a < a_1 + a_2. 
\end{equation}
Let $U$ be a convex neighborhood of $0$ in $E_\mu$. 
Let $\Phi$ be a map from $U$ to $F_0$ such that $\Phi : U \cap E_{a+\mu} \to F_a$ 
is of class $C^2$ for all $a \in [0, a_2 - \mu]$, with 
\begin{align} 
\|\Phi''(u)[v,w] \|_{F_a} 
& \leq M_1(a) \big( \| v \|_{E_{a+\mu}} \| w \|_{E_{a_0}} 
+ \| v \|_{E_{a_0}} \| w \|_{E_{a+\mu}} \big) 
\notag \\ & \quad 
+ \{ M_2(a) \| u \|_{E_{a+\mu}} + M_3(a) \} \| v \|_{E_{a_0}} \| w \|_{E_{a_0}}
\label{Phi sec}
\end{align}
for all $u \in U \cap E_{a+\mu}$, $v,w \in E_{a+\mu}$,
where $M_i : [0, a_2 - \mu] \to \R$, $i = 1,2,3$, are positive, increasing functions. 
Assume that $\Phi'(v)$, for $v \in E_\infty \cap U$ 
belonging to some ball $\| v \|_{E_{a_1}} \leq \d_1$,
has a right inverse $\Psi(v)$ mapping $F_\infty$ to $E_{a_2}$, and that
\begin{equation}  \label{tame in NM}
\| \Psi(v)g \|_{E_a} \leq 
L_1(a) \|g\|_{F_{a + \b - \a}} + 
\{ L_2(a) \| v \|_{E_{a + \b}} + L_3(a) \} \| g \|_{F_0}
\quad \forall a \in [a_1, a_2],
\end{equation}
where $L_i : [a_1, a_2] \to \R$, $i = 1,2,3$, 
are positive, increasing functions.

Then for all $A > 0$ there exists $\d > 0$ such that, 
for every $g \in F_\b$ satisfying
\begin{equation} \label{2705.1}
\sum_{j=0}^\infty \| R_j g \|_{F_\b}^2 \leq A^2 \| g \|_{F_\b}^2, \quad
\| g \|_{F_\b} \leq \d,
\end{equation}
there exists $u \in E_\a$ solving $\Phi(u) = \Phi(0) + g$.
The solution $u$ satisfies 
\begin{equation} \label{qui.01}
\| u \|_{E_\a} \leq C L_{123}(a_2) (1 + A) \| g \|_{F_\b}, 
\end{equation}
where $L_{123} = L_1 + L_2 + L_3$ 
and $C$ is a constant depending on $a_1, a_2, \a, \b$. 
The constant $\d$ is 
\begin{equation} \label{qui.02}
\d = 1/B, \quad 
B = C' L_{123}(a_2) \max \big\{ 1/\d_1, 1+A, (1+A) L_{123}(a_2) M_{123}(a_2-\mu) \big\}
\end{equation}
where $M_{123} = M_1 + M_2 + M_3$ 
and $C'$ is a constant depending on $a_1, a_2, \a, \b$.  
\\[2mm]
{\bf (Higher regularity)} Moreover, let $c > 0$
and assume that \eqref{Phi sec} holds for all $a \in [0, a_2 + c - \mu]$,
$\Psi(v)$ maps $F_\infty$ to $E_{a_2 + c}$, 
and \eqref{tame in NM} holds for all $a \in [a_1, a_2 + c]$. 
If $g$ satisfies \eqref{2705.1} and, in addition, $g \in F_{\b+c}$ with
\begin{equation} \label{0406.1}
\sum_{j=0}^\infty \| R_j g \|_{F_{\b+c}}^2 \leq A_c^2 \| g \|_{F_{\b+c}}^2 
\end{equation}
for some $A_c$, then the solution $u$ belongs to $E_{\a + c}$, 
with 
\begin{equation} \label{0211.10}	
\| u \|_{E_{\a+c}} \leq C_c \big\{ \mG_1 (1+A) \| g \|_{F_\b} 
+ \mG_2 (1+A_c) \| g \|_{F_{\b+c}} \big\} 
\end{equation}
where 
\begin{align} 
\mG_1 & := \tilde L_3 + \tilde L_{12} (\tilde L_3 \tilde M_{12} + L_{123}(a_2) \tilde M_3) 
(1 + z^{N}), 
\quad \mG_2 := \tilde L_{12} (1 + z^N), 
\label{qui.04}
\\ z & := L_{123}(a_1) M_{123}(0) + \tilde L_{12} \tilde M_{12},
\label{qui.03}
\end{align}
$\tilde L_{12} := \tilde L_1 + \tilde L_2$, 
$\tilde L_i := L_i(a_2+c)$, $i = 1,2,3$; 
$\tilde M_{12} := \tilde M_1 + \tilde M_2$, 
$\tilde M_i := M_i(a_2 + c - \mu)$, $i = 1,2,3$;
$N$ is a positive integer depending on $c,a_1,\a,\b$; 
and $C_c$ depends on $a_1, a_2, \a, \b, c$.
\end{theorem}

This theorem is proved in \cite{Baldi-Haus} using an iterative scheme similar to \cite{Hormander1976}. The main advantage with respect to the Nash-Moser implicit function theorems as presented in \cite{Zehnder,BBP} is the optimal regularity of the solution $u$ in terms of the datum $g$ (see \eqref{qui.01}, \eqref{0211.10}).
Theorem \ref{thm:NMH} has the advantage of making explicit all the constants (unlike \cite{Hormander1976}), which is necessary to deduce  the quantitative Theorem \ref{thm:nonlinear transport}.

% \footnotesize  \bibliographystyle{abbrv}

\begin{footnotesize}

\end{footnotesize}

\medskip

\noindent
{\sc Pietro Baldi}, Dipartimento di Matematica e Applicazioni ``R. Caccioppoli",
Universit\`a di Napoli Federico II, Via Cintia, Monte S. Angelo, 
80126, Napoli, Italy. \emph{E-mail:  {\tt pietro.baldi@unina.it}}

\medskip 

\noindent
{\sc Massimiliano Berti}, SISSA,  Via Bonomea 265, 34136, Trieste, Italy. 
\emph{E-mail: {\tt berti@sissa.it}}
 
\medskip 
 
\noindent
{\sc Emanuele Haus}, Dipartimento di Matematica e Applicazioni ``R. Caccioppoli",
Universit\`a di Napoli Fede\-rico II,  Via Cintia, Monte S. Angelo, 
80126, Napoli, Italy. \emph{E-mail:  {\tt emanuele.haus@unina.it}}
 
\medskip 
 
\noindent
{\sc Riccardo Montalto},  University of Z\"urich, Winterthurerstrasse 190,
CH-8057, Z\"urich, Switzerland. \emph{E-mail: {\tt riccardo.montalto@math.uzh.ch}} 

\end{document}